\documentclass[11pt,reqno]{amsart}
\usepackage{enumerate}
\usepackage{mhequ}
\usepackage{soul}
\usepackage[margin=1.2in]{geometry}
\usepackage{amssymb,url,color, booktabs,nccmath}
\usepackage{mathrsfs}
\usepackage{url}
\usepackage{enumitem}
\usepackage{graphicx}
\usepackage{tikz}
\usetikzlibrary{shapes,snakes}
\usetikzlibrary{calc}
\usetikzlibrary{decorations.shapes}
\usepackage[colorinlistoftodos,prependcaption,textsize=tiny]{todonotes}
\usepackage{todonotes}
\usepackage{bbold}
\usepackage{esint}
\usepackage{calc}
\usepackage{amsmath,amsthm,amssymb}
\usepackage{amsfonts} 
\usepackage[all]{xy}

\usepackage[backend=biber,style=alphabetic,maxnames=200,maxalphanames=6,giveninits]{biblatex}
\makeatletter

\usepackage{color}
\usepackage[colorlinks=true]{hyperref}
\hypersetup{
    linkcolor=cyan,          
    citecolor=orange,        
    filecolor=cyan,      
    urlcolor=cyan
}

\definecolor{darkergreen}{rgb}{0.0, 0.5, 0.0}

\numberwithin{equation}{section}

\newcommand{\be}{\begin{eqnarray}}
\newcommand{\ee}{\end{eqnarray}}
\newcommand{\ce}{\begin{eqnarray*}}
\newcommand{\de}{\end{eqnarray*}}
\newtheorem{theorem}{Theorem}[section]
\newtheorem{lemma}[theorem]{Lemma}

\newtheorem{example}[theorem]{Example}
\newtheorem{remark}[theorem]{Remark}
\newtheorem{definition}[theorem]{Definition}
\newtheorem{proposition}[theorem]{Proposition}
\newtheorem{Examples}[theorem]{Example}
\newtheorem{corollary}[theorem]{Corollary}
\newtheorem{assumption}[theorem]{Assumption}

\newenvironment{nouppercase}{%
  \renewcommand{\uppercasenonmath}[1]{}}{}
\newif\ifshowrevisioncolors
\showrevisioncolorsfalse
\ifshowrevisioncolors
  \newcommand{\rmb}[1]{\textcolor{blue}{#1}}
\else
  \newcommand{\rmb}[1]{#1}
\fi

\def\eps{\varepsilon}

\def\d{\hh}

\def\[{{\Big[}}
\def\]{{\Big]}}
\def\<{{\langle}}
\def\>{{\rangle}}
\def\({{\Big(}}
\def\){{\Big)}}

\def\bx{{\mathbf{x}}}

\def\min{{\mathord{{\rm min}}}}

\def\={&\!\!=\!\!&}

\newcommand{\tili}{\tilde{\imath}}

\newcommand{\hE}{\mathbb{E}}

\newcommand{\hg}{\mathbb{g}}

\newcommand{\hS}{\mathbb{S}}

\newcommand{\tn}{{\widetilde\nabla}}

\newcommand{\cD}{\mathcal{D}}
\newcommand{\cH}{\mathcal{H}}
\newcommand{\sd}{\hS^{d-1}}

\def\cS{\mathcal S}

\newcommand{\sym}{\text{sym}}
\newcommand{\anti}{\text{anti}}

\def\cD{{\mathcal D}}

\def\cH{{\mathcal H}}

\def\cS{{\mathcal S}}

\def\1{{\mathbb{1}}}

\def\geq{\geqslant}
\def\leq{\leqslant}
\def\ge{\geqslant}
\def\le{\leqslant}

\def\div{\mathord{{\rm div}}}

\def\eps{\varepsilon}

\def\d{\hh}

\def\[{{\Big[}}
\def\]{{\Big]}}
\def\<{{\langle}}
\def\>{{\rangle}}
\def\({{\Big(}}
\def\){{\Big)}}

\def\bx{{\mathbf{x}}}

\def\min{{\mathord{{\rm min}}}}

\def\={&\!\!=\!\!&}
\def\bt{\begin{theorem}}
\def\et{\end{theorem}}
\def\bl{\begin{lemma}}
\def\el{\end{lemma}}
\def\br{\begin{remark}}
\def\er{\end{remark}}
\def\bx{\begin{Examples}}
\def\ex{\end{Examples}}
\def\bd{\begin{definition}}
\def\ed{\end{definition}}
\def\bp{\begin{proposition}}
\def\ep{\end{proposition}}
\def\bc{\begin{corollary}}
\def\ec{\end{corollary}}

\def\geq{\geqslant}
\def\leq{\leqslant}
\def\ge{\geqslant}
\def\le{\leqslant}

\def\div{\mathord{{\rm div}}}

 \def\R{\mathbb R}
 \def\R{\mathbb R}    
\def\N{\mathbb N}  
\def\dd{\hh}   
\def\<{\langle} \def\>{\rangle}

\def\dd{{a}}

\allowdisplaybreaks

\usepackage[most]{tcolorbox}

\usepackage[most]{tcolorbox}

\newtcolorbox[auto counter, number within=section]{mathblock}[2][]{
  colback=gray!5,
  colframe=black!60,
  title={Approximation~\thetcbcounter},
  label=#1,
  breakable
}

\tikzset{
        dot/.style={circle,fill=black,inner sep=0pt, outer sep=0.7pt, minimum size=1mm},
        Phi/.style={white!40!red,thick,snake=coil,segment amplitude=0.6pt, segment length=2pt},
         Z/.style={black!40!green,thick,snake=coil,segment amplitude=0.6pt, segment length=2pt},
        C/.style={thick,black!20!blue},
          Cr/.style={thick,black!20!red},
            Cg/.style={thick,black!20!green},
       }

       \usepackage[most]{tcolorbox}

\newcommand{\Do}{\R^{d}}
\newcommand{\G}{\R^{2d}}
\newcommand\supp{\operatorname{Supp}}
\newcommand\loc{{\operatorname{loc}}}
\renewcommand{\d}{\partial}

\renewcommand{\dd}{{\,\rm d}}

\begin{document}

\title[Regularised Fluctuating Landau Equations]{\Large The Homogeneous Landau Equation with Regularised Thermal Noise}

\author[H.~Duong]{\large Manh Hong Duong}
\author[Z.~He]{\large Zihui He}
\address[H.~Duong]
{School of Mathematics, University of Birmingham, UK}
\email{h.duong@bham.ac.uk}
\address[Z.~He]
{Fakult\"at f\"ur Mathematik, Universit\"at Bielefeld, Postfach 100131, 33501 Bielefeld, Germany}
\email{zihui.he@uni-bielefeld.de}
\author[Z. Wu]{\large Zhengyan Wu}
\address[Z. Wu]{Department of Mathematics, Technische Universit\"at M\"unchen, Boltzmannstr. 3, 85748 Garching, Germany}
\email{wuzh@cit.tum.de}

\begin{abstract}
We introduce and analyze a fluctuating homogeneous Landau
equation with regularised thermal noise.  The model is motivated by the nonlocal gradient
flow structure of the deterministic Landau equation, the
fluctuation--dissipation principle, and the covariance of the martingale
fluctuations of a Kac-like conservative Landau particle system.  The
noise is written in Landau-divergence form, is antisymmetric in the pair of
velocities, and is interpreted in the Stratonovich sense after introducing a velocity correlation.  To handle the vacuum singularity of
the square-root mobility and the nonlocal Stratonovich-to-It\^o correction, we
replace the mobility by a regular coefficient.

For moderately soft potentials, we prove the existence of probabilistic weak
solutions to the regularised fluctuating homogeneous Landau equation.  The proof is based on a
three-level approximation scheme combining Galerkin approximations, coefficient
regularisations, artificial diffusion, and compactness in both $L^2$ and
$L^1$ frameworks.  The solutions satisfy mass conservation, an energy
inequality, and the entropy dissipation estimate.  Finally, for a
special class of admissible noise bases satisfying a tangential
divergence-free condition, we obtain a refined entropy inequality in which the
expected entropy is non-increasing relative to the initial entropy.
\end{abstract}

\subjclass[2020]{60H15;35Q20;35Q84}
\keywords{}

\date{\today}

\begin{nouppercase}
\maketitle
\end{nouppercase}


\setcounter{tocdepth}{1}
\tableofcontents

%
%


\section{Introduction}\label{sec:intro}

Kinetic theory provides a framework for describing the evolution of large systems of interacting particles through nonlinear kinetic partial differential equations (PDEs). At the deterministic level, the passage from a particle system to its kinetic equation is typically governed by a law of large numbers: as the number of particles tends to infinity, the empirical measure converges to a deterministic density solving a nonlinear kinetic PDE. While such equations successfully capture the average behavior of the underlying particle system, they do not account for its fluctuations. At finer scales, one expects to observe random fluctuations around the deterministic limit. Although these fluctuations are typically small, they may have a significant impact on the dynamics, for instance through the emergence of metastable phenomena. Fluctuating kinetic equations are stochastic partial differential equations (SPDEs) designed to retain this next-order information while preserving the dissipative structure of the underlying kinetic model. They provide an effective macroscopic description of the fluctuations arising from the underlying interacting particle system. 

The Landau equation provides a natural and particularly challenging test
case for fluctuating kinetic theory.  The goal of the present work is to
introduce Dean--Kawasaki-type thermal noise into a regularised homogeneous
Landau equation in a manner compatible both with the fluctuation--dissipation
structure of the Landau operator and with the martingale fluctuations of an
underlying stochastic mean-field particle system.  In contrast with local
Dean--Kawasaki models, Landau collisions couple two different velocities;
hence both the deterministic mobility and the corresponding fluctuation
covariance are intrinsically nonlocal.  Within this framework, we construct
probabilistic weak solutions to the resulting fluctuating Landau equation.

We start from the homogeneous Landau equation
\begin{equation}
    \label{landau}
    \left\{
    \begin{aligned}
    &\d_t f=Q_{\sf L}(f,f)\\
    &Q_{\sf L}(f,f)=\nabla_v \cdot\int_{\R^d}A(|v-v_*|)\Pi_{(v-v_*)^\perp}\big(f_*\nabla_v f-f\nabla_{v_*} f_*\big)\dd v_*.
    \end{aligned}
    \right.
\end{equation}
We write $f=f(v)$ and $f_*=f(v_*)$.
The orthogonal projection
$\Pi_{(v-v_*)^\perp}$ is given by 
\begin{equation*}
\label{projection}
  \Pi_{(v-v_*)^\perp}=\operatorname{Id}-\frac{(v-v_*)\otimes (v-v_*)}{|v-v_*|^2}.
\end{equation*}
In this paper, we consider the moderately soft kinetic kernel of the following form
\begin{equation}
    \label{kernel:A}
A(|v-v_*|)=|v-v_*|^{2+\gamma},\quad \gamma\in(-2,0).
\end{equation}
In \eqref{kernel:A}, $\gamma$ is a physical parameter, where $\gamma>0$ corresponds to the so-called \emph{hard} potentials, $\gamma=0$ to  the \emph{Maxwellian} potential, and $\gamma<0$ to the \emph{soft} potential.
With a further classification, $-2\leq \gamma <0$ is known as the \emph{moderately soft} potentials, and $-d\leq \gamma<-2$ as the \emph{very soft potentials}. In this paper, we consider the moderately soft potential cases and exclude the case $\gamma=-2$ due to the technical difficulties explained in Remark \ref{rmk:-2} below.

To recall the physical origin of the model, the Landau equation describes
collisional dynamics in a spatially homogeneous gas or plasma in the
grazing-collision limit, in which the Boltzmann collision kernel concentrates
on small deflection angles; see, for example, \cite{Villani2002,AV04}.  The
Coulomb kernel ($\gamma=-d$) has the direct plasma interpretation.  By contrast, the
moderately soft potentials considered in this paper exclude the Coulomb case
and correspond to power-law interactions in gases within the
grazing-collision approximation.  The integration over $v_*$ in
\eqref{landau} records collisions between particles of different velocities
and is the source of the difficult nonlocality of the Landau diffusion
matrix.


The Landau equation \eqref{landau} associated with a dissipated Boltzmann
entropy
\begin{equation*}
\cH(f)=\int_{\R^d} f\log f\dd v,    
\end{equation*}
that is 
\begin{align*}
    \frac{\dd }{\dd t}\cH(f_t)=-\cD(f_t)
    \quad\text{where}\quad 
\cD(f)=\frac12\int_{\R^d\times\R^d}Aff_*|\tn \log f|^2\ge0.
\end{align*}
We define the Landau difference gradient
\begin{align}
\label{def:nabla-L}
    \widetilde\nabla f= \Pi_{(v-v_*)^\perp}\big(\nabla_vf-\nabla_{v_*}f_*\big).
\end{align}
Let $G:\G\to\R^d$. We write $G=G(v,v_*)$ and $G_*=G(v_*,v)$. Then the corresponding divergence
$\widetilde\nabla\cdot$ is given by 
\begin{equation}
\label{landau-div}
    \widetilde\nabla\cdot G=\nabla_v\cdot \int_{\R^d}\Pi_{(v-v_*)^\perp} (G-G_*)\dd v_*
\end{equation}
via the integration by parts formula
\begin{align*}
\int_{\R^d\times\R^d}G\cdot \widetilde\nabla f\dd v_*\dd v=- \int_{\R^d}\widetilde\nabla \cdot G  f\dd v.
\end{align*}
With this notation, the Landau equation can be written in the Onsager form
\begin{equation}
\label{Landau-OS}
 \d_t f=\frac12 \tn\cdot\big( Aff_*\tn \log f\big).
\end{equation}
A gradient-flow structure for \eqref{landau} was established in \cite{CDDW24} via a variational characterisation
\begin{align*}
\d_t f=-K_f^L\dd\cH(f),  \quad K_f^L\varphi(v)=-\frac12 \tn\cdot\big(Aff_*\tn\varphi\big).   
\end{align*}
The gradient flow structure is the starting point of the stochastic model considered below.
Formally, it says that the Landau equation is a gradient flow of the
Boltzmann entropy with respect to a nonlocal transport geometry on the space
of velocity distributions.  In this geometry, a tangent vector is not
represented by an ordinary velocity field in the variable $v$, but by a
pairwise flux on the collision space $(v,v_*)$.  The admissible directions are
generated by the Landau difference gradient
$\widetilde\nabla=\Pi_{(v-v_*)^\perp}(\nabla_v-\nabla_{v_*})$: only the
components tangential to the sphere orthogonal to the relative velocity
$v-v_*$ are active.  Thus the metric is degenerate in the radial collision
direction and measures the cost of moving mass through pairwise tangential
collisions.


The fluctuation--dissipation principle then prescribes the covariance of the
thermal noise.  In finite dimensions, for a gradient system
$\dot x=-M(x)\nabla E(x)$, the equilibrium fluctuation has noise covariance
proportional to $M(x)$, i.e. the noise amplitude is formally
$\sqrt{2M(x)}$.  In the present Landau geometry the corresponding mobility is
the pairwise operator $Aff_*$ acting through $\widetilde\nabla\cdot$.
Therefore, the fluctuating correction should be a conservative noise in
Landau-divergence form, with amplitude $\sqrt{Aff_*}$ in the pair variables.
This leads to the following formal fluctuating hydrodynamic ansatz:
\begin{equation}
\label{eq:ideal}
	\partial_tf=\frac12 \tn\cdot\big( Aff_*\tn \log f\big)-\frac12\tn\cdot\left(\sqrt{Aff_*}\dot{W}\right),  
\end{equation}
where $\dot{W}$ is a velocity--velocity--time white noise. Equation~\eqref{eq:ideal} is only formal, since the irregularity of the space--time white noise renders it ill-posed. In the terminology of singular SPDEs, \eqref{eq:ideal} belongs to the supercritical regime; see \cite{Hairer2014}. By the definition of $\tn\cdot$, only the antisymmetric part of $\dot{W}$ contributes. Hence, we replace $\dot{W}$ by a velocity--velocity--time Gaussian noise $\xi$, such that 
\begin{align}
\label{choice:xi}
\xi(t,v,v_*)=-\xi(t,v_*,v). 
\end{align}
Probabilistically, this is the continuum analogue of the
antisymmetric pair Brownian motions in the conservative Landau particle
system, see Subsection \ref{sec-1.1} for more details. 


In the following, we explain more precisely how the above fluctuating Landau equation can be viewed, at a formal level, as describing the fluctuations of an underlying interacting particle system whose mean-field limit is governed by the deterministic Landau equation.

\subsection{From a Kac-like Landau particle system to fluctuations}\label{sec-1.1}

Kac's original programme concerns the use of Kac's binary jump process
and its high-dimensional linear master equation to study the spatially
homogeneous Boltzmann equation.  Its central steps include propagation of
chaos, quantitative comparison with the nonlinear kinetic equation, and the
transfer of relaxation and entropy information from the particle system to
the Boltzmann equation; see \cite{MischlerMouhot2013}.  The diffusion process
used for the Landau equation below is therefore not ``Kac's process'', which
conventionally denotes the binary jump model for Boltzmann's equation.  We
call it a \emph{Kac-like Landau particle system}, following the terminology of
\cite{FournierGuillin2017}.

For the homogeneous Landau equation, the relevant law-of-large-numbers
problem is to prove that \eqref{landau} arises as the mean-field limit of a
conservative system of weakly interacting velocities.  Such particle
approximations and propagation-of-chaos results have been developed, for
different interaction regimes, in
\cite{Fournier2009,FournierHauray2016,FournierGuillin2017}.  The question of
the present paper begins after this deterministic limit has been identified:
we ask which continuum stochastic equation describes the next-order
martingale fluctuations.  Thus our fluctuation problem should be viewed as an
extension beyond the law-of-large-numbers part of Kac's programme, rather than
as part of the original programme itself.  If
\begin{equation*}
    \mu_t^N=\frac1N\sum_{i=1}^N\delta_{V_i(t)}
\end{equation*}
is the empirical measure of the $N$-particle system, then the deterministic
equation describes the formal limit
\begin{equation*}
    \mu_t^N\longrightarrow \bar{f}_t(v)\,\dd v .
\end{equation*}
The fluctuating equation should contain the information which remains at the
next scale, namely in the centred field
$\sqrt N(\mu_t^N-\bar{f}_t(v)\dd v)$.  At this level the decisive object is the
quadratic variation of the empirical-measure martingale.

We recall the formal particle model behind this comparison.  Put
\begin{equation*}
    a(z):=A(|z|)\Pi_{z^\perp},\qquad b(z):=\nabla_z\cdot a(z),
\end{equation*}
so that, for the standard power-law kernel $A(|z|)=|z|^{\gamma+2}$, one has
$b(z)=(1-d)|z|^\gamma z$.  Let
$(B_{ij})_{1\le i<j\le N}$ be independent $\R^d$-valued Brownian motions and
extend them antisymmetrically by setting
\begin{equation*}
    Z_{ij}=B_{ij}\quad (i<j),\qquad
    Z_{ji}=-B_{ij},\qquad Z_{ii}=0 .
\end{equation*}
We consider, at the formal level, the conservative mean-field system
\begin{equation}
\label{intro:particle-landau}
    \dd V_i
    =\frac{2}{N}\sum_{j=1}^N b(V_i-V_j)\,\dd t
    +\sqrt{\frac{2}{N}}\sum_{j=1}^N a(V_i-V_j)^{1/2}\,\dd Z_{ij},
    \qquad 1\le i\le N .
\end{equation}
Here $a(z)^{1/2}$ may be taken as
$\sqrt{A(|z|)}\Pi_{z^\perp}$. We refer to
	\eqref{intro:particle-landau} only as a Kac-like conservative Landau particle
	system.  Unlike Kac's binary jump process for the Boltzmann equation, the
	grazing-collision mechanism is encoded here by a diffusion in the relative
	velocities.  The antisymmetric Brownian family $Z_{ij}=-Z_{ji}$ is
	the key structural feature.  It makes the stochastic force exerted by particle
	$j$ on particle $i$ the negative of the force exerted by particle $i$ on
	particle $j$, and hence preserves the conservative collision structure at the
	particle level.  Early studies of the associated master equation go back to
	\cite{PrigogineBalescu1957,PrigogineBalescu1957b,KiesslingLancellotti2004},
	and mean-field convergence of related Kac-like Landau dynamics to the
	deterministic Landau equation is proved in
	\cite{FournierHauray2016,FournierGuillin2017,feng2025kacsprogramlandauequation}.
	For the present paper, this system
	is important not only because of its law-of-large-numbers limit, but also
	because the quadratic variation of its empirical-measure martingale identifies
	the covariance structure that the fluctuating Landau noise should reproduce.
	The important point is therefore not the particular choice of square
	root, but the antisymmetric pairing of the Brownian motions.  This is the
	particle-level origin of the antisymmetric collision noise in
	\eqref{eq:ideal}.

Testing the empirical measure against
$\phi\in C_c^\infty(\R^d)$ and applying It\^o's formula gives
\begin{equation}
\label{intro:empirical-ito}
    \dd\langle\mu_t^N,\phi\rangle
    =\mathcal L_N\phi(V(t))\,\dd t+\dd M_t^{N,\phi},
\end{equation}
where $\mathcal L_N\phi$ is the drift contribution.  The martingale part is
of order $N^{-1/2}$, and its quadratic variation is given by 
\begin{align}
\label{intro:empirical-bracket}
    \dd [M^{N,\phi}]_t
    &=
    \frac1N\int_{\R^d\times\R^d}
    A(|v-v_*|)\,
    \big|\Pi_{(v-v_*)^\perp}(\nabla\phi(v)-\nabla\phi(v_*))\big|^2
    \mu_t^N(\dd v)\mu_t^N(\dd v_*)\,\dd t .
\end{align}
Thus the fluctuation covariance is concentrated on pair variables
$(v,v_*)$ and sees only the tangential component of the difference gradient.
This is exactly the same geometry as in the deterministic Landau operator.

\rmb{We now pass from the empirical martingale to the fluctuation scale.  The
measure $\bar f_t(v)\dd v$ denotes the deterministic law-of-large-numbers limit
of $\mu_t^N$.  The fluctuation field is
\begin{equation*}
    \eta_t^N:=\sqrt N\big(\mu_t^N-\bar f_t(v)\dd v\big).
\end{equation*}
Therefore the martingale part seen by $\eta_t^N$ is
$\sqrt N M_t^{N,\phi}$, and its quadratic variation is
$N[M^{N,\phi}]_t$.  If, at the formal level, one replaces the empirical measure
in \eqref{intro:empirical-bracket} by its deterministic limit
$\bar f_t(v)\dd v$, then the density in time of the limiting quadratic
variation of the fluctuation martingale is}
\begin{equation}
\label{eq:intro-fluctuation-covariance}
    \rmb{
    \int_{\R^d\times\R^d}
    A(|v-v_*|) \bar{f}_t(v)\bar{f}_t(v_*)
    \big|\Pi_{(v-v_*)^\perp}(\nabla\phi(v)-\nabla\phi(v_*))\big|^2
    \dd v\dd v_* .}
\end{equation}
\rmb{Thus \eqref{eq:intro-fluctuation-covariance} should not be read as the
quadratic variation of the empirical measure itself.  Rather, after integrating
in time, it gives the limiting bracket of the martingale in the Gaussian
fluctuation field around the deterministic Landau solution $\bar f$.  This
distinction explains why the deterministic density $\bar f$ appears in
\eqref{eq:intro-fluctuation-covariance}.}

The guiding principle for the fluctuating Landau equation is that its noise
should reproduce the same quadratic variation when the equation is linearised
around a deterministic profile.  The conservative antisymmetric noise 
\begin{equation*}
    \rmb{
    -\frac12\widetilde\nabla\cdot\big(\sqrt{Aff_*}\,\xi\big)}
\end{equation*}
in \eqref{eq:ideal} has exactly this property: testing it against $\phi$
produces a martingale whose bracket is \eqref{eq:intro-fluctuation-covariance},
up to the normalisation fixed in the definition of
$\widetilde\nabla\cdot$.  In this precise sense, the stochastic term in our
SPDE is chosen as a Landau fluctuation correction: it is designed so that the
continuum SPDE and the Kac-like Landau particle system have the same covariance
structure at the central-limit scale.  The detailed Dean--Kawasaki-type
calculation is given in Appendix~\ref{app:formal-particle-derivation}; here we
record only the covariance-matching mechanism which motivates the model.  After
the regularisations introduced below, this covariance is mild enough for the
Gaussian fluctuation equation associated with the particle system to be
captured by the linearisation of the fluctuating Landau equation.

\subsection{New challenges and regularisation of the model}
The above fluctuation corrections, derived from the fluctuation--dissipation relation, fall within the framework of fluctuating hydrodynamics, a modern theory of nonequilibrium statistical mechanics that connects macroscopic fluctuation theory with large deviation principles. Fluctuating hydrodynamics encompasses a variety of effective SPDE models, and significant progress has been made in recent years in the mathematical analysis of these equations, including the Landau--Lifshitz--Navier--Stokes equations \cite{LL87,GHW23}, Dean--Kawasaki-type equations \cite{FG23,FG24}, fluctuating thin-film equations \cite{DareiotisGessGnannGruen2021,DareiotisGessGnannSauerbrey2026}, and fluctuating geometric PDEs \cite{KawasakiOhta1982,KatsoulakisKho2001}.

Since fluctuating hydrodynamic models are expected to be valid primarily in the mesoscopic regime, the singular space--time white noise is often replaced by a regularised noise that remains white in time while being correlated in space or velocity. The corresponding correlation length typically represents either the grid size of the numerical discretisation or the characteristic interaction distance of the underlying particle system. Moreover, the It\^o formulation of the noise is commonly replaced by either Klimontovich or Stratonovich noise; see \cite{ayala2025reversibilitycovariancecoarsegraininglangevin} for a detailed discussion. This modification is motivated by both physical and mathematical considerations.

From the perspective of nonequilibrium statistical mechanics, \cite{Ottinger} argues that Klimontovich noise formally preserves the Gibbs invariant measure of the dynamics. From the mathematical viewpoint, both Klimontovich and Stratonovich formulations provide additional stochastic coercivity. Since the corresponding correction terms differ only by a constant, we adopt the Stratonovich formulation throughout this paper, as this choice does not affect the final analytical results.

Consequently, following the above considerations, it is natural to consider the modified equation
\begin{equation}
\label{eq:ideal-1}
\partial_tf=\frac12 \tn\cdot\big( Aff_*\tn \log f\big)-\frac{\sqrt{\eps}}{2}\tn\cdot\left(\sqrt{Aff_*}\circ\xi_K\right).
\end{equation}
Here, $\eps\in(0,1)$ denotes the noise intensity, $K\in\mathbb{N}_+$ represents the $K$-dimensional cutoff. The noise is formally defined by
$$
\xi_K(t,v,v_*):=\sum_{k=1}^K\frac{\dd}{\dd t}{B}^k(t)g_k(v,v_*),
$$
where $(B^k)_{k\geq1}$ denotes a sequence of independent one-dimensional standard Brownian motions, and $(g_k)_{k\geq1}$ satisfies
\begin{align*}
g_k(v,v_*)=-g_k(v_*,v).
\end{align*}
This finite-dimensional truncation preserves the antisymmetric pairwise structure of the formal white noise and constitutes the noise model employed throughout the existence theory developed in this paper.

The remaining issue is the square-root mobility in \eqref{eq:ideal-1}.  The
Stratonovich formulation is useful because the Stratonovich-to-It\^o
correction provides stochastic coercivity; however, the same correction also
reveals the singularity of the square root near the vacuum region.  In the classical Dean--Kawasaki equation with correlated noise,
$$
\partial_t\rho=\Delta\rho-\nabla\cdot(\sqrt{\rho}\circ\xi_F),
$$
where $\rho$ depends only on $(x,t)$ and $\xi_F$ is white in time and correlated in space, the corresponding It\^o correction contains a singular term of the form
\begin{equation*}
    \Delta\log\rho.
\end{equation*}
The well-posedness theory developed by Fehrman and Gess \cite{FG24} handles this singularity through the notion of renormalised kinetic solutions, which effectively truncate the solution near the vacuum set. More precisely, for every $S\in C^{\infty}_c(\mathbb{R}_+)$ and $\varphi\in C^{\infty}_c(\mathbb{R}^d)$, the renormalised solution is characterised through an identity for $\dd\langle S(\rho(t)),\varphi\rangle$, in which the singular term is represented by
\begin{equation*}
-\int^t_0\int_{\mathbb{R}^d}\nabla(\varphi S'(\rho))\cdot\frac{1}{\rho}\nabla\rho\,\dd v\,\dd s.
\end{equation*}
Consequently, the truncation of $S'$ together with the available regularity of $\rho$, this singular term is well defined. 

For the Landau equation, the obstruction is substantially more nonlocal. As shown in Appendix~\ref{app:stratonovich-ito}, the correction terms involve not only the visible density $f(v)$ but also factors such as $f(v_*)$ and $f(w)$ that are hidden inside the collision integrals. At the formal square-root level, one should therefore expect singular factors arising simultaneously from $\sqrt{f(v)}$, $\sqrt{f(v_*)}$, and $\sqrt{f(w)}$. More precisely, for every $S\in C^{\infty}_c(\mathbb{R}_+)$ and $\varphi\in C^{\infty}_c(\mathbb{R}^{d})$, the most singular term in the renormalisation formula is given by
\begin{equation*}
\frac{1}{2}\int^t_0\int_{\mathbb{R}^{3d}}\varphi(v_*)S''(f_*)\sqrt{f_*}\sqrt{f_w}\nabla_{v_*} f_*\otimes \nabla_v \log f : F_4\,\dd v\, \dd v_*\,\dd w\, \dd s,
\end{equation*}
where $F_4$ is a regular function depending on the correlation structure of the noise; see Section~\ref{sec:pre-2} for details. Owing to the integration-by-parts formula with respect to the Landau gradient, the renormalisation procedure also generates terms involving the renormalisation of $f$ in the hidden variable $v_*$. Consequently, the classical renormalisation framework is unable to truncate the singular factor $\nabla_v\log f$. This phenomenon constitutes a new challenge within the framework of Dean--Kawasaki-type fluctuating hydrodynamic equations, arising from the nonlocal interaction structure of the noise. 

In the present paper, as the first step, we instead introduce a regularisation. The square
root in the stochastic flux is replaced by a smooth mobility $\sigma$, this leads to the regularised Stratonovich equation
\begin{equation}
\label{eq:sigma-strato-intro}
\partial_t f
=\frac12 \tn\cdot\big(Aff_*\tn \log f\big)
-\frac{\sqrt{\eps}}{2}\tn\cdot\left(A^{1/2}\sigma(f)\sigma(f_*)\circ\xi_K\right).
\end{equation}
We further remark that, when $\sigma$ is chosen as a sequence of approximations of the square-root function, a special choice of the noise basis allows us to establish a uniform entropy dissipation estimate with respect to this approximation. 

A rigorous solution theory for \eqref{eq:sigma-strato-intro} provides a foundation for verifying the coincidence of the fluctuations. More precisely, as a direction for future research, one may investigate the Gaussian fluctuations and large deviations of solutions to
\begin{equation}
\label{eq:sigma-strato-intro-ep}
\partial_t f^N
=\frac12 \tn\cdot\big(Af^Nf^N_*\tn \log f^N\big)
-\frac{1}{2\sqrt{N}}\tn\cdot\left(A^{1/2}\sigma_{n(N)}(f^N)\sigma_{n(N)}(f^N_*)\circ\xi_{K(N)}\right),
\end{equation}
as $N\to\infty$, under a scaling regime satisfying $(N,K(N),n(N))\to(\infty,\infty,\infty)$, where $\sigma_{n(N)}$ denotes a sequence of approximations to the square-root function. Informally, one expects the fluctuations of \eqref{eq:sigma-strato-intro-ep} to coincide with those of the Kac-like Landau particle system \eqref{intro:particle-landau}. In particular, establishing a large deviation principle for \eqref{eq:sigma-strato-intro-ep} would provide a continuum fluctuating hydrodynamic derivation of the rate function conjectured by macroscopic fluctuation theory. Therefore, a rigorous solution theory for \eqref{eq:sigma-strato-intro} constitutes a fundamental first step toward this program.


\subsection{Main results}

We work in the moderately soft potential regime
\begin{align*}
    A(|v-v_*|)=|v-v_*|^{\gamma+2},\quad \gamma\in(-2,0).
\end{align*}
We define
\begin{equation*}
    G_k(v,v_*)=\sqrt{A(|v-v_*|)}\Pi_{(v-v_*)^\perp}g_k(v,v_*),
\end{equation*} 
and, motivated by the discussion in the previous section, consider the following
regularised Stratonovich fluctuating Landau equation
\begin{equation*}
    \dd f
    =Q(f,f)\,\dd t
    -\frac{\sqrt{\eps}}{2}\sum_{k=1}^K
    \widetilde\nabla\cdot\big(G_k(v,v_*)\sigma(f)\sigma(f_*)\circ \dd B_t^k\big).
\end{equation*}
This equation can be written in the It\^o form as follows (see the  Stratonovich-to-It\^o computation  given in
Appendix~\ref{app:stratonovich-ito}) \begin{equation}
\label{eq-app:main}
\begin{aligned}
\partial_t f
&=\frac12\tn\cdot\big(A f f_*\tn \log f\big)-\frac{\sqrt{\eps}}{2}\tn\cdot(A^{1/2}  \sigma(f) \sigma( f_*)\xi_K)\\
&\quad+\frac{\eps}{2}\sum_{k= 1}^K\tn\cdot \Big(G_k(v,v_*) \sigma'(f) \sigma(f_{*}) \tn\cdot\big(G_k(v,w) \sigma(f) \sigma(f_{w})\big) \Big).
\end{aligned}
\end{equation}


We recall the following formula for the Boltzmann entropy \begin{equation}
 \label{def:H-main-boltzmann}
 \cH(f):=\int_{\R^d} f\log f\,\dd v,
 \end{equation}
and the corresponding Landau entropy dissipation
\begin{equation}
\label{def:D-h-sigma-main}
\cD(f):=\frac12\int_{\R^d\times\R^d}A f f_*\big|\tn \log f\big|^2\ge 0.
\end{equation}
The main result of this paper is the following theorem which establishes the existence of probabilistic weak solutions to the fluctuating Landau equation \eqref{eq-app:main}.
\begin{theorem}\label{main:thm}
Let $d\ge 2$. Let $T>0$ be arbitrarily fixed. Assume that the interaction kernel is given by
\eqref{def:cA}, that the stochastic coefficient $\sigma$ satisfies
Assumption~\ref{ass:sigma-R0-0}, and that the noise coefficients satisfy
Assumption~\ref{ass:G-k:app}. Let the initial datum satisfy
Assumption~\ref{ass:initial-main}.  
Then there  exists $\eps_0>0$ such that, for every $\eps\in(0,\eps_0)$,
there exists at least one probabilistic weak solution
$(f,(B^k)_{1\le k\le K})$ to \eqref{eq-app:main} with initial datum $f_0$ in
the sense of Definition~\ref{def:weak-sol:L1}.

Moreover, almost
surely for every $t\in[0,T]$, the mass, momentum and energy bounds hold,
\begin{gather*}
    \int_{\R^d} (1,v)f_t(v)\dd v=\int_{\R^d}(1,v) f_0(v)\dd v,\quad
    \int_v|v|^2  f_t(v)\dd v\le  \int_{\R^d}|v|^2  f_0(v)\dd v.
    \end{gather*}
\end{theorem}

{For the deterministic Landau equation, a fundamental property is that the Boltzmann entropy $\cH(f)$ is decreasing in time, with its decay characterised by the time integral of the entropy dissipation $\cD(f)$. For the fluctuating equation, the corresponding entropy estimate holds for almost every $t\in[0,T]$:
\begin{align}
\hE\big[\cH(f_t)\big]
    +\hE\Big[\int_0^t\cD(f_s)\,\dd s\Big]
    \le \hE\big[\cH(f_0)\big]+C(\|\sigma'\|_{L^\infty}).
    \label{eq: entropy estimate}
\end{align}}
The preceding entropy dissipation estimate is robust for general admissible correlated noises, but
it still contains an additive constant (the constant $C$ in the RHS of \eqref{eq: entropy estimate}) coming from the treatment of the
Stratonovich-to-It\^o correction.  A sharper entropy balance can be recovered
if the active noise modes are chosen more carefully.  More precisely, Appendix
\ref{app-sec:ONB-2} constructs a special orthonormal basis (ONB) $\{g_k\}$ of a subspace of $L^2(\G;\Do)$, which satisfies  the tangential divergence-free condition
\eqref{intro:div-free} below.  This additional cancellation removes the residual error
terms in the entropy computation. 

In the second main result, Theorem \ref{main:thm-2}, we establish
a refined entropy dissipation for a structure-preserving $\theta$-regularised equation, see equation \eqref{eq:appendix-1-intro} below. { In particular, we approximate not only the square root in the noise term by a smooth Lipschitz function $\theta$, which plays the role of $\sigma$ in Theorem \ref{main:thm}, but also 
\begin{align*}
    \text{the mobility $ff_*$ in the Landau collision operator by $\theta^2(f)\theta^2(f_*)$}.
\end{align*}}
More precisely, let $\theta\in C^1(\mathbb{R}_+,\mathbb{R}_+)$ be a regularisation coefficient for the square-root function satisfying
\begin{align*}
 \theta(r)\sim r\quad \text{for $r$ small, while $\theta(r)=\sqrt{r}$ away from $0$}.   
\end{align*}

Instead of the Boltzmann entropy and the full Landau dissipation of \eqref{eq-app:main}, we consider the corresponding $\theta$-entropy $\mathcal H_\theta$ and the associated dissipation $\mathcal D_\theta$
\begin{gather*}
\cH_\theta(f)=\int_{\R^d} h_\theta(f)\dd v,\qquad
h_\theta(s)=\int_0^s\log(\theta^2(r))\dd r,\\
\cD_{\theta}(f)
:=\frac12\int_{\R^d\times \R^d}A\theta^2(f)\theta^2(f_*)
\big|\tn \log(\theta^2(f))\big|^2\dd v\dd v_*\ge 0.
\end{gather*}
The $\theta$-entropy defined above is equivalent to the Boltzmann entropy up to a multiple of the $L^2$-energy, see Remark~\ref{rmk:boltzmann-entropy} below. 
Now we consider the following equation 
\begin{equation}
\label{eq:appendix-1-intro}
\begin{aligned}
\dd f
&=\frac12\tn\cdot\big(A\theta^2(f)\theta^2(f_*)\tn\log\theta^2(f)\big)\dd t
-\frac{\sqrt{\eps}}{2}\tn\cdot\big(A^{1/2}\theta(f)\theta(f_*)\dd\xi_K\big)\\
&\quad+\frac{\eps}{2}\sum_{k=1}^K\tn\cdot \Big(G_k(v,v_*)\theta'(f)\theta(f_*)
\tn\cdot\big(G_k(v,w)\theta(f)\theta(f_w)\big)\Big)\dd t .
\end{aligned}
\end{equation}
The second main result of the paper is the following theorem which provides a refined entropy dissipation for \eqref{eq:appendix-1-intro}. 
\begin{theorem}[Refined entropy dissipation]\label{main:thm-2}
Let the assumptions of Theorem~\ref{main:thm} hold. In addition, assume that the coefficient
$\theta$ is chosen as in Assumption~\ref{ass:7} and that the
noise basis is chosen as in Assumption~\ref{ass:G-k-2}. Equivalently, the active modes satisfy
\begin{equation}
\label{intro:div-free}
   (\nabla_v-\nabla_{v_*})\cdot \Pi_{(v-v_*)^\perp}g_k(v,v_*)=0,\qquad 1\le k\le K .
\end{equation}
Then there exists a probabilistic weak solution to \eqref{eq:appendix-1-intro} such that, for almost every $t\in[0,T]$,
\begin{equation}
\label{main:refined-entropy}
    \hE\big[\cH_\theta(f_t)\big]
    +\hE\Big[\int_0^t\cD_\theta(f_s)\,\dd s\Big]
    \le
    \hE\big[\cH_\theta(f_0)\big].
\end{equation}

Similarly to the deterministic Landau equation, for sufficiently regular solutions, the entropy inequality \eqref{main:refined-entropy} formally holds as an equality, with entropy dissipation arising solely from the explicitly identified dissipation term 
\begin{align*}
\hE\big[\cH_\theta(f_t)\big]
+\hE\Big[\int_0^t\cD_\theta(f_s)\,\dd s\Big]
=\hE\big[\cH_\theta(f_0)\big].
\end{align*}

\end{theorem}

\subsection{Analytic obstacles and the strategy of the proof}

The proof is organised into four steps.
\begin{itemize}
\item \textbf{Step 1.} We employ a Galerkin approximation to establish the existence of solutions to the fully regularised equation. This step is carried out in Section~\ref{sec:Galerkin}.

\item \textbf{Step 2.} We establish uniform estimates for the mass, energy, and entropy. This step is carried out in Section~\ref{sec:L-1:functional}.  

\item \textbf{Step 3.} We pass to the limit in the regularisation parameters within an $L^2$ framework. This step is carried out in Sections~\ref{sec:limit-L2}.

\item \textbf{Step 4.} We pass to the limit in the initial data approximation and the vanishing diffusion parameter within an $L^1$ framework. This step is carried out in Section~\ref{sec:limit-L1}.
\end{itemize}

An overview of the technical details underlying this approximation strategy is provided in Subsection~\ref{sec:pre-2}. 

Compared with the classical deterministic Landau equation, the principal new
source of difficulty in this work is the conservative thermal noise.  The
Landau operator is already nonlocal and singular, but its deterministic weak
formulation is well adapted to the symmetry of binary collisions and to the
entropy dissipation associated with the Landau difference-gradient geometry.
The stochastic flux has to be placed on the same collision space.  Thus it is
not a pointwise multiplicative perturbation in the single velocity variable
$v$, but a pairwise object depending on $(v,v_*)$.  Even after the square-root
amplitude is replaced by a smoother stochastic coefficient $\sigma$, the noise
remains coupled to the singular Landau kernel and to the difference-gradient
structure.  The corresponding Stratonovich-to-It\^o correction is therefore
nonlocal as well, and it contains an additional hidden velocity integration.
For this reason the regularised equation cannot be treated as a standard
semilinear SPDE with smooth multiplicative noise.

A first difficulty is the entropy estimate.  In the deterministic Landau
equation one formally tests the equation by $\log f$.  A rigorous
implementation often relies on approximation procedures that preserve lower
bounds comparable to Maxwellians; this prevents the negative part of
$\log f$ near vacuum from becoming uncontrollable and allows the Boltzmann
entropy to be bounded from below by mass and energy.  Such a comparison
principle is not available for the present SPDE.  In the approximation used
below, the deterministic Landau mobility and the stochastic coefficient are
kept separate.  In Step $1$, the Landau part is first written with a regularised mobility
$\mu_n$, whereas the stochastic flux is governed by a separate coefficient
$\sigma_n$.  Correspondingly, the entropy test function is regularised as
$L_n(\mu_n(f))$, whose limiting form is $\log f$.  This separation is useful
because the final deterministic Landau drift should recover the genuine
mobility $f$, while the stochastic coefficient remains subject to the
structural assumptions needed to handle the thermal noise.

The order of the estimates is important.  In Step $2$, we prove non-negativity and
mass conservation for the approximate solutions.  Only after non-negativity is
available,  we are able to use the quadratic weight and obtain the energy bound
\begin{equation*}
    \int_{\R^d}|v|^2 f_t(v)\,\dd v
    \lesssim
    \int_{\R^d}|v|^2 f_0(v)\,\dd v+1 .
\end{equation*}
Together with mass conservation, this gives a finite lower bound for the
regularised entropy and makes the entropy inequality meaningful.  The
Stratonovich-to-It\^o correction terms must then be combined with the
quadratic variation in precisely the form dictated by the fluctuation--
dissipation structure.  This is the point at which the particular choices of regularisations enter: the approximation has to be regular
enough for the It\^o formula, but it must still preserve the cancellations
which lead to the Landau entropy dissipation in the limit.

A further difference from the deterministic Landau theory appears when we use the weighted Fisher-information estimate of Desvillettes \cite{Des15}, recalled below as Lemma~\ref{lem:Des}.  In the deterministic setting the constant in this estimate, although exponential in the entropy, causes no additional difficulty: the entropy and the entropy dissipation are controlled by the usual entropy inequality.  In the present stochastic setting the estimates needed for compactness are taken after expectation.  Therefore the exponential dependence in Lemma~\ref{lem:Des} cannot be handled by the ordinary averaged entropy dissipation bound alone.  This is why we prove the exponential entropy-dissipation estimate in Proposition~\ref{lem:l-12:log}, namely
{for some constant $\lambda>0$,
 \begin{align*}
 \hE\exp\left[
 \operatorname{esssup}_{s\in[0,t]}\left(\lambda\cH(f_s)+\frac{\lambda}{2}\int_0^s\cD(f_r)\,\dd r\right)\right]
 \le C_\lambda\exp\left[\lambda\cH(f_0)+C\lambda\big(\|f_0\|_{L^1}^3+1\big)\right].
 \end{align*}}
It provides integrability of the exponential entropy factor, together with sufficiently high moments of the accumulated entropy dissipation, and is one of the genuinely stochastic ingredients of the paper.  Its proof uses the exponential supermartingale structure associated with the martingale part of the entropy balance, a probabilistic mechanism which has no counterpart in the classical deterministic Landau argument.

\subsection{Comments on the literature}
We briefly review some of the existing literature related to the present work. 

\textbf{Kinetic theory}
Kinetic theory describes large particle systems through one-particle
densities in phase space.  In law-of-large-numbers regimes the empirical
measure converges to a deterministic kinetic equation, while the next order is
expected to be a fluctuation field around this limit.  This particle-to-PDE
viewpoint, classical in mean-field kinetic theory since Braun--Hepp
\cite{BH77}, is the perspective behind our fluctuating Landau model.

The deterministic Landau equation is the analytic backbone of the present
work.  It arises
from grazing collisions, where many weak small-angle interactions combine into
a nonlinear velocity diffusion; see \cite{Des92,Vil98,AV04}.  The
Boltzmann-to-Landau grazing-collision limit has also been studied from the
gradient-flow viewpoint in \cite{carrillo2022boltzmann}.  Thus the Landau
operator is the averaged effect of conservative binary interactions, and our
noise is designed to represent the corresponding fluctuation of the empirical
collision current with the same pairwise covariance structure. For the spatially homogeneous equation, weak and
$\mathcal H$-solutions, uniqueness, and propagation of regularity under
additional assumptions have been studied extensively; see
\cite{Lio94c,Vil96,Vil98,ADVW00,Wu14,Des15}.  Recent work on Fisher
information \cite{GS25,GS25b,Vil25} further clarifies the regularising
mechanisms of the collision operator.  These results provide the deterministic
reference for our moment, entropy, and compactness estimates.

The inhomogeneous theory illustrates the intrinsic difficulty of Landau
dynamics: global renormalised solutions with a defect measure are known
\cite{Vil96}, whereas global smooth solutions for large data remain open.
Adding conservative thermal noise does not remove the singularity; it creates
additional Stratonovich--It\^o correction terms and makes compactness a
central issue even in the homogeneous setting.  The geometric side is equally
important.  The homogeneous Landau equation has a variational gradient-flow
structure \cite{CDDW24}, related to the grazing-collision limit
\cite{carrillo2022boltzmann}; inhomogeneous and delocalised Landau equations
also fit naturally into the GENERIC framework \cite{DH25a,DH25b,DGH25}.  This
is why our stochastic perturbation is written in conservative
Landau-divergence form with covariance given by the same pairwise mobility as
the entropy dissipation.

\textbf{Fluctuating hydrodynamics and macroscopic fluctuation theory}
Macroscopic fluctuation theory describes nonequilibrium fluctuations through
large-deviation principles for interacting particle systems and their
hydrodynamic limits; see \cite{BDGJL,Derrida}.  Fluctuating hydrodynamics
gives the complementary mesoscopic rule: after the deterministic limit, one
adds conservative random fluxes whose covariance is dictated by the
fluctuation--dissipation relation; see \cite{LL87,HS}.  The Dean--Kawasaki
equation \cite{D96,K98} is the canonical density-level example.  Our equation
follows the same principle at the kinetic level: the deterministic Landau
collision operator is perturbed by an antisymmetric conservative noise whose
quadratic variation matches the Landau fluctuation covariance.  The new
difficulty is that both mobility and covariance are nonlocal in velocity. A particularly relevant microscopic result is due to Bodineau, Gallagher,
Saint-Raymond, and Simonella~\cite{bodineau2023statistical}.  Starting from
deterministic hard-sphere dynamics in the Boltzmann--Grad limit, they prove,
on a short kinetic time interval, convergence of the centered empirical
measure to a Gaussian process governed by the fluctuating Boltzmann equation,
as well as a dynamical large-deviation principle under suitable regularity
assumptions. Their result provides a rigorous microscopic benchmark for the
fluctuating-kinetic picture described above. The present work is
complementary: it is posed directly at the homogeneous Landau
(grazing-collision) level and develops a weak-solution theory for a
regularised nonlinear fluctuating Landau equation.  In particular, we do not
derive this equation from hard-sphere dynamics or obtain it by taking a
grazing-collision limit in the fluctuation and large-deviation results of
\cite{bodineau2023statistical}(see also \cite{rezakhanlou1998large}); the connection lies in the conservative
binary-collision structure of the noise and its covariance.

\textbf{SPDEs with conservative noise} 
For stochastic conservation laws, the kinetic solution framework of
Debussche--Vovelle~\cite{DV10} and its extensions
\cite{GS17,FG19,DG20} provide robust tools for nonlinear conservative noise.
Our equation is also conservative, but the Landau drift and noise couple two
velocity variables through a nonlocal collision kernel, which is the source of
the compactness difficulties treated here.

The closest analytic predecessors are Dean--Kawasaki-type equations.
Fehrman--Gess proved well-posedness for correlated conservative noise
\cite{FG24} and studied small-noise large deviations in \cite{FG23}; see also
\cite{FG25,fehrman2025stochastic}.  Fluctuation expansions were obtained in
\cite{CF23,GWZ24}, and nonlocal or singular interactions were studied in
\cite{WWZ22,WZ24}.  Recent fluctuating kinetic models
\cite{FMJ25,HWZ25} are also related.  The present paper differs in that both
the drift and the noise come from the nonlocal Landau gradient-flow geometry,
which forces a structure-preserving regularisation and a passage from a
regularised $L^2$ theory to an $L^1$ probabilistic weak solution.

\subsection{Structure of the paper}

Section~\ref{sec:pre-1} collects the notation and the assumptions on the
kernel, mobility, noise basis, and initial datum, and
Section~\ref{sec:app-weak-formulation} gives the definition of probabilistic
weak solutions.  Section~\ref{sec:Galerkin} constructs the regularised
Galerkin solutions and proves the basic $L^2$ estimates.  Section~\ref{sec:L-1:functional}
derives the cut-off It\^o formula and the positivity, mass, energy, and
entropy estimates.  Section~\ref{sec:limit-L2} passes to the limit in the
$L^2$ approximation, and Section~\ref{sec:limit-L1} removes the artificial
diffusion and proves the main theorem. In Section~\ref{sec:refined-entropy}, we derive a refined entropy dissipation estimate for a special class of admissible noises and prove the corresponding entropy decay property. The appendices contain the formal
particle derivation of the noise covariance, the Stratonovich-to-It\^o
conversion, and the construction of admissible antisymmetric noise bases.

\subsection*{Acknowledgements}
M. H. D is funded by an EPSRC Standard Grant EP/Y008561/1. Z.~H. is funded by the Deutsche Forschungsgemeinschaft (DFG, German Research Foundation) - Project-ID 317210226 - SFB 1283. 
Z. W. is funded by the US Army Research Office, grant W911NF2310230.

The authors are grateful to Matthias Erbar and Daniel Heydecker for helpful discussions and comments.

\section{Preliminary}

In this section, we collect the notation, functional spaces, structural
assumptions, and elementary identities used throughout the paper.  The
velocity variable is denoted by $v\in\Do$, and we write
$v_*,w\in\R^d$ for independent copies of $v$.  We use the shorthand
\begin{align*}
f=f(v),\qquad f_*=f(v_*),\qquad f_w=f(w),
\end{align*}
and similarly for $\sigma(f)$, $\varsigma(f)$, and other scalar functions of
$f$.  Unless otherwise specified, all gradients and divergences are taken
with respect to the velocity variable $v$.  The Landau difference gradient is
denoted by
\begin{align*}
\tn \phi(v,v_*):=\Pi_{(v-v_*)^\perp}\big(\nabla_v\phi(v)-\nabla_{v_*}\phi(v_*)\big),
\end{align*}
and the collision space is $\G$.  We use the convention
\begin{align*}
\int_v=\int_{\R^d},\qquad
\int_{v,v_*}=\int_{\R^d\times\R^d},\qquad
\int_{v,v_*,w}=\int_{\R^{3d}} .
\end{align*}

We denote by $\langle v\rangle=(1+|v|^2)^{1/2}$ the Japanese bracket.  For
$p\in[1,\infty]$ and $k\ge0$, the weighted Lebesgue space $L^p_k(\R^d)$ is
defined by
\begin{align*}
L^p_k(\R^d)
:=\{f\in L^p(\R^d):\|f\|_{L^p_k}:=\|\langle v\rangle^k f\|_{L^p}<\infty\}.
\end{align*}
In particular,
\begin{align*}
L^1_2(\R^d)
=\left\{f\in L^1(\R^d):\int_{\R^d}\langle v\rangle^2|f(v)|\,\dd v<\infty\right\}.
\end{align*}
We write $L\log L(\R^d)$ for the set of non-negative functions
$f\in L^1(\R^d)$ such that $f|\log f|\in L^1(\R^d)$.  The Sobolev spaces
$W^{k,p}(\R^d)$ and $H^s(\R^d)=W^{s,2}(\R^d)$ are used with their standard
norms; $H^{-s}(\R^d)$ denotes the dual of $H^s(\R^d)$.  Local spaces are
denoted by $L^p_{\loc}(\R^d)$ and $H^s_{\loc}(\R^d)$.  The Schwartz space is
denoted by $\mathcal S(\R^d)$, and $C_c^\infty(\R^d)$ denotes smooth compactly
supported test functions.  The symbol $\langle\cdot,\cdot\rangle$ denotes
the $L^2$-inner product, or the dual pairing when one of the entries is a
distribution.

Let $X$ be a Banach space.  We use the standard Bochner spaces
$L^p(0,T;X)$ and $C([0,T];X)$.  For $\beta\in(0,1)$,
$C^\beta([0,T];X)$ denotes the H\"older space in time.  When a space is
endowed with its weak topology we write, for example,
$(L^2(0,T;L^2(\R^d)),w)$.  Strong convergence in local spaces always means
strong convergence on every compact velocity set.

The letters $C$ and $C_i$ denote positive constants which may change from
line to line.  Their dependence is indicated when relevant.  We write
$A\lesssim B$ if $A\le CB$ for such a constant $C$, and
$A\lesssim_\lambda B$ when the constant is allowed to depend on the
parameter $\lambda$.

\subsection{Assumptions}\label{sec:pre-1}

We collect here the assumptions in Theorem \ref{main:thm}: the interaction kernel, the
noise, the basis, the stochastic coefficient, the Landau mobility used in the
approximation levels, and the initial datum.  We recall the It\^o equation \eqref{eq-app:main} that the
Landau part is the genuine Landau collision operator, whereas the stochastic
flux keeps the regularised coefficient $\sigma$
\begin{equation}
    \label{SDE-1}
\begin{aligned}
\partial_t f
&=Q(f,f)-\frac{\sqrt{\eps}}{2}\tn\cdot(A^{1/2}  \sigma(f) \sigma( f_*)\xi_K)\\
&\quad+\frac{\eps}{2}\sum_{k= 1}^K\tn\cdot \Big(G_k(v,v_*) \sigma'(f) \sigma(f_{*}) \tn\cdot\big(G_k(v,w) \sigma(f) \sigma(f_{w})\big) \Big).
\end{aligned}
\end{equation}
Here
\begin{equation}
    \label{def:Q-sigma}
    Q(f,f):=\frac12\tn\cdot\big(A f f_*\tn \log f\big).
\end{equation}

\medskip
\noindent\textbf{The interaction kernel.}
Throughout the paper we work in the moderately soft potential regime
\begin{align}
\label{def:cA}
    A(|v-v_*|)=|v-v_*|^{\gamma+2},\quad \gamma\in(-2,0).
\end{align}

\medskip
\noindent\textbf{The noise and the basis.}
Let $K\in\N$ be fixed, and let $(B^k)_{1\le k\le K}$ be independent standard
real Brownian motions.  Formally, the driving noise is the finite-dimensional
Gaussian field
\begin{align*}
\xi_K(t,v,v_*):=\sum_{k=1}^K g_k(v,v_*)\,\dot B^k_t .
\end{align*}
For each basis function $g_k$ we define the corresponding noise vector field by
    \begin{gather}
    \label{def:G-k}
        G_k(v,v_*):=\sqrt{A(v-v_*)}\Pi_{(v-v_*)^\perp}g_k(v,v_*)\in\R^d.
    \end{gather}
    

We shall use the weighted spaces $L^p_k(\R^d)$ and the entropy space
$L\log L(\R^d)$ introduced above.  The assumptions on the basis and on the
resulting noise vector fields are as follows.
\begin{assumption}[Active noise modes]
\label{ass:G-k:app}
Let $K\in\N$ be arbitrarily fixed.  We assume that the active family
$\{g_k\}_{1\le k\le K}\subset\cS(\G;\R^d)$ is orthonormal in
$L^2(\G;\R^d)$ and antisymmetric, $g_k(v,v_*)=-g_k(v_*,v)$.  
If a full
ONB is needed, this finite family may be completed by arbitrary orthonormal
modes in the orthogonal complement; the analysis below uses only the first
$K$ active modes.  We assume that the vector fields $(G_k)_{1\le k\le K}$ defined by \eqref{def:G-k} satisfy
\begin{equation}
\label{G-bdd:app-0}
\begin{aligned}
\sum_{k=1}^K\|G_k\|_{W^{2,\infty}(\G)\cap W^{2,1}(\G)}
\le C_K<+\infty .
\end{aligned}
\end{equation}
Following from the anti-symmetry of $g_k$, we have $G_k(v,v_*)=-G_k(v_*,v)$.

We assume that
\begin{align}
\label{Gk:annual}
\cup_{1\le k\le K} \supp\big(G_k\big)\subset \{r\in \R_+\mid  K^{-1}\le r\le K\}\times \R^d\times \sd.
\end{align} 

\end{assumption}

\begin{remark}
\label{rmk:G-k}
One can choose the active family in Assumption~\ref{ass:G-k:app}, and then complete it to an ONB if desired, such that Assumption~\ref{ass:G-k:app} holds even for
the singular kernel \eqref{def:cA}; that is,
\begin{align*}
|v-v_*|^{-\alpha} g_k(v,v_*)\in L^\infty(\G),
\quad \forall k\in\N,\quad \forall \alpha>0.
\end{align*}
Indeed, by Appendix \ref{app-sec:ONB}, the basis functions $g_k$ can be chosen in the form
\begin{align*}
g_k(v,v_*)
=
g_{k_1}(|v-v_*|)
g_{k_2}\left(\frac{v-v_*}{|v-v_*|}\right)
g_{k_3}\left(\frac{v+v_*}{2}\right),
\end{align*}
where ${g_{k_1}}$, ${g_{k_2}}$, and ${g_{k_3}}$ are ONBs of $L^2(r^{d-1}\dd r)$, $L^2(S^{d-1})$, and $L^2(\R^d)$, respectively. It remains to choose ${g_{k_1}}$ so that
\begin{align}
\label{r:sim-0}
r^{-\alpha} g_{k_1}(r)\in L^\infty(\R_+),
\quad \forall k_1\in\N,\quad \forall \alpha>0.
\end{align}
Since $\cS(\R_+\setminus\{0\})$ is dense in $L^2(\R_+)$, we may choose an ONB $\{g_{k_1}\}$ with
\begin{align*}
g_{k_1}\in \cS(\R_+\setminus\{0\}),
\quad \forall k_1\in\N.
\end{align*}
Consequently, each $g_{k_1}$ vanishes near $r=0$ with sufficient regularity and decay, and therefore \eqref{r:sim-0} holds.

Hence, the compact support assumption \eqref{Gk:annual} is harmless for a fixed finite active family $\{g_k\}_{1\le k\le K}$ that
\begin{align}
\cup_{1\le k\le K} \supp\big(g_k\big)\subset \{r\in \R_+\mid  K^{-1}\le r\le K\}\times \R^d\times \sd.
\end{align}
Furthermore, the basis $(g_k)_{k\in\mathbb{N}}$ constructed in Appendix~\ref{app-sec:ONB} can be chosen so that its components are orthonormal in $L^2(\mathbb{R}^{2d})$. In contrast, this property does not hold for the basis constructed in Appendix~\ref{app-sec:ONB-2}.

\end{remark}

\medskip
\noindent\textbf{The stochastic noise coefficient.}
The coefficient $\sigma$ regularizes the square-root amplitude in the
stochastic conservative flux.  It is no longer used to regularize the
deterministic Landau mobility in the limiting equation.  The properties of
$\sigma$ used throughout the paper are collected below.
\begin{assumption}[Stochastic noise coefficient]
  \label{ass:sigma-R0-0}
Let $r_0\in(0,1)$ be fixed. Let $\sigma\in C^1(\R_+;\R_+)$ satisfy
\begin{equation}
\begin{gathered}
\sigma(r)=\sqrt r \quad r\in(r_0,\infty),\quad \sigma(0)=0,\\
0\le \sigma'(r)\le C_{r_0}\quad\text{and}\quad  0\le \sigma(r)\le\sqrt{r}\quad r\in[0,r_0],   
\end{gathered}
\end{equation}
where $C_{r_0}<+\infty$ is a fixed constant depending only on $r_0$ and chosen large enough so that the interpolation is possible.  In particular, the assumption is not tied to the value $1/(2\sqrt{r_0})$.  The estimates below use only the finiteness of $C_{r_0}$, together with the bounds displayed in the assumption.
\end{assumption}

\begin{proposition}[Elementary bounds for $\sigma$]
Let $\sigma$ satisfy Assumption~\ref{ass:sigma-R0-0}. Then
\begin{align*}
 0\le \sigma(r)\le C_{r_0} r,\qquad r\in[0,r_0].
\end{align*}
Moreover, $\sigma^2\in \operatorname{Lip}(\R_+)$. More precisely,
\begin{align*}
   0\le  (\sigma^2(r))'\le
   2C_{r_0}^2 r\mathbb{1}_{\{0\le r\le r_0\}}
   +\mathbb{1}_{\{ r\ge r_0\}}\quad\forall r\in\R_+.
\end{align*}
\end{proposition}

An admissible example of such a mobility coefficient is
\begin{equation}
    \label{exp:sigma-0}
\begin{aligned}
    \sigma (r):=\left\{
        \begin{aligned}
        \frac{3r}{2\sqrt{r_0}}-\frac{r^2}{2r_0\sqrt{r_0}},&\quad  r\in[0,r_0),\\
        \sqrt{r},&\quad r\in[r_0,+\infty).
        \end{aligned}
        \right.
\end{aligned}
\end{equation}
For this example one may take $C_{r_0}=3/(2\sqrt{r_0})$.  Thus the derivative bound in Assumption~\ref{ass:sigma-R0-0} is understood as a generic interpolation bound, rather than the derivative of the square-root branch at $r_0$.

\begin{definition}
We define $\varsigma:\R_+\to\R$ by $\varsigma(0)=0$ and for every 
$r\ge 0$,
\begin{equation}
\label{def:varsigma}
\varsigma(r):=\int_0^r [\sigma '(s)]^2\dd s.
\end{equation}   
\end{definition}

\begin{remark}
\label{rmk:sigma-f:L2}
Notice that
for all $r\ge0$, we have
\begin{align*}
 0\le  \sigma(r)\le \min\big(C_{r_0}r,\sqrt{r}\big)
 \quad \text{and}\quad
 0\le \varsigma(r)\le C_{r_0}^2 r.
\end{align*}
Hence, for every non-negative measurable $f$ for which the right-hand side
is finite, we have
\begin{gather*}
    \| \sigma(f)\|_{L^1}\le C_{r_0}\|f\|_{L^1},\quad  \| \sigma(f)\|_{L^2}\le \min\Big(C_{r_0}\|f\|_{L^2},\|f\|_{L^1}^{\frac12}\Big),\\
    \|\varsigma(f)\|_{L^1}\le C_{r_0}^2\|f\|_{L^1},\quad  \|\varsigma(f)\|_{L^2}\le \min\Big(C_{r_0}\|f\|_{L^2},C_{r_0}^2\|f\|_{L^1}^{\frac12}\Big).
\end{gather*}
\end{remark}

\medskip
\noindent\textbf{The initial datum.}
The main theorem is stated for non-negative initial data satisfying
\begin{equation}
\label{ass:initial-main}
f_0\in L^1_2(\R^d;\R_+)\quad\text{and}\quad \cH(f_0)<+\infty.
\end{equation}
We note that the bounded energy ensures that $\cH(f_0)>-\infty$, see \eqref{LlogL:minus} below. Hence, we have the initial value $f_0\in L\log L(\R^d;\R_+)$.

\subsection{Auxiliary coefficients}\label{sec:pre-3}

We record the auxiliary coefficients generated by the It\^o correction.
Recall that
\begin{gather*}
G_k(v,w):=\sqrt{A(v-w)}\Pi_{(v-w)^\perp}g_k(v,w)\in\R^d,
\qquad G_k(v,w)=-G_k(w,v).
\end{gather*}
For every $(v,v_*,w)\in(\R^d)^3$, define
\begin{equation}
    \label{def:F_i}
\begin{aligned}
&F_1:=F_1(v,v_*,w)= \sum_{k=1}^K G_k(v,v_*) \big(\nabla_v\cdot G_k(v,w)\big)\in \R^d,\\
&F_2:=F_2(v,v_*,w)= \sum_{k=1}^K G_k(v,w)\big(\nabla_v\cdot G_k(v,v_*)\big) \in \R^d,\\ 
&F_3:=F_3(v,v_*,w)= \sum_{k=1}^K\big(\nabla_v\cdot G_k(v,v_*)\big)\big(\nabla_v\cdot G_k(v,w) \big)\in \R,\\ 
&F_4:=F_4(v,v_*,w)= \sum_{k=1}^KG_k(v,v_*)\otimes   G_k(v,w) \in \R^{d\times d}.
\end{aligned}
\end{equation}
These coefficients will enter the rigorous cut-off It\^o formula in
Lemma~\ref{lem:cutoff-ito}.

\subsection{Definition of solutions}
\label{sec:app-weak-formulation}
With the notation introduced above, the It\^o equation \eqref{SDE-1} may be
written in divergence form as
     \begin{align*}
        \dd  f=&Q(f,f)\dd t-\frac{\sqrt{\eps}}{2}\sum_{k=1}^K\tn\cdot\big(G_k(v,v_*)\sigma(f)\sigma(f_*)\big)\dd B_k\\
        &+\eps\,\tn\cdot \Big(\int_{w}\sigma(f_*)\sigma(f_w)\Big(F_4 \nabla_v \varsigma(f)+\sigma(f)\sigma'(f)F_1\big)\Big)\dd t.
\end{align*}

We use the standard Landau coefficients
\begin{equation}
\label{relation:abc}
\begin{aligned}
 &a_{ij}(v-v_*):=\Big(\delta_{ij}-\frac{(v-v_*)_i(v-v_*)_j}{|v-v_*|^2}\Big)A(|v-v_*|),\\
 &
         b_i(v-v_*):=\sum_{j=1}^d\d_j a_{ij}(v-v_*),\\
         &c(v-v_*):=\sum_{i,j=1}^d\d^2_{ij} a_{ij}(v-v_*).
     \end{aligned}
        \end{equation}
     Notice that
   \begin{equation}
   \label{abc}
     \begin{aligned}
        &\operatorname{tr}(a_{ij})=(d-1)A(|v-v_*|),\quad b_i=-(d-1)\frac{(v-v_*)_i}{|v-v_*|^2}A(|v-v_*|),\\
        &\text{and}\quad c=-(d-1)\Big((d-2)\frac{A(|v-v_*|)}{|v-v_*|^2}+\frac{A'(|v-v_*|)}{|v-v_*|}\Big).
     \end{aligned}
     \end{equation}

     Notice that in the case of $A(|v-v_*|)=|v-v_*|^{2+\gamma}$, we have  
        \begin{equation}
   \label{abc:bounds:n}
     \begin{aligned}
        &|a|\lesssim |v-v_*|^{\gamma+2},\quad |b|\lesssim|v-v_*|^{\gamma+1},\quad |c|\lesssim |v-v_*|^{\gamma}.
     \end{aligned}
     \end{equation}
     For any $g\in L^1_\loc(\mathbb{R}^d)$, define
     $\bar a(g)=(\bar a_{ij})\in\R^{d\times d}$,
     $\bar b(g)=(\bar b_i)\in\R^{d}$, and $\bar c(g)\in\R$ by
     \begin{equation}
         \label{def:abc-bar}
     \begin{aligned}
         \bar a_{ij}(g)&:=g*_{v}a_{ij}=\int_{\R^d}a_{ij}(v-v_*)g(v_*)\dd v_*,\\
     \bar b_i(g)&:=g*_{v}b_i=\int_{\R^d}b_i(v-v_*)g(v_*)\dd v_*,\\
     \bar c(g)&:=g*_{v}c=\int_{\R^d}c(v-v_*)g (v_*)\dd v_*.
     \end{aligned}
     \end{equation}

The Landau collision operator $Q(f,f)$ has the following $(\bar a,\bar b)$ (divergence form) and $(\bar a,\bar c)$ (non-divergence form) representations
\begin{align}
\label{form:abc}
  Q(f,f)&=  \nabla_v\cdot\big(\bar a(f)\nabla_v f-\bar b(f) f\big)\\
  &=\sum_{i,j=1}^d\bar a_{ij}(f)\d_{v_iv_j}f-\bar c(f) f.
\end{align}

We now formulate the notion of solution used for the limiting equation.
\begin{definition}[Probabilistic weak solution of \eqref{SDE-1}]
    \label{def:weak-sol:L1}
Let $f_0\in L^1_2(\R^d;\R_+)$ such that $\cH(f_0)<+\infty$. A pair
$(f,(B^k)_{1\le k\le K})$ is called a probabilistic weak solution to
\eqref{SDE-1} on $[0,T]$ with initial datum $f_0$ if there exists a
stochastic basis
$(\Omega,\mathcal F,(\mathcal F_t)_{t\in[0,T]},\mathbb P)$ satisfying the
usual conditions such that $(B^k)_{1\le k\le K}$ are independent standard
$(\mathcal F_t)$-Brownian motions, and $f$ is a non-negative
progressively measurable process satisfying, almost surely,
\begin{align*}
f\in L^\infty(0,T;L^1_2(\R^d;\R_+))\quad\text{and}\quad |\cH(f_t)|<+\infty\quad \text{for a.e. }t\in[0,T].
\end{align*}
The weak formulation below is required for the continuous weak representative of $f$; entropy estimates are understood for the a.e. time representative supplied by the entropy-dissipation bound.
Moreover, for every $\phi\in C_c^\infty(\R^d)$, the process
$t\mapsto \int_{\R^d}f_t(v)\phi(v)\,\dd v$ admits a continuous
$(\mathcal F_t)$-adapted version and, for this version, the following
identity holds almost surely for every $t\in[0,T]$:
\begin{equation}
\label{weak:app-0}
     \begin{aligned}
&\int_v f_t(v)\phi(v)\,\dd v-\int_v f_0(v)\phi(v)\,\dd v\\
={}&\sum_{i,j=1}^d\int_0^t\int_{v,v_*}
ff_*a_{ij}(v-v_*)\d_{v_iv_j}\phi(v)
\,\dd v\dd v_*\dd s\\
&+\sum_{i=1}^d\int_0^t\int_{v,v_*}
ff_*b_i(v-v_*)
\Big(\d_{v_i}\phi(v)-\d_{v_i}\phi(v_*)\Big)
\,\dd v\dd v_*\dd s\\
&+\frac{\sqrt{\eps}}{2}\sum_{k=1}^K\int_0^t\int_{v,v_*}
\tn\phi(v)\cdot G_k(v,v_*)\sigma(f)\sigma(f_*)
\,\dd v\dd v_*\,\dd B^k_s\\
&-\eps\int_0^t\int_{v,v_*,w}
\big(\tn\phi(v)\cdot F_1(v,v_*,w)\big)
\sigma(f_*)\sigma(f_w)\sigma(f)\sigma'(f)
\,\dd v\dd v_*\dd w\dd s\\
&+\eps\int_0^t\int_{v,v_*,w}
\nabla_v\cdot\big(\tn\phi(v)F_4(v,v_*,w)\big)
\sigma(f_*)\sigma(f_w)\varsigma(f)
\,\dd v\dd v_*\dd w\dd s .
\end{aligned}
\end{equation}
\end{definition}

Here and below, in the integrals over $[0,t]\times(\R^d)^m$, we use the
shorthand
\begin{align*}
f=f_s(v),\qquad f_*=f_s(v_*),\qquad f_w=f_s(w).
\end{align*}
All deterministic terms in \eqref{weak:app-0} are understood as Lebesgue
integrals.  They are finite under the above assumptions.  Indeed,
the $a$-term has growth at most quadratic, while in the $b$-term the
difference
$\d_{v_i}\phi(v)-\d_{v_i}\phi(v_*)$ cancels the singularity near
$v=v_*$; hence \eqref{abc:bounds:n}, $\gamma\in(-2,0)$, and
$f\in L^\infty_tL^1_2$ give the required integrability.  The two correction
terms are controlled by the bounds on $F_1,F_4$ in
Assumption~\ref{ass:G-k:app} together with Remark~\ref{rmk:sigma-f:L2}.
The stochastic integral is an It\^o integral;
indeed, for each $k$ and each compactly supported $\phi$,
\begin{align*}
\int_0^T\left|
\int_{v,v_*}\tn\phi(v)\cdot G_k(v,v_*)\sigma(f)\sigma(f_*)
\,\dd v\dd v_*
\right|^2\dd s
\le C_{\phi,K,r_0}\,T\,\|f\|_{L^\infty_tL^1_v}^4<+\infty
\end{align*}
almost surely.

\subsection{Strategy for showing Theorem \ref{main:thm}}
\label{sec:pre-2}

We show Theorem \ref{main:thm} in the following four steps:

\noindent\textbf{Step $1$: Existence of a diffusion approximation It\^o equation.} We show Step $1$ in Section \ref{sec:Galerkin}.

We approximate the interaction kernel $A$, the stochastic coefficient $\sigma$ by smooth bounded $A_n$ and $\sigma_n$ satisfying Assumptions \ref{ass:A} and \ref{ass:sigma-R0}, respectively. In particular, we assume 
\begin{align*}
\sup_{n}\|\sigma_n'\|_{L^\infty}\lesssim \|\sigma'\|_{L^\infty}<+\infty. \end{align*}

In particular, we approximate the Landau mobility $ff_*$ by  $\mu_n^2(f)\mu_n^2(f_*)$ in the drift, where $\mu_n\in \operatorname{Lip}(\R)$ approximates $\sqrt \cdot$ and satisfies Assumption \ref{ass:sigma-R0}. We approximate $\log f$ by $L_n(f)$, where $L_n$ is bounded around $0$ and satisfies Assumption \ref{ass:L-s0}. The detailed constructions of $A_n$, $\sigma_n$, $\mu_n$ and $L_n$ can be founded in Section \ref{sec:limit-L2}. 

 Let $n\in\N$ be fixed. For any fixed  $\alpha\in(0,1)$ , we show the existence of probability weak solutions (defined as in Definition \ref{def:weak-1} below)  in the $L^2$-framework for the following initial value problem
\begin{equation}
\left\{
    \begin{aligned}
        &\dd  f=\alpha \Delta f \dd t+Q_{\mu_n,L_n}(f,f)\dd t\\
        &\qquad-\frac{\sqrt{\eps}}{2}\tn\cdot\left(A_n^{1/2} \sigma_n (f)\sigma_n (f_*) \dd \xi_K\right)\\
        &\qquad+\frac{\eps}{2}\sum_{k= 1}^K\tn\cdot \Big(G_{k,n}(v,v_*) \sigma_n'(f) \sigma_n(f_{*}) \tn\cdot\big(G_{k,n}(v,w) \sigma_n(f) \sigma_n(f_{w})\big) \Big)\\
        &f|_{t=0}=f_0^\alpha,
    \end{aligned}
    \right.
\end{equation}
where $G_{k,n}$ is defined as in \eqref{def:G-k} corresponding to $A_n$, and the initial value $f_0$ is approximated by $f_0^\alpha\in L^2\cap L^1_2\cap L\log L(\R^d;\R_+)$ given as in Section \ref{sec:limit-L1}.

In the above, we define the approximated $(\mu,L)$-Landau collision operator 
\begin{equation}
\label{def:Q-mu-L-pre}
Q_{\mu,L}(f,f):=\frac12\tn\cdot\big(A\mu^2(f)\mu^2(f_*)\tn L(\mu^2(f))\big).
\end{equation}

We denote the approximation solution by $(f_n^\alpha,B_n^\alpha)$ such that
\begin{align*}
f_{n}^\alpha\in L^2\big(\Omega_n^\alpha;L^\infty(0,T;L^2(\R^d))\big)
\cap L^2\big(\Omega_n^\alpha;L^2(0,T;H^1(\R^d))\big).
\end{align*}

\noindent\textbf{Step $2$: Properties of approximation solutions} We show Step $2$ in Section \ref{sec:L-1:functional}.

Let $n\in\N$ and $\alpha\in(0,1)$ be fixed. We show that:
\begin{itemize}
    \item The positivity of the initial value is propagated, i.e. $f_{n}^\alpha\ge 0$ almost everywhere.
    \item The mass and momentum are conserved almost surely
    \begin{gather}
       \int_v (1,v)f_{n}^\alpha(t,v)\dd v = \int_v (1,v)f_0^\alpha(v)\dd v\quad\forall\,t\in[0,T]\label{sec2:ineq-1}
    \end{gather}
    \item The following energy inequality holds almost surely 
    \begin{align}
    \label{sec2:ineq-2}
    \int_v |v|^2 f_{n}^\alpha(t,v) \le \int_v |v|^2 f^\alpha_0(v)+ 2 \alpha d\, T\, \|f^\alpha_0\|_{L^1}\quad\forall\,t\in[0,T].     
    \end{align}
    \item We define the $(\mu,L)$-entropy as follows
\begin{equation}
\mathcal{H}_{\mu,L}(f) := \int_{v} h_{\mu,L}(f(v))\,\dd v\quad\text{and}\quad h_{\mu,L}(s) := \int_0^s L(\mu^2(r))\dd r.
\end{equation}
The following $(\mu,L)$-entropy inequality holds
\begin{equation}
\label{sec2:ineq-3}
\begin{aligned}
\hE\big[\cH_{\mu_n,L_n}(f_{n}^\alpha(t))\big]
+ \hE\Big[\int_0^t \cD_{\mu_n,L_n}(f_{n}^\alpha)\Big]
\le \hE\big[\cH_{\mu_n,L_n}(f^\alpha_0)\big]
+ C \hE\big[\| f^\alpha_0 \|_{L^1}^3+1\big]
\end{aligned}
\end{equation}
where $C=C(\|\sigma'\|_{L^\infty},G_k,T)>0$ is independent of $\alpha$ and $n$. 
\end{itemize}

\noindent\textbf{Step $3$: Pass to the limit in $L^2$-framework by letting $n\to\infty$.} 
We show Step $3$ in Section \ref{sec:limit-L2}.

Let $\alpha\in(0,1)$ be fixed. We show the tightness of $f_n^\alpha$ in an appropriated $L^2$-space, and denote the limit by 
\begin{align*}
    f^\alpha\in L^\infty([0,T];L^2(\Do;\R_+))\quad\text{a.s.}.
\end{align*}
We show that $f^\alpha$ is indeed 
 a probabilistic weak solution
to 
\begin{equation}
\begin{aligned}
\partial_t f
&=\alpha\Delta f+\frac12\tn\cdot\big( A f f_*\tn \log f\big)-\frac{\sqrt{\eps}}{2}\tn\cdot( A^{1/2}  \sigma(f) \sigma( f_*)\xi_K)\\
&\qquad+\frac{\eps}{2}\sum_{k= 1}^K\tn\cdot \Big(G_{k}(v,v_*) \sigma'(f) \sigma(f_{*}) \tn\cdot\big(G_{k}(v,w) \sigma(f) \sigma(f_{w})\big) \Big)
\end{aligned}
\end{equation}
with initial value $f^\alpha|_{t=0}=f^\alpha_0$. 

Moreover, we pass to the limit by letting $n\to\infty$ in \eqref{sec2:ineq-1}, \eqref{sec2:ineq-2} and \eqref{sec2:ineq-3} to show that, almost surely, for all $t\in[0,T]$,
\begin{equation}
\label{ineq:sec-2-2}
\begin{gathered}
    \int_v f^\alpha_t(v)=\int_v f^\alpha_0(v),\quad 
    \int_v|v|^2  f^\alpha_t(v)\le  \int_v|v|^2  f^\alpha_0(v) +2\alpha d T \|f^\alpha_0\|_{L^1}\\
    \hE\big[\cH(f^\alpha_t)\big]+\hE\Big[\int_0^t\cD(f^\alpha)\Big]\le \hE\big[\cH(f^\alpha_0)\big]+C\hE\big[\| f^\alpha_0\|_{L^1}^3+1\big],
    \end{gathered}
    \end{equation}
    where the constant $C$ is the same as in \eqref{sec2:ineq-3}.

\noindent\textbf{Step $4$: Pass to the limit in $L^1$-framework by letting $\alpha\to0$.} We show Step $4$ in Section \ref{sec:limit-L1}.

With the help of the uniformly bounded energy, entropy and entropy dissipation, we show the tightness of $f_n^\alpha$ in an appropriate $L^1$-space, and denote the limit by 
\begin{align*}
    f\in L^\infty([0,T];L^1_2(\Do;\R_+))\quad \text{a.s.}.
\end{align*}
One can construct a probabilistic weak solution
to our goal It\^o equation \eqref{SDE-1} based on the above tightness result. 
Moreover, we pass to the limit by $\alpha\to0$ in \eqref{ineq:sec-2-2} to show the mass conservation law, energy and entropy inequalities stated in Theorem \ref{main:thm}.

\section{Existence of a diffusive approximation Landau equation}
\label{sec:Galerkin}

This section is devoted to the first, fully regularised, existence result.
We freeze one approximation level and suppress the approximation index in
the notation.  Thus the interaction kernel is smooth, bounded, and compactly
supported away from the origin, the Landau mobility $\mu$ and the stochastic
coefficient $\sigma$ are both bounded and Lipschitz, and the entropy variable
$L$ is a smooth truncation of the logarithm.  The small linear diffusion
$\alpha\Delta f$, with $\alpha\in(0,1)$ fixed, is kept throughout the section.

More precisely, we prove the existence of a probabilistic weak solution to
the regularised It\^o equation
\begin{equation}
\label{ito-epsi-1}
    \begin{aligned}
        \dd  f=&\alpha \Delta f \dd t+Q_{\mu ,L }(f,f)\dd t\\
        &-\frac{\sqrt{\eps}}{2}\tn\cdot\left(A^{1/2}\sigma (f)\sigma (f_*) \dd \xi_K\right)\\
        &\quad+\frac{\eps}{2}\sum_{k=1}^K\tn\cdot \Big(G_k(v,v_*)\sigma'(f)\sigma(f_*)
\tn\cdot\big(G_k(v,w)\sigma(f)\sigma(f_w)\big)\Big)\dd t .
    \end{aligned}
\end{equation}
The regularised
Landau collision operator is
\begin{equation}
\label{ef:Q-sigma-h}
    Q_{\mu,L}(f,f):=\frac12 \tn\cdot \big(A\mu^2(f)\mu^2(f_*)\tn L(\mu^2(f))\big).
\end{equation}

We first present the assumptions used only at this regularised level.
\begin{assumption}[Kernel $A$]\label{ass:A}
    Let $A\in C^\infty_c(\R_+;\R_+)$ such that for all $z\in\R^d$
\begin{align*}
    0\le A(|z|)\le A_0,\quad \supp(A(|\cdot|))\subset \{z\in \R^d\mid z_*\le |z|\le z^*\}
\end{align*}
for some constants $0<z_*<z^*<+\infty$.

Moreover, up to a constant depending on $z_*$ and $z^*$, we have 
\begin{align*}
    |a|,\,|b|,\,|c|\lesssim_{z_*,z^*} A_0,\quad \supp(a),\,\supp(b),\,\supp(c)\subset B_{z^*}(0),
\end{align*}
since we have 
\begin{align*}
  b_i=-(d-1)\frac{z_i}{|z|^2}A(|z|)\quad \text{and}\quad c=-(d-1)\Big((d-2)\frac{A(|z|)}{|z|^2}+\frac{A'(|z|)}{|z|}\Big).
\end{align*}

\end{assumption}

The noise coefficients $G_k$ are defined as in \eqref{def:G-k}, with the
kernel $A$ satisfying Assumption~\ref{ass:A}.  In particular, the bound
\eqref{G-bdd:app-0} remains available:
\begin{equation}
\label{G-bdd:app-1}
\begin{aligned}
&\sum_{k=1}^K\big\|G_k\big\|_{W^{2,\infty}_{v,v_*}\cap W^{2,1}_{v,v_*}}\le C_K<+\infty.
\end{aligned}
\end{equation}

\begin{assumption}[Regularised Landau mobility and stochastic coefficient]
  \label{ass:sigma-R0}
Let $\sigma\in C^1(\R;\R_+)$ be a bounded globally Lipschitz function such that
\begin{equation}
\begin{gathered}
\sigma(r)=0\quad \forall\, r\le0,\qquad \sigma(1)=1,\qquad \sigma'(1)=\frac12,\\
0\le \sigma(r)\le \min\big\{C_{\sigma^*},\sqrt r\big\},\qquad
0\le \sigma'(r)\le C_{\sigma_*}\quad \forall\,r\ge0
\end{gathered}
\end{equation}
for some constants $C_{\sigma_*},C_{\sigma^*}>0$.

Let $\mu\in C^1(\R;\R_+)$ be a bounded globally Lipschitz function satisfying
\begin{equation}
\begin{gathered}
\mu(r)=0\quad \forall\, r\le0,\qquad \mu(1)=1,\qquad \mu'(1)=\frac12,\\
0\le \mu(r)\le \min\big\{C_{\mu^*},\sqrt r\big\},\qquad
0\le \mu'(r)\le C_{\mu_*}\quad \forall\,r\ge0
\end{gathered}
\end{equation}
for some constants $C_{\mu_*},C_{\mu^*}>0$.  In particular, for some constant $0<C_{\mu_{**}}\le C_{\mu_*}$, we assume
\begin{align*}
0<C_{\mu_{**}}\le \mu'(r)\le C_{\mu_{*}}\quad \forall r\in[0,1].    
\end{align*}

Moreover, we assume
\begin{align*}
    0\le \sigma(r)\le \mu(r)\quad\forall r\in\R.
\end{align*}

The coefficient $\mu$ appears
only in the deterministic Landau drift, while $\sigma$ appears only in the
stochastic flux and the associated It\^o correction.

\end{assumption}

Recall that $\varsigma$ is defined by \eqref{def:varsigma}. Notice that the $L^1$ and $L^2$ bounds of $\sigma$ and $\varsigma$ in Remark \ref{rmk:sigma-f:L2} still hold, i.e. 
\begin{gather*}
\| \mu(f)\|_{L^1}\le C_{\mu_*}\|f\|_{L^1},\quad  \| \mu(f)\|_{L^2}\le \min\Big(C_{\mu_*}\|f\|_{L^2},\|f\|_{L^1}^{\frac12}\Big),\\
    \| \sigma(f)\|_{L^1}\le C_{\sigma_*}\|f\|_{L^1},\quad  \| \sigma(f)\|_{L^2}\le \min\Big(C_{\sigma_*}\|f\|_{L^2},\|f\|_{L^1}^{\frac12}\Big),\\
    \|\varsigma(f)\|_{L^1}\le C_{\sigma_*}^2\|f\|_{L^1},\quad  \|\varsigma(f)\|_{L^2}\le \min\Big(C_{\sigma_*}\|f\|_{L^2},C_{\sigma_*}^2\|f\|_{L^1}^{\frac12}\Big).
\end{gather*}

\begin{assumption}[Regularisation of logarithm]
   \label{ass:L-s0}
 Let $L\in C^1(\R;\R)$ such that $L$ is monotonically non-decreasing,  
 \begin{align*}
     L(s)\ge -C_L,\quad L(1)=0,\quad |L(s)|\le |\log f|\quad\text{and} \quad 0\le  L'(s)s\le1\quad \forall \,s\in[0,\infty)
 \end{align*}
 for some constant $C_L>0$. 

 Moreover, we assume $L'$ is monotonically non-increasing, $L'(1)=1$, and 
 \begin{align*}
     0\le  L'(s)\le C_L\quad\text{for all}\quad s\in[0,\infty). 
 \end{align*}
\end{assumption}

\begin{example}[Admissible regularisations]
\label{expl:sigma-L}
\begin{itemize}
    \item Let $r_0\in (0,1)$ and $R_0\in(2,\infty)$.  Starting from the
    model choice $\sqrt r$, we may take a $C^1$ function $\sigma$ satisfying
    Assumption~\ref{ass:sigma-R0} by setting
 \begin{equation}
 \label{exm:sigma-R0}
       \sigma (r):=\left\{
        \begin{aligned}
       0,&\quad  r\in(-\infty,0)\\
        \frac{3r}{2\sqrt{r_0}}-\frac{r^2}{2r_0\sqrt{r_0}},&\quad  r\in[0,r_0)\\
        \sqrt{r},&\quad r\in[r_0,R_0)
        \end{aligned}
        \right.,\quad 
        \sigma '(r):=\left\{
        \begin{aligned}
       0,&\quad  r\in(-\infty,0)\\
        \frac{3}{2\sqrt{r_0}}-\frac{r}{r_0\sqrt{r_0}},&\quad r\in[0,r_0)\\
         \frac{1}{2\sqrt{r}},&\quad r\in[r_0,R_0)
        \end{aligned}
        \right..
    \end{equation}
    and then completing the definition smoothly on $[R_0,\infty)$ so that
    $\sigma$ is bounded, non-decreasing, and $\sigma'(r)=\exp(-r^2)$ for
    $r\ge R_0+1$.

For this choice, the derivative on the smoothing region is bounded by a constant of order $r_0^{-1/2}$, while the completion on $[R_0,\infty)$ can be chosen non-increasing to zero and bounded by a constant depending on $R_0$. Hence, we can take, for instance, $C_{\sigma_*}=Cr_0^{-1/2}$ and $C_{\sigma^*}=C R_0^{1/2}$, with $C$ independent of $r_0,R_0$.

The same construction can be used for $\mu$, then $\mu$ satisfying Assumption~\ref{ass:sigma-R0} and $0\le \sigma\le\mu$. Indeed, we have 
\begin{align*}
    0<\frac{1}{2}\le \mu'(r)\le \frac{3}{2\sqrt{r_0}}\quad\,r\in[0,1].
\end{align*}

\item Let $s_0\in(0,1)$.  A convenient regularisation of the logarithm is
    \begin{equation}
    \label{exm:L-s0}
        L(s):=\left\{
        \begin{aligned}
        \log s,&\quad  s\in[s_0,\infty)\\
        s_0^{-1} s+\log s_0-1,&\quad s\in(-\infty,s_0), 
        \end{aligned}
        \right.
    \end{equation}
which satisfies Assumption~\ref{ass:L-s0}.  Indeed,
    \begin{equation*}
     L'(s)=\left\{
        \begin{aligned}
        s^{-1},&\quad  s\in[s_0,\infty)\\
        s_0^{-1},&\quad s\in(-\infty,s_0). 
        \end{aligned}
        \right.
    \end{equation*}
One can take $C_L=1-\log s_0+s_0^{-1}$.
\end{itemize}
    
\end{example}

The main result of this section is the following regularised existence
statement.

\begin{proposition}[Existence for the regularised equation]
\label{lem:app-1:ext}
Let $d\ge 2$. Let  $T>0$, $\alpha\in(0,1)$ and $\eps\in(0,1)$ be arbitrarily fixed.  Assume that
$f_0\in L^2(\R^d;\R_+)$, that $A$ satisfies Assumption~\ref{ass:A}, that
$(\mu,\sigma)$ satisfy Assumption~\ref{ass:sigma-R0}, that $L$ satisfies
Assumption~\ref{ass:L-s0}, and that the finite family
$(G_k)_{1\le k\le K}$ satisfies Assumption~\ref{ass:G-k:app}.  Then there
exist a stochastic basis
$(\Omega,\mathcal F,(\mathcal F_t)_{t\in[0,T]},\mathbb P)$ satisfying the
usual conditions, independent real-valued Brownian motions
$(B^k)_{1\le k\le K}$ on this basis, and a process $f$ such
that
\begin{align*}
f\in L^2\big(\Omega;L^\infty(0,T;L^2(\R^d))\big)
\cap L^2\big(\Omega;L^2(0,T;H^1(\R^d))\big),
\end{align*}
and $(f,(B^k)_{1\le k\le K})$ is a probabilistic weak solution to
\eqref{ito-epsi-1} with initial datum $f_0$ in the sense of
Definition~\ref{def:weak-1}.
\end{proposition}

The remainder of this section is devoted to proving Proposition \ref{lem:app-1:ext}.

In Section \ref{sec-3:1}, we present the weak formulation of \eqref{ito-epsi-1}. 
 The necessary tightness results are established in Section \ref{sec-3:2}, and the proof of Proposition \ref{lem:app-1:ext} is completed in Section \ref{sec-3:3}.


    

\subsection{Weak formulation}
\label{sec-3:1}

Analogously to \eqref{relation:abc}, we define $a,b,c$ corresponding to kernel $A$ satisfying Assumption \ref{ass:A}
\begin{equation}
\label{def:abc-app}
\begin{aligned}
 &a_{ij}(v-v_*):=\Big(\delta_{ij}-\frac{(v-v_*)_i(v-v_*)_j}{|v-v_*|^2}\Big)A(|v-v_*|),\\
 &
         b_i(v-v_*):=\sum_{j=1}^d\d_j a_{ij}(v-v_*),\quad c(v-v_*):=\sum_{i,j=1}^d\d^2_{ij} a_{ij}(v-v_*).
     \end{aligned}
        \end{equation}
        One can define $\bar a,\bar b,\bar c$ as in  \eqref{def:abc-bar}.

Let $L\in C^1(\R)$ satisfy Assumption \ref{ass:L-s0}.  We define function $\tili(s)$ as follows: $\tili(s):=0$ for $s\le 0$, and for $s>0$ 
\begin{equation}
\label{def:h-bar}
    \tili(s):=\int_0^sL' (r)r\dd r.
\end{equation}
 We note that $0\le \tili(s)\le s$.
The collision operator $Q_{\mu,L}(f,f)$ can be written as
\begin{align*}
    Q_{\mu,L}(f,f)&=\nabla_v\cdot \int_{v_*}\mu^2(f)\mu^2(f_*)A\Pi_{(v-v_*)^\perp}\Big[L' (\mu^2(f))\nabla_v \mu^2(f)-L'(\mu^2(f_*))\nabla_{v_*} \mu^2(f_*)\Big]\\
    &= \nabla_v\cdot\big[\bar a(\mu^2(f))\nabla_v \tili\big(\mu^2(f)\big)\big]-\nabla_v\cdot\int_{v_*}\mu^2(f)a(v-v_*)\nabla_{v_*} \tili\big(\mu^2(f_*)\big).
\end{align*}

By integration by parts and the definition of $\bar b$, the second term on the right-hand side can be written as
\begin{align*}
\nabla_v\cdot\int_{\R^d}\mu^2(f)a(v-v_*)\nabla_{v_*} \tili(\mu^2(f_*))\dd v_*=\nabla_v\cdot\big[\mu^2(f)\bar b\big( \tili(\mu^2(f))\big)\big].
\end{align*}
Thus $Q_{\mu,L}(f,f)$ has the following form
\begin{align*}
    Q_{\mu,L}(f,f)=\nabla_v\cdot\big(\bar a(\mu^2(f))\nabla_v \tili (\mu^2(f))-\mu^2(f)\bar b\big[ \tili (\mu^2(f))\big]\big).
\end{align*}

By using $b(v-v_*)=-b(v_*-v)$ and integration by parts, the collision operator $Q_{\mu,L}(f,f)$ has the following weak formulation
    \begin{align*}
        &\int_{\R^d}\phi  Q_{\mu,L}(f,f)\dd v\\
        =&{}\sum_{i,j=1}^d\int_{v,v_*}\mu^2(f_*)\tili(\mu^2(f)) a_{ij}\d_{v_iv_j}\phi+\sum_{i=1}^d\int_{v,v_*} \mu^2(f_*)\tili(\mu^2(f))b_i\big[\d_{v_i}\phi-(\d_{v_i}\phi)_{_*}\big].
    \end{align*}

We define probabilistic weak solutions to \eqref{ito-epsi-1} as follows.
\begin{definition}[Probabilistic weak solution for the regularised equation]
    \label{def:weak-1}
Let $\alpha\in(0,1)$ and let $f_0\in L^2(\R^d)$.  A pair
$(f,(B^k)_{1\le k\le K})$ is called a probabilistic weak solution to
\eqref{ito-epsi-1} on $[0,T]$ if there exists a stochastic basis
$(\Omega,\mathcal F,(\mathcal F_t)_{t\in[0,T]},\mathbb P)$ satisfying the
usual conditions such that $(B^k)_{1\le k\le K}$ are independent
$(\mathcal F_t)$-Brownian motions, and $f$ is a  progressively
measurable process satisfying
\begin{align*}
f\in L^2\big(\Omega;L^\infty(0,T;L^2(\R^d))\big)
\cap L^2\big(\Omega;L^2(0,T;H^1(\R^d))\big).
\end{align*}
Moreover, for every $\phi\in C_c^\infty(\R^d)$, the real-valued process
$t\mapsto\langle f(t),\phi\rangle$ has a continuous
$(\mathcal F_t)$-adapted version and, almost surely, for all $t\in[0,T]$,
\begin{equation}
\label{weak:app-1}
     \begin{aligned}
&\int_{v}f_t(v)\phi(v)\,\dd v-\int_{v}f_0(v)\phi(v)\,\dd v\\
={}&\alpha\int_0^t\int_{v}f\,\Delta\phi\,\dd v\dd s\\
&+\sum_{i,j=1}^d\int_0^t\int_{v,v_*}
\mu^2(f_*)\tili(\mu^2(f))a_{ij}(v-v_*)\d_{v_iv_j}\phi(v)
\,\dd v\dd v_*\dd s\\
&+\sum_{i=1}^d\int_0^t\int_{v,v_*}
\mu^2(f)\tili(\mu^2(f_*))b_i(v-v_*)
\Big(\d_{v_i}\phi(v)-\d_{v_i}\phi(v_*)\Big)
\,\dd v\dd v_*\dd s\\
&-\frac{\sqrt{\eps}}{2}\sum_{k=1}^K\int_0^t\int_{v,v_*}
\tn\phi(v)\cdot G_k(v,v_*)\sigma(f)\sigma(f_*)
\,\dd v\dd v_*\,\dd B^k_s\\
&-\eps\int_0^t\int_{v,v_*,w}
\big(\tn\phi(v)\cdot F_1(v,v_*,w)\big)
\sigma(f_*)\sigma(f_w)\sigma(f)\sigma'(f)
\,\dd v\dd v_*\dd w\dd s\\
&+\eps\int_0^t\int_{v,v_*,w}
\nabla_v\cdot\big(\tn\phi(v)F_4(v,v_*,w)\big)
\sigma(f_*)\sigma(f_w)\varsigma(f)
\,\dd v\dd v_*\dd w\dd s .
\end{aligned}
\end{equation}
All terms in \eqref{weak:app-1} are finite under the above assumptions.  In
particular, since $\phi$ is compactly supported and the regularised
coefficients are bounded with compact support, the deterministic terms are
controlled by $\|f\|_{L^2}$ and $\|\nabla_v f\|_{L^2}$ on bounded velocity
sets.  The stochastic integral is well-defined because, for each $k$,
\begin{align*}
\int_0^T\left|
\int_{v,v_*}\tn\phi(v)\cdot G_k(v,v_*)\sigma(f)\sigma(f_*)
\,\dd v\dd v_*
\right|^2\dd s
\le C_{\phi,G_k}\int_0^T\|f(s)\|_{L^2}^4\,\dd s<+\infty
\end{align*}
almost surely.
\end{definition}

\subsection{Uniform estimates and tightness of the Galerkin scheme}\label{sec-3:2}

We construct solutions by a Galerkin approximation.  Let
$\{e_j\}_{j\ge1}\subset\mathcal S(\R^d)$ be an orthonormal basis of
$L^2(\R^d)$ which is also orthogonal in $H^1(\R^d)$.  For $m\in\N$, set
$E_m:=\operatorname{span}\{e_1,\ldots,e_m\}$ and denote by $\Pi_m$ the
$L^2$-orthogonal projection onto $E_m$:
\begin{align*}
    \Pi_m f:=\sum_{j=1}^m\langle f,e_j\rangle_{L^2}e_j .
\end{align*}

The $m$-th Galerkin approximation is the $E_m$-valued equation
\begin{equation}
\label{app-2:Galerkin}
\begin{aligned}
        \dd  f_m=&\Pi_m\Big[\alpha\Delta f_m\dd t+\nabla_v\cdot\big(\bar a(\mu^2(f_m))\nabla_v \tili(\mu^2(f_m))-\bar b (\tili(\mu^2(f_m))) \mu^2(f_m)\big)\dd t\\
        &-\frac{\sqrt{\eps}}{2}\sum_{k= 1}^K\tn\cdot\big(G_k(v,v_*)\sigma(f_m)\sigma(f_{m,*})\big)\dd B^k\\
        &+\eps\tn\cdot \Big(\int_{w}\sigma(f_{m,*})\sigma(f_{m,w})\Big(F_4 \nabla_v \varsigma(f_m)+\sigma(f_m)\sigma'(f_m)F_1\big)\Big)\dd t\Big],\\
f_m|_{t=0}=&\Pi_m f_0 .
\end{aligned}
\end{equation}
Writing $f_m(t)=\sum_{j=1}^m h_j(t)e_j$, \eqref{app-2:Galerkin} is
equivalent to a finite-dimensional SDE
\begin{align}
\label{galerkin:finite-sde}
\dd h_i(t)=\mathcal B_i(h(t))\,\dd t
+\sum_{k=1}^K\mathcal G_{ik}(h(t))\,\dd B^k_t,
\qquad
h_i(0)=\langle f_0,e_i\rangle_{L^2},
\quad 1\le i\le m,
\end{align}
where $\mathcal B_i$ and $\mathcal G_{ik}$ are obtained by testing the
right-hand side of \eqref{app-2:Galerkin} against $e_i$.  Since all norms on
$E_m$ are equivalent and since $a,b,G_k,F_1,F_4,\sigma,\sigma'$ and
$\varsigma$ are smooth on this regularised level, the maps
$h\mapsto\mathcal B(h)$ and $h\mapsto\mathcal G(h)$ are locally Lipschitz on
$\R^m$.  Classical finite-dimensional SDE theory therefore gives a unique
maximal strong solution $h(t)$, or equivalently an $E_m$-valued process
$f_m(t)$, up to its explosion time $\tau_m$.

For $R>0$ define
\begin{align*}
\tau_{m,R}:=\inf\{t\ge0:\|f_m(t)\|_{L^2}\ge R\}\wedge T .
\end{align*}
The next lemma is first proved on $[0,T\wedge\tau_{m,R}]$, with constants
independent of $R$.  This stopped estimate will then imply that
$\tau_m>T$ almost surely, so the local Galerkin solution is in fact global
on $[0,T]$.

\begin{lemma}[Uniform $L^2$-estimate for the Galerkin solutions]
\label{lem:time-L2}
Assume the hypotheses of Proposition~\ref{lem:app-1:ext}. Let  $m\in\mathbb N$ be arbitrarily fixed. Let $f_m$ be the
maximal local solution to the Galerkin equation \eqref{app-2:Galerkin} with
initial datum $\Pi_m f_0$. Then $\tau_m>T$ almost surely. 

Moreover, for every
$q\ge1$, there exists a constant
\begin{align*}
C=C\big(q,T,d,C_K,\alpha,C_{\sigma^*},C_{\sigma_*},C_{\mu_*},C_{\mu_*},z^*,A_0,
\|f_0\|_{L^2}\big)>0,
\end{align*}
independent of $m$, such that
\begin{align}
\label{galerkin:L2-uniform}
\sup_{m\in\mathbb N}
\mathbb E\Big[
\sup_{t\in[0,T]}\|f_m(t)\|_{L^2_v}^{2q}
+\Big(\alpha\int_0^T\|\nabla_v f_m(s)\|_{L^2_v}^2\,\dd s\Big)^q
\Big]\le C .
\end{align}
\end{lemma}
\begin{proof}
Fix $m\in\mathbb N$ and write $f=f_m$.  We first work on the stopped
interval $[0,T\wedge\tau_{m,R}]$.  Since $\Pi_m$ is an orthogonal projection
in $L^2(\R^d)$ and $f(t)\in E_m$, It\^o's formula may be applied directly to
$\|f(t)\|_{L^2}^2$.  For every $t\in[0,T]$, almost surely,
\begin{align}
\label{galerkin:L2-energy-identity}
\|f(t\wedge\tau_{m,R})\|_{L^2}^2
+&\alpha\int_0^{t\wedge\tau_{m,R}}\|\nabla_vf(s)\|_{L^2}^2\,\dd s\notag\\
\le& \|\Pi_mf_0\|_{L^2}^2
+C\int_0^{t\wedge\tau_{m,R}}(1+\|f(s)\|_{L^2}^2)\,\dd s
+M_{m,R}(t),
\end{align}
where $C$ is independent of $m$, and
\begin{align*}
M_{m,R}(t)
:=\frac12\sum_{k=1}^K\int_0^t\int_{v,v_*}
\mathbb{1}_{\{s\le\tau_{m,R}\}}
\tn f\cdot G_k(v,v_*)\sigma(f)\sigma(f_*)\,\dd v\dd v_*\,\dd B^k(s)
\end{align*}
is a continuous square-integrable martingale.  Let us give the estimates leading to
\eqref{galerkin:L2-energy-identity}.  The artificial diffusion contributes
$-\alpha\|\nabla_v f\|_{L^2}^2$.  For the Landau drift we use
\begin{align*}
Q_{\mu,L}(f,f)
=\nabla_v\cdot\big(\bar a(\mu^2(f))\nabla_v\tili(\mu^2(f))
-\mu^2(f)\bar b(\tili(\mu^2(f)))\big).
\end{align*}
The coercive part is non-positive in the $L^2$ balance:
\begin{align*}
&-\int_v\nabla_v f\cdot
\bar a(\mu^2(f))\nabla_v\tili(\mu^2(f))\,\dd v\\
=&{}-\int_v L'(\mu^2(f))\mu^2(f)(\mu^2)'(f)
(\nabla_v f)^T\bar a(\mu^2(f))\nabla_v f\,\dd v\le0.
\end{align*}
The $b$-part is controlled by the compact support and boundedness of $b$, the boundedness of
$\mu$ and $\tili(\mu^2(\cdot))$, and Young's inequality:
\begin{align*}
\left|\int_v\nabla_v f\cdot
\bar b(\tili(\mu^2(f)))\mu^2(f)\,\dd v\right|
\le
\delta\|\nabla_v f\|_{L^2}^2+C_\delta(1+\|f\|_{L^2}^2).
\end{align*}
The deterministic Stratonovich-to-It\^o corrections are estimated similarly.
For the $F_1$ contribution, the bounds on $\sigma,\sigma'$ and the
$W^{1,1}\cap W^{1,\infty}$ bounds of $F_1$ give
\begin{align*}
&\left|\int_{v,v_*,w}
f\,\sigma(f_*)\sigma(f_w)\sigma(f)\sigma'(f)
\tn f\cdot F_1\,\dd v\dd v_*\dd w\right|\\
&\qquad
\le \delta\|\nabla_vf\|_{L^2}^2
+C_\delta(1+\|f\|_{L^2}^2).
\end{align*}
For the $F_4$ contribution we integrate by parts in $v$ so that no
uncontrolled second derivative of $f$ appears.  Since
$|\varsigma(f)|\lesssim |f|$ and $F_4$ is smooth and compactly supported,
\begin{align*}
&\left|\int_{v,v_*,w}
\sigma(f_*)\sigma(f_w)\tn f\cdot F_4\nabla_v\varsigma(f)
\,\dd v\dd v_*\dd w\right|\\
&\qquad
\le \delta\|\nabla_vf\|_{L^2}^2
+C_\delta(1+\|f\|_{L^2}^2).
\end{align*}
Choosing $\delta>0$ sufficiently small with respect to $\alpha$ gives
\eqref{galerkin:L2-energy-identity}.  The projection $\Pi_m$ does not
increase the $L^2$ norm, and all constants are independent of $m$.

For the martingale part, the Burkholder--Davis--Gundy inequality and the
bound $\sigma(f)\le C_{R_0}$ yield, for every $\delta>0$,
\begin{align*}
\mathbb E\sup_{r\le t}|M_{m,R}(r)|^q
&\le C_q\mathbb E\left[
\int_0^{t\wedge\tau_{m,R}}\sum_{k=1}^K
\left|\int_{v,v_*}\tn f\cdot G_k\sigma(f)\sigma(f_*)\right|^2\dd s
\right]^{q/2}\\
&\le \delta\,\mathbb E\left(\int_0^{t\wedge\tau_{m,R}}
\|\nabla_vf(s)\|_{L^2}^2\,\dd s\right)^q\\
&\qquad
+C_{\delta,q}\mathbb E\left(1+\int_0^{t\wedge\tau_{m,R}}
\|f(s)\|_{L^2}^2\,\dd s\right)^q .
\end{align*}
Taking the supremum in time in \eqref{galerkin:L2-energy-identity}, choosing
$\delta>0$ sufficiently small, and applying Gr\"onwall's inequality gives the
stopped estimate
\begin{align}
\label{galerkin:L2-uniform-stopped}
\sup_{m\in\mathbb N}\sup_{R>0}
\mathbb E\Big[
\sup_{t\in[0,T]}\|f_m(t\wedge\tau_{m,R})\|_{L^2_v}^{2q}+
\Big(\alpha\int_0^{T\wedge\tau_{m,R}}
\|\nabla_v f_m(s)\|_{L^2_v}^2\,\dd s\Big)^q
\Big]\le C .
\end{align}
This is the stopped version of the standard $L^2$ energy computation for the
regularised equation; the indicator $\mathbb{1}_{\{s\le\tau_{m,R}\}}$ only
restricts the time integrals and does not change any of the deterministic
estimates.

The stopped estimate excludes explosion.  Indeed, for $R>\|\Pi_mf_0\|_{L^2}$
and on the event $\{\tau_{m,R}<T\}$, continuity of the finite-dimensional
path gives $\|f_m(\tau_{m,R})\|_{L^2}=R$.  Hence
\begin{align*}
\mathbb P(\tau_{m,R}<T)
\le R^{-2q}
\mathbb E\sup_{t\in[0,T]}
\|f_m(t\wedge\tau_{m,R})\|_{L^2}^{2q}
\le CR^{-2q}\longrightarrow0 .
\end{align*}
Letting $R\to\infty$ yields $\tau_m>T$ almost surely.  Finally, applying
Fatou's lemma to \eqref{galerkin:L2-uniform-stopped} gives
\eqref{galerkin:L2-uniform} for the global Galerkin solution.  The estimate
is uniform in $m$ because $\|\Pi_mf_0\|_{L^2}\le\|f_0\|_{L^2}$.
\end{proof}

\begin{lemma}[Uniform local time regularity for $(f_m)_{m\in\mathbb N}$]
\label{lem:time-regularity}
Assume the hypotheses of Proposition~\ref{lem:app-1:ext}. Let
$\chi\in C_c^\infty(\R^d)$, $l>d/2+2$, and $\beta\in(0,1/2)$. Then, for
every
\begin{align*}
p>\frac{1}{1/2-\beta},
\end{align*}
there exists a constant
\begin{align*}
C=C\big(p,\beta,l,T,\chi,d,C_K,\alpha,C_{\sigma^*},C_{\sigma_*},C_{\mu_*},C_{\mu_*},z^*,A_0,
\|f_0\|_{L^2}\big)>0,
\end{align*}
independent of $m$, such that
\begin{align}
\label{galerkin:time-regularity}
\sup_{m\in\mathbb N}
\mathbb E\big[
\|f_m\chi\|_{C^\beta([0,T];H^{-l}(\R^d))}^{p}
\big]\le C .
\end{align}
In particular, the laws of $(f_m\chi)_{m\in\mathbb N}$ are tight in
$C^{\beta'}([0,T];H^{-l-1}(\R^d))$ for every $\beta'<\beta$.
\end{lemma}

\begin{proof}
Fix $m\in\mathbb N$ and write $f=f_m$.  We decompose the localised equation
as
\begin{align*}
f(t)\chi=f(0)\chi+D_m(t)+M_m(t),
\end{align*}
where $D_m$ contains all deterministic terms in \eqref{app-2:Galerkin}, and
\begin{align*}
M_m(t)
:=-\frac{\sqrt{\eps}}{2}\sum_{k=1}^K\int_0^t
\chi\,\Pi_m\tn\cdot\big(G_k(v,v_*)\sigma(f)\sigma(f_*)\big)\,\dd B^k
\end{align*}
is the localised martingale term.

\medskip

\noindent\textbf{Deterministic increments.}
Let $\phi\in H^l(\R^d)$ with $\|\phi\|_{H^l}\le1$. Since $l>d/2+2$,
\begin{align}
\label{galerkin:Sobolev-time-reg}
\|\chi\phi\|_{W^{2,\infty}}\le C_{\chi,l}.
\end{align}
Using the weak form of \eqref{app-2:Galerkin} and the self-adjointness of
$\Pi_m$, for every $0\le s<t\le T$ we estimate
\begin{align*}
|\langle D_m(t)-D_m(s),\phi\rangle|
\lesssim{}&
\int_s^t\int_v |f|\,|\Delta(\chi\phi)|\,\dd v\dd r\\
&+\int_s^t\int_{v,v_*}
\big|\mu^2(f_*)\tili(\mu^2(f))\big|
|a(v-v_*)|\,|\nabla^2(\chi\phi)|\,\dd v\dd v_*\dd r\\
&+\int_s^t\int_{v,v_*}
\big|\mu^2(f)\tili(\mu^2(f_*))\big|
|b(v-v_*)|\,|\nabla(\chi\phi)(v)-\nabla(\chi\phi)(v_*)|
\,\dd v\dd v_*\dd r\\
&+\int_s^t\int_{v,v_*,w}
|\sigma(f_*)\sigma(f_w)|
\Big(|\varsigma(f)|\,|\nabla_v\cdot(\tn(\chi\phi)F_4)|
\\
&\hspace{5.5cm}
 +|\sigma(f)\sigma'(f)|\,|\tn(\chi\phi)\cdot F_1|\Big)
\dd v\dd v_*\dd w\dd r .
\end{align*}
Here, the $F_4$ term has been integrated by parts in $v$.  Since
$a,b,F_1,F_4$ are smooth and compactly supported in this approximation
section, while $\sigma,\sigma',\varsigma$ and $\tili(\sigma^2(\cdot))$ are
bounded by the assumptions, \eqref{galerkin:Sobolev-time-reg} and
\eqref{galerkin:L2-uniform} imply
\begin{align}
\label{galerkin:deterministic-time-increment}
\mathbb E\|D_m(t)-D_m(s)\|_{H^{-l}}^p
\le C|t-s|^{p/2},
\end{align}
where $C$ is independent of $m$.  The exponent $1/2$ comes only from the
terms estimated by $\int_s^t\|\nabla_vf(r)\|_{L^2}\dd r$; all remaining
terms are in fact of order $|t-s|$.

\medskip

\noindent\textbf{Martingale increments.}
For the martingale part, integration by parts gives, for
$\|\phi\|_{H^l}\le1$,
\begin{align*}
|\langle M_m(t)-M_m(s),\phi\rangle|
\le \frac12\left|
\sum_{k=1}^K\int_s^t\int_{v,v_*}
\nabla_v\Pi_m(\chi\phi)\cdot
G_k(v,v_*)\sigma(f)\sigma(f_*)\,\dd v\dd v_*\,\dd B_k
\right|.
\end{align*}
By the Burkholder--Davis--Gundy inequality, the uniform bounds on $G_k$, and
the boundedness of $\sigma$, for every $p\ge2$,
\begin{align}
\label{galerkin:martingale-time-increment}
\mathbb E\|M_m(t)-M_m(s)\|_{H^{-l}}^p
&\le C_p\mathbb E\Big[
\int_s^t\sum_{k=1}^K
\Big|\int_{v,v_*}\nabla_v\Pi_m(\chi\phi)\cdot
G_k\sigma(f)\sigma(f_*)\,\dd v\dd v_*\Big|^2\dd r
\Big]^{p/2}\notag\\
&\le C_p|t-s|^{p/2}.
\end{align}
Indeed, the last inequality is precisely the point where the higher
$L^2$-moments in Lemma~\ref{lem:time-L2} are used.  Since
$l>d/2+2$, $\|\phi\|_{H^l}\le1$ implies
$\|\nabla_v\Pi_m(\chi\phi)\|_{L^\infty}\le C_{\chi,l}$, uniformly in $m$.
Moreover, by the compact support and boundedness of $G_k$ and by
$|\sigma(z)|\le C_{\sigma_*}|z|$ on the relevant range,
\begin{align*}
\Big|\int_{v,v_*}\nabla_v\Pi_m(\chi\phi)\cdot
G_k(v,v_*)\sigma(f)\sigma(f_*)\,\dd v\dd v_*\Big|
\le C_{\chi,l,C_K}\|\sigma(f)\|_{L^2}^2
\le C_{\chi,l,C_K,C_{\sigma_*}}\|f\|_{L^2}^2 .
\end{align*}
Hence
\begin{align*}
\mathbb E\|M_m(t)-M_m(s)\|_{H^{-l}}^p
&\le C_p\,
\mathbb E\left(\int_s^t\|f(r)\|_{L^2}^{4}\,\dd r\right)^{p/2}  \\
&\le C_p |t-s|^{p/2}
\mathbb E\left[\sup_{r\in[0,T]}\|f(r)\|_{L^2}^{2p}\right]
\le C_p|t-s|^{p/2},
\end{align*}
where the last step follows from \eqref{galerkin:L2-uniform} with
$q=p$.  This explains why the preceding $L^2$ estimate is stated with
arbitrary higher moment $q$.
The constants are independent of $m$, because $\Pi_m$ is bounded on the
finite number of Sobolev norms of the fixed smooth function $\chi\phi$
appearing above.

\medskip

\noindent\textbf{Conclusion.}
Combining \eqref{galerkin:deterministic-time-increment} and
\eqref{galerkin:martingale-time-increment}, we have
\begin{align*}
\mathbb E\|f_m(t)\chi-f_m(s)\chi\|_{H^{-l}}^p
\le C|t-s|^{p/2}.
\end{align*}
Since $p(1/2-\beta)>1$, the quantitative Kolmogorov continuity theorem
implies
\begin{align*}
\sup_m\mathbb E\big[
[f_m\chi]_{C^\beta([0,T];H^{-l})}^{p}\big]\le C.
\end{align*}
The bound on $\sup_t\|f_m(t)\chi\|_{H^{-l}}$ follows from
\eqref{galerkin:L2-uniform}. This proves \eqref{galerkin:time-regularity}.
Finally, the compact embedding
$H^{-l}_{\supp\chi}\Subset H^{-l-1}$ and the Arzel\`a--Ascoli theorem give
\begin{align*}
C^\beta([0,T];H^{-l})\Subset
C^{\beta'}([0,T];H^{-l-1}),\qquad \beta'<\beta,
\end{align*}
and Chebyshev's inequality yields the asserted tightness.
\end{proof}

\begin{proposition}[Tightness of the Galerkin laws]
\label{prop:galerkin-tightness}
Assume the hypotheses of Proposition~\ref{lem:app-1:ext}. Let
$(f_m,(B^k)_{1\le k\le K})$ be the Galerkin solutions to
\eqref{app-2:Galerkin}. Then, for every $\beta'\in(0,1/2)$ and
$l>d/2+2$, the laws of
\begin{align*}
\big(f_m,\nabla_vf_m,(B^k)_{1\le k\le K}\big)
\end{align*}
are tight on
\begin{align*}
\mathbb Y
&:=\Big(L^2([0,T];L^2_{\loc}(\R^d))
\cap C^{\beta'}([0,T];H^{-l-1}_{\loc}(\R^d))\Big)\\
&\qquad\times\big(L^2([0,T];L^2(\R^d)),w\big)
\times C([0,T];\R^K).
\end{align*}
\end{proposition}
\begin{proof}
Let $\chi_R\in C_c^\infty(\R^d)$ be equal to one on $B_R$ and supported in
$B_{2R}$. By Lemma~\ref{lem:time-L2},
\begin{align*}
\sup_{m\in\N}\mathbb E\Big[
\|f_m\|_{L^\infty_tL^2_v}^2
+\|\nabla_v f_m\|_{L^2_tL^2_v}^2\Big]<+\infty.
\end{align*}
By Lemma~\ref{lem:time-regularity}, for every $R>0$, the sequence
$(f_m\chi_R)_{m\ge1}$ is tight in
$C^{\beta'}([0,T];H^{-l-1}(\R^d))$. Moreover, the compact embedding
\begin{align*}
L^2(0,T;H^1(B_{2R}))
\cap C^\beta([0,T];H^{-l}(B_{2R}))
\Subset L^2(0,T;L^2(B_R))
\end{align*}
implies the tightness of $(f_m)_{m\ge1}$ in
$L^2([0,T];L^2_{\loc}(\R^d))$. The gradient component is tight in
$L^2([0,T];L^2(\R^d))$ endowed with the weak topology by the uniform
$L^2$ bound, and the Brownian laws are tight on $C([0,T];\R^K)$. Therefore
the joint laws are tight on $\mathbb Y$.
\end{proof} 

\subsection{Proof of Proposition \ref{lem:app-1:ext}}
\label{sec-3:3}

We follow a compactness--identification strategy, dividing the proof into the following four steps. A similar strategy will also be used in Sections \ref{sec-5:limit} and \ref{sec-6:limit}.

\medskip

\noindent\textbf{Step 1: Compactness and Skorokhod representation.}
Let $B=(B^k)_{1\le k\le K}$ denote the Brownian motions driving the
Galerkin equation.  By the tightness result above and the tightness of
Brownian laws on $C([0,T];\R^K)$, the laws of
$(f_m,\nabla_v f_m,B)$ are tight on
\begin{align*}
\mathbb X\times C([0,T];\R^K),
\end{align*}
where
\begin{align*}
\mathbb{X}
:=
\Big(L^2([0,T];L^2_{\loc}(\R^d))
\cap C^\beta([0,T];H^{-l}_{\loc}(\R^d))\Big)
\times\big(L^2([0,T];L^2(\R^d)),w\big).
\end{align*}
The Skorokhod--Jakubowski representation theorem gives a stochastic basis
$(\bar\Omega,\bar{\mathcal F},\bar{\mathbb P})$ and random variables
$(\bar f_m,\nabla_v\bar f_m,\bar B_m)$ and
$(\bar f,\nabla_v\bar f,\bar B)$ such that, along a subsequence,
\begin{align}
\label{convergence-1}
(\bar f_m,\nabla_v\bar f_m,\bar B_m)
\to
(\bar f,\nabla_v\bar f,\bar B)
\quad
\bar{\mathbb P}\hbox{-a.s. in }\mathbb X\times C([0,T];\R^K),
\end{align}
and, for each $m$,
\begin{equation}\label{identityinlaws}
(\bar f_m,\nabla_v\bar f_m,\bar B_m)
\overset{d}{=}
(f_m,\nabla_vf_m,B).
\end{equation}
In particular, almost surely,
\begin{align}
\label{sec3:strong-local}
\bar f_m\to\bar f
\quad\hbox{strongly in }L^2([0,T];L^2_{\loc}(\R^d)),
\qquad
\nabla_v\bar f_m\rightharpoonup\nabla_v\bar f
\quad\hbox{weakly in }L^2_{t,v}.
\end{align}
Here the second component of the limit is identified with the distributional
velocity gradient of $\bar f$: for every
$\varphi\in C_c^\infty((0,T)\times\R^d)$, the strong local convergence of
$\bar f_m$ and the weak convergence of $\nabla_v\bar f_m$ give
\begin{align*}
\int_0^T\int_v \bar f\,\partial_{v_i}\varphi\,\dd v\dd t
=\lim_{m\to\infty}\int_0^T\int_v \bar f_m\,\partial_{v_i}\varphi\,\dd v\dd t
=-\lim_{m\to\infty}\int_0^T\int_v \partial_{v_i}\bar f_m\,\varphi\,\dd v\dd t,
\end{align*}
which equals $-\int_0^T\int_v \partial_{v_i}\bar f\,\varphi\,\dd v\dd t$.
The uniform estimate of Lemma~\ref{lem:time-L2} is preserved under equality in
law and gives
\begin{align}
\label{sec3:uniform-L2}
\sup_m\bar{\mathbb E}\Big[
\sup_{t\le T}\|\bar f_m(t)\|_{L^2_v}^2
+\alpha\int_0^T\|\nabla_v\bar f_m(s)\|_{L^2_v}^2\,\dd s
\Big]<+\infty .
\end{align}

Let
\begin{align*}
\bar{\mathcal G}_t^0
:=\sigma\big(\bar f|_{[0,t]},\nabla_v\bar f|_{[0,t]},\bar B|_{[0,t]}\big),
\qquad
\bar{\mathcal G}_t
:=\hbox{the usual augmentation of }\bar{\mathcal G}_t^0 .
\end{align*}
The identity in law
\eqref{identityinlaws}, the convergence \eqref{convergence-1}, and Vitali's
theorem imply that $\bar B=(\bar B^1,\ldots,\bar B^K)$ is a family of
independent Brownian motions with respect to
$(\bar{\mathcal G}_t)_{t\in[0,T]}$.  Indeed, for every bounded continuous
functional $\Theta$ of the paths up to time $s$, set
\begin{align*}
\Theta_m:=\Theta\big(\bar f_m|_{[0,s]},\nabla_v\bar f_m|_{[0,s]},\bar B_m|_{[0,s]}\big).
\end{align*}
The identities
\begin{align*}
\bar{\mathbb E}\big[\Theta_m(\bar B_m^k(t)-\bar B_m^k(s))\big]=0,
\end{align*}
and
\begin{align*}
\bar{\mathbb E}\big[
\Theta_m\big((\bar B_m^{k_1}(t)-\bar B_m^{k_1}(s))
(\bar B_m^{k_2}(t)-\bar B_m^{k_2}(s))
-(t-s)\delta_{k_1k_2}\big)
\big]=0
\end{align*}
give the martingale and covariance characterisation of Brownian motion after
letting $m\to\infty$.  The passage to the limit follows from the almost sure
convergence in \eqref{convergence-1}, the boundedness of $\Theta$, and the
uniform integrability of the Brownian increments.

\medskip

\noindent\textbf{Step 2: Passing to the limit in the deterministic terms.}
Fix $\phi\in C_c^\infty(\R^d)$ and choose $R>1$ so large that
$\supp\phi\subset B_R$.  Since $\sigma,\sigma',\varsigma$, $\mu^2$, and
$\tili\circ\mu^2$ are bounded and Lipschitz on $\R$, \eqref{sec3:strong-local} yields
\begin{align}
\label{sec3:sigma-strong}
&\sigma(\bar f_m)\to\sigma(\bar f),\quad
\sigma^2(\bar f_m)\to\sigma^2(\bar f),\quad
\varsigma(\bar f_m)\to\varsigma(\bar f),\notag\\
&\mu^2(\bar f_m)\to\mu^2(\bar f),\quad
\tili(\mu^2(\bar f_m))\to\tili(\mu^2(\bar f))
\end{align}
strongly in $L^2([0,T]\times B_R)$, and also strongly in $L^1$ on every
bounded cylinder.  Since $A,a,b,F_1,F_4$ are smooth and compactly supported,
all product variables are restricted to a fixed bounded set depending only on
$R$ and on the support of the kernels.

For the artificial viscosity term, the weak convergence of $\bar f_m$ gives
\begin{align*}
\int_0^t\int_v\bar f_m\,\Delta\Pi_m\phi\,\dd v\dd s
\to
\int_0^t\int_v\bar f\,\Delta\phi\,\dd v\dd s .
\end{align*}
For the Landau drift, using the weak formulation in Definition
\ref{def:weak-1}, we have
\begin{align*}
&\int_0^t\int_{v,v_*}
\mu^2(\bar f_{m,*})\tili(\mu^2(\bar f_m))
a_{ij}(v-v_*)\,\partial_{ij}\Pi_m\phi(v)\,\dd v\dd v_*\dd s\\
&\qquad\longrightarrow
\int_0^t\int_{v,v_*}
\mu^2(\bar f_*)\tili(\mu^2(\bar f))
a_{ij}(v-v_*)\,\partial_{ij}\phi(v)\,\dd v\dd v_*\dd s,
\end{align*}
because $\tili(\mu^2(\cdot))$ is Lipschitz and bounded, and the two
nonlinear factors converge strongly on the relevant compact set.  The same
argument gives
\begin{align*}
&\int_0^t\int_{v,v_*}
\mu^2(\bar f_m)\tili(\mu^2(\bar f_{m,*}))
b_i(v-v_*)\big(\partial_i\Pi_m\phi(v)
-\partial_i\Pi_m\phi(v_*)\big)\,\dd v\dd v_*\dd s\\
&\qquad\longrightarrow
\int_0^t\int_{v,v_*}
\mu^2(\bar f)\tili(\mu^2(\bar f_*))
b_i(v-v_*)\big(\partial_i\phi(v)-\partial_i\phi(v_*)\big)
\,\dd v\dd v_*\dd s.
\end{align*}

It remains only to record the convergence of the deterministic
Stratonovich-to-It\^o correction.  The convergence $\Pi_m\phi\to\phi$ in $C^2$, the strong convergence
\eqref{sec3:sigma-strong}, and the boundedness of $\sigma'$ imply
\begin{align*}
&\int_0^t\int_{v,v_*,w}
\big(\tn(\Pi_m\phi)(v)\cdot F_1(v,v_*,w)\big)
\sigma(\bar f_{m,*})\sigma(\bar f_{m,w})
\sigma(\bar f_m)\sigma'(\bar f_m)\\
&\to \int_0^t\int_{v,v_*,w}
\big(\tn(\phi)(v)\cdot F_1(v,v_*,w)\big)
\sigma(\bar f_{*})\sigma(\bar f_{w})
\sigma(\bar f)\sigma'(\bar f),\\
&\int_0^t\int_{v,v_*,w}
\nabla_v\cdot\big(\tn(\Pi_m\phi)(v)F_4(v,v_*,w)\big)
\sigma(\bar f_{m,*})\sigma(\bar f_{m,w})
\varsigma(\bar f_m)\\
&\to \int_0^t\int_{v,v_*,w}
\nabla_v\cdot\big(\tn(\phi)(v)F_4(v,v_*,w)\big)
\sigma(\bar f_{*})\sigma(\bar f_{w})
\varsigma(\bar f).
\end{align*}

Thus, the residual obtained from the
Galerkin equation converges to the residual associated with the limiting weak
formulation.

\medskip

\noindent\textbf{Step 3: Identification of the martingale part.}
For each $\phi\in C_c^\infty(\R^d)$ define
\begin{align*}
H_{m,k}^\phi(t)
:=
\int_{v,v_*}
\tn(\Pi_m\phi)(v)\cdot G_k(v,v_*)
\sigma(\bar f_m(t,v))\sigma(\bar f_m(t,v_*))\,\dd v\dd v_* .
\end{align*}
Then, by the equality of laws and the Galerkin equation, the martingale
residual satisfies
\begin{align*}
\bar M_m^\phi(t)
=-\frac{\sqrt{\eps}}{2}\sum_{k=1}^K\int_0^tH_{m,k}^\phi(s)\,\dd\bar B_m^k(s).
\end{align*}
Let
\begin{align*}
H_k^\phi(t)
:=
\int_{v,v_*}
\tn\phi(v)\cdot G_k(v,v_*)
\sigma(\bar f(t,v))\sigma(\bar f(t,v_*))\,\dd v\dd v_* .
\end{align*}
The compact support of $G_k$ and $\phi$, together with
\eqref{sec3:sigma-strong}, gives
\begin{align}
\label{sec3:H-conv}
H_{m,k}^\phi\to H_k^\phi
\quad\hbox{strongly in }L^2(0,T),\qquad 1\le k\le K.
\end{align}
Indeed, the difference is bounded by the sum of the projection error
$\|\Pi_m\phi-\phi\|_{C^1}$ times the uniform $L^2$ bounds of
$\sigma(\bar f_m)$, and two terms with one factor
$\sigma(\bar f_m)-\sigma(\bar f)$, which vanish by
\eqref{sec3:sigma-strong}.

Let $\bar M^\phi$ be the limit residual obtained in Step~$2$.  The convergence
of all deterministic terms implies
\begin{align*}
\bar M_m^\phi(t)\to\bar M^\phi(t)
\quad\hbox{for every }t\in[0,T]\quad\bar{\mathbb P}\hbox{-a.s.}
\end{align*}
Let $0\le s\le t\le T$ and let $\Theta$ be a bounded continuous functional
of $(\bar f,\nabla_v\bar f,\bar B)$ restricted to $[0,s]$.  Passing to the
limit in the martingale identity for $\bar M_m^\phi$ yields
\begin{align*}
\bar{\mathbb E}\Big[
\Theta\big(\bar M^\phi(t)-\bar M^\phi(s)\big)
\Big]=0,
\end{align*}
so $\bar M^\phi$ is a continuous $\bar{\mathcal G}_t$-martingale.
Moreover, It\^o's product formula for the Galerkin martingale gives, for
each $1\le j\le K$,
\begin{align*}
\bar{\mathbb E}\Big[
\Theta_m\Big(
\bar M_m^\phi(t)\bar B_m^j(t)-\bar M_m^\phi(s)\bar B_m^j(s)
+\frac{\sqrt{\eps}}{2}\int_s^tH_{m,j}^\phi(r)\,\dd r
\Big)\Big]=0.
\end{align*}
Using \eqref{sec3:H-conv} and \eqref{convergence-1}, we obtain
\begin{align}
\label{sec3:cross-bracket}
\langle \bar M^\phi,\bar B^j\rangle_t
=-\frac{\sqrt{\eps}}{2}\int_0^tH_j^\phi(r)\,\dd r .
\end{align}
Similarly,
\begin{align*}
\bar{\mathbb E}\Big[
\Theta_m\Big(
(\bar M_m^\phi(t))^2-(\bar M_m^\phi(s))^2
-\frac{\eps}{4}\int_s^t\sum_{k=1}^K|H_{m,k}^\phi(r)|^2\,\dd r
\Big)\Big]=0
\end{align*}
passes to the limit, and therefore
\begin{align}
\label{sec3:quad-bracket}
\langle \bar M^\phi\rangle_t
=\frac{\eps}{4}\int_0^t\sum_{k=1}^K|H_k^\phi(r)|^2\,\dd r .
\end{align}
Equations \eqref{sec3:cross-bracket}--\eqref{sec3:quad-bracket} identify
the martingale completely.  Indeed, set
\begin{align*}
N^\phi(t):=\bar M^\phi(t)
+\frac{\sqrt{\eps}}{2}\sum_{k=1}^K\int_0^tH_k^\phi(s)\,\dd\bar B^k(s).
\end{align*}
Using the Brownian covariance, \eqref{sec3:cross-bracket}, and
\eqref{sec3:quad-bracket}, we get $\langle N^\phi\rangle_t=0$ for all $t$.
Since $N^\phi(0)=0$, it follows that $N^\phi\equiv0$, namely
\begin{align}
\label{sec3:martingale-characterisation}
\bar M^\phi(t)
=-\frac{\sqrt{\eps}}{2}\sum_{k=1}^K\int_0^t\int_{v,v_*}
\tn\phi(v)\cdot G_k(v,v_*)\sigma(\bar f)\sigma(\bar f_*)\,
\dd v\dd v_*\,\dd\bar B^k .
\end{align}

\medskip

\noindent\textbf{Step 4: Conclusion.}
Combining Step~$2$ with \eqref{sec3:martingale-characterisation}, we conclude
that $(\bar f,(\bar B^k)_{1\le k\le K})$ satisfies the weak formulation
\eqref{weak:app-1} for every $\phi\in C_c^\infty(\R^d)$ and every
$t\in[0,T]$, outside a null set independent of $t$ after using the
continuity of all terms.  The lower semicontinuity of the $L^2$ and
$H^1$ bounds in \eqref{sec3:uniform-L2} gives
\begin{align*}
\bar f\in L^\infty(0,T;L^2(\R^d))
\cap L^2(0,T;H^1(\R^d))
\quad\bar{\mathbb P}\hbox{-a.s.}
\end{align*}
Finally, since $f_0\ge0$ and the limiting equation has the same weak
structure as the approximating equation, the negative-part argument carried
out in the proof of non-negativity below applies to $\bar f$ and gives
$\bar f\ge0$ almost everywhere.  Hence $\bar f$ is a probabilistic weak
solution of \eqref{ito-epsi-1} in the sense of Definition~\ref{def:weak-1}.
This completes the proof.

\section{Non-negativity, moment estimates, and entropy dissipation estimates}
\label{sec:L-1:functional}
In this section, we establish non-negativity, conservation of mass and momentum, the energy inequality, and entropy dissipation estimates for the solutions constructed for \eqref{ito-epsi-1} in Proposition \ref{lem:app-1:ext}. We spell out the precise equation used throughout the section, since all estimates below are consequences of this fixed regularised It\^o formulation. In weak form, for each smooth test function $\phi$ with at most quadratic growth,
\begin{equation}
\label{app:eq:weak-cl}
\begin{aligned}
\langle f_t,\phi\rangle
=&\,\langle f_0,\phi\rangle
+\alpha\int_0^t\langle f_s,\Delta \phi\rangle\,\dd s
-\frac12\int_0^t\int_{v,v_*}\tn\phi\cdot A\mu^2(f_s)\mu^2(f_{s,*})\tn L(\mu^2(f_s))\,\dd s\\
&+\frac{\sqrt{\eps}}{2}\sum_{k=1}^K\int_0^t\int_{v,v_*}\tn\phi\cdot G_k(v,v_*)\sigma(f_s)\sigma(f_{s,*})\,\dd B_s^k\\
&+\frac{\eps}{2}\int_0^t\int_{v,v_*,w}\sigma(f_{s,*})\sigma(f_{s,w})
\big(\nabla_v\phi-\nabla_{v_*}\phi_*\big)\cdot
\Big(\sigma'(f_s)\sigma(f_s)F_1+F_4\nabla_v\varsigma(f_s)\Big)\,\dd s .
\end{aligned}
\end{equation}
Equivalently, and only as a compact formal notation, this is the equation
\begin{align}\label{app:eq:cl}
\partial_t f
&= \alpha \Delta f + \frac{1}{2} \tn \cdot \big( A \mu^2(f) \mu^2(f_*) \tn L(\mu^2(f))\big) - \frac{\sqrt{\eps}}{2} \tn \cdot \big( A^{1/2} \sigma(f) \sigma(f_*) \xi_K \big)\notag   \\
&\quad + \frac{\eps}{2} \sum_{k=1}^K \tn \cdot \Big( G_k(v,v_*) \sigma'(f) \sigma(f_*) \tn \cdot \big( G_k(v,w) \sigma(f) \sigma(f_w) \big) \Big), 
\end{align}

We define the mass and energy by
\begin{align*}
   m_t := \int_{v} f_t, 
   \qquad 
   E_t := \int_{v} |v|^2 f_t.
\end{align*}

Motivated by the structure of the regularised Landau equation, we introduce the approximated $(\mu,L)$-entropy.
\begin{definition}
   \label{def:H-h-sigma}
Let $L\in C(\R;\R)$ and $\mu\in C(\R;\R_+)$. We define the $(\mu,L)$-entropy as follows
\begin{equation}
\mathcal{H}_{\mu,L}(f) := \int_{v} h_{\mu,L}(f(v))\,\dd v\quad\text{and}\quad h_{\mu,L}(s) := \int_0^s L(\mu^2(r))\dd r.
\end{equation}

\end{definition}

The functional is finite from below on the class used in Proposition \ref{lem:l-1:app}. Indeed, the assumptions on $L$ and $\mu$ imply
\begin{equation*}
|h_{\mu,L}(r)|\lesssim r|\log r|\mathbb{1}_{\{r>1\}}+r^2\mathbb{1}_{\{0\leq r\leq 1\}},
\end{equation*}
and therefore $\mathcal H_{\mu,L}(f)>-\infty$ whenever $f\in L^2\cap L\log L$. This estimate will be used again in the passage $n\to\infty$ below.

In the particular case $\mu(t)=\sqrt t$ and $L(r)=\log r$, the $(\mu,L)$-entropy (up to mass $m(f)$) coincides with the Boltzmann entropy  $\mathcal{H}(f) =
\int_{v} f \log f$.
For notation convenience, we define $h=0$ for $s\le 0$ and for $s\ge0$ 
\begin{align*}
    h(s):=\int_0^s\log r=s\log s-s.
\end{align*}
We note that 
\begin{align*}
  \int_v h(f)+m(f)=  \cH(f).
\end{align*}

Finally, let $\cD_{\mu,L}$ denote the instantaneous $(\mu,L)$-entropy dissipation, defined by
\begin{equation}
\label{def:D-h-sigma}
\cD_{\mu,L}(f) := \frac{1}{2}\int_{v,v_*} A \mu^2(f)\mu^2(f_*) \big|\tn L(\mu^2(f))\big|^2 \ge 0.
\end{equation}
The non-negativity follows from $A\ge0$ and the square structure. With this convention the cumulative entropy dissipation on $[0,t]$ is $\int_0^t\cD_{\mu,L}(f_s)\,\dd s$.


In this section, we show the following proposition.

\begin{proposition}\label{lem:l-1:app}
Suppose that the initial datum $f_0$ satisfies
\begin{align*}
   f_0 \in L^2 \cap L^1_2 \cap L \log L(\Do;\R_+).
\end{align*}
Let $A$ satisfy Assumption \ref{ass:A}, let $(\mu,\sigma)$ satisfy Assumption \ref{ass:sigma-R0}, let $L$ satisfy Assumption \ref{ass:L-s0}, and let the family $(G_k)_{k \geq 1}$ satisfy Assumption \ref{ass:G-k:app}. Let $d\ge 2$. Let  $\alpha \in (0,1)$ and $\eps\in(0,1)$ be arbitrarily fixed. Let $(f,(B_k)_{k=1,\dots,K})$, $(\Omega,\mathcal{F},(\mathcal{F}(t)),\mathbb{P})$ be a probabilistic weak solution to \eqref{app:eq:cl} with initial data $f_0$. 

Then $f\ge0$ a.e. and $f\in L^\infty(0,T;L^2(\Do))$ almost surely; more precisely the Galerkin paths and the lower semicontinuity argument from Proposition \ref{lem:app-1:ext} give $\hE\sup_{t\le T}\|f_t\|_{L^2}^2<\infty$. Moreover, almost surely for every $t\in[0,T]$,
\begin{gather}
    \int_v f_t(v)\dd v = \int_v f_0(v)\dd v, \label{cl:mass} \\
    \int_v vf_t(v)\dd v = \int_v vf_0(v)\dd v, \label{cl:momentum} \\
    \int_v |v|^2 f_t(v)\dd v \le \int_v |v|^2 f_0(v)\dd v + 2 \alpha d\, T\, \|f_0\|_{L^1(\Do)}, \label{cl:energy-app}
\end{gather}
and for every $t\in[0,T]$, 
\begin{gather}
    \hE\big[\cH_{\mu,L}(f_t)\big]+\hE\Big[\int_0^t\cD_{\mu,L}(f_s)\,\dd s\Big]\le \hE\big[\cH_{\mu,L}(f_0)\big]+ C\hE\big[\| f_0\|_{L^1}^3+1\big]\label{entropy-ineq:mu-L}
\end{gather}
for some constant $C=C(C_{\sigma_*},C_K,T)>0$.

Furthermore, for all $\lambda>0$, there exists $\eps_0\in (0,1)$ such that if $\eps\in(0,\eps_0)$, then we have for every $t\in[0,T]$, 
\begin{equation}
\label{H:esp:exp}
\begin{aligned}
\hE\exp\left[
\sup_{s\in[0,t]}\left(\lambda\cH_{\mu,L}(f_s)+\frac{\lambda}{2}\int_0^s \cD_{\mu,L}(f_r)\,\dd r\right)\right]
\le C_\lambda\exp\left[\lambda\cH_{\mu,L}(f_0)+C\lambda\big(\| f_0 \|_{L^1}^3+1\big)\right].
\end{aligned}
\end{equation}

More precisely, $\eps_0$ and $C_\lambda$ depend only on $\lambda$, $C_{\sigma_*}$ and $C_K$.

\end{proposition}

The remainder of this section is devoted to proving Proposition \ref{lem:l-1:app}.
We present a cut-off It\^o formula in Section \ref{sec:ito}. 

We show Proposition~\ref{lem:l-1:app} by showing the non-negativity of $f$, conservation laws and energy bounds, and the entropy inequality in Section \ref{sec:4-1}, \ref{sec:4-2} and \ref{sec:4-3}, respectively.

\subsection{A cut-off It\^o formula}\label{sec:ito}

In this subsection, we show a cut-off It\^o formula in Lemma \ref{lem:cutoff-ito}.

For orientation, if one ignores the cut-off terms, the localised formula of
Lemma~\ref{lem:cutoff-ito} reduces formally to the following identity for
\eqref{ito-epsi-1} and any smooth $S:\R\to\R$:
\begin{equation}
    \label{dS-1:alpha}
 \begin{aligned}
\frac{\dd}{\dd t} \int_vS(f)=&-\alpha \int_{v}\nabla S'(f)\cdot \nabla f
-\frac12\int_{v,v_*}\tn S'(f)\cdot\big( A\mu^2\mu_*^2\tn L( \mu^2)\big)\\
&+\frac{\sqrt{\eps}}{2}\sum_{k=1}^K\int_{v,v_*}\tn S'(f)\cdot\big(G_k(v,v_*)\sigma(f)\sigma(f_*)\big)\dd B_k\\
          &+\eps\Big[\frac12\int_{v,v_*,w}S''(f_*)\sigma'(f)\sigma(f)\sigma(f_*)\sigma(f_w)\nabla_{v_*} f_*\cdot F_1\\
          &\qquad+S''(f_*)\sigma(f_*)\sigma(f_w)\nabla_{v_*} f_*\otimes \nabla_v \varsigma(f) F_4\\
           &\qquad +\frac12 S''(f)\sigma(f_*)\sigma(f_w)\nabla_v \sigma^2(f)\cdot F_2\\
&\qquad +S''(f)\sigma^2(f)\sigma(f_*)\sigma(f_w)F_3\Big].
\end{aligned}
\end{equation}
More details of the calculations can be found in Remark \ref{rmk:formal-ito-algebra} below.

To use the It\^o formula on the whole space, we begin by introducing a sequence of cut-off functions and performing It\^o's formula calculations. 
\begin{assumption}
    \label{ass:chi}
    Let $\chi=\chi(v)$ be a smooth cut-off function satisfying $\chi(v)=1$ for $0\le |v|\le 1$ and $\chi(v)=0$ for $|v|\ge 2$. Moreover, assume that $|\chi(v)|\le 1$ and $|\nabla \chi(v)|+|\nabla^2\chi(v)|\le 1$. Fix a sufficiently large constant $D\ge 4d$. For each $m\in\mathbb{N}$, define the rescaled function $\chi_m(v)=\chi(v/m^D)$.

   Then we have 
    \begin{align*}
     |\nabla_v\chi_m(v)|\le m^{-D}\mathbb{1}_{\{m^D\le |v|\le 2m^D\}}
     \quad\text{and}\quad
     |\Delta_v\chi_m(v)|\le m^{-2D}\mathbb{1}_{\{m^D\le |v|\le 2m^D\}}.  
    \end{align*}
\end{assumption}

\begin{lemma}\label{lem:cutoff-ito}
Assume the hypotheses of Proposition~\ref{lem:l-1:app}. Let
$S\in C^2(\R)$ be such that the terms below are finite, and let $\chi$ be a
cut-off function satisfying Assumption~\ref{ass:chi}. Then we have almost
surely, for every $t\in[0,T]$,  
\begin{equation}\label{ito:cut-off}
 \begin{aligned}
\int_{v} \chi(v)S(f_t)=& \int_{v} \chi(v)S(f_0)-\alpha\int^t_0\int_{v}\nabla(\chi(v)S'(f))\cdot\nabla f\dd s\\
&-\frac12\int^t_0\int_{v,v_*}\tn \big(\chi(v)S'(f)\big)\cdot\big( A\mu^2(f)\mu^2(f_*)\tn L( \mu^2(f))\big)\dd s\\
&+\frac{\sqrt{\eps}}{2}\sum_{k=1}^K\int^t_0\int_{v,v_*}\tn \big(\chi(v)S'(f)\big)\cdot\big(G_k(v,v_*)\sigma(f)\sigma(f_*)\big)\dd B_k\\
   &+\eps\big(T_1(f,S,\chi)+T_2(f,S,\chi)\big),
\end{aligned}
\end{equation}
where 
\begin{align*}
T_1(f,S,\chi):=  \frac12\int^t_0\int_{v,v_*,w}&\chi(v_*)S''(f_*) \sigma'(f)\sigma(f)\sigma(f_*)\sigma(f_w)\nabla_{v_*} f_*\cdot F_1\\   &+\chi(v)S''(f)\sigma^2(f)\sigma(f_*)\sigma(f_w)F_3\\
   &+\chi(v_*)S''(f_*)\sigma(f_*)\sigma(f_w)\nabla_{v_*} f_*\otimes \nabla_v \varsigma(f) : F_4\\
   & +\frac12 \chi(v)S''(f)\sigma(f_*)\sigma(f_w)\nabla_v \sigma^2(f)\cdot F_2\dd s,  \\
T_2(f,S,\chi):=-\frac12\int^t_0\int_{v,v_*,w}&\sigma(f_*)\sigma(f_w)\big(S'(f)\nabla_v \chi(v)-S'(f_*)\nabla_{v_*} \chi(v_*)\big)\\
&\quad\cdot\Big( \sigma'(f)\sigma(f)F_1+F_4 \nabla_v \varsigma(f)\Big)\dd s.
\end{align*}
\end{lemma}
\begin{proof}
Since $\chi(v)$ also depends on $v$, applying It\^o's formula (\cite{Krylov}) yields the following identity, almost surely, for every $t\in[0,T]$. The factors $-\eps/4$ and $\eps/8$ below come respectively from the It\^o drift correction and the quadratic variation of the noise with coefficient $\sqrt{\eps}/2$:
\begin{align*}
\int_{v} \chi(v)S(f_t)
=& \int_{v} \chi(v)S(f_0)-\alpha\int^t_0\int_{v}\nabla(\chi(v)S'(f))\cdot\nabla f\dd s\\
&-\frac12\int^t_0\int_{v,v_*}\tn \big(\chi(v)S'(f)\big)\cdot\big( A\mu^2(f)\mu^2(f_*)\tn L( \mu^2(f))\big)\dd s\\
&+\frac{\sqrt{\eps}}{2}\sum_{k=1}^K\int^t_0\int_{v,v_*}\tn \big(\chi(v)S'(f)\big)\cdot\big(G_k(v,v_*)\sigma(f)\sigma(f_*)\big)\dd B_k\\
	&-\frac{\eps}{4}\sum_{k=1}^K\int^t_0\int_{v,v_*}\tn \big(\chi(v)S'(f)\big)\cdot 
	\Big(G_k(v,v_*) \sigma'(f)\sigma(f_*)\tn\cdot\big(G_k(v,w)\sigma(f)\sigma(f_w)\big) \Big)\dd s\\
	&+\frac{\eps}{8} \sum_{k=1}^K\int^t_0\int_{v,v_*}\chi(v)S''(f)\big|\tn\cdot\big(G_k(v,v_*)\sigma(f)\sigma(f_*)\big)\big|^2\dd s.
\end{align*}

A direct computation based on the chain rule and the definition of $\tilde{\nabla}$ shows that 
\begin{align*}
   \tn \big(\chi(v)S'(f)\big)
   =&\Pi_{(v-v_*)^\perp}\big(\chi(v)S''(f)\nabla_v f-\chi(v_*)S''(f_*)\nabla_{v_*} f_*\big)\\
   &+\Pi_{(v-v_*)^\perp}\big(S'(f)\nabla_v \chi(v)-S'(f_*)\nabla_{v_*} \chi(v_*)\big).
\end{align*}

We give the calculation of the two correction terms.  For notational convenience, set
\begin{equation*}
\sigma=\sigma(f),\qquad \sigma_*=\sigma(f_*),\qquad
\sigma_w=\sigma(f_w),\qquad
\mathcal R_k(v):=\tn_{(v,w)}\cdot\big(G_k(v,w)\sigma\sigma_w\big).
\end{equation*}
The antisymmetry of $G_k(v,w)\sigma\sigma_w$ gives
\begin{equation}
\label{localised-chain:div}
\mathcal R_k(v)
=2\int_w\Big((\nabla_v\cdot G_k(v,w))\sigma\sigma_w
+G_k(v,w)\cdot\nabla_v\sigma\,\sigma_w\Big).
\end{equation}
Consequently, the definitions in \eqref{def:F_i} imply
\begin{align}
\sum_{k=1}^K G_k(v,v_*)\mathcal R_k(v)
&=2\int_w\sigma_w\Big(\sigma F_1+F_4\nabla_v\sigma\Big), \label{localised-chain:F14}\\
\sum_{k=1}^K(\nabla_v\cdot G_k(v,v_*))\mathcal R_k(v)
&=2\int_w\sigma_w\Big(\sigma F_3+F_2\cdot\nabla_v\sigma\Big). \label{localised-chain:F23}
\end{align}
Here and below tensor contractions are understood in the natural way.  We split
\begin{equation*}
\tn(\chi S'(f))=\mathcal A_\chi+\mathcal B_\chi,
\end{equation*}
where
\begin{align*}
\mathcal A_\chi
&:=\Pi_{(v-v_*)^\perp}\big(\chi(v)S''(f)\nabla_vf
-\chi(v_*)S''(f_*)\nabla_{v_*}f_*\big),\\
\mathcal B_\chi
&:=\Pi_{(v-v_*)^\perp}\big(S'(f)\nabla_v\chi(v)
-S'(f_*)\nabla_{v_*}\chi(v_*)\big).
\end{align*}
Using \eqref{localised-chain:div}--\eqref{localised-chain:F23}, the chain rules
\begin{equation*}
\nabla_v\sigma=\sigma'(f)\nabla_vf,\qquad
\nabla_v\varsigma(f)=[\sigma'(f)]^2\nabla_vf,
\end{equation*}
and integration by parts in $v$ and $v_*$, the part containing $\mathcal A_\chi$ satisfies
\begin{align*}
	&-\frac{\eps}{4}\sum_{k=1}^K\int_0^t\int_{v,v_*}
	\mathcal A_\chi\cdot G_k(v,v_*)\sigma'(f)\sigma_*\mathcal R_k(v)\,\dd s\\
	&\quad+\frac{\eps}{8}\sum_{k=1}^K\int_0^t\int_{v,v_*}
	\chi(v)S''(f)\big|\tn\cdot(G_k(v,v_*)\sigma\sigma_*)\big|^2\,\dd s\\
	&=\eps T_1(f,S,\chi).
\end{align*}
More explicitly, the terms in which the derivative falls on $G_k(v,v_*)$ give the
$F_1$- and $F_3$-terms in $T_1$ through \eqref{localised-chain:F14} and
\eqref{localised-chain:F23}; those in which it falls on $\sigma$ give the
$F_4$- and $F_2$-terms, since
$\sigma'(f)\nabla_v\sigma=\nabla_v\varsigma(f)$ and
$2\sigma\nabla_v\sigma=\nabla_v\sigma^2(f)$.  The remaining cut-off part is
\begin{align*}
	&-\frac{\eps}{4}\sum_{k=1}^K\int_0^t\int_{v,v_*}
	\mathcal B_\chi\cdot G_k(v,v_*)\sigma'(f)\sigma_*\mathcal R_k(v)\,\dd s\\
	&=-\frac{\eps}{2}\int_0^t\int_{v,v_*,w}\sigma_*\sigma_w
\big(S'(f)\nabla_v\chi(v)-S'(f_*)\nabla_{v_*}\chi(v_*)\big)\\
&\hspace{4cm}\cdot\big(\sigma'(f)\sigma F_1+F_4\nabla_v\varsigma(f)\big)\,\dd s
	=\eps T_2(f,S,\chi).
\end{align*}
Thus all the terms produced by the spatial dependence of $\chi$ are retained in
$T_2$, while the localised version of the unweighted chain-rule correction is
exactly $T_1$.

Combining the above identities with the definitions of $T_1(f,S,\chi)$ and $T_2(f,S,\chi)$, we obtain that, almost surely for every $t\in[0,T]$, 
\begin{equation}
\begin{aligned}
\int_{v} \chi(v)S(f_t)
=&\int_{v} \chi(v)S(f_0)
-\alpha\int^t_0\int_{v}\nabla(\chi(v)S'(f))\cdot\nabla f\dd s\\
&-\frac12\int^t_0\int_{v,v_*}\tn \big(\chi(v)S'(f)\big)\cdot\big( A\mu^2(f)\mu^2(f_*)\tn L( \mu^2(f))\big)\dd s\\
&+\frac{\sqrt{\eps}}{2}\sum_{k=1}^K\int^t_0\int_{v,v_*}\tn \big(\chi(v)S'(f)\big)\cdot\big(G_k(v,v_*) \sigma(f)\sigma(f_*)\big)\dd B_k\\
&+\eps\big(T_1(f,S,\chi)+T_2(f,S,\chi)\big).
\end{aligned}
\end{equation}

This completes the proof.

\end{proof}

\begin{remark}[Formal algebra behind the correction terms]
\label{rmk:formal-ito-algebra}
The proof above is the rigorous form of the It\^o computation: the formula is
applied to the scalar semimartingale
$\int_v\chi(v)S(f_t(v))\,\dd v$, or equivalently to Galerkin
approximations followed by a passage to the limit.  It is useful, however,
to record the pointwise algebra which explains the structure of
$T_1$ and $T_2$.  This calculation is only formal, because the limiting SPDE
is not an equation to which It\^o's formula can be applied pointwise in
$v$.

Suppress the cut-off $\chi$, the artificial diffusion, and the regularisation
of the entropy variable.  For the It\^o form of \eqref{SDE-1}, a formal
application of It\^o's formula gives
\begin{align*}
\dd S(f)=&S'(f)Q(f,f)\dd t
-\frac{\sqrt{\eps}}{2}\sum_{k=1}^KS'(f)\tn\cdot\big(G_k(v,v_*)\sigma(f)\sigma(f_*)\big)\dd B_k\\
	&+\frac{\eps}{4}\sum_{k=1}^KS'(f)\tn\cdot \Big(G_k(v,v_*)\sigma'(f)\sigma(f_*)
	\tn\cdot\big(G_k(v,w)\sigma(f)\sigma(f_w)\big) \Big)\dd t\\
	&+\frac{\eps}{8} \sum_{k= 1}^KS''(f)
\big|\tn\cdot\big(G_k(v,v_*)\sigma(f)\sigma(f_*)\big)\big|^2\dd t.
\end{align*}
After integration in $v$, the two correction terms are transformed by
integration by parts.  For the Stratonovich--It\^o correction one obtains,
for each $k$,
\begin{equation}
        \label{st-it}
\begin{aligned}
&\frac14\int_{v,v_*}\tn S'(f)\cdot \Big(G_k(v,v_*)\sigma'(f)\sigma(f_*)
\tn\cdot\big(G_k(v,w)\sigma(f)\sigma(f_w)\big) \Big)\\
=&{}\frac14\int_{v,v_*}S''(f)G_k(v,v_*)\cdot\nabla_v\sigma(f)\sigma(f_*)
\tn\cdot\big(G_k(v,w)\sigma(f)\sigma(f_w)\big)\\
&-\frac14\int_{v,v_*}S''(f_*)\nabla_{v_*} f_*\cdot G_k(v,v_*)\sigma'(f)\sigma(f_*)
\tn\cdot\big(G_k(v,w)\sigma(f)\sigma(f_w)\big).
\end{aligned}
\end{equation}
On the other hand, the quadratic variation term gives
\begin{align*}
&\frac18 \int_vS''(f)
\big|\tn\cdot\big(G_k(v,v_*)\sigma(f)\sigma(f_*)\big)\big|^2\\
=&{}\frac14 \int_{v,v_*}S''(f)G_k(v,v_*)\cdot\nabla_v\sigma(f)\sigma(f_*)
\tn\cdot\big(G_k(v,w)\sigma(f)\sigma(f_w)\big)\\
&+\frac14 \int_{v,v_*}S''(f)\big(\nabla_v\cdot G_k(v,v_*)\big)\sigma(f)\sigma(f_*)
\tn\cdot\big(G_k(v,w)\sigma(f)\sigma(f_w)\big).
\end{align*}
The first terms on the right-hand sides cancel with opposite signs in the
It\^o identity.  Consequently, formally,
\begin{equation}
    \label{dS-00}
 \begin{aligned}
\frac{\dd}{\dd t} \int_vS(f)= 
&-\frac12\int_{v,v_*}\tn S'(f)\cdot\big( A\mu^2\mu_*^2\tn \log ( \mu^2)\big)\\
&+\frac{\sqrt{\eps}}{2}\sum_{k= 1}^K\int_{v,v_*}\tn S'(f)\cdot\big(G_k(v,v_*)\sigma(f)\sigma(f_*)\big)\dd B_k\\
&+\frac{\eps}{4}\sum_{k= 1}^K\int_{v,v_*}S''(f_*)\nabla_{v_*} f_*\cdot G_k(v,v_*)\sigma'(f)\sigma(f_*)
\tn\cdot\big(G_k(v,w)\sigma(f)\sigma(f_w)\big)\\
&+\frac{\eps}{4} \sum_{k= 1}^K\int_{v,v_*}S''(f)\big(\nabla_v\cdot G_k(v,v_*)\big)\sigma(f)\sigma(f_*)
\tn\cdot\big(G_k(v,w)\sigma(f)\sigma(f_w)\big).
\end{aligned}
\end{equation}
Finally, using the antisymmetry of
$G_k(v,w)\sigma(f)\sigma(f_w)$,
\begin{align*}
\tn\cdot\big(G_k(v,w)\sigma(f)\sigma(f_w)\big)
=2\int_w\Big[
(\nabla_v\cdot G_k(v,w))\sigma(f)\sigma(f_w)
 +G_k(v,w)\cdot\nabla_v\sigma(f)\sigma(f_w)
\Big],
\end{align*}
and the definitions of $F_1,\ldots,F_4$ in \eqref{def:F_i}, the last two
lines of \eqref{dS-00} become
\begin{equation}
    \label{dS-1}
 \begin{aligned}
&\eps\Big[\frac12\int_{v,v_*,w}S''(f_*)\sigma'(f)\sigma(f)\sigma(f_*)\sigma(f_w)
\nabla_{v_*} f_*\cdot F_1\\
&\quad+\int_{v,v_*,w}S''(f_*)\sigma(f_*)\sigma(f_w)
\nabla_{v_*} f_*\otimes \nabla_v \varsigma(f):F_4\\
&\quad+\frac12\int_{v,v_*,w}S''(f)\sigma(f_*)\sigma(f_w)
\nabla_v \sigma^2(f)\cdot F_2\\
&\quad+\int_{v,v_*,w}S''(f)\sigma^2(f)\sigma(f_*)\sigma(f_w)F_3\Big] .
\end{aligned}
\end{equation}
Formula \eqref{ito:cut-off} is precisely the localised and rigorous version
of this algebra, with the additional cut-off contribution collected in
$T_2$.
\end{remark}

\subsection{Non-negativity of solutions}\label{sec:4-1}

Let $f^-:=\min(f,0)$.
We use a smooth approximation of the negative part in order to justify the chain rule. Let $S_\epsilon^-\in C^2(\R)$ be convex, non-negative, and such that
\begin{equation*}
S_\epsilon^-(r)\to \frac12(r^-)^2,\qquad
(S_\epsilon^-)'(r)\to r^-,\qquad
0\le (S_\epsilon^-)''(r)\le \mathbb{1}_{\{r\le\epsilon\}}
\end{equation*}
pointwise. Moreover, let
$(S_\epsilon^-)'(r)=0$ for $r\ge\epsilon$ and $(S_\epsilon^-)''(r)=\mathbb{1}$ for $r\le 0$. We apply the cut-off It\^o's formula \eqref{ito:cut-off} to the solution $f$ with cut-off function $\chi_m(v)$ given as in Assumption \ref{ass:chi}. Then, we send $m\to\infty$ and $\epsilon\downarrow0$ to show that $f^-=0$.

 By \eqref{ito:cut-off}, we have 
\begin{equation}
\label{weak-+}
 \begin{aligned}
&\int_vS_\epsilon^-(f_t)\chi_m(v)-\int_vS_\epsilon^-(f_0)\chi_m(v)+\alpha\int_0^t\int_v\nabla_v\big((S_\epsilon^-)'(f)\chi_m(v)\big)\cdot \nabla_v f\\
={}&-\int_0^t\int_{v} \big(\nabla_v \big((S_\epsilon^-)'(f)\chi_m(v)\big)\big)^T\bar a (\mu^2)\nabla_v \tili(\mu^2(f))\\
&
-\int_0^t\int_{v} \nabla_v\big((S_\epsilon^-)'(f)\chi_m(v)\big)\cdot  \bar b(\tili(\mu^2(f)))\mu^2(f)\\
&+\frac{\sqrt{\eps}}{2}\sum_{k=1}^K\int_0^t\int_{v,v_*}\tn \big((S_\epsilon^-)'(f)\chi_m(v)\big)\cdot\big(G_k(v,v_*)\sigma(f)\sigma(f_*)\big)\dd B_k\\
          &+\epsilon\big(T_1(f,S_\epsilon^-,\chi_m)+T_2(f,S_\epsilon^-,\chi_m)\big).
\end{aligned}
\end{equation}

\medskip

\noindent\textbf{The diffusion term:} We note that 
\begin{align*}
    &\int_v\nabla_v\big((S_\epsilon^-)'(f)\chi_m(v)\big)\cdot \nabla_v f\\
    ={}&\int_{v}(S_\epsilon^-)''(f)|\nabla_v f|^2\chi_m+\int_{v}(S_\epsilon^-)'(f)\nabla_v f\cdot \nabla_v\chi_m.
\end{align*}
By using of $|(S_\epsilon^-)'(f)|\le |f|$, the second term on the right-hand side is bounded by 
\begin{align}
\label{est:chi-S-2}
  \Big|\int_0^t\int_{v}(S_\epsilon^-)'(f)\nabla_v f\cdot \nabla_v\chi_m\Big|\lesssim_T m^{-D}\big(\|f\|_{L^\infty_tL^2_v}+\|\nabla_v f\|_{L^2_tL^2_v}\big).
\end{align}

\medskip

\noindent\textbf{The Landau drift terms:} 
Concerning $\bar a$ term,  we have 
\begin{align*}
    &\int_{v} \big(\nabla_v \big((S_\epsilon^-)'(f)\chi_m(v)\big)\big)^T\bar a (\mu^2)\nabla_v \tili(\mu^2(f))\\
    ={}&\int_{v} (S_\epsilon^-)''(f)\chi_m(v)(\nabla_v f)^T\bar a (\mu^2)\nabla_v \tili(\mu^2(f))\\
    &+\int_{v} (S_\epsilon^-)'(f)\nabla_v\chi_m(v)\bar a (\mu^2)\nabla_v \tili(\mu^2(f)),
\end{align*}
where the first term on the right-hand side is non-negative because $\bar a(\mu^2)$ is non-negative and $\tili$ is monotonically increasing. By using of $|(S_\epsilon^-)'(f)|\le |f|$ and  $|\nabla_v \tili(\mu^2(f)) |\lesssim_{C_{\mu_*},C_{\mu^*}}|\nabla_v f|$, the second term is bounded by 
\begin{align}
\label{est:chi-S-1}
   &\Big|\int_0^t\int_{v} (S_\epsilon^-)'(f)\nabla_v\chi_m(v)\bar a (\mu^2)\nabla_v \tili(\mu^2(f))\Big|\\
   \lesssim&{}_{A_0,C_{\mu_*},C_{\mu^*},T}m^{-D}\big(\|f\|_{L^\infty_tL^2_{v}}+\|\nabla_v f\|_{L^2_tL^2_{v}}\big).\notag
\end{align}

Concerning the $\bar b$ term, we have 
\begin{align*}
&\Big|\int_0^t\int_{v} \nabla_v\big((S_\epsilon^-)'(f)\chi_m(v)\big)\cdot  \bar b(\tili(\mu^2(f)))\mu^2(f)\Big|\\
\le&{}\Big|\int_0^t\int_{v} \chi_m(S_\epsilon^-)''(f)\nabla_v f \cdot  \bar b(\tili(\mu^2(f)))\mu^2(f)\Big|+\Big|\int_0^t\int_{v} (S_\epsilon^-)'(f)\nabla_v\chi_m(v)\cdot  \bar b(\tili(\mu^2(f)))\mu^2(f)\Big|.
\end{align*}
The first term on the right-hand side is controlled by the same Young
inequality used in the $L^2$ energy estimate: for every $\delta>0$, we have
\begin{align*}
    &\Big|\int_0^t\int_{v} \chi_m(S_\epsilon^-)''(f)\nabla_v f\cdot  \bar b(\tili(\mu^2(f)))\mu^2(f)\Big|\\
    \lesssim&{}_{A_0,z^*}\int_0^t\int_{v} \sqrt{\chi_m(S_\epsilon^-)''(f)}|\nabla_v f|\big(|f^-|^2+f^2\mathbb{1}_{\{0\le f\le\epsilon\}}\big)\\
    \le&{} \delta\int_0^t\int_{v}(S_\epsilon^-)''(f)|\nabla_v f|^2\chi_m+
C_\delta\int_0^t\| f^-\|_{L^2}^2+\epsilon^2T\| f\|_{L^\infty_tL^2_v}^2.
\end{align*}
The bound $\mu^2(f)\le f^2$ ensures the following bounds for the second term 
\begin{align*}
\Big|\int_0^t\int_{\{0\le f\le\epsilon\}} \nabla_v f\cdot  \bar b(\tili(\mu^2(f)))\mu^2(f)\Big|
\lesssim_{A_0,z^*,T } \epsilon^2\|\nabla_v f\|_{L^2_tL^2_v}.
\end{align*}
To treat the third term, we repeat the arguments as in  \eqref{est:chi-S-1} to derive 
\begin{align*}
\Big|\int_0^t\int_{v} (S_\epsilon^-)'(f)\nabla_v\chi_m(v)\cdot  \bar b(\tili(\mu^2(f)))\mu^2(f)\Big|
\lesssim_{A_0,z^*,T} m^{-D}\big(\|f\|_{L^\infty_tL^2_{v}}+\|\nabla_v f\|_{L^2_tL^2_{v}}\big).
\end{align*}

\medskip

\noindent\textbf{The error term $T_1$:} 
Concerning the error term $T_1$, we use the following fact that
\begin{align*}
 f\mathbb{1}_{\{f\le \epsilon\}}=f^-+  f \mathbb{1}_{\{0\le f\le \epsilon\}}.
\end{align*}
The facts  $\sigma(r)=0$ for $r\le0$ and $|\sigma(r)|+|\varsigma(r)|\lesssim |r|$ imply that
\begin{align*}
    &\left|\int^t_0\int_{v,v_*,w}\chi_m(S_\epsilon^-)''(f_*)\sigma'(f)\sigma(f)\sigma(f_*)\sigma(f_w)\nabla_{v_*} f_*\cdot F_1\dd s\right|\\
\le &{}\delta\int_0^t\int_{v}(S_\epsilon^-)''(f)|\nabla_v f|^2\chi_m+
C_\delta\int_0^t\| f^-\|_{L^2}^2+\epsilon^2\| f\|_{L^\infty_tL^2_v}^2
\end{align*}
for some $C=C(T,C_{\sigma^*},C_K)$. The same bounds also hold for $F_2,F_3,F_4$ terms. Hence,  we have 
\begin{align*}
|T_1(f,S_\epsilon^-,\chi_m)|\le \delta\int_0^t\int_{v}(S_\epsilon^-)''(f)|\nabla_v f|^2\chi_m+
C_\delta\int_0^t\| f^-\|_{L^2}^2+\epsilon^2\| f\|_{L^\infty_tL^2_v}^2.
\end{align*}

\medskip

\noindent\textbf{The error term $T_2$:} 
Concerning the $T_2$ term, we write 
\begin{align*}
T_2&=T_{2,1}+T_{2,2}\\
&:=\frac12\int^t_0\int_{v,v_*,w}\sigma(f_*)\sigma(f_w)\big((S_\epsilon^-)'(f)\nabla_v \chi_m(v)-(S_\epsilon^-)'(f_*)\nabla_{v_*} \chi_m(v_*)\big)\cdot  \sigma'(f)\sigma(f)F_1\dd s\\
&+\frac12\int^t_0\int_{v,v_*,w}\sigma(f_*)\sigma(f_w)\big((S_\epsilon^-)'(f)\nabla_v \chi_m(v)-(S_\epsilon^-)'(f_*)\nabla_{v_*} \chi_m(v_*)\big)\cdot F_4 \nabla_v \varsigma(f)\dd s.
\end{align*}
The boundedness of $\sigma$ and $\sigma'$, and $|(S_\epsilon^-)'(f)|\le |f|$ ensure that  
\begin{align*}
|T_{2,1}|    &\lesssim_{C_K,C_{\sigma_*},C_{\sigma^*},T}\|\nabla\chi_m\|_{L^{\infty}}\|f\|_{L^\infty_tL^2_{v}}.
\end{align*}
By integration by parts, we have  
\begin{align*}
   |T_{2,2}|&=\Big|\int^t_0\int_{v,v_*,w}\nabla_v\cdot\big(\big(\nabla_v \chi_m(v) (S_\epsilon^-)'(f)-\nabla_{v_*} \chi_m (v_*)(S_\epsilon^-)'(f_*)\big) F_4\big) \sigma(f_*)\sigma(f_w)\varsigma(f)\dd s\Big|\\
   &\lesssim_{C_{\sigma_*},C_K,C_{\sigma^*}} \|\nabla\chi_m\|_{W^{1,\infty}}\int^t_0\|f\|_{L^2}+\|\nabla f\|_{L^2}\dd s.
\end{align*}
Combining $T_{2,1}$ and $T_{2,2}$ terms, we have 
\begin{align*}
|T_{2}|    &\lesssim_{C_K,C_{\sigma_*},C_{\sigma^*},T}m^{-D}\big(\|f\|_{L^\infty_tL^2_{v}}+\|\nabla_v f\|_{L^2_tL^2_{v}}\big).
\end{align*}    

\medskip

\noindent\textbf{The stochastic noise term:} 
The stochastic integral in \eqref{weak-+} is a true martingale after the standard stopping-time localisation. Its quadratic variation is bounded by
\begin{align*}
&\sum_{k=1}^K\int_0^t
\Big|\int_{\{|v-v_*|\le R\}}\tn\big((S_\epsilon^-)'(f)\chi_m(v)\big)\cdot G_k(v,v_*)\sigma(f)\sigma(f_*)\Big|^2\dd s\\
\lesssim&{}_{C_K,T,z^*,C_{\sigma^*}} \|\chi_m\|_{W^{1,\infty}}\big(\|f\|_{L^\infty_tL^2_{v}}+\|\nabla_v f\|_{L^2_tL^2_{v}}\big),
\end{align*}
using the compact support and boundedness assumptions on $G_k$ together with $|\sigma(r)|\le C_{\sigma^*}$. Hence, its expectation is zero.

\medskip 

We combine the above terms.  Choosing $\delta>0$ sufficiently small with respect to $\alpha$ and taking expectations in \eqref{weak-+}, we obtain
\begin{align*}
    &\hE\Big[\int_vS_\epsilon^-(f_t)\chi(v)\Big]-\hE\Big[\int_vS_\epsilon^-(f_0)\chi(v)\Big]\\\lesssim&{}_{A_0,C_{\sigma^*},T}(m^{-D}+\epsilon^2)\hE\big[\|f\|_{L^\infty_tL^2_{v}}+\|\nabla_v f\|_{L^2_tL^2_{v}}+1\big]+\hE\Big[\int_0^t\|f^-(s)\|_{L^2_v}^2\,\dd s\Big].
\end{align*}

We pass to the limit by letting $m\to\infty$, then Fatou's lemma implies that
\begin{align*}
    &\hE\Big[\int_vS_\epsilon^-(f_t)\Big]\le \hE\Big[\int_vS_\epsilon^-(f_0)\Big]+C\epsilon^2\hE\big[\|f\|_{L^\infty_tL^2_{v}}+\|\nabla_v f\|_{L^2_tL^2_{v}}+1\big]+\hE\Big[\int_0^t\|f^-(s)\|_{L^2_v}^2\,\dd s\Big].
\end{align*}
Notice that 
\begin{align*}
2\int_vS_\epsilon^-(f)-\int_{\{0\le f\le\epsilon\}}f^2\le \|f^-\|_{L^2_v}^2\le 2\int_vS_\epsilon^-(f).   
\end{align*}
Let $\epsilon\to0$. We have 
\begin{equation*}
\hE\|f^-(t)\|_{L^2}^2\le C\int_0^t\hE\|f^-(s)\|_{L^2}^2\,\dd s .
\end{equation*}
Since $f_0\ge0$, $f^-(0)=0$. Gr\"onwall's inequality gives $f^-\equiv0$, hence $f\ge0$ almost everywhere.

\subsection{Conservation laws and energy bounds}\label{sec:4-2}



Let $\{\chi_m\}_{m\in\mathbb{N}}$ be a sequence of cut-off functions approximating the constant function $1$, as constructed in Assumption~\ref{ass:chi}. We shall then establish mass and momentum conservation laws and energy bounds. 

\medskip

\noindent\textbf{Mass conservation law \eqref{cl:mass}:}
    We test \eqref{app:eq:cl} with $\chi_m$ to obtain that, almost surely, for every $t\in[0,T]$,
\begin{equation}
\label{cl:test}
\begin{aligned}
&\int_{v} f(t,v)\chi_m
=\int_{v} f_0(v)\chi_m
+\alpha\int_0^t\int_{v} f\,\Delta\chi_m\\
&
+\sum_{i,j=1}^d\int_0^t\int_{v,v_*}\tili(\mu^2(f))\mu^2(f_*) a_{ij}\d_{ij}\chi_m+\sum_{i=1}^d\int_0^t\int_{v,v_*}\tili(\mu^2(f))\mu^2(f_*)b_i(\d_i\chi_m-\d_{i*}\chi_{m*})\\
&
-\frac{\sqrt{\eps}}{2}\int_0^t\int_{v,v_*}\tn\chi_m\cdot\left(\sqrt{A}\sigma(f)\sigma(f_*) \dd \xi_K\right)\\
&-\eps\int_0^t\int_{v,v_*}\sigma (f_*)\sigma(f_w)\Big(\big(\tn \chi_m\cdot F_1 \big)\sigma (f)\sigma'(f)
-\nabla_v\cdot\big(\tn \chi_m F_4\big)\varsigma(f)\Big).
\end{aligned}
\end{equation}

We now show that each term containing derivatives of $\chi_m$ vanishes as $m\to\infty$. By It\^o's isometry, for every $t\in[0,T]$,
\begin{align*}
\hE\Big[\Big|\int^t_0\int_{v}\tn\chi_m\cdot\left(\sqrt{A}\sigma(f)\sigma(f_*) \dd \xi_K\right)\Big|^2\Big]^{\frac12}
&\lesssim \hE\Big[\int^t_0\|\sigma(f)\|_{L^2}^4\|\nabla\chi_m\|_{L^\infty}^2\dd s\Big]^{\frac12}\\
&\lesssim m^{-D}\,\hE\Big[\int_0^t\|f_s\|_{L^1}^2\,\dd s\Big]^{\frac12}\to0.
\end{align*}
Hence, along a subsequence, we obtain that almost surely, for every $t\in[0,T]$,
$$
\int^t_0\int_{v}\tn\chi_m\cdot\left(\sqrt{A}\sigma(f)\sigma(f_*) \dd \xi_K\right)\to0.
$$

Similarly, all remaining terms on the right-hand side can be estimated almost surely and uniformly in $m$ by (up to a constant independent of $m$)
\begin{align*}
&\int^t_0\|f\|_{L^1}\|\Delta \chi_m\|_{L^\infty}
+\big(\|f\|_{L^2}^4+\|F_1\|_{L^1}+\|F_4\|_{W^{1,1}}\big)\|\nabla\chi_m\|_{W^{1,\infty}}\dd s\\
\leq&  m^{-D}\int^t_0\|f\|_{L^1}
+\big(\|f\|_{L^2}^4+\|F_1\|_{L^1}+\|F_4\|_{W^{1,1}}\big)\dd s \to 0,
\end{align*}
as $m\to\infty$, almost surely. Consequently, \eqref{cl:mass} follows from Fatou's lemma together with the non-negativity of $f$.

\medskip

\noindent\textbf{Energy bounds \eqref{cl:energy-app}:} Notice that $\tn |v|^2 = 0$. We define $\eta_m(v) := |v|^2 \chi_m(v)$. A straightforward computation yields
\begin{align*}
\nabla \eta_m &= 2v \chi_m + |v|^2 \nabla \chi_m, \\
\Delta \eta_m &= 2d \chi_m + 2v \cdot \nabla_v \chi_m + |v|^2 \Delta \chi_m, \quad \text{for all } v \in \mathbb{R}^d.
\end{align*}

We test \eqref{app:eq:cl} by $\eta_m$, and \eqref{cl:test} holds with $\chi_m$ replaced by $\eta_m$.
We first show the uniform bound of $\int_v f\eta_m$.

We recall that we denote the mass of $f$ by $m_0$. By the properties of $\eta_m$, we obtain
\begin{align*}
\left| \int_0^t \int_{\mathbb{R}^d} f \, \Delta \eta_m \, \dd s \right|
\le \int_0^t \left( 2d m_0 + C\int_{\{m^D\le |v|\le2m^D\}} f_s(v)\,\dd v \right) \dd s.
\end{align*}

We note that, for all $i,j=1,\dots,d$,
\begin{align*}
\d_i\eta_m=2v_i\chi_m+|v|^2\d_i\chi_m,\quad     \d_{ij}\eta_m=2\delta_{ij}\chi_m+2v_i\d_j\chi_m+2v_j\d_i\chi_m+|v|^2\d_{ij}\chi_m.
\end{align*}
The Assumption \ref{ass:chi}  implies that 
 \begin{align}
\label{eta:bdd-0}
|\d_i\eta_m|\lesssim \langle v\rangle \chi_m+1\lesssim \eta_m+1\quad\text{and}\quad     |\d_{ij}\eta_m|\lesssim 1.
\end{align}
The boundedness  of $a$ and $b$ ensure that
\begin{align}
\label{energy:a-b:bdd}
 \big|a_{ij}\d_{ij}\eta_m+b_i(\d_i\eta_m-\d_{i*}\eta_{m*})\big|\lesssim_{A_0} \eta_m+\eta_{m*}+1.
\end{align}
As a consequence, the Landau collision term is bounded by
\begin{align*}
 \Big|\int_0^t\int_{v}Q_{\mu,L}(f,f)\eta_m(v)\Big|\lesssim
m_0\Big(m_0+\int_0^t\int_vf\eta_m(v)\Big).
\end{align*}

By \eqref{eta:bdd-0}, we have 
\begin{align}
\label{eta:bdd-1}
 \big|\tn\eta_m\big|\lesssim 1+\eta_m+\eta_{m*}.
\end{align}
Concerning the stochastic term, by It\^o's isometry, for every $t\in[0,T]$, we have 
\begin{align*}
&\hE\Big[\Big|\int^t_0\int_{v}\tn\eta_m\cdot\left(\sqrt{A}\sigma(f)\sigma(f_*) \dd \xi_K\right)\Big|^2\Big]^{\frac12}
\lesssim \hE\Big[m_0\Big(m_0+\int_0^t\int_vf\eta_m(v)\dd s\Big)\Big]^{\frac12}.
\end{align*}
With the help of \eqref{eta:bdd-0}, we  have
\begin{align*}
\left| \int_0^t \int_{v} \tn \eta_m \cdot F_1 \sigma(f_w)\sigma(f_*)\sigma(f)\sigma'(f) \, \dd s \right|
\lesssim_{C_K,C_{\sigma_*}}m_0^2 \Big(m_0+\int_0^t\int_v f\eta_m(v)\dd s\Big).
\end{align*}

Notice that $\nabla_v\cdot\Pi_{(v-v_*)^\perp}=-(d-1)\frac{v-v_*}{|v-v_*|^2}$, and hence, 
\begin{equation} \label{cl:eta:badd-2} \begin{aligned} \big|\nabla_v\cdot \tn \eta_m \big|&\le \big|\nabla_v\cdot \Pi_{(v-v_*)^\perp}(\nabla_v\eta_m(v)-\nabla_{v_*}\eta_m(v_*))\big|+|\Delta \eta_m|\\ 
&\lesssim \|\nabla^2\eta_m\|_{L^\infty}\big(\big|\nabla_v\cdot \Pi_{(v-v_*)^\perp}\big||v-v_*|+1\big)\lesssim 1.\end{aligned} 
\end{equation}
By \eqref{eta:bdd-0} and \eqref{cl:eta:badd-2}, the same bound as for $F_1$ term holds
\begin{align*}
\Big|\int_0^t \int_{v} \int_{w} \sigma(f_*) \sigma(f_w) \varsigma(f)
\nabla_v \cdot \big( \tn \eta_m F_4 \big) \, \dd s \Big|
\lesssim_{C_K,C_{\sigma_*}}m_0^2 \Big(m_0+\int_0^t\int_v f\eta_m(v)\dd s\Big).
\end{align*}

Combining the above estimates, we deduce
\begin{align*}
\mathbb{E}\Big[\int_{v} \eta_m f_t\Big]
\le \mathbb{E}\Big[\int_{v} \eta_m f_0\Big]
+ C\mathbb{E}\Big[1+\int_0^t\int_v f\eta_m(v)\dd s\Big]
\end{align*}
for some constant $C=C(m_0,C_K,C_{\sigma_*},A_0)>0$ independent of $m$. 
By Gr\"onwall's inequality and the assumption $f_0 \in L^1_2(v)$, we obtain the uniform bound
\begin{align*}
\sup_{t \in [0,T]} \mathbb{E}\Big[\int_{v} \eta_m f_t\Big]
\le C
\end{align*}
for some constant $C> 0$ independent of $m$. By Fatou's lemma, we have 
\begin{align*}
\sup_{t \in [0,T]} \mathbb{E}\Big[\int_{v} |v|^2 f_t\Big]
\le C.
\end{align*}

Now we show the energy equality \eqref{cl:energy-app} holds. We first show the vanishing of the stochastic term by using a more delicate estimate of $\eta_m$.
By using of  $\Pi_{(v-v_*)^\perp}(v-v_*) = 0$, we have 
\begin{equation}
\label{tn:eta}    
\begin{gathered}
    \tn\eta_m=\Pi_{(v-v_*)^\perp}\big(2v(1-\chi_m(v))-2v_*(1-\chi_m(v_*))+|v|^2\nabla\chi_m(v)-|v_*|^2\nabla_*\chi_m(v_*)\big),\\
 \text{and}\quad    |\tn \eta_m|
\lesssim  |v| \mathbb{1}_{B_{m^D}^c}(v) + |v_*| \mathbb{1}_{B_{m^D}^c}(v_*) +m^D\mathbb{1}_{\{m^D\le |v|\le 2m^D\}}+m^D\mathbb{1}_{\{m^D\le |v_*|\le 2m^D\}}.
\end{gathered}
\end{equation}
By It\^o's isometry, we have
\begin{equation}
\label{energy-sto-0}
\begin{aligned}
&\hE\Big[\Big|\int^t_0\int_{v}\tn\eta_m\cdot\left(\sqrt{A}\sigma(f)\sigma(f_*) \dd \xi_K\right)\Big|^2\Big]^{\frac12}\\
\lesssim&{}_{m_0,A_0}\hE\Big[\int_0^t\int_{\{|v|\ge m^D\}} |v|^2 f\Big]^{\frac12}\\
&\quad+\hE\Big[\int_0^t\int_{\{m^D\le |v|\le2m^D\}} |v|^2 f\Big]^{\frac12}
\to0
\end{aligned}
\end{equation}
as $m\to\infty$. 

We pass to the limit by the dominated convergence theorem for the remaining terms of the equation. We only need to verify the following pointwise bounds
\begin{gather*}
\tn \eta_m(v)\to \tn|v|^2=0\quad\text{and}\quad\\ 
a_{ij}\d_{ij}\eta_m+b_i(\d_i\eta_m-\d_{i*}\eta_{m*})=\nabla_v\cdot\big(A\tn \eta_m\big)\to \nabla_v\cdot\big(A\tn |v|^2\big)=0.
\end{gather*}

\medskip

\noindent\textbf{Momentum conservation law\eqref{cl:momentum}:} For $i=1,\dots,d$, we define $\beta^i_m(v) := v_i \chi_m(v)$. A straightforward computation yields  
 \begin{align*}
\nabla \beta^i_m &= e_i \chi_m + v_i \nabla \chi_m, \\
\Delta \beta^i_m &= \d_i\chi_m + v_i \Delta \chi_m, \quad \text{for all } v \in \mathbb{R}^d.
\end{align*}

We test \eqref{app:eq:cl} by $\beta^i_m$, and \eqref{cl:test} holds with $\chi_m$ replaced by $\beta^i_m$. Similar to \eqref{tn:eta}, we have 
\begin{gather*}
    \tn\beta^i_m=\Pi_{(v-v_*)^\perp}\big(e_i(1-\chi_m(v))-e_i(1-\chi_m(v_*))+v_i\nabla\chi_m(v)-{ v_{i*}}\nabla_*\chi_m(v_*)\big),\\
 \text{and}\quad    |\tn \beta^i_m|
\lesssim  \mathbb{1}_{B_{m^D}^c}(v) +  \mathbb{1}_{B_{m^D}^c}(v_*) +m\|\nabla\chi_m\|_{L^\infty}.
\end{gather*}
Then the stochastic term vanishes, similar to \eqref{energy-sto-0}.  Concerning the remaining terms, the following pointwise convergences hold
\begin{gather*}
\tn \beta^i_m(v)\to \tn v_i=0\quad\text{and}\quad\\ 
a_{ij}\d_{ij}\beta^i_m+b_i(\d_i\beta^i_m-\d_{i*}\beta^i_{m*})=\nabla_v\cdot\big(A\tn \beta^i_m\big)\to \nabla_v\cdot\big(A\tn v_i\big)=0.
\end{gather*}
Similar to showing the energy equality, we use the energy bound and the dominated convergence theorem to pass to the limit to show \eqref{cl:momentum}.





%
%

\subsection{Entropy dissipation estimates}\label{sec:4-3}
In this Section, we show the entropy inequalities \eqref{entropy-ineq:mu-L} and \eqref{H:esp:exp}. 

We will repeatedly use the following proposition.
\begin{proposition}
\label{prop:log-chi}
Let $L$ satisfy Assumption~\ref{ass:L-s0}, let $\mu$ satisfy
Assumption~\ref{ass:sigma-R0}, and let
$f\in L^2(\R^d;\R_+)\cap L^1_2(\R^d)$. Then we have
\begin{align}
&|L(\mu^2(f))|\lesssim_{C_{\mu^*},C_L}1\quad \text{and}\quad |L(\mu^2(f))|\le |\log(\mu^2(f))| \label{log-chi:pw}
\end{align}
for all $v\in\Do$. Furthermore, we have
\begin{align}
\int_{v}\mu^2(f)|L(\mu^2(f))|+\int_{v}f|\log f|\lesssim \|f\|_{L^1_2}+\|f\|_{L^2}. \label{LlogL:sigma}
\end{align}
\end{proposition}
\begin{proof}
The bounds in \eqref{log-chi:pw} follow directly from Assumption \ref{ass:L-s0}. In particular, we have
\begin{align}
      |L(\mu^2(f))|
      &\le \big|L(0)\big|+\log(\mu^2(f))\mathbb{1}_{\{f\ge 1\}}
      \lesssim_{C_L,C_{\mu_*},C_{\mu^*}} 1.
      \label{bdd:L}
\end{align}

We next follow a standard argument, as in \cite{JKO98}, to prove
\begin{equation}
\label{LlogL:sigma:detail}
\begin{aligned}
    \int_{v}\mu^2(f)|\log(\mu^2(f))|
    &\le \int_{\{0\le f\le 1\}}\mu^{2(1-l)}(f)
    + \int_{\{f\ge 1\}}\mu^{2(1+l)}(f)\\
    &\le \int_{\{0\le f\le 1\}} f^{1-l} + \int_{\{f\ge 1\}} f^2\\
    &\le \Big(\int_{v}\langle v\rangle^{-\frac{2(1-l)}{l}}\,\dd v\Big)^{l}
    \Big(\int_{v}\langle v\rangle^2 f\,\dd v\Big)^{1-l}
    + \int_{v} f^2\\
    &\lesssim \|f\|_{L^1_2}+\|f\|_{L^2},
\end{aligned}
\end{equation}
where $l\in(0,1)$ is chosen sufficiently small so that $\langle v\rangle^{-\frac{2(1-l)}{l}}\in L^1(\Do)$.

\end{proof}

Let $(\chi_m)_{m\in\mathbb{N}_+}$ be a sequence of cut-off functions as specified in Assumption \ref{ass:chi}. Recall that $\cH_{\mu,L}$ and $h_{\mu,L}$ are defined in Definition \ref{def:H-h-sigma}. For every $m\in\mathbb{N}_+$, we introduce the following approximation of $\cH_{\mu,L}(f)$:
\begin{align*}
    \cH_{\mu,L,m}(f):=\int_{v} \chi_m(v)\,h_{\mu,L}(f)\,\dd v.
\end{align*}
Correspondingly, we define
\begin{align*}
    \cD_{\mu,L,m}(f):= \frac12\int_{v,v_*}\chi_m(v) A \mu^2(f)\mu^2(f_*)\big|\tn L(\mu^2(f))\big|^2.
\end{align*}

%
%

In Section \ref{sec-4:entropy-1}, we show that for every $m\in\mathbb N$ and every
$t\in[0,T]$, almost surely,
\begin{equation}
    \label{H-theorem:1}
 \begin{aligned}
 &\cH_{\mu,L,m}(f_t)+\int_0^t\cD_{\mu,L,m}(f_s)\,\dd s \\
\le &{}\cH_{\mu,L,m}(f_0)+\sqrt{\eps}M_{m,t}
+C_1m^{-D}\big(\|f\|_{L^\infty_tL^2_v}+1\big)^3\big(\|f\|_{L^2_tH^1_v}+1\big)+\eps C_2\big(\|f_0\|_{L^1}^3+1\big),
\end{aligned}
\end{equation}
where the martingale term $M_{m,t}$ is defined in \eqref{entropy:M-mt} below.
Here, the constants $C_1=C_1(C_{\sigma_*},C_{\sigma^*},C_{\mu_*},C_{\mu^*},C_K,d,A_0,z^*,T)>0$ and $C_2=C_2(C_{\sigma_*},C_K,T)>0$ are independent of $m$.  

In Section \ref{sec-4:entropy-2}, we show the estimate \eqref{entropy-ineq:mu-L}.
More precisely, we take the expectation of \eqref{H-theorem:1} then let $m\to\infty$, where we use the usual stopping-time argument applied to the martingale term. 

In Section \ref{sec-4:entropy-3}, we show 
the sharper exponential estimate \eqref{H:esp:exp}, which is obtained after passing first to the limit $m\to\infty$ in the martingale and its quadratic variation.

\subsubsection{Localised entropy dissipation estimates}
\label{sec-4:entropy-1}
By definitions and the chain rule, we have 
\begin{align*}
h_{\mu,L}'(f)\chi_m(v)=L(\mu^2(f))\chi_m(v)\quad\text{and}\quad h_{\mu,L}''(f)\chi_m(v)=2L'(\mu^2(f))\mu'(f)\mu(f)\chi_m(v).
\end{align*}

By \eqref{log-chi:pw} and the chain rule formula \eqref{ito:cut-off}, we have almost surely, for every $t\in[0,T]$,
\begin{equation}
\label{ito:sec-7}
 \begin{aligned}
\cH_{\mu,L,m}(f_t)= &\cH_{\mu,L,m}(f_0)- \alpha\int^t_0\int_{v}\nabla f\cdot \nabla_v\big(L (\mu^2)\chi_m\big) \\
&-\frac12 \int^t_0\int_{v,v_*}A\mu^2\mu_{*}^2\tn L (\mu^2)\cdot \tn \big(L (\mu^2)\chi_m\big)\\
&+\frac{\sqrt{\eps}}{2}\sum_{k=1}^K\int^t_0\int_{v,v_*}\tn \big(L (\mu^2)\chi_m\big)\cdot\big(G_k(v,v_*)\sigma(f)\sigma(f_*)\big)\dd B_k\\
           &+\eps\big(T_1(f,h_{\mu,L},\chi_m)+T_2(f,h_{\mu,L},\chi_m)\big).
\end{aligned}
\end{equation}

In the following, we estimate each term individually.

\medskip

\noindent\textbf{The diffusion term: } By the chain rule and the integration by parts formula, we have 
\begin{align*}
   &\int^t_0\int_{v}\nabla f\cdot \nabla_v\big(L (\mu^2)\chi_m\big)\\
   =&{}\int^t_0\int_{v}\chi_m\nabla f\cdot \nabla_vL (\mu^2)+\int^t_0\int_{v}L (\mu^2)\nabla f\cdot \nabla\chi_m\\
   =&{}\int^t_0\int_{v}\chi_m\nabla f\cdot \nabla_vL (\mu^2)+\int^t_0\int_{v}f L (\mu^2)\Delta\chi_m-\int^t_0\int_{v} f \nabla_vL (\mu^2)\cdot \nabla \chi_m\\
   =:&{}I_1+I_2+I_3.
\end{align*}

Notice that, by the chain rule and the properties of $L$ and $\mu$, we have $I_1 \ge 0$. 
Regarding the term $I_2$, by \eqref{LlogL:sigma}, we obtain  
\begin{align*}
    \big|I_2\big| &\lesssim_{C_{L},C_{\mu_*},C_{\mu^*}} \int_0^t \|f\|_{L^1_2}\,\|\Delta \chi_m\|_{L^\infty}.
\end{align*}

Concerning $I_3$, using the fact that 
\begin{align*}
    |\nabla_v L(\mu^2)| 
    \le 2 L'(\mu^2)\mu(f)\mu'(f)|\nabla_v f|
    \lesssim_{C_L,C_{\mu_*},C_{\mu^*}} |\nabla_v f|,
\end{align*}
we deduce that 
\begin{align*}
    |I_3|
    \lesssim \int_0^t \|\nabla f\|_{L^2}\|f\|_{L^2}\|\nabla \chi_m\|_{L^\infty}.
\end{align*}

Hence, we obtain
\begin{align*}
    -\alpha \int_0^t \int_v \nabla f \cdot \nabla_v\big(L(\mu^2(f))\chi_m\big)
    \lesssim \int^t_0\big(\|\nabla f\|_{L^2}\|f\|_{L^2} + \|f\|_{L^1_2}\big)\|\nabla \chi_m\|_{W^{1,\infty}}.
\end{align*}

\medskip

\noindent\textbf{The Landau collision term:} Concerning the collision term, we have 
\begin{align*}
    &\frac12\int^t_0\int_{v,v_*}A\mu^2(f)\mu^2(f_*)\tn L(\mu^2(f))\cdot \tn \big(L(\mu^2(f))\chi_m\big)\\
    =&{}\int^t_0\cD_{\mu,L,m}(f)\,\dd s+\frac12\int^t_0\int_{v,v_*}L(\mu^2(f))A\mu^2(f)\mu^2(f_*)\tn L(\mu^2(f))\cdot \tn \chi_m.
\end{align*}

Recall that $z^*>0$ be such that $\operatorname{supp} A\subset\{|v-v_*|\le z^*\}$.  Using the pointwise bounds $\mu(r)\le C_{\mu^*}$, $\mu^2(r)\le r$, $|L(\mu^2(f))|\lesssim_{C_L,C_{\mu^*}}1$, and $0\le L'(s)s\le 1$, we estimate the cut-off error without using the $L^1$ or $L^2$ norm of $\nabla\chi_m$.  Here $|L(\mu^2(f))|$ is absorbed into the constant because $L$ is the regularised logarithm and $\mu^2(f)\in[0,C_{\mu^*}^2]$.  Indeed,
\begin{align*}
|\tn\chi_m|\le |\nabla\chi_m(v)|+|\nabla\chi_m(v_*)|\lesssim m^{-D},
\qquad
|\tn L(\mu^2(f))|\lesssim_{C_L,C_{\mu_*},C_{\mu^*}} |\nabla_v f|+|\nabla_{v_*}f_*|.
\end{align*}
Therefore, using the compact support of $A$,
\begin{align*}
&\left|\int^t_0\int_{v,v_*}L(\mu^2(f))A\mu^2(f)\mu^2(f_*)\tn L(\mu^2(f))\cdot \tn \big(\chi_m\big)\right|\\
&\lesssim_{C_{\mu^*},C_L,z^*}m^{-D}\int_0^t\int_{\{|v-v_*|\le z^*\}}f f_*\big(|\nabla_v f|+|\nabla_{v_*} f_*|\big)\,\dd v\dd v_*\dd s\\
&\lesssim_{C_{\mu^*},C_L,z^*,d}m^{-D}\int_0^t\|f_s\|_{L^1_v}\|f_s\|_{L^2_v}\|\nabla f_s\|_{L^2_v}\,\dd s.
\end{align*}

\medskip

\noindent\textbf{The error $T_1$ term:} We define 
\begin{align}
    \label{def:l-i}
    l_{1}(s):=\int_0^sL' (\mu^2(r))\mu'(r)\mu(r)\sigma(r)\dd r\quad\text{and}\quad l_{2}(s):=\int_0^sL' (\mu^2(r))\mu'(r)\mu(r)\sigma'(r)\sigma(r)\dd r.
\end{align}
By assumption $0\le \sigma\le \mu$ and $0\le L'(s)s\le 1$, we have 
\begin{equation}
\label{bdd:mu-sigma}
0\le L'(\mu^2)\sigma\mu\le 1.   
\end{equation}
As a direct consequence, we have 
\begin{align*}
    &0\le l_1(s)\le \mu(s).
\end{align*}
Concerning $l_2$, for $0\le r\le1$, by \eqref{bdd:mu-sigma} and the bounds $0\le \mu'(r)\le \frac{1}{2\sqrt{r}}$ and $0\le \sigma'(r)\le C_{\sigma_*}$, we  have 
\begin{align*}
    0\le l_2(s)\le \int_0^s\mu'(r)\sigma'(r)\le \sqrt{s},\quad \forall\, s\in[0,1].
\end{align*}
For $r\ge1$, we use $0\le \sigma'\sigma,\,\mu'\mu\le \frac12$, and $0\le L'(s)\le 1$ to derive 
\begin{align*}
 0\le l_2(s)=l_2(1)+ \int_1^s1\le s,\quad \forall\, s\in(1,\infty).    
\end{align*}
Hence, we have
\begin{align*}
    0\le l_2(s)\le \sqrt s\mathbb{1}_{\{0\le s\le 1\}}+s\mathbb{1}_{\{s\ge 1\}}.
\end{align*}

Hence, $T_1$ can be written as 
    \begin{align*}
T_1= \int^t_0\int_{v,v_*,w}&\chi_m(v_*)\big(L' (\mu^2_*)\mu'(f_*)\mu(f_*)\big)\sigma'(f)\sigma(f)\sigma(f_*)\sigma(f_w)\nabla_{v_*} f_*\cdot F_1\\   &+2\chi_m(v)\big(L' (\mu^2)\mu'(f)\mu(f)\big)\sigma^2(f)\sigma(f_*)\sigma(f_w)F_3\\
   &+2\chi_m(v_*)\big(L' (\mu^2_*)\mu'(f_*)\mu(f_*)\big)\sigma(f_*)\sigma(f_w)\nabla_{v_*} f_*\otimes \nabla_v \varsigma(f) : F_4\\
   & +\chi_m(v)\big(L' (\mu^2)\mu'(f)\mu(f)\big)\sigma(f_*)\sigma(f_w)\nabla_v \sigma^2(f)\cdot F_2\dd s,  \\
   = \int^t_0\int_{v,v_*,w}&\chi_m(v_*)\sigma'(f)\sigma(f)\sigma(f_w)\big(L' (\mu^2_*)\mu'(f_*)\mu(f_*)\sigma(f_*)\big)\nabla_{v_*} f_*\cdot F_1\\   &+2\chi_m(v)\big(L' (\mu^2)\mu'(f)\mu(f)\big)\sigma^2(f)\sigma(f_*)\sigma(f_w)F_3\\
   &+2\chi_m(v_*)\sigma(f_w)\big(L' (\mu^2_*)\mu'(f_*)\mu(f_*)\sigma(f_*)\big)\nabla_{v_*} f_*\otimes \nabla_v \varsigma(f) : F_4\\
   & +2\chi_m(v)\sigma(f_*)\sigma(f_w)\big(L' (\mu^2)\mu'(f)\mu(f)\sigma'(f)\sigma(f)\big)\nabla_v f\cdot F_2\dd s\\
      = \int^t_0\int_{v,v_*,w}&\chi_m(v_*)\sigma'(f)\sigma(f)\sigma(f_w)\nabla_{v_*} l_1(f_*)\cdot F_1\\   &+2\chi_m(v)\big(L' (\mu^2)\mu'(f)\mu(f)\big)\sigma^2(f)\sigma(f_*)\sigma(f_w)F_3\\
   &+2\chi_m(v_*)\sigma(f_w)\nabla_{v_*} l_1(f_*)\otimes \nabla_v \varsigma(f) : F_4\\
   & +2\chi_m(v)\sigma(f_*)\sigma(f_w)\nabla_v l_2(f)\cdot F_2\dd s.
\end{align*}

Applying integration by parts to the $F_1$, $F_2$ and $F_4$ terms and then taking absolute values, we may bound $T_1$ by the absolute value of
\begin{align*}
\int^t_0\int_{v,v_*,w}&-\sigma'(f)\sigma(f)\sigma(f_w) l_1(f_*)\nabla_{v_*}\cdot (\chi_m(v_*)F_1)\\
&-2\varsigma(f)\sigma(f_w)l_1(f_*)\nabla_{v_*} \cdot (\chi_m(v_*)\nabla_v \cdot F_4)\\
&+2\sigma(f_*)\sigma(f_w)l_2(f)\nabla_v \cdot (\chi_m(v)F_2)\\
&+ 2\chi_m(v)\big(L' (\mu^2)\mu'(f)\mu(f)\sigma^2(f)\big)\sigma(f_*)\sigma(f_w)F_3\dd s.
\end{align*}
By similar argument as in \eqref{bdd:mu-sigma}, we have $L' (\mu^2)\sigma^2\le 1$. Combining this with $0\le \mu'\mu\le \frac12$, we have 
\begin{align*}
    \Big|2\chi_m(v)\big(L' (\mu^2)\mu'(f)\mu(f)\sigma^2(f)\big)\sigma(f_*)\sigma(f_w)F_3\Big|\le \sigma(f_*)\sigma(f_w)|F_3|.
\end{align*}
Hence, the $T_1$ term is bounded by
(up to a constant depending only on $C_{\sigma_*}$ and $C_K$)
\begin{align*}
&\int^t_0\int_{v,v_*,w}
\sigma(f)\sigma(f_w) \sigma(f_*)\big|\nabla_{v_*}\cdot (\chi_m(v_*)F_1)\big|+\varsigma(f)\sigma(f_w)\sigma(f_*)\big|\nabla_{v_*} \cdot (\chi_m(v_*)\nabla_v \cdot F_4)\big|\\
           &\qquad +\sigma(f_*)\sigma(f_w)\sigma(f)\big|\nabla_v \cdot (\chi_m(v)F_2)\big|+ \sigma(f_*)\sigma(f_w)\big|F_3\big|\dd s\\
\lesssim&{}_{C_{\sigma_*},C_K}\int^t_0(\| f_0\|_{L^1}+1)^3\dd s,
\end{align*}
where we use the bounds in Remark \ref{rmk:sigma-f:L2}.

\medskip

\noindent
\textbf{The error $T_2$ term:} 
We write 
\begin{align*}
T_2&=T_{2,1}+T_{2,2}\\
&:=\frac12\int^t_0\int_{v,v_*,w}\sigma(f_*)\sigma(f_w)\big(L (\mu^2)\nabla_v \chi_m(v)-L (\mu^2_*)\nabla_{v_*} \chi_m(v_*)\big)\cdot  \sigma'(f)\sigma(f)F_1\dd s\\
&+\frac12\int^t_0\int_{v,v_*,w}\sigma(f_*)\sigma(f_w)\big(L (\mu^2)\nabla_v \chi_m(v)-L (\mu^2_*)\nabla_{v_*} \chi_m(v_*)\big)\cdot F_4 \nabla_v \varsigma(f)\dd s.
\end{align*}
The boundedness of $\sigma$ and $\sigma'$ ensure that  
\begin{align*}
|T_{2,1}|    &\lesssim_{C_{\mu^*},C_K,C_L,C_{\sigma_*},C_{\sigma^*},T}\|\nabla \chi_m\|_{L^\infty}.
\end{align*}
By integration by parts and the bounds in Remark \ref{rmk:sigma-f:L2}, we have  
\begin{align*}
   |T_{2,2}|&=\Big|\int^t_0\int_{v,v_*,w}\nabla_v\cdot\big(\big(\nabla_v \chi_m(v) L (\mu^2)-\nabla_{v_*} \chi_m (v_*)L (\mu^2_*)\big) F_4\big) \sigma(f_*)\sigma(f_w)\varsigma(f)\dd s\Big|\\
   &\lesssim_{C_L,C_K,C_{\mu_*},C_{\mu^*},C_{\sigma_*},C_{\sigma^*}} \|\nabla\chi_m\|_{W^{1,\infty}}\int^t_0\|f\|_{L^2}\|\nabla f\|_{L^2}\dd s.
\end{align*}

\medskip

\noindent
\textbf{The stochastic noise term:} Let
\begin{equation}
\label{entropy:M-mt}
M_{m,t}:=\frac12\sum_{k=1}^K\int_0^t\int_{v,v_*}\tn\big(L(\mu^2(f_s))\chi_m\big)\cdot
\big(G_k(v,v_*)\sigma(f_s)\sigma(f_{s,*})\big)\,\dd B_s^k .
\end{equation}
By the ONB identity for $G_k$, its quadratic variation satisfies
\begin{align*}
\langle M_m\rangle_t
&\lesssim \int_0^t\int_{v,v_*}A\sigma^2(f_s)\sigma^2(f_{s,*})
\big|\tn \big(L(\mu^2(f_s))\chi_m\big)\big|^2\,\dd s .
\end{align*}
Notice that $\chi_m^2\le\chi_m$ and 
\begin{align*}
\big|\tn \big(L(\mu^2)\chi_m\big)\big|^2\le &{} |\tn L(\mu^2) |^2\chi_m(v)+\big|\nabla_vL(\mu_*^2)\big|^2\big|\chi_m(v)-\chi_m(v_*)\big|^2    \\
&+|L(\mu^2)||\nabla_v\chi_m(v)|+|L(\mu_*^2)||\nabla_v\chi_m(v_*)|.
\end{align*}
By using of $|\nabla_v L(\mu^2)|\lesssim_{C_L,C_{\mu_*},C_{\mu^*}}  |\nabla_v f|$, we have 
\begin{align*}
\big|\tn \big(L(\mu^2)\chi_m\big)\big|^2\lesssim &{}|\tn L(\mu^2) |^2\chi_m(v)+C|\nabla_v f||v-v_*|^2\|\nabla \chi_m\|_{L^\infty}^2 \\
&+\|\nabla \chi_m\|_{L^\infty}\big(|L(\mu^2)|+|L(\mu_*^2)|\big).
\end{align*}
Then we have
\begin{align}\label{martingale-control}
    &\int^t_0\int_{v,v_*}A\mu^2(f)\mu^2(f_*)\big|\tn \big(L(\mu^2)\chi_m\big)\big|^2\dd s\notag\\
    \le&{} 2\int^t_0\cD_{\mu,L,m}(f)\dd s\notag\\
   &+ C \int^t_0\|\nabla\chi_m\|_{L^\infty}\int_{v}f|\nabla f|+\mu^2(f)|\log \mu^2(f)|\dd s.\notag\\
   \le&{} 2\int^t_0\cD_{\mu,L,m}(f)\dd s\notag\\
   &+ C \int^t_0\|\nabla\chi_m\|_{L^\infty}\big(\|f\|_{L^2}\|\nabla f\|_{L^2}+\|f\|_{L^1_2}\big)\dd s,
\end{align}
where, in the last inequality, we use 
\eqref{LlogL:sigma}.

\medskip

Put the above estimates together.  Then, for every $m\in\mathbb N$ and every
$t\in[0,T]$, almost surely,
\begin{align}
\label{localised-entropy-pathwise-exp-prep}
&\cH_{\mu,L,m}(f_t)+\int_0^t\cD_{\mu,L,m}(f_s)\,\dd s \notag\\
\le &{}\cH_{\mu,L,m}(f_0)+\sqrt{\eps}M_{m,t}
+C_1m^{-D}\mathcal R_T(f)+\eps C_2\big(\|f_0\|_{L^1}^3+1\big),
\end{align}
where
\begin{equation*}
\mathcal R_T(f):=\big(\|f\|_{L^\infty_tL^2_v}+1\big)^3\big(\|f\|_{L^2_tH^1_v}+1\big).
\end{equation*}
Here, $C_1>0$ depends only on the fixed regularisation constants and
$C_2=C_2(C_{\sigma_*},C_K,T)$.  Notice that $\mathcal R_T(f)<\infty$ almost
surely by the $L^2$ estimate, the energy estimate, and the artificial
diffusion estimate.

\subsubsection{The entropy inequality \eqref{entropy-ineq:mu-L}}\label{sec-4:entropy-2}

We recall the stochastic noise term \eqref{entropy:M-mt}. By \eqref{martingale-control}, 
 $\langle M_m\rangle_t$ is integrable after the usual stopping-time localisation, and the stopped martingales have zero expectation. Letting the stopping level go to infinity gives $\hE M_{m,t}=0$. This justifies dropping the stochastic term in the expected entropy inequality.

 Hence, \eqref{localised-entropy-pathwise-exp-prep} implies that
 \begin{equation}
 \begin{aligned}
&\hE\big[\cH_{\mu,L,m}(f_t)\big]+\hE\Big[\int_0^t\cD_{\mu,L,m}(f_s)\,\dd s\Big]\le \hE\big[\cH_{\mu,L,m}(f_0)\big]\\
&\quad + C_1 m^{-D}\,\hE\Big[\mathcal R_T(f)\Big]
+ C_2\hE\big[\| f_0\|_{L^1}^3+1\big],
\end{aligned}
\end{equation}
Then to show the entropy inequality \eqref{entropy-ineq:mu-L}, we only need to pass to the limit in entropy and entropy dissipation terms. 

\medskip

\noindent\textbf{ The $\cH_{\mu,L}(f_0)$ term:}
We have the following pointwise convergence 
\begin{align*} 
   \chi_m(v) h_{\mu,L}(f_0)\to h_{\mu,L}(f_0)\quad\forall v\in\Do.
\end{align*}
By \eqref{log-chi:pw}, we have 
\begin{align*}
       |L(\mu^2(f_0))|\le |\log (\mu(f_0)^2)|,
    \end{align*}
   moreover, by definition,
    \begin{align*}
        \Big|\log(\mu^2(r))\Big|\le&{}-\log( \mu^2(r))\mathbb{1}_{\{0< r\le 1\}}+\log(\mu^2(r))\mathbb{1}_{\{r> 1\}}\\
        \le&{}-\log(C_{\mu_*}^2r^2)\mathbb{1}_{\{0< r\le 1\}}+r\mathbb{1}_{\{r\ge 1\}}.
    \end{align*}
    As a consequence, we have 
    \begin{equation}
   \label{H:bdd-epsilon-0}
    \begin{aligned}
        |h_{\mu,L}(s)|\le&{} \int_0^s \Big(\big|\log(\mu^2(r))\big|\mathbb{1}_{\{0< r\le 1\}}+r\mathbb{1}_{\{r\ge 1\}}\Big)\,\dd r\\
        \lesssim&{} s\big|\log s\big|\mathbb{1}_{\{0\le s\le 1\}}+s^2\mathbb{1}_{\{s\ge 1\}}.
    \end{aligned}
    \end{equation}
    Since $f_0\in L\log L\cap L^2(\Do;\R_+)$, we have 
     \begin{align*}
        &\chi_m(v)| h_{\mu,L}(f_0)|\lesssim  f_0|\log f_0|+f_0^2\in L^1(\Do).
    \end{align*}
By the dominated convergence theorem, we have  
    \begin{align*}
\lim_{m\to\infty}\cH_{\mu,L,m}(f_0)= \cH_{\mu,L}(f_0).
\end{align*}

\medskip

\noindent\textbf{ The entropy dissipation:} By assumptions of $A$, $\mu$ and $L$, we have 
\begin{align*}
0\le A\mu^2\mu_*^2\big|\tn L(\mu^2)\big|^2\lesssim_{A_0,C_L,C_{\mu_*},C_{\mu^*}} \big(|\nabla_v f|^2 +|\nabla_{v_*} f_*|^2\big)\mathbb{1}_{\{|v-v_*|\le z^*\}}\in L^1([0,T]\times \G) .   
\end{align*}
By the monotonicity convergence theorem, we have 
    \begin{align*}
   \int_0^t\cD_{\mu,L}(f_s)\,\dd s=\lim_{m\to\infty} \int_0^t\cD_{\mu,L,m}(f_s)\,\dd s\quad \forall t\in[0,T].
\end{align*}

    \medskip

\noindent\textbf{ The $\cH_{\mu,L}(f_t)$ term:}
   By using of \eqref{H:bdd-epsilon-0} and \eqref{LlogL:sigma}, we have 
    \begin{align*}
        |h_{\mu,L}(f)|\lesssim&{} f|\log f|+f^2\in L^1(\Do).
    \end{align*}
By a direct computation, we observe that for every $r\in[0,1]$,
\begin{align*}
(r\log r)^- = -r\log r \leq ar + e^{-a}, \quad \text{for every } a\geq 0.
\end{align*}
For every $v\in\mathbb{R}^{d}$, choosing $a=|v|^2$, it follows that almost surely, for every $t\in[0,T]$,
\begin{equation}
    \label{LlogL:minus}
\begin{aligned}
\int_{\{f\in[0,1]\}} f \log f
\geq& -\int_{\mathbb{R}^{d}} |v|^2 f
-\int_{\mathbb{R}^{d}} e^{-|v|^2}\\
\geq& C\left(-\int_{\mathbb{R}^{d}} |v|^2 f_{0}-C(d,T,f_0)
-\int_{\mathbb{R}^{d}} e^{-|v|^2}\right) > -\infty. 
\end{aligned}
\end{equation}
This shows that the entropy functional is bounded from below. Consequently, $\mathcal{H}_{\mu,L}(f_t)$ is almost surely finite for every $t\in[0,T]$.
By  Fatou's lemma, we have 
\begin{align*}
    \cH_{\mu,L}(f_t)\le \liminf_{m\to\infty} \cH_{\mu,L,m}(f_t)\quad \forall t\in[0,T].
\end{align*}

\subsubsection{The exponential entropy inequality \eqref{H:esp:exp}}\label{sec-4:entropy-3}

To show \eqref{H:esp:exp}, we let $m\to\infty$ in \eqref{localised-entropy-pathwise-exp-prep} before applying the exponential martingale argument.
This order is important because the estimate to be exponentiated has to be
written in terms of the limiting martingale and its quadratic variation.

By the previous subsection, we have 
\begin{gather*}
\lim_{m\to\infty}\cH_{\mu,L,m}(f_0)=\cH_{\mu,L}(f_0),\quad \cH_{\mu,L}(f_t)\le \liminf_{m\to\infty}\cH_{\mu,L,m}(f_t),\\
\text{and}\quad \int_0^t\cD_{\mu,L}(f_s)\,\dd s
=\lim_{m\to\infty}\int_0^t\cD_{\mu,L,m}(f_s)\,\dd s.
\end{gather*}

It remains to identify the limiting martingale.  Define
\begin{equation*}
H_{m,k}(s):=\int_{v,v_*}\tn\big(L(\mu^2(f_s))\chi_m\big)
\cdot G_k(v,v_*)\sigma(f_s)\sigma(f_{s,*})\,\dd v\dd v_* .
\end{equation*}
The estimate \eqref{martingale-control}, the bound $0\le \sigma\le\mu$, and
the convergence $\chi_m\uparrow1$ imply that $(H_{m,k})_{m\ge1}$ is Cauchy in
$L^2(\Omega\times[0,T])$ after the standard stopping-time localisation; the
localising sequence is then removed by monotone convergence and
\eqref{martingale-control}.  Its limit is
\begin{equation*}
H_k(s):=\int_{v,v_*}\tn L(\mu^2(f_s))
\cdot G_k(v,v_*)\sigma(f_s)\sigma(f_{s,*})\,\dd v\dd v_* .
\end{equation*}
Consequently, by It\^o isometry, we have 
\begin{equation}
\label{martingale-m-limit-exp}
M_{m,t}\longrightarrow
M_t:=\frac12\sum_{k=1}^K\int_0^t H_k(s)\,\dd B_s^k
\end{equation}
in probability, and along a subsequence almost surely.
Furthermore,
\begin{equation}
\label{bracket-m-limit-exp}
\langle M_m\rangle_t=\frac14\sum_{k=1}^K\int_0^t|H_{m,k}(s)|^2\,\dd s
\longrightarrow
\langle M\rangle_t:=\frac14\sum_{k=1}^K\int_0^t|H_k(s)|^2\,\dd s
\end{equation}
in probability, again along the same subsequence almost surely.  Passing to
the limit by letting $m\to\infty$ in \eqref{martingale-control} gives the bracket estimate
\begin{equation}
\label{limit-bracket-control-exp}
\langle M\rangle_t\le 4\int_0^t\cD_{\mu,L}(f_s)\,\dd s,
\qquad t\in[0,T].
\end{equation}
Combining \eqref{localised-entropy-pathwise-exp-prep}--\eqref{martingale-m-limit-exp} and using $m^{-D}\mathcal R_T(f)\to0$, we obtain, for every
$t\in[0,T]$, almost surely,
\begin{equation}
\label{entropy-limit-before-exp}
\cH_{\mu,L}(f_t)+\int_0^t\cD_{\mu,L}(f_s)\,\dd s
\le \cH_{\mu,L}(f_0)+\sqrt{\eps}M_t+\eps C_2\big(\|f_0\|_{L^1}^3+1\big).
\end{equation}

We now apply the exponential martingale argument to the limiting martingale
$M$.  Fix an arbitrary $\lambda>0$.  Since the stochastic term in
\eqref{entropy-limit-before-exp} is $\sqrt{\eps}M_t$, the exponential
martingale contains the compensator
$\frac{\lambda^2\eps}{2}\langle M\rangle_t$.  Choose
$\eps_0(\lambda)>0$ small such that
\begin{equation*}
2\lambda^2\eps\le \frac{\lambda}{2}
\qquad \forall\, \eps\in(0,\eps_0).
\end{equation*}
Multiplying \eqref{entropy-limit-before-exp} by $\lambda$, subtracting
$\frac{\lambda^2\eps}{2}\langle M\rangle_t$, and using
\eqref{limit-bracket-control-exp}, we obtain
\begin{align*}
&\lambda\cH_{\mu,L}(f_t)+\frac{\lambda}{2}\int_0^t\cD_{\mu,L}(f_s)\,\dd s\\
&\le \lambda\cH_{\mu,L}(f_0)+C_2\lambda\big(\|f_0\|_{L^1}^3+1\big)
+\lambda\sqrt{\eps}M_t-\frac{\lambda^2\eps}{2}\langle M\rangle_t,
\end{align*}
where we used $\eps\le1$ to absorb the factor $\eps$ in the deterministic
correction error into the constant.  The Dol\'eans exponential
\begin{equation*}
\mathcal E_t^{\lambda,\eps}:=
\exp\left(\lambda\sqrt{\eps}M_t-\frac{\lambda^2\eps}{2}\langle M\rangle_t\right)
\end{equation*}
is a nonnegative local martingale and hence a supermartingale.  Therefore,
since $f_0$ is deterministic,
\begin{align}\label{exponential-bound-1}
&\hE\exp\left[\lambda\cH_{\mu,L}(f_t)+\frac{\lambda}{2}\int_0^t\cD_{\mu,L}(f_s)\,\dd s\right]\notag\\
&\le
\hE\left[
\exp\left(\lambda\cH_{\mu,L}(f_0)+C_2\lambda\big(\|f_0\|_{L^1}^3+1\big)\right)
\mathcal E_t^{\lambda,\eps}
\right]\notag\\
&\le
\hE\exp\left(\lambda\cH_{\mu,L}(f_0)+C_2\lambda\big(\|f_0\|_{L^1}^3+1\big)\right).
\end{align}

The preceding argument gives the endpoint estimate \eqref{exponential-bound-1}.
We now derive the maximal version.  Since
\eqref{entropy-limit-before-exp} holds for every $s\in[0,t]$, multiplying it
by $\lambda$ and keeping only one half of the dissipation gives
\begin{align}
\label{entropy-limit-before-exp-2}
&\sup_{s\in[0,t]}
\left(\lambda\cH_{\mu,L}(f_s)+\frac{\lambda}{2}\int_0^s\cD_{\mu,L}(f_r)\,\dd r\right)\notag\\
&\le \lambda\cH_{\mu,L}(f_0)
+C_2\lambda\big(\|f_0\|_{L^1}^3+1\big)
+\sup_{s\in[0,t]}\left(\lambda\sqrt{\eps}M_s-\frac{\lambda}{2}\int_0^s\cD_{\mu,L}(f_r)\,\dd r\right).
\end{align}
By \eqref{limit-bracket-control-exp},
\begin{equation*}
\frac{\lambda}{2}\int_0^s\cD_{\mu,L}(f_r)\,\dd r
\ge \frac{\lambda}{8}\langle M\rangle_s .
\end{equation*}
Thus, with $a:=\lambda\sqrt\eps$,
\begin{align}
\label{maximal-martingale-reduction}
&\sup_{s\in[0,t]}
\left(\lambda\cH_{\mu,L}(f_s)+\frac{\lambda}{2}\int_0^s\cD_{\mu,L}(f_r)\,\dd r\right)\notag\\
&\le \lambda\cH_{\mu,L}(f_0)
+C_2\lambda\big(\|f_0\|_{L^1}^3+1\big)
+\sup_{s\in[0,t]}\left(aM_s-\frac{\lambda}{8}\langle M\rangle_s\right).
\end{align}
It remains to estimate the last supremum.  Choose $\eps_0(\lambda)>0$ so
small that $a^2\le\lambda/8$ for $0<\eps<\eps_0(\lambda)$, and set
\begin{equation*}
Z_s:=\exp\left(2aM_s-2a^2\langle M\rangle_s\right),\qquad s\in[0,t].
\end{equation*}
Then $(Z_s)_{s\in[0,t]}$ is a nonnegative local martingale and hence a
supermartingale.  Since $a^2\le\lambda/8$, we have
\begin{equation*}
\exp\left[
\sup_{s\in[0,t]}\left(aM_s-\frac{\lambda}{8}\langle M\rangle_s\right)
\right]
\le \left(\sup_{s\in[0,t]}Z_s\right)^{1/2}.
\end{equation*}
Doob's maximal inequality in its weak $L^1$ form yields
\begin{equation*}
\mathbb P\left(\sup_{s\in[0,t]}Z_s>R\right)\le \frac{1}{R},
\qquad R>0.
\end{equation*}
Consequently,
\begin{align}
\label{doob-exp-max-bound}
\hE\left(\sup_{s\in[0,t]}Z_s\right)^{1/2}
&=\int_0^\infty \mathbb P\left(\left(\sup_{s\in[0,t]}Z_s\right)^{1/2}>y\right)\,\dd y\notag\\
&\le 1+\int_1^\infty y^{-2}\,\dd y
\le 2.
\end{align}
Taking the exponential in \eqref{maximal-martingale-reduction}, then using
\eqref{doob-exp-max-bound}, gives
\begin{align}\label{exponential-bound-2}
&\hE\exp\left[
\sup_{s\in[0,t]}\left(\lambda\cH_{\mu,L}(f_s)
+\frac{\lambda}{2}\int_0^s\cD_{\mu,L}(f_r)\,\dd r\right)\right]\notag\\
&\le
2\exp\left(\lambda\cH_{\mu,L}(f_0)
+C_2\lambda\big(\|f_0\|_{L^1}^3+1\big)\right).
\end{align}
This proves \eqref{H:esp:exp}, after absorbing the harmless factor $2$ into
the constant $C_\lambda$ and renaming $C_2$ as $C$.

\section{Pass to the limits in \texorpdfstring{$L^2$-framework}{}}
\label{sec:limit-L2}

In this section, we show the existence results of the following approximation It\^o's equation in the $L^2\cap L^1$-framework
\begin{equation}
    \label{eq:ito-2}
\begin{aligned}
\partial_t f
&=\alpha\Delta f+\frac12\tn\cdot\big( A f f_*\tn \log f\big)-\frac{\sqrt{\eps}}{2}\tn\cdot( A^{1/2}  \sigma(f) \sigma( f_*)\xi_K)\\
&\quad+\frac{\eps}{2}\sum_{k= 1}^K\tn\cdot \Big(G_{k}(v,v_*) \sigma'(f) \sigma(f_{*}) \tn\cdot\big(G_{k}(v,w) \sigma(f) \sigma(f_{w})\big) \Big),
\end{aligned}
\end{equation}
where $A$ is given by \eqref{def:cA}
\begin{align}
A(|v-v_*|)=|v-v_*|^{2+\gamma},\quad \gamma\in(-2,0),    
\end{align}
and $\sigma(r)\in C^1(\R_+;\R_+)$ satisfies Assumption \ref{ass:sigma-R0-0}
\begin{equation}
\label{def:sigma-final}
\begin{gathered}
\sigma(r)=\sqrt r \quad r\in(r_0,\infty),\\
0\le \sigma'(r)\le C_{r_0}:=\frac{1}{2\sqrt{r_0}}\quad\text{and}\quad  0\le \sigma(r)\le\sqrt{r}\quad r\in[0,r_0].    
\end{gathered}
\end{equation}

\begin{definition}
   \label{def:entropy-sigma}
  We denote by $\cH(f)$ the Boltzmann-Shannon entropy for $f\in L^1(\Do;\R_+)$
\begin{equation}
\cH(f):=\int_v f\log f\dd v
\end{equation}
whenever the positive part $\max(0,f\log f)$ is integrable, and otherwise set $\cH(f)=+\infty$.

Let $\cD(f)$ denote the Landau entropy dissipation
\begin{align*}
\cD(f):=\frac12\int_{v,v_*}A f f_*\big|\tn \log f\big|^2\ge 0.
\end{align*}
\end{definition}

\begin{proposition}\label{lem:l-12:log}
Let the initial datum $f_0$ satisfy 
\begin{align*}
   f_0\in L^2\cap L^1_2\cap L\log L(\Do;\R_+).
\end{align*}
Let $d\ge 2$. Let 
$\alpha\in(0,1)$ and $\eps\in(0,1)$ be arbitrarily fixed.
Let $\sigma$ satisfy Assumption~\ref{ass:sigma-R0-0}. Let the family
$(G_k)_{k\geq 1}$ satisfy Assumption~\ref{ass:G-k:app}.  Then there exists a probabilistic weak solution
$(f,(B_k)_{k=1,\ldots,K})$ to \eqref{eq:ito-2} with initial datum $f_0$. Moreover, we have $f\in L^\infty([0,T];L^2(\Do;\R_+))$, and we have, almost surely, for every $t\in[0,T]$,
\begin{gather*}
    \int_v (1,v)f_t(v)\dd v=\int_v (1,v)f_0(v)\dd v,\\
    \int_v|v|^2  f_t(v)\dd v\le  \int_v|v|^2  f_0(v)\dd v +2\alpha d T \|f_0\|_{L^1},\\
    \end{gather*}
    and for almost every $t\in[0,T]$, 
    \begin{gather*}
    \hE\big[\cH(f_t)\big]+\hE\Big[\int_0^t\cD(f)\Big]\le \hE\big[\cH(f_0)\big]+C\hE\big[\| f_0\|_{L^1}^2\big] 
    \end{gather*}
    for some $C=C(\|\sigma'\|_{L^\infty},C_K)>0$.

Moreover, for every $\lambda>0$ there exists $\eps_0(\lambda)>0$ such that, if $0<\eps\le\eps_0(\lambda)$, then for almost every $t\in[0,T]$, 
\begin{equation}
\label{H:limit-L2-exp}
\hE\exp\left[
\operatorname{esssup}_{s\in[0,t]}\left(\lambda\cH(f_s)+\frac{\lambda}{2}\int_0^s\cD(f_r)\,\dd r\right)\right]
\le C_\lambda\exp\left[\lambda\cH(f_0)+C\lambda\big(\|f_0\|_{L^1}^3+1\big)\right].
\end{equation}

\end{proposition}

The remainder of this section is devoted to proving Proposition \ref{lem:l-12:log}.

To show the existence result, we approximate the equation, we approximate \eqref{eq:ito-2} by \eqref{ito-epsi-1}, for which the existence and properties of solutions were established in Sections \ref{sec:Galerkin} and \ref{sec:L-1:functional}. The necessary tightness results are established in Section \ref{sec-5:tight}, and the proof of Proposition \ref{lem:l-12:log} is completed in Section \ref{sec-5:limit}.

\medskip

We present the details of approximations here.

\begin{mathblock}[quadruple]{quadruple}

The logarithm function, Landau mobility, $\sigma$ and kernel $A$ are approximated by:
\begin{itemize}
    \item \textbf{Approximation to $\log$:} Let $n\in\N$.  We take a sequence of $L_{n}(s)$ as in Example \ref{expl:sigma-L} by taking $s_0=n^{-1}$. We note that $L_n$ satisfies  Assumption \ref{ass:L-s0}. For all $s\in(0,\infty)$, the following pointwise convergence holds
\begin{align}
\label{def:L-n}
   L_n(s)\to L(s):=\log s\quad \text{as}\quad n\to\infty.
\end{align}

\item \textbf{Approximation to the Landau mobility:} We take $\mu_n$ as in Example  \ref{expl:sigma-L} by choosing $r_0=n^{-1}$ and $R_0=n$.
We note that $\mu_n$
satisfies assumption \ref{ass:sigma-R0-0}. Moreover, we have
\begin{equation}
\label{def:mu-n}
\begin{gathered}
    0\le \mu_n(r)\le \sqrt r\quad\text{and}
    \quad 0\le \mu_n'(r)\le  \frac{1}{2\sqrt{r}}\quad\forall\, r\in[0,\infty),\\
    \frac{\sqrt{n}}{2}\le \mu_n'(r)\le  \frac{3\sqrt{n}}{2} \quad \forall\,r\in[0,n^{-1}],\\\mu_n(r)=\sqrt{r}\quad \forall\,r\in[n^{-1},n],\\
    \text{and}\quad \mu_n^2(r)\to r\quad\text{locally uniformly on }[0,\infty).
\end{gathered}
\end{equation}



\item \textbf{Approximation to $\sigma$:} Let $r_0\in(0,1)$ be fixed as given in Assumption \ref{ass:sigma-R0-0}.
We take $\sigma_n(r)$ as in Example  \ref{expl:sigma-L} by choosing $r_0=r_0$ and $R_0=n$.
We note that $\sigma_n$
satisfies assumption \ref{ass:sigma-R0-0}. Moreover, we have, for all $r\in[0,\infty)$ and $n\in\N$,
{\begin{equation}
\label{def:sigma-n}
\begin{gathered}
    \sigma_{n}(r)\le \sigma_{n+1}(r)\le \sigma(r)\quad\text{and} \quad \sigma_{n}'(r)\le \sigma'(r) \le C_{r_0}.
    \end{gathered}
    \end{equation}}
    

   \item \textbf{Approximation to $A$:}    Let $\chi^1=\chi^1(r)$ be a smooth cut-off function such that $0\le \chi^1\le 1$, $\chi^1(r)=1$ for $r\in[0,1]$ and $\chi^1(r)=0$ for $r\in[2,\infty)$.  We define the rescaled function $\chi^1_n(r)=\chi(r/n)$.
We take the approximation sequence as follows 
\begin{equation}
\label{def:A-n}  
A_n(|z|):=A(|z|)\big(\chi^1_n-\chi^1_{n^{-1}}\big)(|z|)\quad z\in\Do.
\end{equation}
We note that $A_n$ satisfies  Assumption \ref{ass:A}. Moreover, we have 
$$0\le A_n\le A_{n+1}\le A \quad \forall \,n\in\N,$$
 and  the following pointwise convergence holds
\begin{align*}
 A_n(|v-v_*|)= A(|v-v_*|)\quad \text{as}\quad n\to\infty.    
\end{align*} 
\end{itemize}
    
\end{mathblock}

Corresponding to the approximation kernel $A_n$, we define  $(a_{n},b_n,c_n)$ analogous to \eqref{relation:abc}
\begin{align*}
 &a_{n,ij}(v-v_*):=\Big(\delta_{ij}-\frac{(v-v_*)_i(v-v_*)_j}{|v-v_*|^2}\Big)A_n(|v-v_*|),\\
         &b_{n,i}(v-v_*):=\sum_{j=1}^d\d_j a_{n,ij}(v-v_*),\quad c_n(v-v_*):=\sum_{i,j=1}^d\d^2_{ij} a_{n,ij}(v-v_*).
\end{align*}
In particular, by \eqref{abc} the bounds \eqref{abc:bounds:n} holds uniformly in $n\in\N$
\begin{align}
\label{bdd:ab-n}
    |a_n|\lesssim |v-v_*|^{\gamma+2}\quad\text{and}\quad |b_n|\lesssim |v-v_*|^{\gamma+1},
\end{align}
and  using, for all $r\in(0,\infty)$, 
$$|(\chi^1_n-\chi^1_{n^{-1}})'(r)|\le n^{-1}\mathbb{1}_{\{0\le r\le2n\}}(r)+n\mathbb{1}_{\{0\le r\le2n^{-1}\}}(r)\lesssim r^{-1},$$ we have 
\begin{align}
\label{bdd:c-n}
    |c_n|\lesssim\frac{A_n(|v-v_*|)}{|v-v_*|^2}+\frac{|A_n'(|v-v_*|)|}{|v-v_*|}\lesssim |v-v_*|^\gamma.
\end{align}

We define the family $G_{k,n}$ associated to the approximation kernel $A_n$ 
\begin{align}
\label{def:G-k:n}
    G_{k,n}(v,v_*):=\sqrt{A_n(v-v_*)}\Pi_{(v-v_*)^\perp}g_k(v,v_*)\in\R^d.
\end{align}
We note that the uniform bounds \eqref{bdd:ab-n} and \eqref{bdd:c-n} ensure that, up to a constant, the bound of $G_k$ in Assumption \ref{ass:G-k:app} holds also here and uniformly in $n\in\N$
\begin{equation}
\begin{aligned}
\sum_{k=1}^K\|G_{k,n}\|_{W^{2,\infty}\cap W^{2,1}(\G)}
\le C_K<+\infty.
\end{aligned}
\end{equation}
Correspondingly, we define $F_{i,n}$ as in \eqref{def:F_i} corresponding to $G_{k,n}$. Moreover, we have the following convergences
\begin{align}
\label{G-k:conv}
    G_{k,n}\to G_{k}\quad\text{in}\quad W^{2,1}\cap W^{2,\infty}(\G),\quad F_{i,n}\to F_{i}\quad\text{in}\quad W^{1,1}\cap W^{1,\infty}(\R^{3d})
\end{align}
as $n\to\infty$. In particular, the convergence is immediate by choosing an active family ${g_k}$ satisfying \eqref{Gk:annual} that
\begin{align*}
\cup_{1\le k\le K} \supp\big(g_k\big)\subset \{r\in \R_+\mid  K^{-1}\le r\le K\}\times \R^d\times \sd.
\end{align*}

Furthermore, we define 
\begin{align*}
&\cH_n(f):=\int_vh_n(f)\dd v,\quad h_n (s):=\int_0^s L_n(\mu_n^2(r))\dd r,.
\end{align*}

We denote the regularised family by
$(A_n,\sigma_n,\mu_n^2(f)\mu_n^2(f_*),L_n,\cH_n,\cD_n,G_{k,n},F_{i,n})$,
and the corresponding limiting objects by
$(A,\sigma,ff_*,\log,\cH,\cD,G_k,F_i)$. The notation emphasises that $\mu_n$ regularises the deterministic
Landau mobility and converges to the square root, while $\sigma_n$ regularises the
stochastic coefficient and converges to $\sigma$.

For each $n\in\mathbb{N}$, let $f_n$ be a probabilistic weak solution of \eqref{eq:ito-2:app-n} on a probability space $(\Omega_n,\mathcal{F}_n,\mathbb{P}_n)$, whose existence is guaranteed by Proposition~\ref{lem:app-1:ext}:
\begin{equation}
    \label{eq:ito-2:app-n}
\begin{aligned}
\partial_t f
&=\alpha\Delta f+\frac12\tn\cdot\big( A_n \mu_n^2(f) \mu_n^2(f_*)\tn L_n (\mu_n^2(f))\big)-\frac{\sqrt{\eps}}{2}\tn\cdot( A_n^{1/2}  \sigma_n(f) \sigma_n( f_*)\xi_K)\\
&\quad+\frac{\eps}{2}\sum_{k= 1}^K\tn\cdot \Big(G_{k,n}(v,v_*) \sigma_n'(f) \sigma_n(f_{*}) \tn\cdot\big(G_{k,n}(v,w) \sigma_n(f) \sigma_n(f_{w})\big) \Big).
\end{aligned}
\end{equation}
To simplify the notation, we write $\mathbb{E}$ for the expectation with respect to $\mathbb{P}_n$ and suppress the dependence on $n\in\mathbb{N}$ whenever no confusion can arise.

\medskip

In Section \ref{sec-5:tight},  we establish an $L^2$-uniform bound in Lemma~\ref{lem:L-2:app-refine} and a time-regularity estimate in Lemma~\ref{lem:time-regularity:L-2-2}, both uniformly with respect to $n\in\mathbb{N}$. As a consequence, we shall prove the tightness of $(f_n,\nabla f_n)$ in the space
$$
\mathbb{X}:=L^2([0,T];L^2_{\mathrm{loc}}(\mathbb{R}^d))
\cap \mathcal{C}^{\beta}([0,T];H^{-l}_{\mathrm{loc}}(\mathbb{R}^d))
\times \big(L^2([0,T];L^2(\mathbb{R}^d)),w\big).
$$

In Section \ref{sec-5:limit}, we pass to the limit by letting $n\to\infty$ in the weak formulation \eqref{weak:app-1}.
Moreover, by Proposition~\ref{lem:l-1:app}, for every $n\in\mathbb{N}$, the mass conservation law as well as the energy and entropy inequalities hold uniformly in $n$:
\begin{gather}
    \int_v f_n(t,v)\dd v=\int_v f_0(v)\dd v,\quad a.s. \text{ for every }t\in[0,T],\label{mass:h-eps}\\
     \int_v vf_n(t,v)\dd v = \int_v vf_0(v)\dd v, \quad a.s. \text{ for every }t\in[0,T],\label{momentum:h-eps}\\
    \int_v|v|^2  f_n(t,v)\dd v\le  \int_v|v|^2  f_0(v)\dd v +2\alpha d T \|f_0\|_{L^1(\Do)},\quad a.s. \text{ for every }t\in[0,T]\label{energy:h-eps}
    \\
\hE\big[\cH_{n}(f_n(t))\big]+\hE\Big[\int_0^t\cD_{n}(f_n)\Big]
\le \hE\big[\cH_{n}(f_0)\big]+C\hE\big[\| f_0\|_{L^1}^2\big],\quad \text{for every }t\in[0,T], \label{H:h-eps}\\
\hE\exp\left[
\sup_{s\in[0,t]}\left(\lambda\cH_n(f_n(s))+\frac{\lambda}{2}\int_0^s\cD_n(f_n(r))\,\dd r\right)\right]
\le C_\lambda\exp\left[\lambda\cH_n(f_0)+C\lambda\big(\|f_0\|_{L^1}^3+1\big)\right],\label{H:h-eps-exp}
\end{gather}
for every $t\in[0,T]$, $\lambda>0$ and $0<\eps\le\eps_0(\lambda)$, where $\eps_0(\lambda)$ is as in Proposition~\ref{lem:l-1:app}.  The constant $C=C(r_0,C_K)>0$ is independent of $n$ and $\eps$.

\subsection{Uniform estimates and tightness results}\label{sec-5:tight}

We emphasise that the $L^2$-estimate at the fixed regularised level depends
on the compact support and boundedness of $b_n$. By exploiting the uniform
mass and energy bounds, we now derive a refined $L^2$-estimate that holds
uniformly with respect to $n\in\mathbb{N}$.
\begin{lemma}[Uniform $L^2$-estimate for $(f_n)_{n\in\mathbb N}$]
    \label{lem:L-2:app-refine}
 Let $f_0\in L^2(\R^d)\cap L^1_2(\R^d;\R_+)$. Let the approximation pair $(L_n,\mu_n,\sigma_n,A_n)$ be given as in Approximation \ref{quadruple}. For each
$n\in\mathbb N$, let $f_n$ be a probabilistic weak solution to
\eqref{eq:ito-2:app-n} on
$(\Omega_n,\mathcal F_n,(\mathcal F_{n,t})_{t\in[0,T]},\mathbb P_n)$
with initial datum $f_0$. We denote expectation with
respect to $\mathbb P_n$ by
$\mathbb E_n$. 
Then there exists a constant
\begin{align*}
C=C\big(d,\gamma,\alpha,r_0,T,C_K,\|f_0\|_{L^1_2}\big)>0,
\end{align*}
independent of $n$, such that 
\begin{align}
\label{L-2:app-re-integrated}
\sup_{n\in\mathbb N}\mathbb E\Big[
\sup_{t\in[0,T]}\|f_n(t)\|_{L^2}^2\Big]
+\alpha\sup_{n\in\mathbb N}\mathbb E
\int_0^T\|\nabla_v f_n(s)\|_{L^2}^2\,\dd s
\le C\big(1+\|f_0\|_{L^2}^2\big).
\end{align}

\end{lemma}

\begin{proof}
We fix an arbitrary $n\in\mathbb N$. Throughout the proof, we write
\begin{equation*}
f=f_n,\quad\mu=\mu_n,\quad \sigma=\sigma_n,\quad \tili=\tili_n,\quad F_i=F_{i,n},\quad A=A_n,\quad (a,b,c)=(a_n,b_n,c_n), 
\end{equation*}
only to simplify notation. All constants below are independent of $n$.
Recall that $(\chi_m)_{m\geq 1}$ denotes the sequence of cutoff functions introduced in Assumption~\ref{ass:chi}. Applying It\^o's formula to $\int_v \chi_m |f_n|^2$ and then passing to the limit as $m\to\infty$, we obtain that, almost surely, for every $t\in[0,T]$,
\begin{equation*}
    \begin{aligned}
 &\|f(t)\|_{L^2}^2+\alpha\int^t_0\|\nabla_vf\|_{L^2}^2\dd s\\
 =&{}\|f_0\|_{L^2}^2-\int^t_0\int_v (\nabla_v f)^T\bar a (\mu^2)\nabla_v \tili (\mu^2(f))\dd s-\int^t_0\int_v \nabla_vf\cdot  \bar b(\tili (\mu^2(f)))\mu^2(f)\dd s\\
&+\frac{\sqrt{\eps}}{2}\sum_{k=1}^K\int^t_0\int_{v,v_*}\tn f\cdot\big(G_k(v,v_*)\sigma\sigma_*\big)\dd B_k\\
&+\frac{\eps}{2}\int^t_0\int_{v,v_*,w}\Big(\sigma'(f)\sigma(f)\sigma(f_*)\sigma(f_w)\nabla_{v_*} f_*\cdot F_1 +\sigma'(f)\sigma(f)\sigma(f_*)\sigma(f_w)\nabla_v f\cdot F_2\\
&\quad+\sigma^2(f)\sigma(f_*)\sigma(f_w)F_3+\sigma(f_*)\sigma(f_w)\nabla_{v_*} f_*\otimes \nabla_v \varsigma(f) :F_4\Big)\dd s.
 \end{aligned}
\end{equation*}

We next estimate each term on the right-hand side uniformly with respect to $n$. This shows that, for every $\delta\in(0,1)$, the right-hand side of the above equality is bounded by 
\begin{align*}
   \int^t_0   \delta\|\nabla_v f\|_{L^2(\Do)}^2
    + C\big(\|f\|_{L^2(\Do)}^2+ 1\big)\dd s.
\end{align*}

\medskip

\noindent\textbf{The collision term:}
Compared with the fixed regularised $L^2$ estimate, the only new point is to
replace the bound for the $b$ terms by one that is uniform in $n$.

We define 
\begin{align*}
\tilde\mu(s)=\int_0^s \mu^2(r)\dd r,\quad 0\le \tilde\mu(f)\le \frac{f^2}{2}.    
\end{align*}
Then using the integration by parts formula to see that almost surely, for every $t\in[0,T]$, 
    \begin{align*}
        &\Big|\int^t_0\int_v \nabla_vf\cdot   {\bar b}(\tili(\mu^2(f)))\mu^2(f)\dd s\Big|\\
        =&{}\Big|\int^t_0\int_{v,v_*} \nabla_v\tilde\mu(f)\cdot  b(v-v_*)\tili(\mu^2(f_*))\dd s\Big|\\
         =&{}\Big|\int^t_0\int_{v,v_*} \tilde\mu(f)   c(v-v_*)\tili(\mu^2(f_*))\dd s\Big|\\
        \lesssim&{}\int^t_0\int_{v,v_*} f^2 f_* |v-v_*|^\gamma\dd s,
             \end{align*}
             where we use the uniform bound \eqref{bdd:c-n}.

In the following, we show that for every $\delta\in(0,1)$, there exists a constant $C_\delta>0$ such that:
    \begin{align}\label{claim-bn}
    \int^t_0\int_{v,v_*} f^2 f_* |v-v_*|^\gamma\dd s
    \le \int^t_0 \delta \|\nabla_v f\|_{L^2(\Do)}^2
    + C_\delta \|f\|_{L^2(\Do)}^2 \dd s.
\end{align}

The proof is adapted from the propagation of $L^p$-norm estimates for the Landau equation established in \cite{Wu14}.
We decompose the domain as $\Do=\{|v-v_*|\le 1\}\cup \{|v-v_*|>1\}$. On the region $\{|v-v_*|>1\}$, we have
\begin{align*}
    \int^t_0\int_{\{|v-v_*|>1\}} f^2 f_* |v-v_*|^\gamma\dd s
    \le
    \int^t_0 \|f\|_{L^2}^2 \|f_0\|_{L^1}\dd s.
\end{align*}

On the region $\{|v-v_*|\le 1\}$, we obtain that almost surely, for every $t\in[0,T]$, 
\begin{align*}
  \int^t_0\int_{\{|v-v_*|\le 1\}} f^2 f_* |v-v_*|^\gamma\dd s
  \lesssim{}&
  \int^t_0\int_{v_*}
  f_*
  \Big(
  \int_{\{|v-v_*|\le 1\}}
  |v-v_*|^\gamma f^2\dd v
  \Big)
  \dd v_*\dd s
  \\
  \le{}&
  \int^t_0
  \|f_0\|_{L^1(\Do)}
  \sup_{v_*\in\mathbb{R}^d}
  \int_{\{|v-v_*|\le 1\}}
  |v-v_*|^\gamma f^2\dd v
  \dd s.
\end{align*}

Let $\hat f$ denote the Fourier transform of $f$ with respect to the variable $v$:
\begin{equation*}
    \hat f(x,\xi)
    =(2\pi)^{-\frac{d}{2}}
    \int_v f(v)e^{-iv\cdot\xi}\dd v.
\end{equation*}
Applying Pitt's inequality (see, for instance, \cite{Bec08,Bec08b,Bec12}), we obtain, for every $v_*\in\mathbb{R}^d$,
\begin{equation}
\label{soft-pitt}
\begin{aligned}
\int_v |v-v_*|^\gamma f^2\dd v
=
\int_v |v|^\gamma f^2(v+v_*)\dd v
\le
C_{pitt}
\int_\xi |\xi|^{-\gamma} |\widehat{f}|^2\dd \xi.
\end{aligned}
\end{equation}

Next, we decompose $\mathbb{R}^d$ as
$$
\mathbb{R}^d=\{|\xi|\le R\}\cup\{|\xi|>R\},
$$
for some $R>0$ large to be determined later. On the region $\{|\xi|\le R\}$, Parseval's theorem yields
\begin{align*}
 \int_{\{|\xi|\le R\}}
 |\xi|^{-\gamma}
 |\widehat{f}|^2
 \dd \xi
 \le
 R^{|\gamma|}
 \|f\|_{L^2(\Do)}^2.
\end{align*}

Similarly, on the region $\{|\xi|>R\}$,
\begin{align*}
 \int_{\{|\xi|>R\}}
 |\xi|^{-\gamma}
 |\widehat{f}|^2
 \dd \xi
 \le
 R^{|\gamma|-2}
 \int_\xi |\xi|^2 |\widehat{f}|^2 \dd \xi
 =
 R^{|\gamma|-2}
 \|\nabla_v f\|_{L^2}^2.
\end{align*}

Combining the above estimates and choosing $R>0$ large yields \eqref{claim-bn}.

\medskip

\noindent\textbf{The stochastic noise term:} 
The stochastic integral is a continuous local martingale. After the
standard stopping-time localisation, its expectation is zero and hence it
does not contribute to \eqref{L-2:app-re-integrated}. To obtain the estimate of the
time supremum in \eqref{L-2:app-re-integrated}, the
Burkholder--Davis--Gundy inequality, the uniform coefficient bound, and
$\sigma(f)\le \sqrt{f}$ and $\sigma(f)\le C_{r_0}f$ give, for every $\delta>0$,
 \begin{align*} 
&\sum_{k=1}^K\hE\Big[\sup_{t\in[0,T]}\Big|\int^t_0\int_{v,v_*}\tn f\cdot\big(G_{k,n}(v,v_*)\sigma(f)\sigma(f_*)\big)\dd B_k\Big|\Big]\\
    \le&{}C\sum_{k=1}^K\hE\Big[\Big(\int^T_0\Big|\int_{v,v_*}\sigma(f)\sigma(f_*)\tn  f\cdot G_{k,n}(v,v_*)\Big|^2\dd s\Big)^{\frac12}\Big]\\
    \lesssim&{}_{C_K,C_{r_0}}\hE\Big[\Big(\int^T_0\Big(\int_{v,v_*}f_*\sqrt{f}|\nabla_v f|\Big)^2\dd s\Big)^{\frac12}\Big]\\
     \le&{}\delta \hE\Big[\int^T_0\|\nabla_v f\|_{L^2}^2\dd s\Big]+C(\delta,C_{r_0},C_K,T)\|f_0\|_{L^1}^2.
\end{align*}

\medskip

\noindent\textbf{The Stratonovich-to-It\^o correction terms:} We make use of the bound $\sigma(f)\le \sqrt{f}$. By H\"older's and Young's inequalities, we obtain, almost surely, for every $t\in[0,T]$,
\begin{equation*}
 \begin{aligned}
 &\Big|\int^t_0\int_{v,v_*,w}
 \sigma'(f)\sigma(f)\sigma(f_*)\sigma(f_w)\nabla_{v_*} f_*
 \cdot F_1\dd s\Big|\\
 &\le
 C_{r_0}
 \Big|
 \int^t_0\int_{v,v_*,w}
 \sigma(f)\sigma(f_*)\sigma(f_w)\nabla_{v_*} f_*
 \cdot F_1\dd s
 \Big|
 \\
 &\le
 C_{r_0}
 \int^t_0
 \|f\|_{L^1}^3
 \|\nabla_{v_*} f_*\|_{L^2}
 \|F_1\|_{L^2_{v,w}L^\infty_{v_*}}
 \dd s
 \\
 &\le
 \delta \int^t_0
 \|\nabla_v f\|_{L^2}^2\dd s
 +
 C(\delta,C_{r_0},C_K,T)\|f_0\|_{L^1}^{6}.
 \end{aligned}
\end{equation*}

Similarly, almost surely, for every $t\in[0,T]$,
\begin{equation}
 \begin{aligned}
 \Big|
 \int^t_0\int_{v,v_*,w}
 \sigma'(f)\sigma(f)\sigma(f_*)\sigma(f_w)
 \nabla_v f\cdot F_2
 \dd s
 \Big|
 \le
 \delta \int^t_0
  \|\nabla_v f\|_{L^2}^2\dd s
 +
 C(\delta,C_{r_0},C_K,T)\|f_0\|_{L^1}^{6},
 \end{aligned}
\end{equation}
and
\begin{align*}
 \Big|
 \int^t_0\int_{v,v_*,w}
 \sigma^2(f)\sigma(f_*)\sigma(f_w)F_3
 \dd s
 \Big|
 \le
 C(C_K,T)
 \|f_0\|_{L^1}^3.
\end{align*}

For the $F_4$ term, we first integrate by parts with respect to $v$. Using the bound
$$
|\varsigma(f)|\lesssim_{C_{r_0}} |f|,
$$
together with the same Young inequality used for the $F_2$-type correction,
we deduce that, almost surely, for every $t\in[0,T]$,
\begin{align*}
 &\Big|
 \int^t_0\int_{v,v_*,w}
 \sigma(f_*)\sigma(f_w)
 \nabla_{v_*} f_*
 \otimes
 \nabla_v \varsigma(f)
 :F_4
 \dd s
 \Big|\\
 =&{}
 \Big|
 \int^t_0\int_{v,v_*,w}
 \varsigma(f)\sigma(f_*)\sigma(f_w)
 \nabla_{v_*} f_*
 \cdot
 \big(\nabla_v\cdot F_4\big)
 \dd s
 \Big|
 \\
 \le&{}
 \delta \int^t_0
 \|\nabla_v f\|_{L^2}^2\dd s
 +
 C(\delta,C_{r_0},C_K,T)
 \|f_0\|_{L^1}^{6}.
\end{align*}

\end{proof}

\begin{remark}[$L^2$-estimate in the case of $\gamma=-2$]
    \label{rmk:-2}
For the deterministic Landau equation, the existence of weak solutions in $L^1$-framework in the moderately soft potential case $\gamma\in[-2,0)$ can be treated in a unified way. However, for the fluctuating equation, we first establish existence results for the approximate equations in an $L^2$-framework. The uniform $L^2$-bounds in Lemma \ref{lem:L-2:app-refine} fail in the case $\gamma=-2$. More precisely, the estimate \eqref{claim-bn} fails
\begin{align*}
    \int^t_0\int_{v,v_*} f^2 f_* |v-v_*|^\gamma\dd s
    \le \int^t_0 \delta \|\nabla_v f\|_{L^2(\Do)}^2
    + C_\delta \|f\|_{L^2(\Do)}^2 \dd s.
\end{align*}
The proof of \eqref{claim-bn} in Lemma \ref{lem:L-2:app-refine} relies on the condition $\gamma>-2$, where the positivity of $2+\gamma>0$ ensures the smallness of the coefficient multiplying $\|\nabla_v f\|_{L^2(\Do)}^2$. In the deterministic case $\gamma=-2$ considered in \cite{Wu14}, this smallness is instead obtained from the superlinear integrability of $f$. More precisely, this requires a uniform-in-time bound on the Boltzmann entropy $|\cH(f)|$. However, in the SPDE case, we only have uniform bounds on the expectation of the entropy, as in \eqref{H:h-eps}. Hence, this strategy is not available here. A similar difficulty also appears in Lemma \ref{lem:Des}, in the derivation of the weighted Fisher information bound, where superlinear integrability is likewise used.

We explain in detail why the strategy of \cite{Wu14} fails in the present setting in the case of $\gamma=-2$. We  split the domain $\{|v-v_*|\le 1\}$ into $\{|v-v_*|\le 1\}\cap \{f_*\le K\}$ and $\{|v-v_*|\le 1\}\cap \{f_*> K\}$ for some $K>1$ to be determined later.

On the domain of $f_*\ge K$, we have
\begin{equation}
\label{omega-2}
\begin{aligned}
\int_{|v-v_*|\le 1,\,f_*\ge K} |v-v_*|^{-2}f^2 f_*
\le\int_{\{f_*> K\}}f_*\Big(\sup_{v_*\in\R^d}\int_{\R^d} |v-v_*|^{-2}f^2\Big).
\end{aligned}
\end{equation}
Similar to \eqref{soft-pitt}, we use Pitt's and Parseval's theorems to derive 
\begin{align*}
\int_{\Do} |v-v_*|^{-2}f^2\dd v\le C_{pitt}\int_{\R^d} |\xi|^{2}|\hat f |^2\dd \xi=C_{pitt}\|\nabla_vf \|_{L^2_v}^2.
\end{align*}
Notice that \eqref{omega-2} can be controlled by $\alpha \|\nabla_vf \|_{L^2_v}^2$ by choosing $K\sim \exp(C |\cH(f)|)$ such that
\begin{align*}
\int_{\{f>K\}}f\le \frac{\int_vf |\log f|}{\log K}\le \frac{|\cH(f)|+CE}{\log K}\le \alpha/4,
\end{align*}
where we use the estimate \eqref{LlogL:minus}. 

However, on the domain of $f_*\le K$, by changing the variable, we have
\begin{align*}
&\int_{|v-v_*|\le 1,\,f_*\le K} |v-v_*|^{-2}f^2 f_*\le K\int_{\{|v|\le1\}} |v|^{-2}\int_{v}f^2\sim  e^{C|\cH(f)|}\|f\|_{L^2_v}^2.
\end{align*}
Due to the lack of pathwise entropy bounds, we cannot control the right-hand side of the above inequality almost surely.

\end{remark}

We establish a refined version of Lemma~\ref{lem:time-regularity} that yields a time-regularity estimate uniform in $n\in\mathbb{N}$. 
\begin{lemma}[Uniform local time regularity for $(f_n)_{n\in\mathbb N}$]
\label{lem:time-regularity:L-2-2}
Let $f_0\in L^2(\R^d)\cap L^1_2(\R^d;\R_+)$. Let $n\in\N$. Let $f_n$ be the same as in Lemma \ref{lem:L-2:app-refine}.
Let
$\chi\in C_c^\infty(\R^d)$, $l>d/2+2$, and $\beta\in(0,1/2)$. Then,
for every
\begin{align*}
p>\frac{1}{1/2-\beta},
\end{align*}
there exists a constant
\begin{align*}
C=C\big(p,\beta,l,T,\chi,d,\gamma,r_0,C_K,
\|f_0\|_{L^1_2}\big)>0,
\end{align*}
independent of $n$, such that
\begin{align}
\label{time-regularity:L-2-2}
\sup_{n\in\mathbb N}
\mathbb E_n\big[
\|f_n\chi\|_{C^\beta([0,T];H^{-l}(\R^d))}^{p}
\big]\le C.
\end{align}
In particular, the laws of $(f_n\chi)_{n\in\mathbb N}$ are tight in
$C^{\beta'}([0,T];H^{-l-1}(\R^d))$ for every $\beta'<\beta$.
\end{lemma}

\begin{proof}
We fix an arbitrary $n\in\mathbb N$. As in Lemma~\ref{lem:L-2:app-refine}, we suppress
the index $n$ and write
\begin{equation*}
f=f_n,\quad \mu=\mu_n,\quad \sigma=\sigma_n,\quad \tili=\tili_n,\quad
G_k=G_{k,n},\quad F_i=F_{i,n},\quad (a,b)=(a_n,b_n).
\end{equation*}
All constants below are independent of $n$. Write the localised weak equation
as
\begin{align*}
f(t)\chi=f_0\chi+D(t)+M(t),
\end{align*}
where $D$ contains all deterministic terms and
\begin{align*}
M(t)
=-\frac{\sqrt{\eps}}{2}\sum_{k=1}^K\int_0^t
\chi\,\tn\cdot\big(G_k(v,v_*)\sigma(f)\sigma(f_*)\big)\,\dd B_k
\end{align*}
is the localised martingale term.

\medskip

\noindent\textbf{Deterministic increments.}
Let $\phi\in H^l(\R^d)$ with $\|\phi\|_{H^l}\le1$. Since
$l>d/2+2$, Sobolev embedding yields
\begin{align}
\label{Sobolev-time-reg}
\|\chi\phi\|_{W^{2,\infty}}\le C_{\chi,l}.
\end{align}
Using the weak formulation of \eqref{eq:ito-2:app-n}, together with the bounds
$\mu_n^2(f_n),\,\sigma_n^2(f_n)\le f_n$, $|\sigma_n'(f_n)|\le C_{r_0}$, and the uniform coefficient estimates 
$$
a(v-v_*)\lesssim \langle v-v_*\rangle^2,
\qquad
|b(v-v_*)(\nabla_v\varphi-\nabla_{v_*}\varphi_*)|
\lesssim
\|\nabla^2\varphi\|_{L^{\infty}}|v-v_*|^{2+\gamma}
\le
C(\varphi)\langle v-v_*\rangle^2,
$$
we deduce that, for every $0\le s<t\le T$,
\begin{align*}
|\langle D(t)-D(s),\phi\rangle|
\lesssim{}&
\int_s^t\int_v f\,|\Delta(\chi\phi)|\,\dd v\dd r\\
&+\int_s^t\int_{v,v_*}ff_*
\big(\langle v\rangle^2+\langle v_*\rangle^2\big)
|\nabla^2(\chi\phi)|\,\dd v\dd v_*\dd r\\
&+\int_s^t\int_{v,v_*,w}ff_*f_w
\big(1+|F_1|+|F_4|+|\nabla_vF_4|\big)
\|\chi\phi\|_{W^{2,\infty}}\,\dd v\dd v_*\dd w\dd r.
\end{align*}
The mass and energy estimates \eqref{mass:h-eps}--\eqref{energy:h-eps},
\eqref{Sobolev-time-reg}, and the uniform bounds for $F_{i,n}$ therefore
give
\begin{align}
\label{deterministic-time-increment}
\|D(t)-D(s)\|_{H^{-l}}
\le C|t-s|\big(1+\|f_0\|_{L^1_2}^3\big),
\end{align}
where $C$ is independent of $n$.

\medskip

\noindent\textbf{Martingale increments.}
For every $\phi\in H^l(\R^d)$ with $\|\phi\|_{H^l}\le1$, integration by
parts gives
\begin{align*}
|\langle M(t)-M(s),\phi\rangle|
\le \frac12\Big|
\sum_{k=1}^K\int_s^t\int_{v,v_*}
\nabla_v(\chi\phi)\cdot G_{k,n}(v,v_*)\sigma(f)\sigma(f_*)\,\dd v\dd v_*
\dd B_k\Big|.
\end{align*}
By the Burkholder--Davis--Gundy inequality, the uniform bounds on
$G_{k,n}$, and $\sigma^2(f)\le f$, for every $p\ge2$,
\begin{align}
\label{martingale-time-increment}
\mathbb E\|M(t)-M(s)\|_{H^{-l}}^p
&\le C_p\mathbb E\Big[
\int_s^t\sum_{k=1}^K
\Big|\int_{v,v_*}\nabla_v(\chi\phi)\cdot
G_{k,n}\sigma(f)\sigma(f_*)\,\dd v\dd v_*\Big|^2\dd r
\Big]^{p/2}\notag\\
&\le C_p\|f_0\|_{L^1}^{2p}|t-s|^{p/2}.
\end{align}
Let us spell out the last estimate.  By
\eqref{Sobolev-time-reg}, we have $\|\nabla_v(\chi\phi)\|_{L^\infty}\le C_{\chi,l}$.
Using the uniform coefficient bounds and the elementary estimate
$\|\sigma(f)\|_{L^1}\le C_{r_0}\|f\|_{L^1}$ from
Remark~\ref{rmk:sigma-f:L2}, we have, for every $r\in[0,T]$,
\begin{align*}
\Big|\int_{v,v_*}\nabla_v(\chi\phi)\cdot
G_{k,n}(v,v_*)\sigma(f)\sigma(f_*)\,\dd v\dd v_*\Big|
&\le C_{\chi,l,C_K}\|\sigma(f(r))\|_{L^1}^2  \\
&\le C_{\chi,l,C_K,r_0}\|f_0\|_{L^1}^2.
\end{align*}

 Therefore the quadratic variation on $[s,t]$ is
bounded by $C\|f_0\|_{L^1}^4|t-s|$, and BDG gives
\eqref{martingale-time-increment}.  This is the main difference from
Lemma~\ref{lem:time-regularity}: in the Galerkin $L^2$ framework of
Section~\ref{sec:Galerkin}, the corresponding martingale increment is
controlled by $\|f_m\|_{L^2}^2$ and hence requires the higher
$L^2$-moment estimate \eqref{galerkin:L2-uniform}, whereas here the
$L^1$ conservation closes the estimate directly.
The constants in \eqref{martingale-time-increment} are independent of
$n$.

\medskip

\noindent\textbf{Conclusion.}
Combining \eqref{deterministic-time-increment} and
\eqref{martingale-time-increment}, we obtain
\begin{align*}
\mathbb E\|f(t)\chi-f(s)\chi\|_{H^{-l}}^p
\le C_p|t-s|^{p/2}.
\end{align*}
Since $p(1/2-\beta)>1$, the quantitative Kolmogorov continuity theorem
implies
\begin{align*}
\mathbb E\big[
[f\chi]_{C^\beta([0,T];H^{-l})}^{p}\big]\le C.
\end{align*}
Moreover,
\begin{align*}
\mathbb E
\sup_{t\in[0,T]}\|f(t)\chi\|_{H^{-l}}^p
\lesssim_{\chi,l}\|f_0\|_{L^1}^p+C,
\end{align*}
which proves \eqref{time-regularity:L-2-2}. Finally, since all paths are
supported in the fixed compact set $\supp\chi$, the compact local
embedding $H^{-l}\Subset H^{-l-1}$ and the Arzel\`a--Ascoli theorem imply
\begin{align*}
C^\beta([0,T];H^{-l})\Subset
C^{\beta'}([0,T];H^{-l-1}),\qquad \beta'<\beta.
\end{align*}
Chebyshev's inequality then yields the asserted tightness.
\end{proof}

With the aid of the uniform estimates established above, we are now in a position to prove the tightness of the approximation sequence. 
\begin{proposition}[Tightness of the approximation laws]
\label{prop:limit-L2-tightness}
Assume the hypotheses of Proposition~\ref{lem:l-12:log}.  Let
$(f_n,(B^k_n)_{1\le k\le K})$ be the sequence of probabilistic weak
solutions to the regularised equation \eqref{eq:ito-2:app-n}.  Then the
laws of
\begin{align*}
\big(f_n,\nabla_v f_n,(B^k_n)_{1\le k\le K}\big)
\end{align*}
are tight on
\begin{align*}
\mathbb X
&:=\Big(L^2([0,T];L^2_{\loc}(\R^d))
\cap C^{\beta'}([0,T];H^{-l-1}_{\loc}(\R^d))\cap L^1([0,T];L^1(\R^d))\Big)\\
&\qquad\times \big(L^2(0,T;L^2(\R^d)),w\big)
\times C([0,T];\R^K),
\end{align*}
for every $\beta'\in(0,1/2)$ and $l>d/2+2$. 
\end{proposition}
\begin{proof}
 Let $R>1$ be arbitrarily fixed. Let $\chi_R\in C_c^\infty(\R^d;\R_+)$ be a smooth cut-off function given as in Assumption \ref{ass:chi}.
By Lemma~\ref{lem:L-2:app-refine}, we have the unifor bound
\begin{align*}
\sup_{n\in\N}\mathbb E\Big[
\|f_n\|_{L^\infty_tL^2_v}^2
+\|\nabla_v f_n\|_{L^2_tL^2_v}^2\Big]<+\infty.
\end{align*}
By Lemma~\ref{lem:time-regularity:L-2-2}, for every $R>0$ the sequence
$(f_n\chi_R)_{n\ge1}$ is tight in
$C^{\beta'}([0,T];H^{-l-1}(\R^d))$.  The compact embedding
\begin{align*}
L^2([0,T];H^1(B_{2R}))\cap C^\beta([0,T];H^{-l}(B_{2R}))
\Subset L^2([0,T];L^2(B_R))
\end{align*}
then gives tightness of $(f_n)_{n\ge1}$ in
$L^2([0,T];L^2_{\loc}(\R^d))$. Furthermore, by H\"older's inequality, the above estimate implies the tightness of $(f_n)_{n\ge1}$ in $L^1([0,T];L^1_{\mathrm{loc}}(\mathbb{R}^d))$. 
The uniform energy bound \eqref{energy:h-eps} ensures that the mass tails are controlled uniformly small. It follows that $(f_n)_{n\ge1}$ is tight in $L^1([0,T];L^1(\mathbb{R}^d))$. The gradient component is tight in
$L^2([0,T];L^2(\R^d))$ endowed with the weak topology by the uniform
$L^2$-bound, and the Brownian laws are tight on $C([0,T];\R^K)$.  Hence
the joint laws are tight on $\mathbb X$.
\end{proof}

\subsection{Proof of Proposition \ref{lem:l-12:log}}\label{sec-5:limit}

We divide the proof into the following five steps.

\medskip

\noindent\textbf{Step 1: Skorokhod representation.}
To simplify the notation, we denote the collection of Brownian motions by $B_n=(B_n^k)_{1\leq k\leq K}$. By Proposition~\ref{prop:limit-L2-tightness}, the laws of
$(f_n,\nabla_v f_n,B_n)$ are tight on $\mathbb X$.

Since the weak topology on the gradient component is not metrisable on
the whole space, we use the Jakubowski--Skorokhod representation theorem.
There are a new probability space
$(\bar\Omega,\bar{\mathcal F},\bar{\mathbb P})$ and random variables
\begin{align*}
(\bar f_n,\nabla_v\bar f_n,\bar B_n)
\quad\hbox{and}\quad
(\bar f,\nabla_v\bar f,\bar B)
\end{align*}
such that
\begin{align}
\label{limit-L2:skorokhod}
(\bar f_n,\nabla_v\bar f_n,\bar B_n)
\to(\bar f,\nabla_v\bar f,\bar B)
\quad\bar{\mathbb P}\hbox{-almost surely in }\mathbb X,
\end{align}
and the law of $(\bar f_n,\nabla_v\bar f_n,\bar B_n)$ is the same as that
of $(f_n,\nabla_v f_n,B_n)$. 

In particular, we have almost surely, for every $R>0$,
\begin{align}
\label{limit-L2:conv}
&\bar f_n\to\bar f
\quad\hbox{strongly in }L^2([0,T];L^2(B_R))\cap C^{\beta'}([0,T];H^{-l-1}_{\loc}(\R^d))\cap L^1([0,T];L^1(\mathbb{R}^{2d})),\notag\\
&\nabla_v\bar f_n\rightharpoonup\nabla_v\bar f
\quad\hbox{weakly in }L^2([0,T];L^2(\R^d)),\\
&\bar B_n\to\bar B
\quad\hbox{in }C([0,T];\R^K).\notag
\end{align}
The non-negativity of $\bar f$ follows from the strong local convergence
and the non-negativity of $\bar f_n$.

\medskip

\noindent\textbf{Step 2: Convergence of the nonlinear coefficients.}
We note that
\begin{equation}
\label{sigma:d}
\sigma_n'( \bar{f}_n)\le C_{r_0}\quad \forall\, n\in\N,\quad   \sigma_n'( \bar{f}_n)\to  {\sigma}'(\bar{f})\quad \text{pointwisely as $n\to\infty$} .
\end{equation}
 By the construction in Approximation  \ref{quadruple}, we have the following uniform bounds for all $n\in\N$ and pointwise convergence as $n\to\infty$
\begin{equation}
    \label{coe:n:conv}
\begin{gathered}
0\le \sigma_n( \bar{f}_n)\lesssim_{C_{r_0}} \min\big(\sqrt{\bar{f}_n},\bar{f}_n\big),\quad \sigma_n( \bar{f}_n) \to  {\sigma}(\bar f),\\
0\le \varsigma_n( \bar{f}_n)\lesssim_{C_{r_0}} \min\big(\sqrt{\bar{f}_n},\bar{f}_n\big),\quad \varsigma_n( \bar{f}_n) \to  {\varsigma}(\bar{f}),\\
0\le \mu_n( \bar{f}_n)\le \sqrt{\bar{f}_n},\quad \mu_n( \bar{f}_n) \to  \sqrt{\bar{f}},\\
0\le \tili_n (\mu_n^2(\bar{f}_n))\le \mu_n^2(\bar{f}_n),\quad\tili_n (\mu_n^2(\bar{f}_n))\to \bar{f}.
\end{gathered}
\end{equation}
By the dominated convergence theorem, we have, for every $R>0$,
\begin{align}
\label{L2-local}
\mu_n(\bar f_n)\to  \sqrt{\bar f},\qquad \sigma_n(\bar f_n)\to\sigma(\bar f),\qquad
\varsigma_n(\bar f_n)\to\varsigma(\bar f)
\quad\hbox{in }L^2(0,T;L^2(B_R)),
\end{align}
and
\begin{align}
\label{DCT:energy}
\mu_n^2(\bar f_n)\to \bar f,\qquad
\tili_n(\mu_n^2(\bar f_n))\to \bar f
\quad\hbox{in }L^1(0,T;L^1(B_R)).
\end{align}
The uniform  energy bounds \eqref{energy:h-eps}
control the velocity tails
\begin{align}
\label{tail}
\int_{B_R^c}\bar f_n
\le \frac{1}{R^2}
\int_v |v|^2\bar f_n\to0
\end{align}
as $R\to\infty$ uniformly in $t$ and $n$.  The upper bounds in \eqref{coe:n:conv} ensure that the above 
convergences hold globally with the same exponents
\begin{equation}
\label{L12-global}
    \begin{gathered}
\mu_n(\bar f_n)\to  \sqrt{\bar f},\quad \sigma_n(\bar f_n)\to\sigma(\bar f),\quad
\varsigma_n(\bar f_n)\to\varsigma(\bar f)
\quad\hbox{in }L^2(0,T;L^2(\Do)),\\ 
\mu_n^2(\bar f_n)\to \bar f,\quad
\tili_n(\mu_n^2(\bar f_n))\to \bar f
\quad\hbox{in }L^1(0,T;L^1(\Do)).
    \end{gathered}
\end{equation}

\medskip

\noindent\textbf{Step 3: Pass to the limit in weak formulation \eqref{weak:app-1}: deterministic terms.}
Let $\phi\in C_c^\infty(\R^d)$. We show that
\begin{align}
&\int_0^t\int_{v}\bar f_n\,\Delta\phi\to \int_0^t\int_{v}\bar f\,\Delta\phi,\label{Delta:n+}\\
&\int_0^t\int_{v,v_*}
\mu_n^2(\bar f_{n*})\tili_n(\mu_n^2(\bar f_n))a_{n}:\nabla^2\phi\to \int_0^t\int_{v,v_*}
\bar f\bar f_{*}a:\nabla^2\phi,
\label{a:n+}\\
&\int_0^t\int_{v,v_*}
\mu^2_n(\bar f_{n*})\tili_n(\mu^2_n(\bar f_n))b_n
\cdot \big(\nabla\phi-\nabla_{*}\phi_*\big)
\to \int_0^t\int_{v,v_*}
\bar f\bar f_* b
\cdot \big(\nabla\phi-\nabla_{*}\phi_*\big),\label{b:n+}\\
&\int_0^t\int_{v,v_*,w}
\big(\tn\phi(v)\cdot F_{1,n}(v,v_*,w)\big)\sigma_n(\bar f_n)\sigma_n'(\bar f_n)
\sigma_n(\bar f_{n*})\sigma_n(\bar f_{nw})\notag\\
&\qquad \to \int_0^t\int_{v,v_*,w}
\big(\tn\phi(v)\cdot F_{1}(v,v_*,w)\big)\sigma(\bar f)\sigma'(\bar f)
\sigma(\bar f_{*})\sigma(\bar f_{w}),\label{F1:n+}\\
&\int_0^t\int_{v,v_*,w}
\nabla_v\cdot\big(\tn\phi(v)F_{4,n}(v,v_*,w)\big)
\varsigma_n(\bar f_n)
\sigma_n(\bar f_{n*})\sigma_n( \bar f_{nw})\notag\\
&\qquad \to \int_0^t\int_{v,v_*,w}
\nabla_v\cdot\big(\tn\phi(v)F_{4}(v,v_*,w)\big)
\varsigma(\bar f)
\sigma(\bar f_{*})\sigma(\bar f_{w}).\label{F4:n+}
\end{align}

The limit \eqref{Delta:n+} holds as a direct consequence of \eqref{limit-L2:conv}.

We next pass to the limit for Landau drift terms \eqref{a:n+} and \eqref{b:n+}.  
By uniform bound \eqref{bdd:ab-n}, we have 
\begin{align*}
 |a_n|\le  | v-v_*|^{2+\gamma}\quad\text{and}\quad 
|b_n\cdot (\nabla\varphi-\nabla_{*}\varphi_*)|
\lesssim
\|\nabla^2\varphi\|_{L^{\infty}}| v-v_*|^{2+\gamma}.
\end{align*}
Since $2+\gamma\in(0,2)$ and
$|v-v_*|^{2+\gamma}\le \langle v\rangle^2+\langle v_*\rangle^2$,
the uniform energy bound implies that the contributions of \eqref{a:n+} and \eqref{b:n+} over the region $\{|v|+|v_*|\ge R\}$ are uniformly small in $n$ by choosing $R$ sufficiently large, using an argument similar to that in \eqref{tail}.
On the bounded region $\{|v|+|v_*|\le R\}$, we pass to the limit by using the dominated convergence theorem and the pointwise convergence \eqref{coe:n:conv}.

We now pass to the Stratonovich-to-It\^o correction terms \eqref{F1:n+} and \eqref{F4:n+}.  By \eqref{sigma:d} and \eqref{L12-global}, we have 
\begin{gather*}
\sigma_n(\bar f_n)\sigma_n'(\bar f_n)
\sigma_n(\bar f_{n*})\sigma_n(\bar f_{nw})\to \sigma(\bar f)\sigma'(\bar f)
\sigma(\bar f_{*})\sigma(\bar f_{w})\quad \text{in}\quad L^2_tL^2_{v,v_*,w}\\
\text{and}\quad \varsigma_n(\bar f_n)
\sigma_n(\bar f_{n*})\sigma_n(\bar f_{nw})\to
\varsigma(\bar f)
\sigma(\bar f_{*})\sigma(\bar f_{w})\quad \text{in}\quad L^2_tL^2_{v,v_*,w}.
\end{gather*}
Combining with the convergence of $F_{i,n}$ in \eqref{G-k:conv}, the convergences \eqref{F1:n+} and \eqref{F4:n+} hold.

\medskip

\noindent\textbf{Step $4$: Identification of the martingale part.}
Let $\bar{\mathcal G}_t$ be the augmented filtration generated by
$\bar f|_{[0,t]}$, $\nabla_v\bar f|_{[0,t]}$, and
$\bar B|_{[0,t]}$.  As in Proposition~\ref{lem:app-1:ext}, the identity
in law and the convergence \eqref{limit-L2:skorokhod}, combined with the
same martingale-characterisation argument as in Step~3 of
Proposition~\ref{lem:app-1:ext}, imply that
$(\bar B^k)_{1\le k\le K}$ is a family of independent Brownian motions
with respect to $(\bar{\mathcal G}_t)_{t\in[0,T]}$.

	For each $n$ and each test function $\phi$, define the 
	martingale $\bar M_n^\phi$ by subtracting from
	$\langle \bar f_n(t),\phi\rangle$ all deterministic terms in the weak
	formulation of \eqref{eq:ito-2:app-n}.  By the equality of laws,
	setting
	\begin{align*}
	H_{n,k}^\phi(t)
	:=
	\int_{v,v_*}
	\tn\phi(v)\cdot G_{k,n}(v,v_*)\sigma_n(\bar f_n(t,v))
	\sigma_n(\bar f_n(t,v_*))\,\dd v\dd v_*,
	\end{align*}
	we have
	\begin{align*}
	\bar M_n^\phi(t)
	=-\frac{\sqrt{\eps}}{2}\sum_{k=1}^K\int_0^t H_{n,k}^\phi(s)\,\dd\bar B_n^k(s).
	\end{align*}
	The convergence proved in Step~$2$ implies that the residuals converge,
	for every $t\in[0,T]$,
	\begin{align*}
	\bar M_n^\phi(t)\to\bar M^\phi(t),
	\end{align*}
	where $\bar M^\phi$ is the residual obtained from the limit equation.
	We now identify this residual.

	Let
	\begin{align*}
	H_{k}^\phi(t)
	:=
	\int_{v,v_*}
	\tn\phi(v)\cdot G_{k}(v,v_*)\sigma(\bar f(t,v))
	\sigma(\bar f(t,v_*))\,\dd v\dd v_* .
	\end{align*}
	We first record the convergence of the covariation densities.  By the convergence of $G_{k,n}$ in \eqref{G-k:conv}, the convergence in
	Step~$2$ gives, almost surely,
	\begin{align*}
	H_{n,k}^\phi\to H_k^\phi
	\quad\hbox{in }L^2([0,T]),\qquad 1\le k\le K.
	\end{align*}
	Indeed, the difference is bounded by the sum of
	\begin{align*}
	&\int_{v,v_*}|\tn\phi|\,|G_{k,n}-G_k|\,
	|\sigma_n(\bar f_n)|\,|\sigma_n(\bar f_{n,*})|\,\dd v\dd v_*,
	\\
	&\int_{v,v_*}|\tn\phi|\,|G_k|\,
	|\sigma_n(\bar f_n)-\sigma(\bar f)|\,|\sigma_n(\bar f_{n,*})|\,\dd v\dd v_*,
	\\
	&\int_{v,v_*}|\tn\phi|\,|G_k|\,
	|\sigma(\bar f)|\,|\sigma_n(\bar f_{n,*})-\sigma(\bar f_*)|\,\dd v\dd v_*,
	\end{align*}
	and each term converges to zero in $L^2(0,T)$ by the strong
	$L^2_{t,v}$ convergence of $\sigma_n(\bar f_n)$, the uniform
	$L^\infty_tL^2_v$ bounds following from $\sigma_n^2(r)\le r$, and the
	uniform compact support of $G_{k,n}$.
	
	 Let $0\le s\le t\le T$ and let
	$\Theta$ be a bounded continuous functional of
	$(\bar f,\nabla_v\bar f,\bar B)$ restricted to $[0,s]$.  The equality
	of laws, applied first to the approximating system, gives
	\begin{align}
	\label{limit-L2:mart-id-1}
	\bar{\mathbb E}\Big[
	\Theta_n\big(\bar M_n^\phi(t)-\bar M_n^\phi(s)\big)
	\Big]=0,
	\end{align}
	where $\Theta_n$ denotes the same functional evaluated at
	$(\bar f_n,\nabla_v\bar f_n,\bar B_n)$.  Passing to the limit, using
	the convergence above and a standard localisation by the uniform
	$L^\infty_tL^2_v$ bound, shows that $\bar M^\phi$ is a continuous
	$\bar{\mathcal G}_t$-martingale.
	
	Next, for each $1\le j\le K$, It\^o's product rule for the
	approximating stochastic integral yields
	\begin{align}
	\label{limit-L2:cross-id-n}
	\bar{\mathbb E}\Big[
	\Theta_n\Big(
	&\bar M_n^\phi(t)\bar B_n^j(t)
	-\bar M_n^\phi(s)\bar B_n^j(s)
	+\frac{\sqrt{\eps}}{2}\int_s^t H_{n,j}^\phi(r)\,\dd r
	\Big)\Big]=0 .
	\end{align}
	The sign comes from
	$\dd\langle \bar M_n^\phi,\bar B_n^j\rangle_t
	=-\frac{\sqrt{\eps}}{2}H_{n,j}^\phi(t)\dd t$.  Letting $n\to\infty$ in
	\eqref{limit-L2:cross-id-n} gives
	\begin{align}
	\label{limit-L2:cross-id}
	\bar M^\phi(t)\bar B^j(t)
	+\frac{\sqrt{\eps}}{2}\int_0^t H_j^\phi(r)\,\dd r
	\quad\hbox{is a }\bar{\mathcal G}_t\hbox{-martingale}.
	\end{align}
	Equivalently,
	\begin{align}
	\label{limit-L2:cross-bracket}
	\langle \bar M^\phi,\bar B^j\rangle_t
	=
	-\frac{\sqrt{\eps}}{2}\int_0^t H_j^\phi(r)\,\dd r.
	\end{align}
	Similarly, again by It\^o's formula,
	\begin{align}
	\label{limit-L2:quad-id-n}
	\bar{\mathbb E}\Big[
	\Theta_n\Big(
	(\bar M_n^\phi(t))^2-(\bar M_n^\phi(s))^2
	-\frac{\eps}{4}\int_s^t\sum_{k=1}^K |H_{n,k}^\phi(r)|^2\,\dd r
	\Big)\Big]=0 .
	\end{align}
	The convergence $H_{n,k}^\phi\to H_k^\phi$ in $L^2(0,T)$ implies
	\begin{align}
	\label{limit-L2:quad-id}
	(\bar M^\phi(t))^2
	-\frac{\eps}{4}\int_0^t\sum_{k=1}^K |H_k^\phi(r)|^2\,\dd r
	\quad\hbox{is a }\bar{\mathcal G}_t\hbox{-martingale},
	\end{align}
	and hence
	\begin{align}
	\label{limit-L2:quad-bracket}
	\langle \bar M^\phi\rangle_t
	=
	\frac{\eps}{4}\int_0^t\sum_{k=1}^K |H_k^\phi(r)|^2\,\dd r.
	\end{align}
	As a consequence, 
	\begin{align}\label{martingale-characterisation}
	\bar M^\phi(t)
	=-\frac{\sqrt{\eps}}{2}\sum_{k=1}^K\int_0^t\int_{v,v_*}
	\tn\phi(v)\cdot G_k(v,v_*)\sigma(\bar f)
\sigma(\bar f_*)\,\dd v\dd v_*\,\dd\bar B^k .
\end{align}
Thus $\bar f$ satisfies the weak formulation corresponding to
\eqref{eq:ito-2}.  This proves that
$(\bar f,(\bar B^k)_{1\le k\le K})$ is a probabilistic weak solution.

\medskip

\noindent\textbf{Step $5$: Pass to the limit in conservation laws, energy and entropy inequalities.}
It remains to pass to the estimates \eqref{mass:h-eps},
\eqref{momentum:h-eps},
\eqref{energy:h-eps}, and \eqref{H:h-eps}.  Since
$(\bar f_n,\bar B_n)$ and $(f_n,B_n)$ have the same law, the three estimates
hold for $\bar f_n$ with the same constants.  By \eqref{limit-L2:conv}, after
possibly extracting a further subsequence, $\bar f_n(t)\to \bar f(t)$ in
$L^1(\R^d)$ for a.e. $t\in[0,T]$ as well as pointwisely.  Combining this convergence
with the uniform second moment bound in \eqref{energy:h-eps}, we have for almost every $t\in[0,T]$
\begin{align}\label{mass-energy-aet}
\int_v(1,v)\bar f_t(v)\,\dd v
&=\lim_{n\to\infty}\int_v(1,v)\bar f_n(t,v)\,\dd v
=\int_v (1,v) f_0(v)\,\dd v,\notag\\
\int_v |v|^2\bar f_t(v)\,\dd v
&\le \liminf_{n\to\infty}\int_v |v|^2\bar f_n(t,v)\,\dd v\notag\\
&\le \int_v |v|^2f_0(v)\,\dd v+2\alpha dT\|f_0\|_{L^1}.
\end{align}
In the following, we show that the above preservation of mass and the energy bounds hold almost surely for every $t\in[0,T]$. 

For each $R>1$, let $\chi_R\in C_c^\infty(\mathbb{R}^d)$ be a smooth cut-off function such that $0\le \chi_R\le 1$,
with $\chi_R=1$ on $\{|v|\le R\}$ and $\chi_R=0$ on $\{|v|\ge 2R\}$. By the weak continuity in time, it follows that, almost surely, for every $t\in[0,T]$, there exists a sequence $(s_j)_{j\ge1}$ with $s_j\to t$ as $j\to\infty$ such that
\begin{align*}
\int_v \bar f_t\,\chi_R
=
\lim_{j\to\infty}
\int_v \bar f_{s_j}\,\chi_R,
\end{align*}
for every $R>1$, and \eqref{mass-energy-aet} holds at each time $s_j$.

Consequently,
\begin{align*}
\int_v \bar f_t\,\chi_R
=
\lim_{j\to\infty}
\int_v \bar f_{s_j}\,\chi_R
\le
\|f_0\|_{L^1},
\end{align*}
which yields
\begin{align*}
\|\bar f_t\|_{L^1}\le \|f_0\|_{L^1}.
\end{align*}
We next show that this inequality is in fact an equality. Indeed,
\begin{align*}
\int_v \bar f_t
&=
\lim_{R\to\infty}
\int_v \bar f_t\,\chi_R
=
\lim_{R\to\infty}
\lim_{j\to\infty}
\int_v \bar f_{s_j}\,\chi_R
\\
&=
\lim_{R\to\infty}
\lim_{j\to\infty}
\left(
\int_v \bar f_{s_j}
-
\int_v (1-\chi_R)\bar f_{s_j}
\right)
\\
&=
\|f_0\|_{L^1}
-
\lim_{R\to\infty}
\lim_{j\to\infty}
\int_v (1-\chi_R)\bar f_{s_j}.
\end{align*}

Repeat the argument as in \eqref{tail}, we have 
\begin{align*}
\lim_{R\to\infty}
\lim_{j\to\infty}
\int_v (1-\chi_R)\bar f_{s_j}
&\le
\limsup_{R\to\infty}
\limsup_{j\to\infty}
\int_v
\frac{|v|^2}{R^2}
(1-\chi_R)
\bar f_{s_j}=0.
\end{align*}
Therefore, the mass conservation law holds almost surely for every $t\in[0,T]$.

By a similar argument, one also obtains that the momentum conservation law and the energy estimate in \eqref{mass-energy-aet} hold almost surely for every $t\in[0,T]$.

\medskip

We next pass to the limit by letting $n\to\infty$ for the entropy inequalities \eqref{H:h-eps} and \eqref{H:h-eps-exp}.  We recall the definition
\begin{align*}
h_n(r)=\int_0^r L_n(\mu_n^2(s))\,\dd s\quad\text{and}\quad
h(r)=r\log r-r.
\end{align*}
The convergence in \eqref{def:L-n} and \eqref{def:mu-n}, together with the dominated convergence theorem, implies $h_n(r)\to h(r)$ for
every $r\ge0$. Moreover, by the bounds on $\mu_n$ and $L_n$, together with
\eqref{LlogL:sigma}, there is a constant $C$ independent
of $n$ such that
\begin{align*}
|h_n(r)|\le C\big(1+r|\log r|\mathbb 1_{\{0<r\le1\}}+r^2
\big),\qquad r\ge0.
\end{align*}
Since $f_0\in L^1_2\cap L^2(\Do)$, the dominated convergence yields
\begin{align}
\lim_{n\to\infty}\cH_n(f_0)=\cH(f_0)-m_0.
\label{limit-L2:initial-entropy}
\end{align}

It remains to show that almost surely, 
\begin{align}
\label{limit-L2:diss-lsc-final}
\int_0^t\cD(\bar{f})
\le
\liminf_{n\to\infty}
\int_0^t\cD_{n}(\bar{f}_{n}),
\qquad
\forall\, t\in[0,T].
\end{align}
We prove \eqref{limit-L2:diss-lsc-final} by first freezing a bounded Landau kernel and then removing the truncation. Define $q_n(s)=0$ for $s\le0$, and for $s\ge0$,
\begin{align*}
q_n(s):=\int_0^s2L_n'\big(\mu_n^2(r)\big)\mu_n(r)^2\mu_n'(r)\,\dd r.
\end{align*}
The factor $2$ is the chain-rule factor in $\frac{\dd}{\dd r}L_n(\mu_n^2(r))=2L_n'(\mu_n^2(r))\mu_n(r)\mu_n'(r)$. With this convention, $\mu_n(f_{n,*})\nabla q_n(f_n)-\mu_n(f_n)\nabla q_n(f_{n,*})$ is exactly the square-root flux associated with $\mu_n^2(f_n)\mu_n^2(f_{n,*})|\widetilde\nabla L_n(\mu_n^2(f_n))|^2$.
We note that 
\begin{align*}
\int_0^t\cD_n(\bar f_n)
&=
\frac12
\int_0^t
\int_{v,v_*}
A_n \Big|\Pi_{(v-v_*)^\perp}
\Big(
\mu_n(\bar f_{n,*})\nabla_vq_n(\bar f_n)
-
\mu_n(\bar f_{n})\nabla_{v_*}q_n(\bar f_{n,*})
\Big)\Big|^2,
\\
\text{and}\quad\int_0^t\cD(\bar f)
&=
\frac12
\int_0^t
\int_{v,v_*}
A \Big|\Pi_{(v-v_*)^\perp}
\Big(
2\sqrt{\bar f_*}\nabla_v\sqrt{\bar f}
-
2\sqrt{\bar f}\nabla_{v_*}\sqrt{\bar f_*}
\Big)\Big|^2.
\end{align*}
We note that 
\begin{gather*}
0\le q_n(s)\le 2\mu_n(s),
\qquad
0\le \mu_n(s)\le \sqrt{s},\\
\mu_n(\bar f_n)\to \sqrt{\bar f},
\qquad
q_n(\bar f_n)\to 2\sqrt{\bar f},
\qquad
\text{almost surely, pointwisely as }n\to\infty.
\end{gather*}

Fix $t\in[0,T]$ and fix
$\omega$ in the full-probability set on which the convergences in
\eqref{limit-L2:conv}, \eqref{L12-global}, and the almost everywhere convergence
of $\bar f_n$ hold.  If
\begin{align*}
\liminf_{n\to\infty}\int_0^t\cD_n(\bar f_n)=+\infty,
\end{align*}
then there is nothing to prove.  Otherwise we choose a subsequence, not relabelled,
such that the above liminf is finite.  For each fixed $m\in\mathbb N$, the
monotone construction of $A_n$ gives $A_m\le A_n$ for all $n\ge m$, and therefore
\begin{align*}
\sqrt{a_m}
\Big(
\mu_n(\bar f_{n,*})\nabla_vq_n(\bar f_n)
-
\mu_n(\bar f_{n})\nabla_{v_*}q_n(\bar f_{n,*})
\Big)
\end{align*}
is bounded in $L^2((0,t)\times\mathbb{R}^{2d})$.  Here
\begin{align*}
a_m(v-v_*):=A_m(|v-v_*|)\Pi_{(v-v_*)^\perp}.
\end{align*}
Hence, along a further subsequence, it converges weakly in
$L^2((0,t)\times\mathbb{R}^{2d})$.  We identify the weak limit.  It is enough to test
against any $\Phi\in C_c^\infty((0,t)\times\mathbb{R}^{2d})$ and prove
\begin{align}
\label{doal:X}
&\int_0^t\int_{v,v_*}
\Phi\sqrt{a_m}
\Big(
\mu_n(\bar f_{n,*})\nabla_vq_n(\bar f_n)
-
\mu_n(\bar f_{n})\nabla_{v_*}q_n(\bar f_{n,*})
-
2(\nabla_v-\nabla_{v_*})
\sqrt{\bar f\bar f_*}
\Big)
\longrightarrow0 .
\end{align}
By integration by parts in $(v,v_*)$ and the exchange of the two velocity
variables, the absolute value of the left-hand side of \eqref{doal:X} is bounded by
\begin{align*}
\Big|
\int_0^t\int_{v,v_*}
\nabla_v\cdot
\big(
\Phi\sqrt{a_m}
-
\Phi_*\sqrt{a_m}
\big)
\Big(
\mu_n(\bar f_{n,*})q_n(\bar f_n)
-
2\sqrt{\bar f\bar f_*}
\Big)
\Big|,
\end{align*}
where $\Phi_*(s,v,v_*):=\Phi(s,v_*,v)$.  We note that
$\sqrt{a_m}=\sqrt{A_m}\Pi_{(v-v_*)^\perp}$.  Similar to deriving the uniform
bound of $(a_m,b_m)$ in \eqref{bdd:ab-n},
\begin{align*}
|\sqrt{a_m}|
\lesssim
|v-v_*|^{1+\frac{\gamma}{2}},
\qquad
\big|\nabla_v\cdot \sqrt{a_m}\big|
\lesssim
|v-v_*|^{\frac{\gamma}{2}}.
\end{align*}
Together with
$|\Phi-\Phi_*|\lesssim \|\nabla\Phi\|_{L^\infty}|v-v_*|$, this gives
\begin{equation}
\label{bdd:phi-a}
\begin{aligned}
\big|
\nabla_v\cdot
(
\Phi \sqrt{a_m}
-
\Phi_* \sqrt{a_m}
)
\big|
\lesssim
\|\Phi\|_{W^{1,\infty}}
|v-v_*|^{1+\frac{\gamma}{2}}.
\end{aligned}
\end{equation}
On the compact support of $\Phi$, the right-hand side is bounded.  Moreover,
from \eqref{L12-global}, the inequality $0\le q_n\le2\mu_n$, and
$q_n(\bar f_n)\to2\sqrt{\bar f}$ pointwise, we have
\begin{align*}
\mu_n(\bar f_{n,*})q_n(\bar f_n)
\longrightarrow
2\sqrt{\bar f\bar f_*}
\quad\text{strongly in }L^1_{\rm loc}((0,t)\times\mathbb{R}^{2d}).
\end{align*}
Therefore \eqref{doal:X} follows.  Consequently,
\begin{align*}
&\sqrt{a_m}
\Big(
\mu_n(\bar f_{n,*})\nabla_vq_n(\bar f_n)
-
\mu_n(\bar f_{n})\nabla_{v_*}q_n(\bar f_{n,*})
\Big)\\
&\qquad\rightharpoonup
2\sqrt{a_m}
\Big(
\sqrt{\bar f_*}\nabla_v\sqrt{\bar f}
-
\sqrt{\bar f}\nabla_{v_*}\sqrt{\bar f_*}
\Big)
\end{align*}
weakly in $L^2((0,t)\times\mathbb{R}^{2d})$.  The weak lower semicontinuity of the
$L^2$ norm gives
\begin{equation}
\label{limit:H:fixed-m}
\begin{aligned}
&\frac12
\int_0^t
\int_{v,v_*}
A_m \Big|\Pi_{(v-v_*)^\perp}
\Big(
2\sqrt{\bar f_*}\nabla_v\sqrt{\bar f}
-
2\sqrt{\bar f}\nabla_{v_*}\sqrt{\bar f_*}
\Big)\Big|^2\\
\le&{}
\frac12
\liminf_{n\to\infty}
\int_0^t
\int_{v,v_*}
A_m \Big|\Pi_{(v-v_*)^\perp}
\Big(
\mu_n(\bar f_{n,*})\nabla_vq_n(\bar f_n)
-
\mu_n(\bar f_{n})\nabla_{v_*}q_n(\bar f_{n,*})
\Big)\Big|^2.
\end{aligned}
\end{equation}

Finally, by the monotone convergence theorem, 
\begin{align*}
\frac12
\int_0^t
\int_{v,v_*}
A_m \Big|\Pi_{(v-v_*)^\perp}
\Big(
2\sqrt{\bar f_*}\nabla_v\sqrt{\bar f}
-
2\sqrt{\bar f}\nabla_{v_*}\sqrt{\bar f_*}
\Big)\Big|^2
\uparrow
\int_0^t
\cD(\bar f)\,\dd s,
\end{align*}
while $A_m\le A_n$ for $n\ge m$ implies
\begin{align*}
\frac12
\int_0^t
\int_{v,v_*}
A_m \Big|\Pi_{(v-v_*)^\perp}
\Big(
\mu_n(\bar f_{n,*})\nabla_vq_n(\bar f_n)
-
\mu_n(\bar f_{n})\nabla_{v_*}q_n(\bar f_{n,*})
\Big)\Big|^2
\le
\int_0^t
\cD_n(\bar f_n)\,\dd s.
\end{align*}

It remains to consider the entropy at time $t$. We use the elementary pointwise bound
$(r\log r)^-\le |v|^2r+e^{-|v|^2}$ for $r\ge0$, which is the same estimate as in \eqref{LlogL:minus}. Together with the mass--energy bounds already established in the limit, this gives that, for every finite $T$, there exists $C_T>0$ such that
\begin{align*}
h_n(\bar f_n(t,v))
+
C_T(1+|v|^2)\bar f_n(t,v)
&\ge0,
\\
h(\bar f(t,v))
+
C_T(1+|v|^2)\bar f(t,v)
&\ge0.
\end{align*}
Fatou's lemma, the almost everywhere convergence of $\bar f_n$ along the subsequence, the lower semicontinuity of the second moment, and conservation of mass imply
\begin{align}
\cH(\bar f_t)-m_0
\le
\liminf_{n\to\infty}
\cH_n(\bar f_n(t))
\qquad
\text{for a.e. }t\in[0,T].
\label{limit-L2:entropy-lsc}
\end{align}
On the one hand, the entropy inequality implies that the right-hand side is finite, and therefore $\cH(\bar f)<+\infty$. On the other hand, \eqref{LlogL:minus} yields $\cH(\bar f)>-\infty$. Combining \eqref{limit-L2:initial-entropy}, \eqref{limit-L2:entropy-lsc}, and \eqref{limit-L2:diss-lsc-final} with \eqref{H:h-eps}, and then applying Fatou's lemma in expectation, we obtain
\begin{align*}
\bar{\mathbb E}\big[\cH(\bar f_t)\big]
+
\bar{\mathbb E}
\left[
\int_0^t
\cD(\bar f_s)\,\dd s
\right]
\le
\cH(f_0)
+
C\|f_0\|_{L^1}^2.
\end{align*}

It remains to pass the maximal exponential entropy estimate to the limit.  By
\eqref{limit-L2:entropy-lsc} and \eqref{limit-L2:diss-lsc-final}, we have
\begin{align*}
&\operatorname{esssup}_{s\in[0,t]}\left(\lambda\cH(\bar f_s)+\frac{\lambda}{2}\int_0^s\cD(\bar f_r)\,\dd r\right)\\
\le&{}
\liminf_{n\to\infty}
\operatorname{esssup}_{s\in[0,t]}\left(\lambda\cH_n(\bar f_n(s))+\frac{\lambda}{2}\int_0^s\cD_n(\bar f_n(r))\,\dd r\right).
\end{align*}
Indeed, the entropy lower semicontinuity holds for almost every time and the dissipation term is lower semicontinuous on every interval $[0,s]$; the supremum may therefore be taken over a countable dense subset of full-measure times, which gives the displayed inequality.  Since the exponential function is increasing, Fatou's lemma and the equality of laws imply
\begin{align*}
&\bar{\mathbb E}\exp\left[
\operatorname{esssup}_{s\in[0,t]}\left(\lambda\cH(\bar f_s)+\frac{\lambda}{2}\int_0^s\cD(\bar f_r)\,\dd r\right)\right]\\
\le&{}
\liminf_{n\to\infty}\mathbb E
\exp\left[
\operatorname{esssup}_{s\in[0,t]}\left(\lambda\cH_n(f_n(s))+\frac{\lambda}{2}\int_0^s\cD_n(f_n(r))\,\dd r\right)\right]\\
\le&{}
\liminf_{n\to\infty} C_\lambda\exp\left[\lambda\cH_n(f_0)+C\lambda\big(\|f_0\|_{L^1}^3+1\big)\right].
\end{align*}
Since $\cH_n(f_0)\to \cH(f_0)-m_0$ by \eqref{limit-L2:initial-entropy}, the last term is bounded by the right-hand side of \eqref{H:limit-L2-exp}, after increasing the constant $C$.  This proves \eqref{H:limit-L2-exp}.
This proves all the claimed estimates and completes the proof.

\section{Pass to the limits in \texorpdfstring{$L^1$-framework}{}}
\label{sec:limit-L1}

We complete the proof of Theorem \ref{main:thm} in this section. We show the existence of the probability weak solutions to the following It\^o equation
\begin{equation}
\label{app:eq:L-1}
\begin{aligned}
\partial_t f
&=\frac12\tn\cdot\big(A f f_*\tn \log f\big)-\frac{\sqrt{\eps}}{2}\tn\cdot(A^{1/2}  \sigma(f) \sigma( f_*)\xi_K) \\
&\quad+\frac{\eps}{2}\sum_{k= 1}^K\tn\cdot \Big(G_k(v,v_*) \sigma'(f) \sigma(f_{*}) \tn\cdot\big(G_k(v,w) \sigma(f) \sigma(f_{w})\big) \Big).
\end{aligned}
\end{equation}

 We show in Theorem \ref{thm:L-1} 
\begin{theorem}\label{thm:L-1}
Let $d\ge 2$. Let the interaction kernel $A$ be given by \eqref{def:cA}. Let 
$\sigma$ satisfy Assumption~\ref{ass:sigma-R0-0}. Let the family 
$(G_k)_{1\leq k\le K}$ satisfy Assumption~\ref{ass:G-k:app}.  Let
the initial datum satisfy \eqref{ass:initial-main}, and set $m_0:=\|f_0\|_{L^1}$. 

For  every $\lambda>0$, there  exists $\eps_0>0$ such that, for every $\eps\in(0,\eps_0)$, there exists at least one probabilistic weak solution
$(f,(B_k)_{k=1,\ldots,K})$ to \eqref{app:eq:L-1} with initial datum $f_0$ in
the sense of Definition~\ref{def:weak-sol:L1}.

Moreover, we have almost surely, $f\in L^\infty([0,T];L^1_2(\Do;\R_+))$, $|\cH(f_t)|<+\infty$ for almost every $t\in[0,T]$. Furthermore, we have, almost surely, for every $t\in[0,T]$, 
\begin{gather*}
    \int_v (1,v)f_t(v)\dd v=\int_v (1,v)f_0(v)\dd v,\quad
    \int_v|v|^2  f_t(v)\dd v\le  \int_v|v|^2  f_0(v)\dd v,\end{gather*}
    and for almost every $t\in[0,T]$, 
    \begin{gather*}
    \hE\big[\cH(f_t)\big]+\hE\Big[\int_0^t\cD(f)\Big]\le \hE\big[\cH(f_0)\big]+C
    \end{gather*}
for some $C=C(\|\sigma'\|_{L^\infty},C_K,m_0)>0$. 

Furthermore, we have for almost every $t\in[0,T]$, 
\begin{align*}
\hE\exp\left[
\operatorname{esssup}_{s\in[0,t]}\left(\lambda\cH(f_s)+\frac{\lambda}{2}\int_0^s\cD(f_r)\,\dd r\right)\right]
\le C_\lambda\exp\left[\lambda\cH(f_0)+C\lambda\big(\|f_0\|_{L^1}^3+1\big)\right].
\end{align*}
\end{theorem}

The remainder of this section is devoted to proving Theorem \ref{thm:L-1}.

To show the existence result, we approximate the equation \eqref{app:eq:L-1} by \eqref{eq:ito-2}, for which the existence and properties of solutions were established in Section \ref{sec:limit-L2}. The necessary tightness results are established in Section \ref{sec-6:tight}, and the proof of Theorem \ref{thm:L-1} is completed in Section \ref{sec-6:limit}.

\medskip

We present the approximation in detail.
We do not normalise the mass. Throughout this section the constants are allowed to depend on $m_0=\|f_0\|_{L^1}$, and the approximation below preserves this mass.
Let $\rho\in C^\infty_c(\R^d;\R_+)$ such that $\supp(\rho)\subset B_1(0)$ and $\int_v \rho=1$. Let $\rho^\alpha(v):=\alpha^{-d}\rho(v/\alpha)$, $\alpha\in(0,1)$ be a sequence of mollifiers. Let $M(v)=(2\pi)^{-d/2}e^{-|v|^2/2}$. We define
\begin{align}
\label{def:f-0-alpha}
 f_0^\alpha
:=(1-\alpha)f_0*\rho^\alpha
\,+\,\alpha m_0 M.
\end{align}
Let $\alpha\in(0,1)$ be fixed. Let $f_\alpha$ denote a probabilistic weak solution of the initial value problem of \eqref{eq:ito-2} given in Proposition~\ref{lem:l-12:log}
\begin{equation}\label{eq:ito-2-2} 
\left\{
\begin{aligned}
&\partial_t f
=\alpha\Delta f+Q(f,f)-\frac{\sqrt{\eps}}{2}\tn\cdot( A^{1/2}  \sigma(f) \sigma( f_*)\xi_K)\\
&\qquad+\frac{\eps}{2}\sum_{k= 1}^K\tn\cdot \Big(G_{k}(v,v_*) \sigma'(f) \sigma(f_{*}) \tn\cdot\big(G_{k}(v,w) \sigma(f) \sigma(f_{w})\big) \Big)\\
&f|_{t=0}=f^\alpha_0.
\end{aligned}
\right.
\end{equation}
In particular, the following mass conservation law, energy and entropy inequalities hold uniformly in $\alpha\in(0,1)$
\begin{equation}
    \label{bdd:alpha-m-n}
\begin{gathered}
    \int_v (1,v)f_\alpha(t,v)\dd v=\int_v (1,v)f^\alpha_0(v)\dd v,\quad a.s., \text{ for every }t\in[0,T],\\
    \int_v|v|^2  f_\alpha(t,v)\dd v\le  \int_v|v|^2  f^\alpha_0(v)\dd v +4\alpha dT  \|f_0\|_{L^1(\Do)},\quad a.s., \text{ for every }t\in[0,T],
    \\
\hE\big[\cH(f_\alpha(t))\big]+\hE\Big[\int_0^t\cD(f_\alpha)\Big]\le \hE\big[\cH(f^\alpha_0)\big]+C,\text{ for almost every }t\in[0,T],\\
\hE\exp\left[
\operatorname{esssup}_{s\in[0,t]}\left(\lambda\cH(f_\alpha(s))+\frac{\lambda}{2}\int_0^s\cD(f_\alpha(r))\,\dd r\right)\right]
\le C_\lambda\exp\left[\lambda\cH(f_0^\alpha)+C\lambda\big(\|f_0^\alpha\|_{L^1}^3+1\big)\right],
\end{gathered}
\end{equation}
for almost every $t\in[0,T]$, and for some $C=C\big(m_0,C_{r_0},C_K\big)>0$. The last estimate holds for every $\lambda>0$ and every $0<\eps\le\eps_0(\lambda)$, where $\eps_0(\lambda)$ is the threshold in Proposition~\ref{lem:l-12:log} and is independent of $\alpha$.

\subsection{Tightness results}\label{sec-6:tight}

We begin with some properties of the approximate initial datum $f^\alpha_0$.

\begin{proposition}
Let $f_0\in L^1_2\cap L\log L(\R^d;\R_+)$. Let $f^\alpha_0$ be defined as in \eqref{def:f-0-alpha}. Then we have $f_0^\alpha\in C^\infty(\Do;\R_+)\cap L^1_2(\Do)\cap L\log L(\Do)$. The following uniform bounds hold
\begin{align}
\label{f-alpha-0:bdd}
\sup_{\alpha\in(0,1)}\|f^\alpha_0\|_{L^1}\le  \|f_0\|_{L^1}+1\quad\text{and}\quad
\sup_{\alpha\in(0,1)}\|f^\alpha_0\|_{L^1_2}\le C\|f_0\|_{L^1_2}
\end{align} 
for some constant $C>0$ independent of $\alpha$.

For any $s\in[0,2)$, as $\alpha\to0$, we have 
\begin{align}
\label{f-alpha-0:limit}
f_0^\alpha\to f_0\quad\hbox{in }L^1(\Do),\quad\text{and}\quad \cH(f^\alpha_0)\to\cH(f_0).
\end{align}   
\end{proposition}

\begin{proof}
The uniform bounds  \eqref{f-alpha-0:bdd} hold, since, for all $\alpha\in(0,1)$,  
\begin{align*}
\|f^\alpha_0\|_{L^1}\le  \|f_0\|_{L^1}+\|M\|_{L^1}\quad\text{and}\quad
\sup_{\alpha\in(0,1)}\|f^\alpha_0\|_{L^1_2}\le \|\rho\|_{L^1_2}\|f_0\|_{L^1_2}+\|M\|_{L^1_2}, 
\end{align*}
where we use the fact 
\begin{align*}
    \langle v\rangle^2\lesssim \langle w\rangle^2\langle v-w \rangle^2, \quad \text{for all }v,w\in\mathbb{R}^d.
\end{align*}

The first limit in \eqref{f-alpha-0:limit} holds by definition. We show the convergence of entropy. By weak lower semi-continuity of $\cH$, we have 
\begin{align*}
    \cH(f_0)\le\liminf_{\alpha\to0}\cH(f^\alpha_0).
\end{align*}
On the other hand, the convexity of $r\mapsto r\log r$ and Jensen's inequality imply that 
\begin{align*}
  \cH(f^\alpha_0)&=\cH\big((1-\alpha)f_0*\rho^\alpha+\alpha M\big)\\
  &\le \cH\big((1-\alpha)f_0*\rho^\alpha\big)+\alpha\cH\big( M\big)\\
  &\le (1-\alpha)\int_{v,w}f_0(w)\log f_0(w)\rho^\alpha(v-w)+\alpha\cH\big( M\big)\\
  &=  (1-\alpha)\cH(f_0)+\alpha \cH\big(M\big)\to\cH(f_0)
\end{align*}
as $\alpha\to0$, since $f_0\in L\log L(\R^d;\R_+)$.

\end{proof}

{
\textcite{Des15} showed that the weighted Fisher information can be bounded by the Landau entropy dissipation up to a constant depending on the entropy $\cH(f)$
 \begin{equation}
 \label{landau-entropy:ineq}
        \begin{aligned}
    \int_v\langle v\rangle ^{\gamma}|\nabla_v \sqrt{f}|^2&\le C(\cH(f))|\big( 1+\cD(f)\big).
        \end{aligned}
     \end{equation}
For deterministic Landau equations, the entropy is uniformly bounded 
\begin{align*}
 \cH(f_t)\le\cH(f_0)\quad \text{for almost all $t\in[0,T]$}. 
\end{align*}
However, in the stochastic case, we only have uniform bounds of expectation \eqref{bdd:alpha-m-n}
\begin{align*}
\hE\big[\cH(f_t)\big]\le\hE\big[\cH(f_0)\big]+C\quad \text{for almost all $t\in[0,T]$}. 
\end{align*}
We follow \cite{Des15} to derive the following refined weighted Fisher information estimate, in which the constant in \eqref{landau-entropy:ineq} has explicit exponential dependence on the entropy.

We denote the positive part of Boltzmann entropy by 
\begin{equation*}
    \cH_+(f):=\int_v (f\log f)^+\dd v.
\end{equation*}
     
\begin{lemma}[\cite{Des15}, Lemma 2, Theorem 3]\label{lem:Des} Let $\gamma\in(-2,0)$.  Let $f\in L^1_2\cap L\log L(\Do;\R_+)$ with $\int_v f>0$. 
Then we have
        \begin{equation*}
        \label{ineq:Des}
        \begin{aligned}
  \int_v\langle v\rangle ^{\gamma}|\nabla_v \sqrt{f}|^2&\le C_1\exp(C_2 \cH_+(f))(1+\cD(f))
        \end{aligned}
     \end{equation*}
     for some constants $C_i=C_i(d,\gamma,\|f\|_{L^1_2})>0$, $i=1,2$.
\end{lemma}
\begin{proof}

We first follow \cite[Theorem 3]{Des15} to establish the weighted Fisher information bounds, keeping explicit all constants that depend on $H$.

Set
\begin{equation*}
m:=\int_v f(v)\,\dd v,\qquad
u:=\frac1m\int_vvf(v)\,\dd v,\qquad
g(z):=\frac1m f(z+u).
\end{equation*}
Then $\int g=1$, $\int zg(z)\,\dd z=0$, and
\begin{equation}
\label{Des-normalization}
\|g\|_{L^1_2}\le \frac{\|f\|_{L^1_2}}{m},\qquad
\cH_+(g)\le \frac{\cH_+(f)}{m}+\log^+\frac1{m},
\qquad
\cD(g)=m^{-2}\cD(f).
\end{equation}
Without loss of generality, we relabel $g$ as $f$ and consider $f$ to have zero momentum and unit mass. Set
$E:=\|f\|_{L^1_2}$ and write
$H:=\cH_+(f)$.

For $x,\,y\in\R^d$, we  use the notation $[x,y]_{ij}= x_iy_j-x_jy_i$. By using of the identity 
\begin{align*}
|x|^2(y\cdot\Pi_{x^\perp}y)
&=\frac{1}{2}\sum_{ij}|x_iy_j-x_j y_i|^2,
\end{align*}
the entropy dissipation in Definition \ref{def:entropy-sigma} can be written as
\begin{align*}
\cD(f)&=\frac12\int_{v,v_*}|v-v_*|^{\gamma+2}ff_*|\tn \log f|^2\\
&=\frac12\int_{v,w}|v-w|^{\gamma}f(v)f(w)|v-w|^2\big((\nabla_v \log f-\nabla_{w}\log f_w)\\
&\qquad\cdot \Pi_{(v-w)^\perp}(\nabla_v \log f-\nabla_{w}\log f_w)\big)\\
&=\frac14\int_{v,w}|v-w|^\gamma f(v)f(w)|q^f|^2.
\end{align*}
In the above,  $q^f=(q_{ij})_{i,j}$ and $q_{ij}$ is defined via
    \begin{equation}
    \label{eq:q_ij}
    \begin{aligned}
&\Big[v,\frac{\nabla_v f}{f}(v)\Big]_{ij}-w_j  \frac{\d_if}{f}(v)  +w_i  \frac{\d_jf}{f}(v)\\
    =&{}v_i  \frac{\d_jf}{f}(w)  -v_j  \frac{\d_if}{f}(w)-\Big(w_i\frac{\d_j f}{f}(w)-w_j\frac{\d_i f}{f}(w)\Big)+q_{ij}.
    \end{aligned}
    \end{equation}

Let $\lambda\in(0,1/2)$ be a constant to be determined later. 
    We test $q^f$ by $e^{-\lambda|w|^2}f(w)$ and $w_ke^{-\lambda|w|^2}f(w)$, $k=1,\dots, d$,  to drive
    \begin{align*}
&\int_{w}e^{-\lambda|w|^2}f(w) q_{ij}= \Big(v_i\frac{\d_jf}{f}(v)-v_j\frac{\d_if}{f}(v)\Big)\int_{w}e^{-\lambda|w|^2} f(w) \\
&\qquad +\frac{\d_if}{f}(v)\int_{w}e^{-\lambda|w|^2}w_jf(w) -\frac{\d_jf}{f}(v)\int_{w}e^{-\lambda|w|^2} w_if(w) \\
&\qquad -2\lambda v_i\int_{w}e^{-\lambda|w|^2}w_jf(w) +2\lambda v_j\int_{w}e^{-\lambda|w|^2}w_if(w),
    \end{align*}
and 
    \begin{align*}
        &\int_{w}e^{-\lambda|w|^2} q_{ij}w_kf(w)=\Big(v_i\frac{\d_jf}{f}(v)-v_j\frac{\d_if}{f}(v)\Big)\int_{w} e^{-\lambda|w|^2} w_if(w) \\
        &\qquad +\frac{\d_if}{f}(v)\int_{w} e^{-\lambda|w|^2} w_iw_jf(w) -\frac{\d_jf}{f}(v)\int_{w}e^{-\lambda|w|^2} w_i^2f(w) \\
        &\qquad +v_i\int_{w}(\delta_{jk}-2\lambda w_jw_k)e^{-\lambda|w|^2} f(w) -v_j\int_{w}(\delta_{ik}-2\lambda w_iw_k)e^{-\lambda|w|^2} f(w)\\
        &\qquad +\int_{\Do}(\delta_{ik}w_j-\delta_{jk}w_i)e^{-\lambda|w|^2} f(w).
    \end{align*}

Choosing $k=i,j$, we  derive a $3\times3$-system of linear equations for $v_i\frac{\d_jf}{f}(v)-v_j\frac{\d_if}{f}(v)$, $\frac{\d_if}{f}(v)$ and $-\frac{\d_jf}{f}(v)$. By Cramer's formula, $\frac{\d_if}{f}(v)$ is given by 
\begin{equation*}
    \frac{\d_if}{f}(x,v)=\frac{\Delta_i(f)}{\Delta_\lambda(f)}.
\end{equation*}
In the above, $\Delta_\lambda(f)$ and $\Delta_i(f)$ are defined as follows
\begin{align*}
\Delta_\lambda(f)&=\operatorname{Det}\Big(\int_{w}e^{-\lambda|w|^2}f(w)\begin{pmatrix}
    1&w_j&w_i\\
    w_i&w_iw_j&w_i^2\\
    w_j&w_j^2&w_iw_j
    \end{pmatrix}\Big),\\
\Delta_i(f)&=\operatorname{Det}\Big(\int_{w}e^{-\lambda|w|^2}f(w)\begin{pmatrix}
    1&Z_1&w_i\\
    w_i&Z_2&w_i^2\\
    w_j&Z_3&w_iw_j
    \end{pmatrix}\Big),
\end{align*}
where
\begin{align*}
    Z_1&=q_{ij}-2\lambda(v_jw_i-v_iw_j),\\
       Z_2&=w_iq_{ij}+v_j-w_j-2\lambda (v_jw_i^2-v_iw_iw_j),\\
          Z_3&=w_jq_{ij}+w_i-v_i-2\lambda(v_jw_iw_j-v_iw_j^2).
\end{align*}

We note the fact $e^{-\lambda|w|^2} \langle w\rangle^2\le 1+\lambda^{-1}$.
Then we have 
\begin{align*}
    \Big| \frac{\nabla_v f}{f}(v)\Big|
    &\le |\Delta_\lambda(f)|^{-1}\Big(\int_{w}e^{-\lambda|w|^2} \langle w\rangle^2f(w) \Big)^2\Big(\int_{w}e^{-\lambda|w|^2}\langle w\rangle (\sum_{i}|Z_i|)f(w) \Big)\\
     &\le C(E) |\Delta_\lambda(f)|^{-1}(1+\lambda^{-1})^2\Big(\langle v\rangle+\int_{w}e^{-\lambda|w|^2}\langle w\rangle f(w)|q| \Big).
\end{align*}

By Cauchy--Schwarz inequality, we have 
    \begin{align*}
        &\int_{v}\langle v\rangle^{\gamma}f(v)\Big| \frac{\nabla_v f}{f}(v)\Big|^2\\
\le&{}C(1+\lambda^{-1})^4|\Delta_\lambda(f)|^{-2}\Big(\int_vf(v)\langle v\rangle^{\gamma}\Big(\langle v\rangle+\int_{w}e^{-\lambda|w|^2}\langle w\rangle f(w)|q| \Big)^2\Big)\\
\le&{}C(1+\lambda^{-1})^4|\Delta_\lambda(f)|^{-2}\Big(E+\int_vf(v)\langle v\rangle^{\gamma}\Big(\int_{w}e^{-\lambda|w|^2}\langle w\rangle f(w)|q| \Big)^2\Big).
    \end{align*} 

    Notice that 
    \begin{align*}
   &\int_vf(v)\langle v\rangle^{\gamma}\Big(\int_{w}e^{-\lambda|w|^2}\langle w\rangle f(w)|q| \Big)^2\\
   \le&{}\Big(\int_{v,w}f(v)f(w)|q|^2|v-w|^{\gamma}\Big)\Big(\sup_{v\in\R^d}\langle v\rangle^{\gamma}\int_{w}|v-w|^{-\gamma}e^{-\lambda|w|^2}\langle w\rangle^2 f(w) \Big)\\
    \le&{}4(1+(2\lambda)^{-1})\cD(f)\Big(\sup_{v\in\R^d}\langle v\rangle^{\gamma}\int_{w}|v-w|^{-\gamma} f(w) \Big)\\
    \le&{}4E(1+(2\lambda)^{-1})\cD(f),
    \end{align*}
    where we use (see e.g. \cite[Lemma 2.10]{DH25a})
   that 
   \begin{align*}
   \sup_{v\in\R^d}\int_{w}|v-w|^{-\gamma} f(w) \le E  \langle v\rangle^{-\gamma} .
   \end{align*}

   Hence, we have
       \begin{align}
       \label{des-final-1}
        &\int_{v}\langle v\rangle^{\gamma}f(v)\Big| \frac{\nabla_v f}{f}(v)\Big|^2\le C(1+\lambda^{-1})^5|\Delta_\lambda(f)|^{-2}\big(1+\cD(f)\big),
    \end{align}
    where $C=C(E)>0$.

Now we give the precise lower bound of $|\Delta_\lambda(f)|$. 
By \cite[Lemma 1\&2]{Des15},  for any $S>0$,
\begin{align*}
    |\Delta_\lambda(f)|\ge \frac12 e^{-2d\lambda}\Big(\inf_{|\theta|\le 1}\inf_{k\neq l}\int_we^{-\lambda|w|^2}|w_k-\theta w_l|^2f(w)\dd w \Big)^2-4d (S(1-e^{-\lambda S^2})+2dS^{-1})^2.
\end{align*}
For any $\delta>0$ and $R>0$, we have 
\begin{align*}
    &\inf_{|\theta|\le 1}\inf_{k\neq l}\int_we^{-\lambda|w|^2}|w_k-\theta w_l|^2f(w)\\
    \ge&{}\delta^2\inf_{|\theta|\le 1}\inf_{k\neq l}\int_{|w_k-\theta w_l|\ge\delta}e^{-\lambda|w|^2}f(w)\\
    \ge&{}\delta^2e^{-\lambda R^2}\inf_{|\theta|\le 1}\inf_{k\neq l}\int_{|w_k-\theta w_l|\ge\delta,\, |w|\le R}f(w)\\
    =&{}\delta^2e^{-\lambda R^2}\Big(1-\int_{|w|\ge R}f(w)-\sup_{|\theta|\le 1}\sup_{k\neq l}\int_{|w_k-\theta w_l|\le\delta,\, |w|\le R}f(w)\Big)\\
    \ge&{}\delta^2e^{-\lambda R^2}\Big(1-\frac{E}{R^2}-\sup_{|\theta|\le 1}\sup_{k\neq l}\int_{\frac{|w_k-\theta w_l|}{\sqrt{1+\theta^2}}\le\frac{\delta}{\sqrt{1+\theta^2}},\, |w|\le R}f(w)\Big).
\end{align*}

Notice that 
\begin{align*}
\Big|\Big\{\frac{|w_k-\theta w_l|}{\sqrt{1+\theta^2}}\le\frac{\delta}{\sqrt{1+\theta^2}},\, |w|\le R|\Big\}\Big|\le 2\delta (2R)^{d-1}.    
\end{align*}

We split the integration into $f\ge K$ and $0\le f\le K$ for  $K:=e^{4H+1}>1$. Then we have 
\begin{align*}
    &\inf_{|\theta|\le 1}\inf_{k\neq l}\int_we^{-\lambda|w|^2}|w_k-\theta w_l|^2f(w)\\
    \ge&{}\delta^2e^{-\lambda R^2}\Big(1-\frac{E}{R^2}-2\delta K(2R)^{d-1}-\frac{H}{\log K}\Big).
\end{align*}
Then we arrive at the following lower bound as in \cite[Lemma 2]{Des15}
\begin{align*}
    |\Delta_\lambda(f)|\ge \frac12 e^{-2d\lambda}\delta^4e^{-2\lambda R^2}\Big(1-\frac{E}{R^2}-2\delta K(2R)^{d-1}-\frac{H}{\log K}\Big)^2-4d (S(1-e^{-\lambda S^2})+2dS^{-1})^2.
\end{align*}

We take $R=2\sqrt{E}$ and $\delta=e^{-4H-1}2^{-(d+2)}R^{-(d-1)}$, which lead to  $\Big(1-\frac{E}{R^2}-2\delta K(2R)^{d-1}-\frac{H}{\log K}\Big)\ge 1/4$. We recall that $\lambda\in(0,1/2)$. Then we have 
\begin{align*}
    |\Delta_\lambda(f)|\ge c_0e^{-16H}-4d (S(1-e^{-\lambda S^2})+2dS^{-1})^2,\quad c_0:=2^{-8d-11} e^{-d-4- 4E^2}E^{-2(d-1)}.
\end{align*}
We note that, for $S\ge0$, the quantity $4d(S(1-e^{-\lambda S^2})+2dS^{-1})^2$ attains its minimum, whose value is $c_1 d^{3/2}\sqrt{\lambda}$ for some constant $c_1>0$.
Now we take 
\begin{align*}
\lambda=\min\Big(\frac12,\frac{c_0^2}{4d^3c_1^2}  e^{-32H}\Big).
\end{align*}
As a consequence, we have 
\begin{align*}
    |\Delta_\lambda(f)|\ge \frac{c_0}{2}e^{-16H}.
\end{align*}

Substituting $\lambda$ to the weighted Fisher information bound \eqref{des-final-1}, we have 
 \begin{align}
        &\int_{v}\langle v\rangle^{\gamma}f(v)\Big| \frac{\nabla_v f}{f}(v)\Big|^2\lesssim_E (1+e^{160 H})e^{32 H}\big(1+\cD(f)\big)\le e^{192 H}\big(1+\cD(f)\big).
    \end{align}
    Returning to the density before normalization, we use
$\langle v-u\rangle\le\sqrt2\langle v\rangle\langle u\rangle$ and
$\gamma<0$ to obtain
\begin{align*}
\int_v\langle v\rangle^\gamma|\nabla\sqrt {f(v)}|^2\,\dd v
&\le 2^{-\gamma/2}\langle u\rangle^{-\gamma}m
\int_z\langle z\rangle^\gamma|\nabla\sqrt {g(z)}|^2\,\dd z.
\end{align*}
This incurs only an additional constant factor. Therefore, the proof is complete. 
\end{proof}
}






By using of Lemma \ref{lem:Des}, we establish uniform gradient estimates. 
\begin{lemma}[Uniform gradient estimates for $(f_\alpha)_{\alpha\in(0,1)}$]
    \label{lem:gradient-es}
Let $f_0$ satisfy Assumption~\ref{ass:initial-main}, and let
$(f_0^{\alpha})_{\alpha\in(0,1)}$ denote the regularisation of the initial
datum given in \eqref{def:f-0-alpha}. For each $\alpha\in(0,1)$, let
$(f_\alpha,B_\alpha)$ be a probabilistic weak solution to \eqref{eq:ito-2} given in Proposition \ref{lem:l-12:log} on
$$
(\Omega_\alpha,\mathcal F_\alpha,(\mathcal F_{\alpha,t})_{t\in[0,T]},\mathbb P_\alpha)
$$
with initial datum $f_0^\alpha$. We denote by $\mathbb E_\alpha$ the expectation with respect to $\mathbb P_\alpha$.

Let $C_2^{\rm Des}$ be the constant $C_2$ in Lemma~\ref{lem:Des}, and assume $0<\eps\le\eps_0(C_2^{\rm Des})$.
Then, for every $R>1$, there exists a constant
$$
C=C\big(d,R,\gamma,r_0,C_K,\|f_0\|_{L^1}\big)>0,
$$
independent of $\alpha$, such that
\begin{align}
\label{gradient-es}
\sup_{\alpha\in(0,1)}
\mathbb E_\alpha
\int_0^T
\|\nabla_v f_\alpha(s)\|_{L^1(B_R)}
\,\dd s
\le C.
\end{align}
\end{lemma}
\begin{proof}
With the aid of the entropy dissipation estimates for $(f_{\alpha})_{\alpha\in(0,1)}$, for every $R>1$, we obtain
\begin{align*}
\mathbb{E}\int_0^T\int_{B_R} |\nabla_v f_\alpha|\leq m_0^{\frac12}\left(\mathbb{E}\int_0^T\int_{B_R} |\nabla_v \sqrt{f_\alpha}|^2\right)^{1/2}
\lesssim
\left(\mathbb{E}\int_0^T\int_v
\langle v\rangle^\gamma
|\nabla_v \sqrt{f_\alpha}|^2\right)^{1/2}. 
\end{align*}

We now use Lemma~\ref{lem:Des} together with the exponential entropy estimate from Proposition~\ref{lem:l-12:log}.  The decomposition used below is the pointwise positive--negative decomposition of the integrand $g\log g$, not the positive--negative part of the number $\cH(g)$.  More precisely, set
\begin{equation*}
    \cH_+(g):=\int_v (g\log g)^+\,\dd v,
    \qquad
    \cH_-(g):=\int_v (g\log g)^-\,\dd v.
\end{equation*}
Then
\begin{equation*}
    \cH(g)=\cH_+(g)-\cH_-(g),
    \qquad
    |\cH(g)|\le \cH_+(g)+\cH_-(g).
\end{equation*}
First we control the negative part of the integrand $g\log g$ by the mass and energy.  For any $\ell\in(0,1)$,
\begin{equation*}
    (r\log r)^-\le C_\ell r^{1-\ell},
    \qquad r\ge0.
\end{equation*}
Choosing $\ell>0$ sufficiently small so that $\langle v\rangle^{-2(1-\ell)/\ell}\in L^1(\R^d)$ and using H\"older's inequality, we get
\begin{align}
\label{entropy-negative-energy-control}
\cH_-(f_\alpha(t))
&\le C_\ell\int_v f_\alpha(t,v)^{1-\ell}\,\dd v\notag\\
&\le C_\ell
\left(\int_v \langle v\rangle^2 f_\alpha(t,v)\,\dd v\right)^{1-\ell}
\left(\int_v \langle v\rangle^{-2(1-\ell)/\ell}\,\dd v\right)^\ell
\le C,
\end{align}
uniformly in $\alpha$ and for all $t\in[0,T]$, almost surely, by the energy bound in \eqref{bdd:alpha-m-n}.  Consequently,
\begin{equation}
\label{entropy-abs-positive-control}
   \cH_+(f_\alpha(t))=\cH(f_\alpha(t))+\cH_-(f_\alpha(t))\le \cH(f_\alpha(t))+C.
\end{equation}
This is the point where the positive part in Lemma~\ref{lem:Des} is used.  Applying Lemma~\ref{lem:Des} with \eqref{entropy-abs-positive-control}, we obtain
\begin{equation}
\label{desvillettes-alpha-use}
\int_v\langle v\rangle^\gamma|\nabla_v\sqrt{f_\alpha(t)}|^2\,\dd v
\le C\exp\big(C_2^{\rm Des}\cH_+(f_\alpha(t))\big)
\big(1+\cD(f_\alpha(t))\big).
\end{equation}
Moreover, thanks to \eqref{entropy-abs-positive-control}, the exponential estimate in \eqref{bdd:alpha-m-n} controls the positive part of the integrand $f_\alpha\log f_\alpha$: for every $\lambda>0$ and $0<\eps\le\eps_0(\lambda)$,
\begin{equation}
\label{positive-entropy-exp-control}
\sup_{\alpha\in(0,1)}
\mathbb E_\alpha\operatorname{esssup}_{t\in[0,T]}\exp\big(\lambda\cH_+(f_\alpha(t))\big)
\le C_\lambda.
\end{equation}
We next justify the factor $1+\cD(f_\alpha)$ in \eqref{desvillettes-alpha-use}.  Fix $p,q>1$ with $p^{-1}+q^{-1}=1$.  By \eqref{entropy-negative-energy-control},
\begin{equation*}
    \exp\big(C_2^{\rm Des}|\cH(f_\alpha(t))|\big)
    \le C\exp\big(C_2^{\rm Des}\cH_+(f_\alpha(t))\big).
\end{equation*}
Hence H\"older's inequality gives
\begin{align}
\label{holder-desvillettes-factor}
&\mathbb E_\alpha\int_0^T
\exp\big(C_2^{\rm Des}|\cH(f_\alpha(t))|\big)
\big(1+\cD(f_\alpha(t))\big)\,\dd t\notag\\
&\quad\le C
\left(\mathbb E_\alpha
\exp\big(\operatorname{esssup}_{t\in[0,T]}pC_2^{\rm Des}\cH_+(f_\alpha(t))\big)\right)^{1/p}
\left(\mathbb E_\alpha\Big(\int_0^T
1+\cD(f_\alpha(t))\,\dd t\Big)^q\right)^{1/q}.
\end{align}
The first factor is bounded by \eqref{positive-entropy-exp-control}, with $\lambda=pC_2^{\rm Des}$.  For the second factor, it can be controlled by the elementary inequality $x^q\le C_{q,\delta}\exp(\delta x)$, $x\ge0$, combining with \eqref{exponential-bound-2}. Therefore we conclude that 
\begin{equation}
\label{exp-desvillettes-time-control}
\sup_{\alpha\in(0,1)}
\mathbb E_\alpha\int_0^T
\exp\big(C_2^{\rm Des}|\cH(f_\alpha(t))|\big)
\big(1+\cD(f_\alpha(t))\big)\,\dd t
\le C.
\end{equation}
Here the smallness of $\eps$ is exactly the small-noise condition entering Proposition~\ref{lem:l-12:log}: in the exponential martingale argument the quadratic-variation compensator is multiplied by $\eps\lambda^2$, and this term is absorbed by the dissipation part by taking $\eps\le\eps_0(\lambda)$.
Combining Lemma~\ref{lem:Des} with \eqref{exp-desvillettes-time-control}, we obtain
\begin{equation}
\label{weighted-fisher-alpha-bound}
\sup_{\alpha\in(0,1)}
\mathbb E_\alpha\int_0^T\int_v
\langle v\rangle^\gamma|\nabla_v\sqrt{f_\alpha(t)}|^2\,\dd v\dd t
\le C.
\end{equation}
Finally, for $R>1$, since $\langle v\rangle^{-\gamma}$ is bounded on $B_R$ and the mass is conserved,
\begin{align*}
\mathbb E_\alpha\int_0^T\int_{B_R}|\nabla_v f_\alpha|\,\dd v\dd t
&=2\mathbb E_\alpha\int_0^T\int_{B_R}\sqrt{f_\alpha}\,|\nabla_v\sqrt{f_\alpha}|\,\dd v\dd t\\
&\le C_R\|f_0\|_{L^1}^{1/2}
\left(\mathbb E_\alpha\int_0^T\int_v
\langle v\rangle^\gamma|\nabla_v\sqrt{f_\alpha}|^2\,\dd v\dd t\right)^{1/2}
\le C_R.
\end{align*}



This completes the proof.
\end{proof}

We establish a refined version of Lemmas~\ref{lem:time-regularity} and \ref{lem:time-regularity:L-2-2}, yielding a time-regularity estimate. Moreover, as follows from the proof of Proposition~\ref{lem:l-12:log}, this estimate holds uniformly with respect to $\alpha\in(0,1)$. 
\begin{lemma}[Uniform local time regularity for $(f_\alpha)_{\alpha\in(0,1)}$]
\label{lem:time-regularity:L-1}
Assume the hypotheses of Lemma~\ref{lem:gradient-es}. Let
$\chi\in C_c^\infty(\R^d)$, $l>d/2+2$, and $\beta\in(0,1/2)$. Then,
for every
\begin{align*}
p>\frac{1}{1/2-\beta},
\end{align*}
there exists a constant
\begin{align*}
C=C\big(p,\beta,l,T,\chi,d,\gamma,r_0,C_K,
\|f_0\|_{L^1_2}\big)>0,
\end{align*}
independent of $\alpha$, such that
\begin{align}
\label{time-regularity:L-1}
\sup_{\alpha\in(0,1)}
\mathbb E_\alpha\big[
\|f_\alpha\chi\|_{C^\beta([0,T];H^{-l}(\R^d))}^{p}
\big]\le C.
\end{align}
In particular, the laws of $(f_\alpha\chi)_{\alpha\in(0,1)}$ are tight in
$C^{\beta'}([0,T];H^{-l-1}(\R^d))$ for every $\beta'<\beta$.
\end{lemma}

\begin{lemma}\label{lem:tight-L1}
Assume the hypotheses of Lemma~\ref{lem:gradient-es}. Then
$(f_{\alpha})_{\alpha\in(0,1)}$ is tight in
$L^1([0,T];L^1(\mathbb{R}^d))\cap
C^{\beta'}([0,T];H^{-l-1}_{\loc}(\R^d))$, for every
$\beta'\in(0,1/2)$ and $l>d/2+2$. 
\end{lemma}
\begin{proof}
Combining Lemmas~\ref{lem:gradient-es} and \ref{lem:time-regularity:L-1}, and applying the Aubin--Lions--Simon compactness criterion, we deduce that $(f_\alpha)_{\alpha\in(0,1)}$ is tight in $L^1([0,T];L^1_{\mathrm{loc}}(\mathbb{R}^d))\cap C^{\beta'}([0,T];H^{-l-1}_{\loc}(\R^d))$, for every $\beta'\in(0,1/2)$ and $l>d/2+2$. Furthermore, the energy estimate \eqref{bdd:alpha-m-n} yields the tightness of $(f_\alpha)_{\alpha\in(0,1)}$ in $L^1([0,T];L^1(\mathbb{R}^d))\cap C^{\beta'}([0,T];H^{-l-1}_{\loc}(\R^d))$. This completes the proof.
\end{proof}

\subsection{Proof of Theorem~\ref{thm:L-1}}\label{sec-6:limit}
We are now in a position to prove Theorem~\ref{thm:L-1}.

We divide the proof into four steps, following the structure of the proof in
Subsection~5.2.  The approximation parameter is now the viscosity parameter
$\alpha\downarrow0$, and the approximating solutions are the solutions
$(f_\alpha,B_\alpha)$ to \eqref{eq:ito-2} with initial datum
$f_0^\alpha$.

\medskip

\noindent\textbf{Step 1: Compactness and Skorokhod representation.}
Let $B_\alpha=(B_\alpha^k)_{1\le k\le K}$.  By Lemma~\ref{lem:tight-L1}
and the tightness of Brownian laws on $C([0,T];\R^K)$, the laws of
$(f_\alpha,B_\alpha)$ are tight on
\begin{align*}
\mathbb X_1
:=
\Big(L^1([0,T];L^1(\R^d))
\cap C^{\beta'}([0,T];H^{-l-1}_{\loc}(\R^d))\Big)
\times C([0,T];\R^K),
\end{align*}
for every $\beta'\in(0,1/2)$ and $l>d/2+2$.  Hence, along a sequence
$\alpha\downarrow0$, the Skorokhod representation theorem gives a new
stochastic basis
$(\bar\Omega,\bar{\mathcal F},\bar{\mathbb P})$ and random variables
$(\bar f_\alpha,\bar B_\alpha)$ and $(\bar f,\bar B)$ such that
\begin{align}
\label{limit-L1:skorokhod}
(\bar f_\alpha,\bar B_\alpha)\to(\bar f,\bar B)
\quad\bar{\mathbb P}\hbox{-a.s. in }\mathbb X_1,
\end{align}
and $(\bar f_\alpha,\bar B_\alpha)$ has the same law as
$(f_\alpha,B_\alpha)$ for every $\alpha$.  In particular, almost surely, 
\begin{align}
\label{limit-L1:strong}
\bar f_\alpha\to\bar f
\quad\hbox{strongly in }L^1([0,T]\times\R^d),
\end{align}
and, after extracting a further subsequence,
$\bar f_\alpha\to\bar f$ almost surely, almost everywhere on $[0,T]\times\R^d$.  The
non-negativity of $\bar f$ follows from the non-negativity of
$\bar f_\alpha$.

The estimates \eqref{bdd:alpha-m-n} are preserved under equality in law.
Using the convergence of the initial data and \eqref{f-alpha-0:limit},
we have
\begin{align}
\label{limit-L1:uniform-est}
\sup_{\alpha\in(0,1)}
\bar{\mathbb E}\Big[
\sup_{t\le T}\|\bar f_\alpha(t)\|_{L^1_2}
+\int_0^T\cD(\bar f_\alpha(s))\,\dd s
\Big]<+\infty,
\end{align}
where the entropy is understood through its positive and negative parts as \eqref{LlogL:minus}.

\medskip

\noindent
\textbf{Step 2: Convergence of the deterministic terms.}
From \eqref{limit-L1:strong} and the fact $|\sqrt s-\sqrt t|^2\le |s-t|$, $s,\,t>0$, we obtain
\begin{align}
\label{limit-L1:sigma-conv}
\sqrt{\bar f_\alpha}\to\sqrt{\bar f},
\qquad
\sigma(\bar f_\alpha)\to\sigma(\bar f)
\quad\hbox{strongly in }L^2([0,T]\times\R^d).
\end{align}
We note that $\sigma'$ and $\varsigma'$ are bounded and 
\begin{align}
\label{limit-L1:sigma2-conv}
\sigma'(\bar f_\alpha)\to\sigma'(\bar f) \quad\text{pointwisely},
\qquad
\varsigma(\bar f_\alpha)\to\varsigma(\bar f)
\quad\hbox{strongly in }L^1([0,T]\times\R^d).
\end{align}
The convergence of the products in the variables $(v,v_*)$ and
$(v,v_*,w)$ follows from \eqref{limit-L1:sigma-conv}--\eqref{limit-L1:sigma2-conv}
and the uniform mass bound.  For instance,
\begin{align*}
\|\sigma(\bar f_\alpha)\sigma(\bar f_{\alpha,*})
-\sigma(\bar f)\sigma(\bar f_*)\|_{L^1_{t,v,v_*}}
\le C\|\sigma(\bar f_\alpha)-\sigma(\bar f)\|_{L^2_{t,v}},
\end{align*}
and the analogous estimates hold for all factors appearing in the It\^o
correction.

Let $\phi\in C_c^\infty(\R^d)$.
The artificial viscosity term vanishes since
\begin{align*}
\left|\alpha\int_0^t\int_v\bar f_\alpha\,\Delta\phi\,\dd v\dd s\right|
\le \alpha T\|\Delta\phi\|_{L^\infty}\|f_0\|_{L^1}\longrightarrow0 .
\end{align*}

The passage to the limit in the Landau
	drift follows from the compact--tail decomposition carried out in Step~$2$ of
	Subsection~\ref{sec:limit-L2}.  The only difference is that there is no
	approximation of the kernel here, and the nonlinear coefficient convergence
	comes from the global strong $L^1$ convergence
	\eqref{limit-L1:strong}, whereas in
	Subsection~\ref{sec:limit-L2} it was obtained from the $L^2$ compactness
	framework and the convergence of
	$\tili_n(\mu_n^2(\bar f_n))$ and $\mu_n^2(\bar f_n)$.  Indeed, on every
	bounded set, the products in the Landau drift converge by
	\eqref{limit-L1:strong} and the uniform mass bound.  On the complement,
	the kernel estimates used in Subsection~\ref{sec:limit-L2}, together with
	the uniform second-moment bound \eqref{bdd:alpha-m-n} and the fact that
	$\gamma<0$, make the tail contributions uniformly small.  Passing first
	to the limit $\alpha\to0$ on the bounded region and then letting the
	truncation radius tend to infinity gives the convergence of the full
	Landau drift.

For the It\^o correction terms we use the weak form
\eqref{weak:app-0}.  The term containing $F_4\nabla\varsigma(\bar f_\alpha)$
is first integrated by parts in $v$; therefore only
$\varsigma(\bar f_\alpha)$, $\sigma(\bar f_\alpha)$ and
$\sigma(\bar f_\alpha)\sigma'(\bar f_\alpha)$ appear without derivatives,
tested against the fixed kernels $F_1,F_4$ and their first derivatives.
By \eqref{limit-L1:sigma-conv}--\eqref{limit-L1:sigma2-conv}, the boundedness
of $\sigma'$ and the assumed bounds on $F_1,F_4$, every deterministic
correction term converges to the corresponding term with $\bar f$.

\medskip

\noindent
\textbf{Step 3: Identification of the martingale part.}
Let $\bar{\mathcal G}_t$ be the augmented filtration generated by
$\bar f|_{[0,t]}$ and $\bar B|_{[0,t]}$.  As in Subsection~5.2, the
identity in law and the convergence \eqref{limit-L1:skorokhod} imply that
$(\bar B^k)_{1\le k\le K}$ is a family of independent Brownian motions with
respect to $(\bar{\mathcal G}_t)_{t\in[0,T]}$.

For a test function $\phi\in C_c^\infty(\R^d)$, let
$\bar M_\alpha^\phi$ be the residual obtained by subtracting from
$\langle \bar f_\alpha(t),\phi\rangle$ all deterministic terms in the weak
formulation of \eqref{eq:ito-2}.  Then
\begin{align*}
\bar M_\alpha^\phi(t)
=-\frac{\sqrt{\eps}}{2}\sum_{k=1}^K\int_0^t H_{\alpha,k}^{\phi}(s)\,\dd\bar B_\alpha^k(s),
\end{align*}
where
\begin{align*}
H_{\alpha,k}^{\phi}(s)
:=
\int_{v,v_*}
\tn\phi(v)\cdot G_k(v,v_*)\,
\sigma(\bar f_\alpha(s,v))\sigma(\bar f_\alpha(s,v_*))
\,\dd v\dd v_* .
\end{align*}
By \eqref{limit-L1:sigma-conv},
\begin{align}
\label{limit-L1:H-conv}
H_{\alpha,k}^{\phi}\to H_k^\phi
\quad\hbox{in }L^2(0,T),
\end{align}
where
\begin{align*}
H_k^\phi(s)
:=
\int_{v,v_*}
\tn\phi(v)\cdot G_k(v,v_*)\,
\sigma(\bar f(s,v))\sigma(\bar f(s,v_*))
\,\dd v\dd v_* .
\end{align*}
The deterministic convergence established in Step~$2$ shows that
$\bar M_\alpha^\phi(t)\to\bar M^\phi(t)$ for every $t$, where
$\bar M^\phi$ is the residual associated with $\bar f$.

Let $0\le s\le t\le T$ and let $\Theta$ be a bounded continuous functional
of $(\bar f,\bar B)$ restricted to $[0,s]$.  Applying the martingale
identities to the approximating solutions and passing to the limit by
\eqref{limit-L1:skorokhod} and \eqref{limit-L1:H-conv}, we obtain
\begin{align*}
\bar{\mathbb E}\Big[
\Theta\big(\bar M^\phi(t)-\bar M^\phi(s)\big)
\Big]=0,
\end{align*}
so $\bar M^\phi$ is a continuous $\bar{\mathcal G}_t$-martingale.  Moreover,
for every $1\le j\le K$,
\begin{align*}
\bar M^\phi(t)\bar B^j(t)
+\frac{\sqrt{\eps}}{2}\int_0^tH_j^\phi(r)\,\dd r
\end{align*}
is a martingale, and
\begin{align*}
(\bar M^\phi(t))^2
-\frac{\eps}{4}\int_0^t\sum_{k=1}^K|H_k^\phi(r)|^2\,\dd r
\end{align*}
is a martingale.  Equivalently,
\begin{align*}
\langle \bar M^\phi,\bar B^j\rangle_t
=-\frac{\sqrt{\eps}}{2}\int_0^tH_j^\phi(r)\,\dd r,
\qquad
\langle \bar M^\phi\rangle_t
=\frac{\eps}{4}\int_0^t\sum_{k=1}^K|H_k^\phi(r)|^2\,\dd r .
\end{align*}
Therefore,
\begin{align*}
\bar M^\phi(t)
=-\frac{\sqrt{\eps}}{2}\sum_{k=1}^K\int_0^t H_k^\phi(s)\,\dd\bar B^k(s).
\end{align*}
Combining this identity with Step~$2$ proves that
$(\bar f,(\bar B^k)_{1\le k\le K})$ satisfies the weak formulation of
\eqref{app:eq:L-1}.  Thus it is a probabilistic weak solution.

\medskip

\noindent
\textbf{Step 4: Passage of the estimates and entropy dissipation.}
Since $(\bar f_\alpha,\bar B_\alpha)$ and $(f_\alpha,B_\alpha)$ have the
same law, \eqref{bdd:alpha-m-n} holds for $\bar f_\alpha$.  By
\eqref{limit-L1:strong}, Fatou's lemma, and the convergence of
$f_0^\alpha$ in $L^1_2$, we obtain that almost surely, for almost every $t\in[0,T]$,
\begin{align*}
\int_v(1,v)\bar f_t(v)\,\dd v
&=\int_v(1,v)f_0(v)\,\dd v,\quad 
\int_v|v|^2\bar f_t(v)\,\dd v
\le \int_v|v|^2f_0(v)\,\dd v .
\end{align*}
Consequently $\bar f\in L^\infty([0,T];L^1_2(\R^d))$ almost surely.

It remains to pass to the entropy inequality.  The initial entropy
converges by \eqref{f-alpha-0:limit}.  

We now prove the lower semicontinuity of the dissipation. The proof closely follows the argument for showing entropy inequality in the step $4$ of the proof of Proposition \ref{lem:l-12:log}. We define $a(v-v_*)=A(|v-v_*|)\Pi_{(v-v_*)^\perp}$. We note that
\begin{align*}
\cD(f_\alpha)=2\int_{v,v_*}\big|\sqrt{a}(\nabla_v-\nabla_{v_*}\big)\sqrt{\bar f_\alpha\bar f_{\alpha,*}}\big|^2.
\end{align*}
Fix 
$\omega$ in the full-probability set on which \eqref{limit-L1:strong} and the
a.e. convergence of $\bar f_\alpha$ hold, and fix $t\in[0,T]$. Then \eqref{LlogL:minus} and \eqref{f-alpha-0:limit} ensure that $\int_0^t\cD(f_\alpha)\le \cH(f_0)+C\|f_0\|_{L^1_2}$, and hence, $\sqrt{a}(\nabla_v-\nabla_{v_*}\big)\sqrt{\bar f_\alpha\bar f_{\alpha,*}}$ is uniformly bounded in
$L^2([0,T]\times\R^{2d})$. We show that 
\begin{align*}
\sqrt{a}\big(\nabla_v-\nabla_{v_*}\big)\sqrt{\bar f_\alpha\bar f_{\alpha,*}}\rightharpoonup
\sqrt{a}\big(\nabla_v-\nabla_{v_*}\big)\sqrt{\bar f\bar f_{*}}\quad\text{as}\quad \alpha\to0
\end{align*}
in $L^2([0,T]\times\R^{2d})$.
 We only need to show that, for every $\Psi\in C_c^\infty([0,t]\times\R^{2d};\R^d)$, we have
\begin{align}
\label{conv:6}
\lim_{\alpha\to0}\int_0^t\int_{v,v_*}(\nabla_v-\nabla_{v_*})\cdot(\sqrt{a}\Psi) \sqrt{\bar f_\alpha\bar f_{\alpha,*}}
=\int_0^t\int_{v,v_*}(\nabla_v-\nabla_{v_*})\cdot(\sqrt{a}\Psi) \sqrt{\bar f \bar f_{*}}.
\end{align} 
By similar arguments as in \eqref{bdd:phi-a}, we have
\begin{equation}
\begin{aligned}
\big|(\nabla_v-\nabla_{v_*})\cdot(\sqrt{a}\Psi) 
\big|
\lesssim 
\|\Psi\|_{W^{1,\infty}}
|v-v_*|^{1+\frac{\gamma}{2}}.
\end{aligned}
\end{equation}
Since $\Psi$ is compactly supported and $\bar f_\alpha\to\bar f$ strongly in
$L^1((0,T)\times\R^d)$, we have
$\sqrt{\bar f_\alpha\bar f_{\alpha,*}}\to\sqrt{\bar f\bar f_*}$ strongly in
$L^1$ on the support of $\Psi$. Hence \eqref{conv:6} follows. Therefore the
weak limit is
$\sqrt{a}(\nabla_v-\nabla_{v_*})\sqrt{\bar f\bar f_*}$, and the weak lower
semicontinuity of the $L^2$ norm gives
\begin{align}
\label{limit-L1:D-lsc}
\int_0^t\cD(\bar f_s)\,\dd s
\le
\liminf_{\alpha\to0}
\int_0^t\cD(\bar f_\alpha(s))\,\dd s .
\end{align}

For the entropy at time $t$,
the lower semicontinuity of $\cH$ on $L^1$, together with the
uniform second moment bound controlling the negative part as in \eqref{LlogL:minus}, gives for almost every $t$,
\begin{align}
\label{limit-L1:H-lsc}
\cH(\bar f_t)
\le
\liminf_{\alpha\to0}\cH(\bar f_\alpha(t)).
\end{align}
The lower bound as in  \eqref{LlogL:minus} and the second moment
estimate imply that $|\cH(\bar f_t)|<+\infty$ for almost every
$t$.

Combining \eqref{f-alpha-0:limit}, \eqref{limit-L1:H-lsc} and 
\eqref{limit-L1:D-lsc}, and then using Fatou's
lemma in expectation, yields for almost every $t\in[0,T]$,
\begin{align*}
\bar{\mathbb E}\big[\cH(\bar f_t)\big]
+\bar{\mathbb E}\left[\int_0^t\cD(\bar f_s)\,\dd s\right]
\le
\cH(f_0)+C .
\end{align*}

It remains to pass the maximal exponential entropy estimate to the limit.  Fix $\lambda>0$ and assume $0<\eps\le\eps_0(\lambda)$.  Since $(\bar f_\alpha,\bar B_\alpha)$ has the same law as $(f_\alpha,B_\alpha)$, the last estimate in \eqref{bdd:alpha-m-n} gives
\begin{align*}
&\bar{\mathbb E}\exp\left[
\sup_{s\in[0,t]}\left(\lambda\cH(\bar f_\alpha(s))+\frac{\lambda}{2}\int_0^s\cD(\bar f_\alpha(r))\,\dd r\right)\right]\\
&\le C_\lambda\exp\left[\lambda\cH(f_0^\alpha)+C\lambda\big(\|f_0^\alpha\|_{L^1}^3+1\big)\right].
\end{align*}
By \eqref{limit-L1:H-lsc} and \eqref{limit-L1:D-lsc}, using again a countable dense set of full-measure times for the entropy term, we have
\begin{align*}
&\sup_{s\in[0,t]}\left(\lambda\cH(\bar f_s)+\frac{\lambda}{2}\int_0^s\cD(\bar f_r)\,\dd r\right)\\
&\le \liminf_{\alpha\to0}
\sup_{s\in[0,t]}\left(\lambda\cH(\bar f_\alpha(s))+\frac{\lambda}{2}\int_0^s\cD(\bar f_\alpha(r))\,\dd r\right).
\end{align*}
Since the exponential function is increasing, Fatou's lemma yields
\begin{align*}
&\bar{\mathbb E}\exp\left[
\sup_{s\in[0,t]}\left(\lambda\cH(\bar f_s)+\frac{\lambda}{2}\int_0^s\cD(\bar f_r)\,\dd r\right)\right]\\
&\le \liminf_{\alpha\to0}\bar{\mathbb E}\exp\left[
\sup_{s\in[0,t]}\left(\lambda\cH(\bar f_\alpha(s))+\frac{\lambda}{2}\int_0^s\cD(\bar f_\alpha(r))\,\dd r\right)\right].
\end{align*}
Finally, \eqref{f-alpha-0:limit} and the uniform mass bound for $f_0^\alpha$ imply
\begin{align*}
&\bar{\mathbb E}\exp\left[
\sup_{s\in[0,t]}\left(\lambda\cH(\bar f_s)+\frac{\lambda}{2}\int_0^s\cD(\bar f_r)\,\dd r\right)\right]\\
&\le C_\lambda\exp\left[\lambda\cH(f_0)+C\lambda\big(\|f_0\|_{L^1}^3+1\big)\right].
\end{align*}

 This proves all assertions of Theorem~\ref{thm:L-1}.

\section{Refined entropy dissipation estimates}
\label{sec:refined-entropy}

In this section, we isolate the refinement that is available for a particular choice of the conservative noise. More precisely, we consider a structure-preserving regularisation of the fluctuating Landau equation:
\begin{equation}
\label{eq:appendix-1}
\begin{aligned}
\dd f
&=\frac12\tn\cdot\big(A\theta^2(f)\theta^2(f_*)\tn\log\theta^2(f)\big)\dd t
-\frac{\sqrt{\eps}}{2}\tn\cdot\big(A^{1/2}\theta(f)\theta(f_*)\dd\xi_K\big)\\
&\quad+\frac{\eps}{2}\sum_{k=1}^K\tn\cdot \Big(G_k(v,v_*)\theta'(f)\theta(f_*)
\tn\cdot\big(G_k(v,w)\theta(f)\theta(f_w)\big)\Big)\dd t .
\end{aligned}
\end{equation}

The coefficient $\theta$ regularizes the square-root amplitude in the stochastic conservative flux. Furthermore, we use the same regularisation to modify the Landau collision operator. The entropy estimate in Theorem~\ref{thm:L-1} contains an additive constant arising from the Stratonovich-to-It\^o correction. If the active collision fields are divergence-free in the Landau pair variables, this correction vanishes in the entropy balance. Consequently, the expected $\theta$-entropy is non-increasing.
\begin{assumption}[Stochastic coefficient]
  \label{ass:7}
Let $r_0\in(0,1)$ be fixed. Let $\theta\in C^1(\R_+;\R_+)$ satisfy
\begin{equation}
\begin{gathered}
\theta(r)=\sqrt r \quad r\in(r_0,\infty),\quad \theta(0)=0,\\
C_0\le \theta'(r)\le C_{r_0}\quad\text{and}\quad  0\le \theta(r)\le\sqrt{r}\quad r\in[0,r_0].    
\end{gathered}
\end{equation}
for some constants $0<C_0\le C_{r_0}<+\infty$.
\end{assumption}
 \begin{remark}[Compared with $\sigma$ in Assumption \ref{ass:sigma-R0-0}]
In this section, $\theta$ plays the role of the stochastic noise coefficient satisfying Assumption \ref{ass:sigma-R0-0} in Theorem \ref{main:thm}. Compared with $\sigma$, we assume here a positive lower bound on $\theta'$, namely
\begin{gather*}
0<C_0\le \theta'(r)\le C_{r_0}.
\end{gather*}
This ensures that $\theta$ behaves linearly near $0$. The reason is that we also use $\theta^2(f)\theta^2(f_*)$ here to approximate the Landau mobility $ff_*$. Instead of studying the evolution of the Boltzmann entropy $\cH(f)$ as in Theorem \ref{main:thm}, we study the approximate $\theta$-entropy \eqref{def:entropy-theta} here. The above positive lower bounds on $\theta'$ ensure that the $\theta$-entropy $\cH_\theta(f)$ is equivalent to the Boltzmann entropy $\cH(f)$ up to bounded mass and energy; see Remark \ref{rmk:boltzmann-entropy} below.

 \end{remark}

\begin{assumption}
\label{ass:G-k-2}
Let $K\in \N$ be fixed. Let $\{g_k\}_{1\le k\le K}\subset \cS(\G;\R^d)$ be the active finite family satisfying Assumption~\ref{ass:G-k:app}.  We assume that these active modes satisfy
\begin{equation}
\label{div-free}
   \big(\nabla_v-\nabla_{v_*}\big)\cdot \Pi_{(v-v_*)^\perp}g_k(v,v_*)=0,
   \qquad 1\le k\le K.
\end{equation}
\end{assumption}

Appendix~\ref{app-sec:ONB-2} gives a construction of such admissible active
modes.

We define the $\theta$-entropy and dissipation as follows
\begin{gather}
\cH_\theta(f)=\int_v h_\theta(f)\dd v,\qquad
h_\theta(s)=\int_0^s\log(\theta^2(r))\dd r,\label{def:entropy-theta}\\
\cD_{\theta}(f)
:=\frac12\int_{v,v_*}A\theta^2(f)\theta^2(f_*)
\big|\tn \log(\theta^2(f))\big|^2\dd v\dd v_*\ge 0.
\end{gather}

We show that the approximation $\theta$-entropy, up to energy, is equivalent to the Boltzmann entropy.
\begin{remark}[Compared to Boltzmann entropy]
\label{rmk:boltzmann-entropy}
We compare the $\theta$-entropy defined in \eqref{def:entropy-theta} with the Boltzmann entropy $\cH(f)=\int_v f\log f \dd v$.

Let $\theta$ satisfy Assumption~\ref{ass:7} and let $r_0\in(0,1)$ be small.
For $r\in(0,r_0)$, we have  $\theta^2(r)\le 1$ and  $C_0 r\le \theta(r)\le \sqrt{r}$, which implies that 
\begin{align}
\label{bdd-1:rmk}
2(s\log C_0+s\log s-s)=\int_{0}^s \log (C_0 r)^2 \dd r\le h_\theta(s)\le \int_{0}^s \log r \dd r=s\log s-s\le0.
\end{align}
For $r\in(r_0,\infty)$, we have 
\begin{align}
\label{bdd-2:rmk}
h_\theta(s)=
         \int_{r_0}^s \log r\dd r+h_\theta(r_0)=
         s\log s-s+C_{h_\theta,r_0},
\end{align}
where 
\begin{align*}
 C_{h_\theta,r_0}:=-r_0\log r_0+r_0+h_\theta(r_0)\le 0.
\end{align*}

Then for any $f\in L^1_2(\Do;\R_+)$, we have 
\begin{align*}
&\int_{\{0\le f\le  r_0\}}(2\log C_{0}-1)f+f\log f+\int_{\{f\ge r_0\}}C_{h_\theta,r_0}+\cH(f)-\|f\|_{L^1_v}\\
    &\le\cH_\theta(f)\le \cH(f).
\end{align*}
Concerning the left-hand side of the above inequality, the estimate \eqref{LlogL:minus} ensures that 
\begin{align*}
\int_{\{0\le f\le  r_0\}}f\log f\gtrsim-\|f\|_{L^1_2},
\end{align*}
and $f\in L^1_v$ and $C_{h_\theta,r_0}\le 0$ ensure that 
\begin{align*}
  \int_{\{f\ge r_0\}} C_{h_\theta,r_0}\ge C_{h_\theta,r_0}\frac{\|f\|_{L^1}}{r_0}. 
\end{align*}
Hence, we have 
\begin{align*}
\cH(f)-C \|f\|_{L^1_2}\le\cH_\theta(f)\le \cH(f) 
\end{align*}
for some constant $C=C(r_0)>0$.

For the particular model choice $\theta^2(r)=r^2/r_0$ on $[0,r_0]$ and $\theta^2(r)=r$ on $(r_0,\infty)$, the entropy density takes the explicit form
\begin{equation}
\label{h-sigma:log}
\begin{aligned}
    h_\theta(r)&=\left\{
    \begin{aligned}
        \int_0^r\log(s^2/r_0)\dd s&\quad r\in[0,r_0]\\
         \int_{r_0}^r \log s\dd s+r_0\log r_0&\quad r\in(r_0,\infty)
    \end{aligned}\right.\\
    &=\left\{
    \begin{aligned}
        2 (r\log r-r)-r\log r_0 &\quad r\in[0,r_0]\\
         r\log r-r_0&\quad r\in(r_0,\infty)
    \end{aligned}\right..
\end{aligned}
\end{equation}

\end{remark}
Let us first record the formal cancellation.  Applying It\^o's formula
\eqref{dS-00} to $\cH_\theta(f)$ gives
\begin{equation}
 \begin{aligned}
\frac{\dd}{\dd t} \cH_{\theta}(f)= 
&-\frac12\int_{v,v_*}\tn \log(\theta^2)\cdot\big( A\theta^2\theta_*^2\tn \log( \theta^2)\big)\\
&+\frac{\sqrt{\eps}}{2}\sum_{k= 1}^K\int_{v,v_*}\tn \log(\theta^2)\cdot\big(G_k(v,v_*)\theta(f)\theta(f_*)\big)\dd B_k\\
         &+\frac12\sum_{k= 1}^K\int_{v,v_*}G_k(v,v_*)\cdot \theta'(f)\nabla_{v_*} \theta(f_*) \tn\cdot\big(G_k(v,w)\theta(f)\theta(f_w)\big)\\
         &+\frac12 \sum_{k= 1}^K\int_{v,v_*}\big(\nabla_v\cdot G_k(v,v_*)\big)\theta'(f)\theta(f_*)\tn\cdot\big(G_k(v,w)\theta(f)\theta(f_w)\big).
\end{aligned}
\end{equation}
The last two terms can be combined by integration by parts in $v_*$ and by
using the antisymmetry of the Landau pair divergence.  They reduce to
\begin{align}
\label{error-0}
             \sum_{k= 1}^K\int_{v,v_*}\big(\nabla_v-\nabla_{v_*}\big)\cdot G_k(v,v_*)\theta'(f)\theta(f_*)\tn\cdot\big(G_k(v,w)\theta(f)\theta(f_w)\big)=0,
\end{align}
where we use the form $A=A(|v-v_*|)$, the divergence free condition \eqref{div-free}, and 
\begin{align*}
&\big(\nabla_v-\nabla_{v_*}\big)\cdot G_k(v,v_*)\\
=&{}A^{1/2} (|v-v_*|)   \big(\nabla_v-\nabla_{v_*}\big)\Pi_{(v-v_*)^\perp}g_k(v,v_*)\\
&+\big(\Pi_{(v-v_*)^\perp}(\nabla_v-\nabla_{v_*})A^{1/2} (|v-v_*|)\big)\cdot g_k(v,v_*)=0.
\end{align*}
Thus \eqref{div-free} removes the extra It\^o drift in the entropy balance,
and formally
\begin{align*}
    \frac{\dd}{\dd t}\cH_{\theta}(f)+\cD_\theta(f)=\frac{\sqrt{\eps}}{2}\sum_{k= 1}^K\int_{v,v_*}\tn \log(\theta^2)\cdot\big(G_k(v,v_*)\theta(f)\theta(f_*)\big)\dd B_k.
\end{align*}
After localisation, the stochastic term has zero expectation, as in the proof
of Proposition~\ref{prop:log-chi}.  The next lemma makes this computation
rigorous through the approximation scheme of Sections~5 and~6.

\begin{lemma}
\label{lem:sec-7}
Assume Assumption~\ref{ass:G-k-2}.  Let $f$ be the probabilistic weak
solution of \eqref{eq:appendix-1} constructed by the same approximation scheme as in Theorem~\ref{thm:L-1}.  Then,
for almost every $t\in[0,T]$,
\begin{equation}
\label{H:dissip}
\hE\big[\cH_{\theta}(f_t)\big]
+\hE\Big[\int_0^t\cD_{\theta}(f_s)\dd s\Big]
\le \hE\big[\cH_{\theta}(f_0)\big].
\end{equation}
\end{lemma}

\begin{proof}
The proof follows the same compactness scheme as Theorem~\ref{thm:L-1}; only
the entropy estimate for the smooth approximations is sharpened.  Fix
$\alpha\in(0,1)$ and $n\in\N$.  We approximate $A$ and the logarithm
by $A_n$ and $L_n$ as in the proof of Proposition~\ref{lem:l-12:log}, and we
replace the initial datum by $f_0^\alpha$.  Moreover, we approximate $\theta$ by a smooth bounded sequence $\theta_n$ as in the construction of $\sigma_n$ in \eqref{def:sigma-n}. Compared with the proof of Proposition~\ref{lem:l-12:log}, we approximate both the Landau mobility and the noise coefficient by $\theta_n$, that is, we take $\mu_n=\theta_n$. We choose $\theta_n$ so that $\theta_n=\theta$ near the origin and $\theta_n\to\theta$ locally uniformly on $\R_+$. In particular, we do not replace the fixed linear behaviour of $\theta$ near $0$ by the square-root coefficient.

The corresponding regularised
equation is
\begin{equation}
\label{ito-epsi-1:app}
    \begin{aligned}
        \dd f
        &=\alpha \Delta f\,\dd t+Q_{\theta_n,L_n}(f,f)\dd t\\
        &\quad-\frac{\sqrt{\eps}}{2}\tn\cdot\left(A_n^{1/2}\theta_n(f)\theta_n(f_*)\dd \xi_K\right)\\
        &\quad+\frac{\eps}{2}\tn\cdot \Big(\int_w\theta_n(f_*)\theta_n(f_w)
        \big(F_{4,n}\nabla_v\varsigma_n(f)+\theta_n(f)\theta_n'(f)F_{1,n}\big)\Big)\dd t ,
    \end{aligned}
\end{equation}
where $F_{1,n}$ and $F_{4,n}$ are associated with $A_n$ through
\eqref{def:G-k:n}.

Let $f_{n,\alpha}$ be the probabilistic weak solution given by
Proposition~\ref{lem:app-1:ext}.  For
\begin{equation*}
\mathcal{H}_{\theta,L}(g):=\int_v h_{\theta,L}(g(v))\,\dd v,\qquad
h_{\theta,L}(s):=\int_0^s L(\theta^2(r))\dd r,
\end{equation*}
the refined localised estimate proved in Subsection~\ref{sec-7:n-1} gives
\begin{equation}
         \label{app-goal-1}
     \begin{aligned}
&\hE\big[\cH_{\theta_n,L_n}(f_{n,\alpha}(t))\big]
-\hE\big[\cH_{\theta_n,L_n}(f^{\alpha}_0)\big]
+\hE\Big[\int_0^t\cD_{\theta_n,L_n}(f_{n,\alpha}(s))\dd s\Big]\\
&\qquad\le
C\hE\Big[\int_0^t
\Big(\int_{\{0\le \theta_n^2(f_{n,\alpha})\le n^{-1}\}}
\theta_n^2(f_{n,\alpha})\,\dd v\Big)^{1/2}\dd s\Big],
\end{aligned}
\end{equation}
where $C$ depends only on the fixed structural constants and on
$\|f_0^\alpha\|_{L^1}$.

We now let $n\to\infty$ with $\alpha$ fixed.  By the compactness and
identification argument of Proposition~\ref{lem:l-12:log}, a subsequence of
$f_{n,\alpha}$ converges to a probabilistic weak solution $f_\alpha$ of the
equation with coefficients $(A,\theta,\log)$ and viscosity $\alpha$. Passing to the limit as in the proof of Proposition~\ref{lem:l-12:log}, we have 
\begin{align}
\label{sec7:lsc-n}
&\hE\big[\cH_{\theta}(f_{\alpha}(t))\big]
-\hE\big[\cH_{\theta}(f^\alpha_0)\big]
+\hE\Big[\int_0^t\cD_{\theta}(f_\alpha(s))\dd s\Big]\notag\\
&\le
\liminf_{n\to\infty}
\Big\{
\hE\big[\cH_{\theta_n,L_n}(f_{n,\alpha}(t))\big]
-\hE\big[\cH_{\theta_n,L_n}(f^\alpha_0)\big]
+\hE\Big[\int_0^t\cD_{\theta_n,L_n}(f_{n,\alpha}(s))\dd s\Big]
\Big\}.
\end{align}

It remains to show that the right-hand side error in \eqref{app-goal-1}
vanishes.  Since $\theta_n$ is linear near the origin and
$\theta_n^2(r)\le r$, there is $c_0>0$, independent of $n$, such that
\begin{align*}
\{0\le \theta_n^2(f_{n,\alpha})\le n^{-1}\}
\subset \{0\le f_{n,\alpha}\le c_0 n^{-1/2}\}.
\end{align*}
By the energy bound \eqref{energy:h-eps},
\begin{align*}
        \sup_n\sup_{t\le T}\int_v|v|^2  f_{n,\alpha}(t,v)\dd v\le C_\alpha
        \quad \hbox{a.s.}
\end{align*}
Therefore, for every $R>0$,
\begin{align*}
\int_{\{0\le f_{n,\alpha}\le c_0n^{-1/2}\}} f_{n,\alpha}(t,v)\dd v
&\le
\int_{\{|v|>R\}} f_{n,\alpha}(t,v)\dd v
+\int_{\{|v|\le R,\ f_{n,\alpha}\le c_0n^{-1/2}\}}f_{n,\alpha}(t,v)\dd v\\
&\le C_\alpha R^{-2}+c_0|B_R|n^{-1/2}.
\end{align*}
Since $\theta_n^2(r)\le r$, taking square roots in the preceding bound gives
\begin{align*}
&\hE\Big[\int_0^t
\Big(\int_{\{0\le \theta_n^2(f_{n,\alpha})\le n^{-1}\}}
\theta_n^2(f_{n,\alpha})\,\dd v\Big)^{1/2}\dd s\Big]\\
&\qquad\le t\big(C_\alpha R^{-2}+c_0|B_R|n^{-1/2}\big)^{1/2}.
\end{align*}
Letting first $n\to\infty$ and then $R\to\infty$ proves
\begin{align}
\label{app-goal-2}
\limsup_{n\to\infty}
\hE\Big[\int_0^t
\Big(\int_{\{0\le \theta_n^2(f_{n,\alpha})\le n^{-1}\}}
\theta_n^2(f_{n,\alpha})\,\dd v\Big)^{1/2}\dd s\Big]=0.
\end{align}
Combining \eqref{app-goal-1}, \eqref{sec7:lsc-n} and \eqref{app-goal-2}, we
obtain
\begin{equation}
\label{app-goal-3}
\hE\big[\cH_{\theta}(f_\alpha(t))\big]
-\hE\big[\cH_{\theta}(f^\alpha_0)\big]
+\hE\Big[\int_0^t\cD_{\theta}(f_\alpha(s))\dd s\Big]\le0.
\end{equation}

 By applying the preceding argument in the same way as in Proposition~\ref{lem:l-12:log} for showing \eqref{H:limit-L2-exp}, we obtain the following refined maximal exponential estimate. For every $\lambda>0$, there exists $\eps_0(\lambda)>0$ such that, if $0<\eps\le\eps_0(\lambda)$, then
\begin{equation*}
\hE\exp\left[
\operatorname{esssup}_{s\in[0,t]}\left(\lambda\cH_\theta(f_\alpha(s))+\frac{\lambda}{2}\int_0^s\cD_\theta(f_\alpha(r))\dd r\right)\right]
\le C_\lambda\exp\left[\lambda\cH_\theta(f_0^\alpha)\right].
\end{equation*}

Finally, we let $\alpha\downarrow0$.  The compactness, identification of the
limit equation, and lower semicontinuity of the entropy and dissipation are
the same as in the proof of Theorem~\ref{thm:L-1}. However, we emphasize that establishing the uniform $L^1_tW^{1,1}_v$ estimate in the spirit of Lemma~\ref{lem:gradient-es} requires a more refined argument for a general choice of $\theta$. Precisely, by the assumptions on $\theta$, its derivative is bounded. Hence,
\begin{align*}
    &\mathbb{E}\int_0^T
    \Big(
    \int_{\{\{0\le f_\alpha\le r_0\}\cap B_R\}}
    |\nabla_v f_\alpha|^2
    +
    \int_{\{\{f_\alpha\ge r_0\}\cap B_R\}}
    \frac{|\nabla_v f_\alpha|^2}{f_\alpha}
    \Big)
    \\
    \le{}\,&
    C\,
    \mathbb{E}\int_0^T
    \Big(
    \int_{\{\{0\le f_\alpha\le r_0\}\cap B_R\}}
    \theta'(f_\alpha)^2
    |\nabla_v f_\alpha|^2
    +
    \int_{\{\{f_\alpha\ge r_0\}\cap B_R\}}
    \frac{|\nabla_v f_\alpha|^2}{f_\alpha}
    \Big)
    \\
    \le{}\,&
    C\,
    \mathbb{E}\int_0^T\int_{B_R}
    |\nabla_v \theta(f_\alpha)|^2
    \le
    C(f_0).
\end{align*}

Here, the last step follows directly from \cite[Theorem~1]{Des15} applied to $\theta^2$. Indeed, the bound $\theta(f)\le \sqrt{f}$ implies that $\theta^2(f)\in L^1_2$. It therefore remains to verify that $|\cH(\theta^2(f))|<+\infty$.

By Assumption~\ref{ass:7}, we have
\begin{align*}
&C_0^2s^2\le \theta^2(s)\le s\quad s\in(0,r_0),\quad \theta^2(s)=s\quad s\in[r_0,\infty). 
\end{align*}

By Remark~\ref{rmk:boltzmann-entropy}, we have 
\begin{align*}
&s\log (C_{r_0}s)^2\le\theta^2(s)\log \theta^2(s)\le0\quad s\in(0,r_0),\quad \theta^2(s)\log \theta^2(s)=s\log s\quad s\in[r_0,\infty). 
\end{align*}
Hence, we have 
\begin{align*}
2\log C_{r_0}\|f\|_{L^1}+2\int_{\{0\le f\le 1\}}f\log f \le \cH(\theta^2(f))\le \int_{\{f\ge 1\}}f\log f.
\end{align*}
By using of the estimate \eqref{LlogL:minus}, we have 
\begin{align*}
\int_{\{0\le f\le  1\}}f\log f\gtrsim-\|f\|_{L^1_2}.
\end{align*}
Hence, we have 
\begin{align*}
-C \|f\|_{L^1_2}\le \cH(\theta^2(f))\le \cH(f)+C\|f\|_{L^1_2},
\end{align*}
for some constant $C=C(r_0)>0$.

Using the equivalence between $\cH(f)$ and $\cH_\theta(f)$ in Remark~\ref{rmk:boltzmann-entropy}, we have 
\begin{align*}
 |\cH(\theta^2(f))|\le |\cH_\theta(f)|+C\|f\|_{L^1_2}.
\end{align*}

As a consequence, for every $R>1$,
\begin{align*}
    \|\nabla_v f_\alpha\|_{L^1_tL^1(B_R)}
    \le&
    \|\nabla_v f_\alpha \,\mathbb{1}_{\{0\le f_\alpha\le r_0\}}\|_{L^1_tL^2_v}
    |B_R|^{\frac12}
    \\
    &+
    \|f_\alpha\|_{L^1}^{\frac12}
    \|\nabla_v \sqrt{f_\alpha}\,\mathbb{1}_{\{f_\alpha\ge r_0\}}\|_{L^1_tL^2(B_R)}
    \le C(R,f_0).
\end{align*}

 We pass to the limit by letting $\alpha\to0$ in \eqref{app-goal-3} to obtain \eqref{H:dissip}. Although Theorem~\ref{thm:L-1} concerns the Boltzmann entropy $\cH(f)$, here we consider the $\theta$-entropy $\cH_\theta$; the term $\cH_\theta$ can be treated in the same way as in the entropy computation of Proposition~\ref{lem:l-1:app}. We recall
\begin{gather*}
\cH_\theta(f)=\int_v h_\theta(f)\dd v,\qquad
h_\theta(s)=\int_0^s\log(\theta^2(r))\dd r.
\end{gather*}
Since $h_\theta$ is continuous and convex, and satisfies $h_\theta(s)\ge -C(1+s)$ for all $s\ge0$, it follows, as for the Boltzmann entropy $\cH$, that $\cH_\theta(f)$ is weakly lower semicontinuous in $L^1(\R^d)$ under the uniform mass bound; see, for example, \cite[Theorem 5.14]{FL07}. Indeed, the convexity of $h_\theta$, together with $h_\theta(0)=0$ and $h_\theta(s)>0$ for $s\gg1$, ensures such an affine lower bound. Moreover, the linear behaviour of $\theta(r)$ near $0$ ensures the equivalence between $\cH_\theta$ and $\cH$, as stated in Remark \ref{rmk:boltzmann-entropy}.

The refined maximal exponential estimate also passes to the limit.  Indeed, the lower semicontinuity of $\cH_\theta$ and $\cD_\theta$, together with the same countable-dense-time argument used in the proof of Theorem~\ref{thm:L-1}, gives
\begin{align*}
&\operatorname{esssup}_{s\in[0,t]}\left(\lambda\cH_\theta(f_s)+\frac{\lambda}{2}\int_0^s\cD_\theta(f_r)\,\dd r\right)\\
&\le
\liminf_{\alpha\to0}
\operatorname{esssup}_{s\in[0,t]}\left(\lambda\cH_\theta(f_\alpha(s))+\frac{\lambda}{2}\int_0^s\cD_\theta(f_\alpha(r))\,\dd r\right).
\end{align*}
Applying Fatou's lemma to the refined maximal estimate for $f_\alpha$ and using $\cH_\theta(f_0^\alpha)\to\cH_\theta(f_0)$, we obtain
\begin{align*}
\hE\exp\left[
\operatorname{esssup}_{s\in[0,t]}\left(\lambda\cH_\theta(f_s)+\frac{\lambda}{2}\int_0^s\cD_\theta(f_r)\,\dd r\right)\right]
\le C_\lambda\exp\left[\lambda\cH_\theta(f_0)\right].
\end{align*}
This proves the refined maximal estimate in Theorem~\ref{main:thm-2}.

\end{proof}

\subsection{The localised refined entropy estimate}\label{sec-7:n-1}

We prove the estimate \eqref{app-goal-1}. 
The argument is the localised
entropy computation of Proposition~\ref{lem:l-1:app}; the only new point is
that the $G_k$-representation of the stochastic correction allows us to use
the divergence-free condition \eqref{div-free}.  

Let
$f\in L^\infty([0,T];L^1_2\cap L^2(\R^d;\R_+))$ be a regularised solution as
in Proposition~\ref{lem:l-1:app}.  Assume that the hypotheses of that
proposition hold, that $L$ is given by \eqref{exm:L-s1}, and that the active
modes satisfy Assumption~\ref{ass:G-k-2}.  We shall prove
\begin{equation}
         \label{app-goal-4}
\begin{aligned}
\hE\big[\cH_{\theta,L}(f(t))\big]
-\hE\big[\cH_{\theta,L}(f_0)\big]
+\hE\Big[\int_0^t\cD_{\theta,L}(f_s)\dd s\Big]
\le
C\hE\Big[\int_0^t
\Big(\int_{\{0\le\theta^2(f_s)\le s_0\}}\theta^2(f_s)\dd v\Big)^{1/2}\dd s\Big].
\end{aligned}
\end{equation}

We recall the regularised logarithm
\begin{equation}
\label{exm:L-s1}
        L(s)=
        \begin{cases}
        \log s, & s\in[s_0,\infty),\\
        s_0^{-1} s+\log s_0-1, & s\in(-\infty,s_0).
        \end{cases}
    \end{equation}
In the proof of Lemma~\ref{lem:sec-7}, the function $L_n$ is obtained from
\eqref{exm:L-s1} by taking $s_0=n^{-1}$.

Following Section~\ref{sec:4-3}, we first prove a localised entropy
inequality and then remove the cut-off.  Let $\chi_m$ be the cut-off functions
from Assumption~\ref{ass:chi}.  Set
    \begin{align*}
        \cH_{\theta,L,m}(f)=\int_v \chi_m(v)h_{\theta,L}(f)\dd v,\quad 
    \cD_{\theta,L,m}(f)=\frac12\int_{v,v_*}\chi_m(v)A\theta^2(f)\theta^2(f_*)\big|\tn L(\theta^2(f))\big|^2\dd v\dd v_*.
\end{align*}

Analogously to \eqref{H-theorem:1}, we show that 
\begin{equation}
    \label{H-theorem:1-app}
 \begin{aligned}
&\hE\big[\cH_{\theta,L,m}(f_t)\big]
+\hE\Big[\int_0^t\cD_{\theta,L,m}(f_s)\dd s\Big]\le \hE\big[\cH_{\theta,L,m}(f_0)\big]\\
&\quad+ C_1 m^{-d}
\hE\Big[\big(\|f\|_{L^\infty_tL^2_v}+1\big)^3
\big(\|f\|_{L^2_tH^1_v}+1\big)\Big]\\
&\quad+C_2\hE\Big[\int_0^t
\Big(\int_{\{0\le \theta^2(f_s)\le s_0\}}\theta^2(f_s)\dd v\Big)^{1/2}\dd s\Big]
\end{aligned}
\end{equation}
for some constants $C_1=C_1(C_{r_0},C_{R_0},C_K,d,A_0,z^*,T)>0$ and $C_2=C_2(C_{r_0},C_K,\|f_0\|_{L^1})>0$.
Letting $m\to\infty$ in \eqref{H-theorem:1-app}, the cut-off errors vanish by
the uniform estimates of Proposition~\ref{lem:l-1:app}; monotone convergence
for the nonnegative dissipation and lower semicontinuity of the entropy then
give \eqref{app-goal-4}.

It remains to prove \eqref{H-theorem:1-app}. Compared with
\eqref{H-theorem:1}, the only new contribution is the small-density term
containing $C_2$; it comes from the stochastic correction term $T_1$ in the
localised It\^o formula
\begin{equation}
 \begin{aligned}
\cH_{\theta,L,m}(f_t)= &\cH_{\theta,L,m}(f_0)- \alpha\int^t_0\int_{v}\nabla f\cdot \nabla_v\big(L (\theta^2)\chi_m\big) \\
&-\frac12 \int^t_0\int_{v,v_*}A\theta^2\theta_{*}^2\tn L (\theta^2)\cdot \tn \big(L (\theta^2)\chi_m\big)\\
&+\frac{\sqrt{\eps}}{2}\sum_{k=1}^K\int^t_0\int_{v,v_*}\tn \big(L (\theta^2)\chi_m\big)\cdot\big(G_k(v,v_*)\theta(f)\theta(f_*)\big)\dd B_k\\
           &+T_1(f,h_{\theta,L},\chi_m)+T_2(f,h_{\theta,L},\chi_m).
\end{aligned}
\end{equation}

We only need to estimate $T_1$ term by taking \eqref{div-free} into account. Showing \eqref{H-theorem:1}, we use the form of $T_1$ as given in \eqref{ito:cut-off} ($F_i$ representation). Here, the key is to keep the $G_k$ presentation to use the divergence-free condition \eqref{div-free}
\begin{align*}
   &T_1(f,h_{\theta,L},\chi_m)\\
   =&{}2\int_{v,v_*}\chi_m(v_*)\big(L'(\theta^2(f_*))\theta^2(f_*)\big)G_k(v,v_*)\cdot \theta'(f)\nabla_{v_*} \theta(f_*) \tn\cdot\big(G_k(v,w)\theta(f)\theta(f_w)\big)\\
&+2\int_{v,v_*}\chi_m(v)\big(L'(\theta^2(f))\theta^2(f)\big)\big(\nabla_v\cdot G_k(v,v_*)\big)\theta'(f)\theta(f_*)\tn\cdot\big(G_k(v,w)\theta(f)\theta(f_w)\big).
\end{align*}

Moreover, using \eqref{div-free}, we have 
    \begin{align*}
            &T_1(f,h_{\theta,L},\chi_m)\\
          =&{}2\int_{v,v_*}\chi_m(v)\big(\nabla_v\cdot G_k(v,v_*)-\nabla_{v_*}\cdot G_k(v,v_*)\big)\theta'(f)\theta(f_*)\tn\cdot\big(G_k(v,w)\theta(f)\theta(f_w)\big)\\
          &+2\int_{v,v_*}\big(\chi_m(v)-\chi_m(v_*)\big)\nabla_{v_*}\cdot G_k(v,v_*)\theta'(f)\theta(f_*)\tn\cdot\big(G_k(v,w)\theta(f)\theta(f_w)\big)\\
          &+2\int_{v,v_*}\chi_m(v_*)\big(L'(\theta^2(f_*))\theta^2(f_*)-1\big)G_k(v,v_*)\cdot \theta'(f)\nabla_{v_*} \theta(f_*) \tn\cdot\big(G_k(v,w)\theta(f)\theta(f_w)\big)\\
         &+2\int_{v,v_*}\chi_m(v)\big(L'(\theta^2(f))\theta^2(f)-1\big)\big(\nabla_v\cdot G_k(v,v_*)\big)\theta'(f)\theta(f_*)\tn\cdot\big(G_k(v,w)\theta(f)\theta(f_w)\big)\\
          =&{}4\int_{v,v_*,w}\big(\chi_m(v)-\chi_m(v_*)\big)\nabla_{v_*}\cdot G_k(v,v_*)\theta'(f)\theta(f_*)\big(\nabla_v\cdot G_k(v,w)\theta(f)\theta(f_w)+G_k(v,w)\nabla_v\theta(f)\theta(f_w)\big)\\
          &+4\int_{v,v_*}\chi_m(v_*)\big(L'(\theta^2(f_*))\theta^2(f_*)-1\big)G_k(v,v_*)\cdot \theta'(f)\nabla_{v_*} \theta(f_*) \tn\cdot\big(G_k(v,w)\theta(f)\theta(f_w)\big)\\
         &+4\int_{v,v_*}\chi_m(v)\big(L'(\theta^2(f))\theta^2(f)-1\big)\big(\nabla_v\cdot G_k(v,v_*)\big)\theta'(f)\theta(f_*)\tn\cdot\big(G_k(v,w)\theta(f)\theta(f_w)\big).
         \end{align*}
         

By the definition of $F_i$ in \eqref{def:F_i}, we have 
         \begin{align*}
          &T_1(f,h_{\theta,L},\chi_m)\\
          =&{}4\int_{v,v_*,w}\big(\chi_m(v)-\chi_m(v_*)\big)\big(F_3 \theta'(f)\theta(f_*)\theta(f)\theta(f_w)-F_2\cdot \nabla_vf [\theta'(f)]^2\theta(f_*)\theta(f_w)\big)\\
          &+4\int_{v,v_*,w}\chi_m(v_*)\big(L'(\theta^2(f_*))\theta^2(f_*)-1\big) \theta'(f)\nabla_{v_*} \theta(f_*) \big(F_4\nabla_v\theta(f)\theta(f_w)+F_1\theta(f)\theta(f_w)\big)\\
&+4\int_{v,v_*,w}\chi_m(v)\big(L'(\theta^2(f))\theta^2(f)-1\big)\theta'(f)\theta(f_*)\big(F_2\nabla_v\theta(f)\theta(f_w)+F_3\theta(f)\theta(f_w)\big).
\end{align*}
Concerning the first term on the right-hand side, by  using of $|\chi_m(v)-\chi_m(v_*)|\le\|\nabla\chi_m\|_{L^\infty}|v-v_*|$, the regularised kernel $\supp (A)\subset \{z\in\R^d\mid z_*\le |z|\le z^*\}$, and the boundedness of $\theta$ and $\theta'$, we have 
\begin{align*}
&\Big|\int_{v,v_*,w}\big(\chi_m(v)-\chi_m(v_*)\big)\big(F_3 \theta'(f)\theta(f_*)\theta(f)\theta(f_w)-F_2\cdot \nabla_vf [\theta'(f)]^2\theta(f_*)\theta(f_w)\big)\Big|\\
    &\lesssim_{C_K,C_{r_0},C_{R_0}}\|\nabla \chi_m\|_{L^\infty}\big(1+\|\nabla_v f\|_{L^2}\big).
\end{align*}

We recall the definition of $l_1$ in  \eqref{def:l-i}
\begin{align}
    l_{1}(s):=\int_0^sL' (\theta^2(r))\theta'(r)\theta^2(r)\dd r,\quad 0\le l_1(s)\le\theta(s).
\end{align}
Then the second term on the right-hand side is bounded by 
\begin{align*}
              &\Big|\int_{v,v_*,w}\chi_m(v_*)\big(L'(\theta^2(f_*))\theta^2(f_*)-1\big) \theta'(f)\nabla_{v_*} \theta(f_*) \big(F_4\nabla_v\theta(f)\theta(f_w)+F_1\theta(f)\theta(f_w)\big)\Big|\\
              =&{}\Big|\int_{v,v_*,w}\chi_m(v_*)\theta(f_w)\nabla_{v_*} \big(l_1(f_*)-\theta(f_*)\big)\big( \nabla_v\cdot F_4\varsigma(f)- \theta'(f)F_1\theta(f)\big)\Big|\\
              =&{}\Big|\int_{v,v_*,w}\nabla_{v_*}\chi_m(v_*)\theta(f_w) \big(l_1(f_*)-\theta(f_*)\big)\big( \nabla_v\cdot F_4\varsigma(f)- \theta'(f)F_1\theta(f)\big)\Big|\\
              \lesssim&{}_{C_{r_0},C_{R_0},C_K}\|\nabla\chi_m\|_{L^\infty}\to0
              \end{align*}
              as $m\to\infty$.

We recall the definition of $l_2$ in  \eqref{def:l-i}
\begin{align}
    l_{2}(s):=\int_0^sL' (\theta^2(r))\theta'(r)^2\theta^2(r)\dd r,\quad 0\le l_2(s)\le C_{r_0}\theta(s).
\end{align}
              Concerning the third term on the right-hand side, we have 
 \begin{align*}
&\Big|\int_{v,v_*,w}\chi_m(v)\big(L'(\theta^2(f))\theta^2(f)-1\big)\theta'(f)\theta(f_*)\big(F_2\nabla_v\theta(f)\theta(f_w)+F_3\theta(f)\theta(f_w)\big)\Big|\\
\le&{}\Big|\int_{v,v_*,w}\chi_m(v)\theta(f_*)F_2\nabla_v\big(l_2(f)-\theta(f)\big)\theta(f_w)\Big|\\
&+\Big|\int_{v,v_*,w}\chi_m(v)\big(L'(\theta^2(f))\theta^2(f)-1\big)\theta'(f)\theta(f_*)F_3\theta(f)\theta(f_w)\Big|\\
=:&{}A_{1}+A_2.
\end{align*}  
By integration by parts, we have 
\begin{align*}
0\le A_{1}&=\Big|\int_{v,v_*,w}\nabla_v\chi_m(v)\theta(f_*)F_2\big(l_2(f)-\theta(f)\big)\theta(f_w)\Big|\\
&\lesssim_{C_{r_0},C_{R_0},C_K} \|\nabla\chi_m\|_{L^\infty}\to0
              \end{align*}
              as $m\to\infty$.   
Concerning the $A_2$ term, we use $\big|L'(\theta^2(f))\theta^2(f)-1\big|\le \mathbb{1}_{\{0\le \theta^2(f)\le s_0\}} $ to show  
\begin{align*}
0\le A_2\le &\int_{\{0\le \theta^2(f)\le s_0\}}\theta'(f)\theta(f)
\Big(\int_{v_*,w}\theta(f_*)\theta(f_w)|F_3|\dd v_*\dd w\Big)\dd v\\
\le&{} C_{r_0}
\Big(\int_{\{0\le \theta^2(f)\le s_0\}}\theta^2(f)\dd v\Big)^{1/2}
\Big\|\int_{v_*,w}\theta(f_*)\theta(f_w)|F_3|\dd v_*\dd w\Big\|_{L^2_v}\\
\lesssim&{}_{C_{r_0},C_K}\|f\|_{L^1}
\Big(\int_{\{0\le \theta^2(f)\le s_0\}}\theta^2(f)\dd v\Big)^{1/2}.
\end{align*}
In the last inequality we used the compact support and boundedness of the
kernels defining $F_3$: for a smooth compactly supported kernel $K$,
$|F_3(v,v_*,w)|\le C_KK(v-v_*)K(v-w)$, and hence, by Young's inequality and
$\theta^2(f)\le f$,
\begin{align*}
\Big\|\int_{v_*,w}\theta(f_*)\theta(f_w)|F_3|\dd v_*\dd w\Big\|_{L^2_v}
\le C_K\|K*\theta(f)\|_{L^4_v}^2
\le C_K\|\theta(f)\|_{L^2_v}^2
\le C_K\|f\|_{L^1_v}.
\end{align*}
Combining the cut-off estimate, the bound for the second logarithmic
regularisation error, and the bounds for $A_1$ and $A_2$, we obtain
\begin{align*}
|T_1(f,h_{\theta,L},\chi_m)|
\le C\|\nabla\chi_m\|_{L^\infty}\big(1+\|\nabla_v f\|_{L^2}\big)
+C\Big(\int_{\{0\le\theta^2(f)\le s_0\}}\theta^2(f)\dd v\Big)^{1/2} .
\end{align*}
After integrating in time and taking expectations, this is precisely the
additional contribution displayed in \eqref{H-theorem:1-app}.  All other terms
in \eqref{ito:sec-7} are estimated exactly as in the proof of
Proposition~\ref{lem:l-1:app}.  This proves \eqref{H-theorem:1-app}, and hence
\eqref{app-goal-4}.  Taking $s_0=n^{-1}$ gives \eqref{app-goal-1}.

\appendix

\section{A formal particle derivation of the fluctuating Landau noise}
\label{app:formal-particle-derivation}

We record the formal computation behind the noise appearing in
\eqref{eq:ideal}.  The point is not to construct the particle system
rigorously, but to identify the covariance of the fluctuation martingale.
This is the Landau analogue of the Dean--Kawasaki derivation: the noise in the
fluctuating hydrodynamic equation is chosen so that its second moment agrees
with the second moment of the martingale part of the empirical measure.

Set
\begin{equation*}
    a(z):=A(|z|)\Pi_{z^\perp},\qquad
    b(z):=\nabla_z\cdot a(z)=(1-d)|z|^\gamma z
    \quad\text{if } A(|z|)=|z|^{\gamma+2}.
\end{equation*}
In dimension $d=3$ one may write a square root of $a$ explicitly as
\begin{equation*}
\delta(z)=|z|^{\frac{\gamma}{2}}
\begin{bmatrix}
	z_2 & 0 & z_3\\
	-z_1 & z_3 & 0\\
	0 & -z_2 & -z_1
\end{bmatrix},
\qquad
\delta(z)\delta(z)^T=a(z).
\end{equation*}
Equivalently, since $\Pi_{z^\perp}$ is an orthogonal projection,
$a(z)^{1/2}=\sqrt{A(|z|)}\Pi_{z^\perp}$.

Let $(B_{ij})_{1\le i<j\le N}$ be independent $\R^d$-valued Brownian
motions and define the antisymmetric family $Z_{ij}$ by
$Z_{ij}=B_{ij}$ for $i<j$, $Z_{ji}=-B_{ij}$, and $Z_{ii}=0$.  Formally
consider the conservative Landau particle system
\begin{equation}
\label{app:particle-landau}
    \dd V_i
    =\frac{2}{N}\sum_{j=1}^N b(V_i-V_j)\,\dd t
    +\sqrt{\frac{2}{N}}\sum_{j=1}^N a(V_i-V_j)^{1/2}\,\dd Z_{ij},
    \qquad 1\le i\le N.
\end{equation}
The antisymmetry of $Z_{ij}$ is the relevant structural point: the noise acts
pairwise and conservatively.  For indices $i,j,k,l$,
\begin{equation}
\label{app:anti-brownian-cov}
    \dd [Z_{ij}^\alpha,Z_{kl}^\beta]_t
    =\delta_{\alpha\beta}
    \big(\delta_{ik}\delta_{jl}-\delta_{il}\delta_{jk}\big)\,\dd t.
\end{equation}

Let
\begin{equation*}
    \mu_t^N:=\frac1N\sum_{i=1}^N\delta_{V_i(t)}.
\end{equation*}
For $\phi\in C_c^\infty(\R^d)$, It\^o's formula gives
\begin{equation}
\label{app:empirical-ito}
    \dd \langle \mu_t^N,\phi\rangle
    =\mathcal L_N\phi(V(t))\,\dd t+\dd M_t^{N,\phi},
\end{equation}
where $\mathcal L_N\phi$ denotes the drift part and
\begin{equation}
\label{app:particle-martingale}
    \dd M_t^{N,\phi}
    =\frac{1}{N}\sqrt{\frac{2}{N}}
    \sum_{i,j=1}^N
    \nabla\phi(V_i)\cdot a(V_i-V_j)^{1/2}\,\dd Z_{ij}.
\end{equation}
By the antisymmetry of $Z_{ij}$, this martingale can be symmetrised as
\begin{equation}
\label{app:particle-martingale-sym}
    \dd M_t^{N,\phi}
    =\frac{1}{N}\sqrt{\frac{2}{N}}
    \sum_{1\le i<j\le N}
    \big(\nabla\phi(V_i)-\nabla\phi(V_j)\big)
    \cdot a(V_i-V_j)^{1/2}\,\dd B_{ij}.
\end{equation}
Consequently,
\begin{align}
\label{app:particle-bracket}
    \dd [M^{N,\phi}]_t
    &=\frac{2}{N^3}
    \sum_{1\le i<j\le N}
    \big|a(V_i-V_j)^{1/2}
    \big(\nabla\phi(V_i)-\nabla\phi(V_j)\big)\big|^2\,\dd t \notag\\
    &=\frac{1}{N^3}
    \sum_{i,j=1}^N
    \big(\nabla\phi(V_i)-\nabla\phi(V_j)\big)^T
    a(V_i-V_j)
    \big(\nabla\phi(V_i)-\nabla\phi(V_j)\big)\,\dd t \notag\\
    &=\frac1N
    \int_{\R^d\times\R^d}
    A(|v-v_*|)\mu_t^N(\dd v)\mu_t^N(\dd v_*)\,
    \big|\Pi_{(v-v_*)^\perp}(\nabla\phi(v)-\nabla\phi(v_*))\big|^2
    \,\dd t.
\end{align}
Thus the empirical measure has noise of order $N^{-1/2}$, while the
fluctuation field $\sqrt N(\mu^N-f)$ has a non-trivial Gaussian covariance
given by the last integral with $\mu_t^N$ replaced by the limiting density
$f_t(v)\dd v$.

We now compare this with the formal fluctuating Landau noise.  Let $\dot{W}$ be a
velocity--velocity--time white noise. Then for every smooth vector fields $H$ and $\widetilde H$,
\begin{align*}
    \hE&\left[
    \int_0^t\!\!\int_{v,v_*} H_s(v,v_*)\cdot \dot{W}(\dd s,\dd v,\dd v_*)
    \int_0^t\!\!\int_{v,v_*} \widetilde H_s(v,v_*)\cdot \dot{W}(\dd s,\dd v,\dd v_*)
    \right]\\
    =&\hE\int_0^t\!\!\int_{v,v_*}H_s\cdot \widetilde H_s\,\dd v\dd v_*\dd s.
\end{align*}
The density-level fluctuating equation with the particle scaling would contain
the martingale term
\begin{equation}
\label{app:formal-spde-noise}
    \frac1{\sqrt N}\int_0^t\!\!\int_{v,v_*}
    \sqrt{A(|v-v_*|)f_s(v)f_s(v_*)}\,
    \widetilde\nabla\phi(v,v_*)\cdot \dot{W}(\dd s,\dd v,\dd v_*),
\end{equation}
where
\begin{equation*}
    \widetilde\nabla\phi(v,v_*)
    =\Pi_{(v-v_*)^\perp}\big(\nabla\phi(v)-\nabla\phi(v_*)\big).
\end{equation*}
Its second moment is
\begin{align}
\label{app:spde-bracket}
    &\hE\left[
    \frac1{\sqrt N}\int_0^t\!\!\int_{v,v_*}
    \sqrt{A f_s f_{s,*}}\,
    \widetilde\nabla\phi\cdot \dot{W}(\dd s,\dd v,\dd v_*)
    \right]^2 \notag\\
    &\qquad
    =\frac1N\hE\int_0^t\!\!\int_{\R^d\times\R^d}
    A(|v-v_*|)f_s(v)f_s(v_*)
    \big|\Pi_{(v-v_*)^\perp}(\nabla\phi(v)-\nabla\phi(v_*))\big|^2
    \dd v\dd v_*\dd s.
\end{align}
This is exactly the formal limit of the particle covariance
\eqref{app:particle-bracket}.  Hence the multiplicative conservative noise
in \eqref{eq:ideal} has the same covariance as the empirical-measure
martingale.  Equivalently, after multiplying the density fluctuation
$\mu^N-f$ by $\sqrt N$, both the particle system and the fluctuating Landau
equation produce the same Gaussian covariance form.

In the body of the paper we use two regularisations of this formal structure.
First, the space-time noise is projected onto finitely many antisymmetric
tangential modes $g_k$, which gives the finite-dimensional noise
$\xi_K=\sum_{k=1}^K \dot\beta_k g_k$.  Second, the singular factor
$\sqrt f$ is replaced by the mobility $\sigma(f)$, with
$\sigma(r)=\sqrt r$ away from the origin and $\sigma(r)\sim r$ near zero.
These changes preserve the covariance structure at the formal level while
making the stochastic equation analytically tractable.

\section{Stratonovich-to-It\^o conversion}
\label{app:stratonovich-ito}

We give the formal computation which converts the Stratonovich equation used
in the introduction into the It\^o equation \eqref{SDE-1}.  The calculation is
purely algebraic; in the body of the paper the corresponding identities are
used only after testing against smooth functions and applying the localised
It\^o formula of Lemma~\ref{lem:cutoff-ito}.

Throughout this appendix we use the abbreviations
\begin{equation*}
    \sigma=\sigma(f(v)),\qquad
    \sigma_*=\sigma(f(v_*)),\qquad
    \sigma_w=\sigma(f(w)),
\end{equation*}
and similarly for $\sigma'$ and $f_*$.  Recall that
\begin{equation*}
    G_k(v,v_*)=\sqrt{A(|v-v_*|)}\Pi_{(v-v_*)^\perp}g_k(v,v_*),
    \qquad G_k(v,v_*)=-G_k(v_*,v).
\end{equation*}
For a function $h=h(v)$ we write
\begin{equation*}
    \tn_{(v,w)}\cdot\big(G_k(v,w)h(v,w)\big)
\end{equation*}
for the Landau divergence in the pair of variables $(v,w)$.  Define
\begin{equation}
\label{app:Rk-def}
    R_k(f)(v):=
    \tn_{(v,w)}\cdot\big(G_k(v,w)\sigma(f(v))\sigma(f(w))\big).
\end{equation}
With these conventions the regularised Stratonovich equation is formally
\begin{equation}
\label{app:strato-main}
    \dd f
    =Q(f,f)\,\dd t
    -\frac{\sqrt{\eps}}{2}\sum_{k=1}^K
    \tn\cdot\big(G_k(v,v_*)\sigma\sigma_*\circ \dd B_t^k\big).
\end{equation}

Let
\begin{equation*}
    \mathcal B_k(f):=
    -\frac{\sqrt{\eps}}{2}\tn\cdot\big(G_k(v,v_*)\sigma(f)\sigma(f_*)\big).
\end{equation*}
Then \eqref{app:strato-main} has the abstract form
\begin{equation*}
    \dd f=Q(f,f)\,\dd t+\sum_{k=1}^K\mathcal B_k(f)\circ \dd B_t^k .
\end{equation*}
The Stratonovich-to-It\^o rule gives
\begin{equation}
\label{app:strato-ito-rule}
    \mathcal B_k(f)\circ \dd B_t^k
    =
    \mathcal B_k(f)\,\dd B_t^k
    +\frac12 D\mathcal B_k(f)[\mathcal B_k(f)]\,\dd t .
\end{equation}
The Fr\'echet derivative of $\mathcal B_k$ in the direction $h$ is
\begin{equation}
\label{app:DBk}
    D\mathcal B_k(f)[h]
    =
    -\frac{\sqrt{\eps}}{2}\tn\cdot\Big(
    G_k(v,v_*)\big(\sigma'(f)h(v)\sigma_*
    +\sigma\sigma'_*
    h(v_*)\big)\Big).
\end{equation}
Since $\mathcal B_k(f)(v)=-\frac{\sqrt{\eps}}{2}R_k(f)(v)$, substituting
$h=\mathcal B_k(f)$ into \eqref{app:DBk} yields
\begin{equation}
\label{app:first-correction}
    \frac12D\mathcal B_k(f)[\mathcal B_k(f)]
    =
    \frac{\eps}{8}\tn\cdot\Big(
    G_k(v,v_*)\sigma'\sigma_*R_k(f)(v)
    +G_k(v,v_*)\sigma\sigma'_*R_k(f)(v_*)\Big).
\end{equation}

The two terms in \eqref{app:first-correction} are the same after applying the
Landau divergence.  Indeed, using the antisymmetry of $G_k$ and changing
$v$ and $v_*$ in the second contribution, one obtains, in the sense of
distributions,
\begin{equation}
\label{app:two-corrections-equal}
\begin{aligned}
&\tn\cdot\Big(
G_k(v,v_*)\sigma\sigma'_*R_k(f)(v_*)\Big)\\
&\qquad =
\tn\cdot\Big(
G_k(v,v_*)\sigma'\sigma_*R_k(f)(v)\Big).
\end{aligned}
\end{equation}
Equivalently, testing both sides against $\phi\in C_c^\infty(\R^d)$ and using
the integration-by-parts formula for $\tn\cdot$ gives
\begin{align*}
&-\int_{v,v_*}\tn\phi(v,v_*)\cdot
G_k(v,v_*)\sigma\sigma'_*R_k(f)(v_*)\,\dd v\dd v_*\\
&\quad
=-\int_{v,v_*}\tn\phi(v_*,v)\cdot
G_k(v_*,v)\sigma_*\sigma'R_k(f)(v)\,\dd v\dd v_*\\
&\quad
=-\int_{v,v_*}\tn\phi(v,v_*)\cdot
G_k(v,v_*)\sigma'\sigma_*R_k(f)(v)\,\dd v\dd v_*,
\end{align*}
where we used
$\tn\phi(v_*,v)=-\tn\phi(v,v_*)$ and
$G_k(v_*,v)=-G_k(v,v_*)$.  Therefore the two half contributions in
\eqref{app:first-correction} combine into
\begin{equation}
\label{app:ito-correction-G}
    \frac12D\mathcal B_k(f)[\mathcal B_k(f)]
    =
    \frac{\eps}{4}\tn\cdot\Big(
    G_k(v,v_*)\sigma'(f)\sigma(f_*)R_k(f)(v)\Big).
\end{equation}
With the convention used in the main text for the finite-dimensional noise
$\xi_K=\sum_{k=1}^K\dot B^k g_k$, this is equivalently written as
\begin{equation}
\label{app:ito-correction-main-normalisation}
    \frac{\eps}{2}\sum_{k=1}^K
    \tn\cdot \Big(G_k(v,v_*) \sigma'(f) \sigma(f_{*})
    \tn\cdot\big(G_k(v,w) \sigma(f) \sigma(f_{w})\big) \Big),
\end{equation}
which is the correction term appearing in \eqref{SDE-1}.  The difference
between \eqref{app:ito-correction-G} and
\eqref{app:ito-correction-main-normalisation} is only the normalisation of
the antisymmetric pair divergence: the operator $\tn\cdot$ counts both
ordered pairs $(v,v_*)$ and $(v_*,v)$.

Combining \eqref{app:strato-ito-rule}--\eqref{app:ito-correction-main-normalisation}
we obtain the It\^o form
\begin{equation}
\label{app:ito-main}
\begin{aligned}
\dd f
&=Q(f,f)\,\dd t
-\frac{\sqrt{\eps}}{2}\sum_{k=1}^K
\tn\cdot\big(G_k(v,v_*)\sigma(f)\sigma(f_*)\big)\,\dd B_t^k\\
&\quad
+\frac{\eps}{2}\sum_{k=1}^K
\tn\cdot \Big(G_k(v,v_*) \sigma'(f) \sigma(f_{*})
\tn\cdot\big(G_k(v,w) \sigma(f) \sigma(f_{w})\big) \Big)\,\dd t .
\end{aligned}
\end{equation}
This is precisely \eqref{SDE-1}.

Finally we record the form used in the weak formulation. Since $G_k(v,w)\sigma(f)\sigma(f_w)$ is antisymmetry,
\begin{equation}
\label{app:Rk-expanded}
\begin{aligned}
R_k(f)(v)
&=
2\int_w\Big[
(\nabla_v\cdot G_k(v,w))\sigma(f)\sigma(f_w)
 +G_k(v,w)\cdot\nabla_v\sigma(f)\,\sigma(f_w)
\Big]\dd w .
\end{aligned}
\end{equation}
After summing in $k$ and using the definitions of $F_1$ and $F_4$ in
\eqref{def:F_i}, the correction tested against a smooth $\phi$ becomes
\begin{align*}
&-\frac{\eps}{2}\int_{v,v_*,w}
\tn\phi(v,v_*)\cdot
\sigma(f_*)\sigma(f_w)
\Big(F_4(v,v_*,w)\nabla_v\varsigma(f)
\,+\,
\sigma(f)\sigma'(f)F_1(v,v_*,w)\Big)
\dd v\dd v_*\dd w,
\end{align*}
where $\varsigma'(r)=[\sigma'(r)]^2$.  This is the deterministic correction
term in Definition~\ref{def:weak-sol:L1} and in the approximation equations
used throughout the proof.

\section{Construction of anti-symmetry bases}\label{app-sec:ONB}

In this appendix, we present a construction of an ONB of
\begin{align*}
 L^2_{\anti}(\G;\R^d):=\{f\in L^2(\G;\R^d)\mid f(v,v_*)=-f(v_*,v)\}.   
\end{align*}

Let $(e_1,\dots,e_d)$ denote the standard orthonormal basis of $\R^d$. Then we only need to construct an ONB $\{\psi_n\}_{n\ge 1}$ for $L^2_{\anti}(\G;\R)$. A desired ONB $\{g_k\}_{k\ge 1}$ of $L^2_{\anti}(\G;\R^d)$ is then given by
\begin{align*}
    g_{(n-1)d+i}(v,v_*)=\psi_n(v,v_*)e_i,
    \qquad n\ge1,\quad 1\le i\le d.
\end{align*}

The rest of the section is devoted to the construction of $\{\psi_n\}_{n\ge 1}$.
We use the following mass-centred coordinate
\begin{equation*}
    z=\frac{v+v_*}{2},\qquad x=v-v_*,\qquad
    r=|x|,\qquad \omega=\frac{x}{|x|}\in \sd .
\end{equation*}
We construct $\psi_n$ of the following form
\begin{equation}
    \label{ONB:L-2-anti}
    \begin{aligned}
&\psi_n(v,v_*):=\zeta_{p(n)}(z)\eta_{q(n)}(r)Y_{l(n),m(n)}(\omega),
\end{aligned}
\end{equation}
where $\{\zeta_p\}_{p\ge1}$ and $\{\eta_q\}_{q\ge1}$ are ONBs of
$L^2(\R^d;\dd z)$ and $L^2(\R_+;r^{d-1}\dd r)$, respectively, and
$\{Y_{l,m}:l\ge0,\ 1\le m\le N(d,l)\}$ is the usual ONB of spherical
harmonics in $L^2(\sd)$.  Here $n\mapsto (p(n),q(n),l(n),m(n))$ is any enumeration of all indices with $l(n)$ odd and $1\le m(n)\le N(d,l(n))$.

Notice that $\zeta_{p(n)}(z)$ and $\eta_{q(n)}(r)$ are symmetric in $(v,v_*)$. It remains to choose spherical harmonics satisfying
$$Y_{l(n),m(n)}(\omega)=-Y_{l(n),m(n)}(-\omega)\quad\text{for all $n\in\N$ and $\omega\in\sd$}.$$ 
Let $l\in \N$. Let $Y_{l,m}$ be the $l$-eigenfunction of the Laplace--Beltrami operator such that
\begin{align*}
    -\Delta_{\sd} Y_{l,m}=l(l+d-2) Y_{l,m},
\end{align*}
where $m=1,\dots,N(d,l)$ labels the $N(d,l)$ spherical harmonics with degree $l$.
Equivalently, $Y_{l,m}$ is the restriction to $\sd$ of an $l$-homogeneous harmonic polynomial. 
As a consequence, we have  
$$Y_{l,m}(- \omega)=(-1)^lY_{l,m}(\omega),$$
and hence $Y_{l,m}(-\omega)=-Y_{l,m}(\omega)$ whenever $l$ is odd. 

Hence, \eqref{ONB:L-2-anti}, with all odd degrees $l$ included in the enumeration of the indices, gives the desired scalar basis functions. 

The above construction gives an ONB of the antisymmetric subspace.  A full ONB of $L^2(\G;\R^d)=L^2_{\anti}(\G;\R^d)\oplus L^2_{\sym}(\G;\R^d)$ is obtained by adjoining any ONB of the symmetric orthogonal complement.

\section{A family of divergence free bases}\label{app-sec:ONB-2}
In this appendix, we present a construction of an ONB $\{g_k\}_{k\in \N}$ of 
\begin{align*}
L^2_{\anti,\tan}(\G;\R^d):=\big\{ f\in L^2_{\anti}(\G;\R^d)\mid (\nabla_v-\nabla_{v_*})\cdot \Pi_{(v-v_*)^\perp}f=0\text{ weakly} \big\}.
\end{align*}

We use the following mass-centred coordinate
\begin{equation*}
    z=\frac{v+v_*}{2},\qquad x=v-v_*,\qquad
    r=|x|,\qquad \omega=\frac{x}{|x|}\in \sd .
\end{equation*}
We note that
\begin{equation}
\label{app:onb-grad-rel}
    \nabla_v-\nabla_{v_*}=2\nabla_x \quad\text{and}\quad |x| \Pi_{x^\perp}\nabla_x =\nabla_{\sd}.
\end{equation}
This implies that 
\begin{equation}
    \label{eq-ONB}
\begin{aligned}
   &(\nabla_v-\nabla_{v_*}) \cdot \Pi_{(v-v_*)^\perp}g_k(v,v_*)
   =\frac{2}{r}\nabla_{\sd}\cdot g_k(z,r,\omega).
\end{aligned}
\end{equation}
We shall construct $g_k$ so that the right-hand side of \eqref{eq-ONB}
vanishes.

Similarly to Section \ref{app-sec:ONB}, we take $g_k$ of the following form 
\begin{equation}
    \label{ONB:L-2-anti-div}
    \begin{aligned}
&g_k(v,v_*):=\zeta_{p(k)}(z)\eta_{q(k)}(r)T_{l(k),m(k)}(\omega),
\end{aligned}
\end{equation}
where $\zeta_{p(n)}(z)$ and $\eta_{q(n)}(r)$ are ONBs of $L^2(\R^d;\dd z)$ and  $L^2(\R_+;r^{d-1}\dd r)$, respectively. In particular, we take $T_{l(k),m(k)}(\omega)\in\R^d$ as an ONB of the tangential divergence-free subspace
\begin{align*}
    L^2_{\tan,\div}(\sd;T\sd):=\big\{\hg\in L^2(\sd;\R^d)\mid
    \hg\cdot \omega=0,\ \nabla_{\sd}\cdot \hg=0\big\}.
\end{align*}

In dimension $d=2$, we  parametrise $\omega\in\mathbb S^1$ by
\begin{equation*}
\omega=\omega(\theta)=(\cos\theta,\sin\theta)^T,
\qquad \theta\in[0,2\pi).
\end{equation*}
Let ${Y_{l,m}}$ denote the real normalised scalar spherical harmonics on $\mathbb S^1$. More explicitly, $Y_0(\theta)=\frac{1}{\sqrt{2\pi}}$ and, for $l\ge 1$,
\begin{align*}
Y_{l,1}(\theta)
=
\frac{1}{\sqrt{\pi}}\cos(l\theta),\quad 
Y_{l,2}(\theta)
=
\frac{1}{\sqrt{\pi}}\sin(l\theta).
\end{align*}
For $l\ge 1$, define
\begin{equation*}
\label{app:toroidal-harmonics-2}
T_{l,m}(\omega)
:=
\frac{\mathsf R\nabla_{\mathbb S^1}Y_{l,m}(\omega)}{l},\quad \mathsf{R}=\begin{pmatrix}
        0&1\\-1&0
    \end{pmatrix},
\qquad m=1,2.
\end{equation*}
More precisely, for $l\ge 1$, we have 
\begin{align*}
T_{l,1}(\omega)
=
-\frac{1}{\sqrt{\pi}}\sin(l\theta),\omega,\quad 
T_{l,2}(\omega)
=
\frac{1}{\sqrt{\pi}}\cos(l\theta),\omega.
\end{align*}
Moreover, these modes are normalised in $L^2(\mathbb S^1;\mathbb R^2)$:
\begin{equation*}
\int_{\mathbb S^1}
T_{l,m}(\omega)\cdot T_{l',m'}(\omega)\dd \omega=\delta_{ll'}\delta_{mm'}.
\end{equation*}

 In dimension $d=3$, we recall $Y_{l,m}$ denoting the $l$-eigenfunction of $-\Delta_{\mathbb S^2}$, with $l\ge1$. Define the toroidal vector spherical harmonics
\begin{equation}
\label{app:toroidal-harmonics}
    T_{l,m}(\omega)
    :=\frac{\omega\times\nabla_{\mathbb S^2}Y_{l,m}(\omega)}
    {\sqrt{l(l+1)}} .
\end{equation}
Then $T_{l,m}$ is tangential, orthonormal in $L^2(\mathbb S^2;\R^3)$,
and divergence-free on the sphere.  Indeed,
\begin{equation*}
    \omega\cdot T_{l,m}=0\quad\text{and}\quad
    \nabla_{\mathbb S^2}\cdot T_{l,m}=0,
\end{equation*}
where the second identity being the standard fact that the surface curl of a scalar
field is surface-divergence free.  Orthogonality follows from integration by
parts:
\begin{align*}
\int_{\mathbb S^2}T_{l,m}\cdot T_{\ell' m'}\,\dd\omega
&=
\frac1{\sqrt{l(l+1)\ell'(\ell'+1)}}
\int_{\mathbb S^2}
\nabla_{\mathbb S^2}Y_{l,m}\cdot
\nabla_{\mathbb S^2}Y_{\ell' m'}\,\dd\omega\\
&=\delta_{\ell\ell'}\delta_{mm'}.
\end{align*}
We take $l$ odd to ensure the antisymmetry as in Section \ref{app-sec:ONB}.
Indeed, the toroidal harmonics have the same antipodal parity as the scalar
harmonics, namely $T_{l,m}(-\omega)=(-1)^lT_{l,m}(\omega)$.

 In dimensions $d\ge 2$ one uses
the analogous co-closed vector spherical harmonics supplied by the Hodge
decomposition on $\sd$. Since the Hodge Laplacian and the antipodal map
commute, these co-closed eigenspaces can be decomposed into even and odd
parts; selecting the odd part gives tangential, spherical-divergence-free
modes with the required antisymmetry.


\printbibliography

@article{rezakhanlou1998large,
  title={Large deviations from a kinetic limit},
  author={Rezakhanlou, Fraydoun},
  journal={Annals of probability},
  pages={1259--1340},
  year={1998},
  publisher={JSTOR}
}

@article{bodineau2023statistical,
  title={Statistical dynamics of a hard sphere gas: fluctuating Boltzmann equation and large deviations},
  author={Bodineau, Thierry and Gallagher, Isabelle and Saint-Raymond, Laure and Sergio, Simonella},
  journal={Annals of Mathematics},
  volume={198},
  number={3},
  pages={1047--1201},
  year={2023},
  publisher={Department of Mathematics, Princeton University Princeton, New Jersey, USA}
}

@article{CDDW24,
 author = {Carrillo, Jos{\'e} A. and Delgadino, Matias G. and Desvillettes, Laurent and Wu, Jeremy S.-H.},
 title = {The {Landau} equation as a gradient flow},
 fjournal = {Analysis \& PDE},
 journal = {Anal. PDE},
 issn = {2157-5045},
 volume = {17},
 number = {4},
 pages = {1331--1375},
 year = {2024},
 language = {English},
 doi = {10.2140/apde.2024.17.1331},
 zbMATH = {7852592},
 Zbl = {1542.35382}
}

@article{PrigogineBalescu1957,
  author  = {I. Prigogine and R. Balescu},
  title   = {Irreversible Processes in Gases. I. The Diagram Technique},
  journal = {Physica},
  volume  = {23},
  number  = {1--5},
  pages   = {28--42},
  year    = {1957},
  doi     = {10.1016/S0031-8914(57)94080-8}
}

@article{PrigogineBalescu1957b,
  author  = {I. Prigogine and R. Balescu},
  title   = {Irreversible Processes in Gases. II. The Equations of Evolution},
  journal = {Physica},
  volume  = {23},
  number  = {1--5},
  pages   = {225--242},
  year    = {1957},
  doi     = {10.1016/S0031-8914(57)94095-X}
}

@article{KiesslingLancellotti2004,
  author  = {Michael K.-H. Kiessling and Carlo Lancellotti},
  title   = {On the Master-Equation Approach to Kinetic Theory: Linear and Nonlinear Fokker--Planck Equations},
  journal = {Transport Theory and Statistical Physics},
  volume  = {33},
  number  = {5--7},
  pages   = {379--401},
  year    = {2004},
  doi     = {10.1081/TT-200053929}
}

@misc{feng2025kacsprogramlandauequation,
      title={Kac's Program for the Landau Equation}, 
      author={Xuanrui Feng and Zhenfu Wang},
      year={2025},
      eprint={2506.14309},
      archivePrefix={arXiv},
      primaryClass={math.AP},
      url={https://arxiv.org/abs/2506.14309}, 
}

@book {LL87,
    AUTHOR = {Landau, Lew D. and Lifshitz, Evgeny M.},
     TITLE = {Course of theoretical physics. {V}ol. 6},
   EDITION = {Second},
      NOTE = {Fluid mechanics,
              Translated from the third Russian edition by J. B. Sykes and
              W. H. Reid},
 PUBLISHER = {Pergamon Press, Oxford},
      YEAR = {1987},
     PAGES = {xiv+539},
      ISBN = {0-08-033933-6},
   MRCLASS = {00A05 (00-01 76-01)},
  MRNUMBER = {961259},
}

@misc{GHW23,
      title={Landau--Lifshitz--Navier--Stokes Equations: Large Deviations and Relationship to the Energy Equality},
      author={Benjamin Gess and Daniel Heydecker and Zhengyan Wu},
      year={2024},
      eprint={2311.02223},
      archivePrefix={arXiv},
      primaryClass={math.PR}
}

@article {FG23,
    AUTHOR = {Fehrman, Benjamin and Gess, Benjamin},
     TITLE = {Non-equilibrium large deviations and parabolic-hyperbolic
              {PDE} with irregular drift},
   JOURNAL = {Invent. Math.},
  FJOURNAL = {Inventiones Mathematicae},
    VOLUME = {234},
      YEAR = {2023},
    NUMBER = {2},
     PAGES = {573--636},
      ISSN = {0020-9910},
   MRCLASS = {35Q84 (37H05 60F10 60H15 60K35 82B21)},
  MRNUMBER = {4651008},
       DOI = {10.1007/s00222-023-01207-3},
       URL = {https://doi.org/10.1007/s00222-023-01207-3},
}

@article {FG24,
    AUTHOR = {Fehrman, Benjamin and Gess, Benjamin},
     TITLE = {Well-{P}osedness of the {D}ean--{K}awasaki and the {N}onlinear
              {D}awson--{W}atanabe {E}quation with {C}orrelated {N}oise},
   JOURNAL = {Arch. Ration. Mech. Anal.},
  FJOURNAL = {Archive for Rational Mechanics and Analysis},
    VOLUME = {248},
      YEAR = {2024},
    NUMBER = {2},
     PAGES = {Paper No. 20},
      ISSN = {0003-9527},
   MRCLASS = {35R60 (35B09 35K40 60H15)},
  MRNUMBER = {4716244},
       DOI = {10.1007/s00205-024-01963-3},
       URL = {https://doi.org/10.1007/s00205-024-01963-3},
}

@article{DareiotisGessGnannGruen2021,
  author  = {Konstantinos Dareiotis and Benjamin Gess and
             Manuel V. Gnann and G{\"u}nther Gr{\"u}n},
  title   = {Non-negative Martingale Solutions to the Stochastic Thin-Film Equation with Nonlinear Gradient Noise},
  journal = {Archive for Rational Mechanics and Analysis},
  volume  = {242},
  number  = {1},
  pages   = {179--234},
  year    = {2021},
  doi     = {10.1007/s00205-021-01682-z}
}

@article{DareiotisGessGnannSauerbrey2026,
  author  = {Konstantinos Dareiotis and Benjamin Gess and
             Manuel V. Gnann and Max Sauerbrey},
  title   = {Solutions to the stochastic thin-film equation for initial values with non-full support},
  journal = {Transactions of the American Mathematical Society},
  volume  = {379},
  pages   = {2343--2383},
  year    = {2026}
}

@article{KawasakiOhta1982,
  author  = {Kyozi Kawasaki and Takao Ohta},
  title   = {Kinetic Drumhead Model of Interface. I},
  journal = {Progress of Theoretical Physics},
  volume  = {67},
  number  = {1},
  pages   = {147--163},
  year    = {1982},
  month   = jan,
  doi     = {10.1143/PTP.67.147}
}

@article{KatsoulakisKho2001,
  author  = {Markos A. Katsoulakis and Alvin T. Kho},
  title   = {Stochastic Curvature Flows: Asymptotic Derivation, Level Set Formulation and Numerical Experiments},
  journal = {Interfaces and Free Boundaries},
  volume  = {3},
  number  = {3},
  pages   = {265--290},
  year    = {2001},
  doi     = {10.4171/IFB/41}
}

@misc{ayala2025reversibilitycovariancecoarsegraininglangevin,
      title={Reversibility, covariance and coarse-graining for Langevin dynamics: On the choice of multiplicative noise}, 
      author={Mario Ayala and Nicolas Dirr and Grigorios A. Pavliotis and Johannes Zimmer},
      year={2025},
      eprint={2511.03347},
      archivePrefix={arXiv},
      primaryClass={math.PR},
      url={https://arxiv.org/abs/2511.03347}, 
}

@book{Ottinger, place={}, edition={}, series={}, title={Beyond equilibrium thermodynamics}, DOI={}, publisher={John Wiley and Sons}, author={Hans Christian \"{O}ttinger}, year={2005}, collection={}}

@article{Des15,
 author = {Desvillettes, L.},
 title = {Entropy dissipation estimates for the {Landau} equation in the {Coulomb} case and applications},
 fjournal = {Journal of Functional Analysis},
 journal = {J. Funct. Anal.},
 issn = {0022-1236},
 volume = {269},
 number = {5},
 pages = {1359--1403},
 year = {2015},
 language = {English},
 doi = {10.1016/j.jfa.2015.05.009},

}

@article{BH77,
	author = {Braun, W. and Hepp, K.},
	fjournal = {Communications in Mathematical Physics},
	issn = {0010-3616},
	journal = {Comm. Math. Phys.},
	mrclass = {82.60},
	mrnumber = {475547},
	mrreviewer = {H. Wakita},
	number = {2},
	pages = {101--113},
	title = {The {V}lasov dynamics and its fluctuations in the {$1/N$} limit of interacting classical particles},
	url = {http://projecteuclid.org/euclid.cmp/1103901139},
	volume = {56},
	year = {1977}}

@article{Des92,
 author = {Desvillettes, L.},
 title = {On asymptotics of the {Boltzmann} equation when the collisions become grazing},
 fjournal = {Transport Theory and Statistical Physics},
 journal = {Transp. Theory Stat. Phys.},
 issn = {0041-1450},
 volume = {21},
 number = {3},
 pages = {259--276},
 year = {1992},
 language = {English},
 doi = {10.1080/00411459208203923},

}

@Article{Vil98,
 Author = {Villani, C.},
 Title = {On the spatially homogeneous {Landau} equation for {Maxwellian} molecules},
 FJournal = {$M^3$AS. Mathematical Models \& Methods in Applied Sciences},
 Journal = {Math. Models Methods Appl. Sci.},
 ISSN = {0218-2025},
 Volume = {8},
 Number = {6},
 Pages = {957--983},
 Year = {1998},

}

@article{AV04,
 author = {Alexandre, R. and Villani, C.},
 title = {On the {Landau} approximation in plasma physics.},
 fjournal = {Annales de l'Institut Henri Poincar{\'e}. Analyse Non Lin{\'e}aire},
 journal = {Ann. Inst. Henri Poincar{\'e}, Anal. Non Lin{\'e}aire},
 issn = {0294-1449},
 volume = {21},
 number = {1},
 pages = {61--95},
 year = {2004},
 language = {English},
 doi = {10.1016/j.anihpc.2002.12.001},

}

@article{carrillo2022boltzmann,
  title={Boltzmann to Landau from the gradient flow perspective},
  author={Carrillo, Jose A and Delgadino, Matias G and Wu, Jeremy},
  journal={Nonlinear Analysis},
  volume={219},
  pages={112824},
  year={2022},
  publisher={Elsevier}
}

@article {BDGJL,
    AUTHOR = {Bertini, Lorenzo and De Sole, Alberto and Gabrielli, Davide
              and Jona-Lasinio, Giovanni and Landim, Claudio},
     TITLE = {Macroscopic fluctuation theory},
   JOURNAL = {Rev. Modern Phys.},
  FJOURNAL = {Reviews of Modern Physics},
    VOLUME = {87},
      YEAR = {2015},
    NUMBER = {2},
     PAGES = {593--636},
      ISSN = {0034-6861},
   MRCLASS = {82B05 (80A10)},
  MRNUMBER = {3403268},
       DOI = {10.1103/RevModPhys.87.593},
       URL = {https://doi.org/10.1103/RevModPhys.87.593},
}

@article {Derrida,
    AUTHOR = {Derrida, Bernard},
     TITLE = {Non-equilibrium steady states: fluctuations and large
              deviations of the density and of the current},
   JOURNAL = {J. Stat. Mech. Theory Exp.},
  FJOURNAL = {Journal of Statistical Mechanics: Theory and Experiment},
      YEAR = {2007},
    NUMBER = {7},
     PAGES = {P07023, 45},
   MRCLASS = {82C05 (82-02 82C22)},
  MRNUMBER = {2335699},
MRREVIEWER = {Davide Gabrielli},
       DOI = {10.1088/1742-5468/2007/07/p07023},
       URL = {https://doi.org/10.1088/1742-5468/2007/07/p07023},
}

@book {HS,
    AUTHOR = {Herbert Spohn},
     TITLE = {Large scale dynamics of interacting particles},
 PUBLISHER = {Springer Science \&  Business Media},
      YEAR = {2012},
}

@article {D96,
    AUTHOR = {Dean, David S.},
     TITLE = {Langevin equation for the density of a system of interacting
              {L}angevin processes},
   JOURNAL = {J. Phys. A},
  FJOURNAL = {Journal of Physics. A. Mathematical and General},
    VOLUME = {29},
      YEAR = {1996},
    NUMBER = {24},
     PAGES = {L613--L617},
      ISSN = {0305-4470},
   MRCLASS = {82C31},
  MRNUMBER = {1446882},
       DOI = {10.1088/0305-4470/29/24/001},
       URL = {https://doi.org/10.1088/0305-4470/29/24/001},
}

@article{K98,
	author = {Kawasaki, Kyozi},
	doi = {10.1023/B:JOSS.0000033240.66359.6c},
	fjournal = {Journal of Statistical Physics},
	issn = {0022-4715},
	journal = {J. Statist. Phys.},
	mrclass = {82C05 (82D15)},
	mrnumber = {1666565},
	mrreviewer = {Richard K. Jordan},
	number = {3-4},
	pages = {527--546},
	title = {Microscopic analyses of the dynamical density functional equation of dense fluids},
	url = {https://doi.org/10.1023/B:JOSS.0000033240.66359.6c},
	volume = {93},
	year = {1998}}

@Article{Lio94c,
 Author = {Lions, P. L.},
 Title = {On {Boltzmann} and {Landau} equations},
 FJournal = {Philosophical Transactions of the Royal Society of London. Series A},
 Journal = {Philos. Trans. R. Soc. Lond., Ser. A},
 ISSN = {0962-8428},
 Volume = {346},
 Number = {1679},
 Pages = {191--204},
 Year = {1994},

}

@Article{Vil96,
 Author = {Villani, C{\'e}dric},
 Title = {On the {Cauchy} problem for {Landau} equations: {Sequential} stability, global existence},
 FJournal = {Advances in Differential Equations},
 Journal = {Adv. Differ. Equ.},
 ISSN = {1079-9389},
 Volume = {1},
 Number = {5},
 Pages = {793--816},
 Year = {1996},

}

@article{ADVW00,
 author = {Alexandre, R. and Desvillettes, L. and Villani, C. and Wennberg, B.},
 title = {Entropy dissipation and long-range interactions},
 fjournal = {Archive for Rational Mechanics and Analysis},
 journal = {Arch. Ration. Mech. Anal.},
 issn = {0003-9527},
 volume = {152},
 number = {4},
 pages = {327--355},
 year = {2000},

}

@Article{Wu14,
 Author = {Wu, Kung-Chien},
 Title = {Global in time estimates for the spatially homogeneous {Landau} equation with soft potentials},
 FJournal = {Journal of Functional Analysis},
 Journal = {J. Funct. Anal.},
 ISSN = {0022-1236},
 Volume = {266},
 Number = {5},
 Pages = {3134--3155},
 Year = {2014},
}

@misc{GS25,
 author = {Guillen, Nestor and Silvestre, Luis E.},
 title = {The {Landau} equation and {Fisher} information},
 year = {2025},
 howpublished = {Preprint, {arXiv}:2507.05167 [math.{AP}] (2025)},
 url = {https://arxiv.org/abs/2507.05167},
 arXiv = {arXiv:2507.05167}
}

@article{GS25b,
 author = {Guillen, Nestor and Silvestre, Luis},
 title = {The {Landau} equation does not blow up},
 fjournal = {Acta Mathematica},
 journal = {Acta Math.},
 issn = {0001-5962},
 volume = {234},
 number = {2},
 pages = {315--375},
 year = {2025},
 language = {English},
 doi = {10.4310/ACTA.2025.v234.n2.a2},
 zbMATH = {8061615}
}

@misc{Vil25,
      title={Fisher Information in Kinetic Theory}, 
      author={C{\'e}dric Villani},
      year={2025},
      eprint={2501.00925},
      archivePrefix={arXiv},
      primaryClass={math.AP},
      url={https://arxiv.org/abs/2501.00925}, 
}

@misc{DH25a,
    title={On a fuzzy Landau Equation: Part I. A variational approach},
    author={Manh Hong Duong and Zihui He},
    year={2025},
    eprint={2504.07666},
    archivePrefix={arXiv},
    primaryClass={math.AP}
}

@misc{DH25b,
 author = {Duong, Manh Hong and He, Zihui},
 title = {On a fuzzy {Landau} {Equation}: {Part} {II}. {Solvability} results},
 year = {2025},
 howpublished = {Preprint, {arXiv}:2507.10288 [math.{AP}] (2025)},
 url = {https://arxiv.org/abs/2507.10288},
 arXiv = {arXiv:2507.10288}
}

@misc{DGH25,
      title={On a fuzzy Landau Equation: Part III. The grazing collision limit}, 
      author={Manh Hong Duong and Boris Golubkov and Zihui He},
      year={2025},
      eprint={2512.04713},
      archivePrefix={arXiv},
      primaryClass={math.AP},
      url={https://arxiv.org/abs/2512.04713}, 
}

@article {DV10,
    AUTHOR = {Debussche, Arnaud and Vovelle, Julien},
     TITLE = {Scalar conservation laws with stochastic forcing},
   JOURNAL = {J. Funct. Anal.},
  FJOURNAL = {Journal of Functional Analysis},
    VOLUME = {259},
      YEAR = {2010},
    NUMBER = {4},
     PAGES = {1014--1042},
      ISSN = {0022-1236},
   MRCLASS = {60H15 (35L65 35R60)},
  MRNUMBER = {2652180},
MRREVIEWER = {Marko Nedeljkov},
       DOI = {10.1016/j.jfa.2010.02.016},
       URL = {https://doi.org/10.1016/j.jfa.2010.02.016},
}

@article {GS17,
    AUTHOR = {Gess, Benjamin and Souganidis, Panagiotis E.},
     TITLE = {Long-time behavior, invariant measures, and regularizing
              effects for stochastic scalar conservation laws},
   JOURNAL = {Comm. Pure Appl. Math.},
  FJOURNAL = {Communications on Pure and Applied Mathematics},
    VOLUME = {70},
      YEAR = {2017},
    NUMBER = {8},
     PAGES = {1562--1597},
      ISSN = {0010-3640,1097-0312},
   MRCLASS = {35R60 (35B40 35L65 60H15 76M35)},
  MRNUMBER = {3666564},
MRREVIEWER = {Mark\ C.\ Veraar},
       DOI = {10.1002/cpa.21646},
       URL = {https://doi.org/10.1002/cpa.21646},
}

@article {FG19,
    AUTHOR = {Fehrman, Benjamin and Gess, Benjamin},
     TITLE = {Well-posedness of nonlinear diffusion equations with
              nonlinear, conservative noise},
   JOURNAL = {Arch. Ration. Mech. Anal.},
  FJOURNAL = {Archive for Rational Mechanics and Analysis},
    VOLUME = {233},
      YEAR = {2019},
    NUMBER = {1},
     PAGES = {249--322},
      ISSN = {0003-9527},
   MRCLASS = {35R60 (35B30 35K59 76S05)},
  MRNUMBER = {3974641},
MRREVIEWER = {Alp O. Eden},
       DOI = {10.1007/s00205-019-01357-w},
       URL = {https://doi.org/10.1007/s00205-019-01357-w},
}

@article{DG20,
	author = {Dareiotis, Konstantinos and Gess, Benjamin},
	doi = {10.1214/20-ejp436},
	fjournal = {Electronic Journal of Probability},
	journal = {Electron. J. Probab.},
	mrclass = {60H15 (35K59 35K65 35R60)},
	mrnumber = {4089785},
	mrreviewer = {Dominic Breit},
	pages = {Paper No. 35, 43},
	title = {Nonlinear diffusion equations with nonlinear gradient noise},
	url = {https://doi.org/10.1214/20-ejp436},
	volume = {25},
	year = {2020}}

@article{FG25,
  author  = {Benjamin J. Fehrman and Benjamin Gess},
  title   = {Conservative stochastic PDEs on the whole space},
  journal = {Stochastics and Partial Differential Equations: Analysis and Computations},
  year    = {2025},
  publisher = {Springer},
  doi     = {10.1007/s40072-025-00369-w},
  url     = {https://doi.org/10.1007/s40072-025-00369-w}
}

@misc{fehrman2025stochastic,
  title={Stochastic PDEs with Correlated, Non-stationary Stratonovich Noise of Dean--Kawasaki Type},
  author={Benjamin Fehrman},
  year={2025},
  eprint={2504.18370},
  archivePrefix={arXiv},
  primaryClass={math.PR}
}

@article{CF23,
      title={A Central Limit Theorem for Nonlinear Conservative {SPDE}s},
      author={Andrea Clini and Benjamin Fehrman},
      journal={Stochastics and Partial Differential Equations: Analysis and Computations},
      volume={13},
      pages={1407--1450},
      year={2025},
      doi={10.1007/s40072-025-00359-y},
      url={https://doi.org/10.1007/s40072-025-00359-y}
}

@article{GWZ24,
  title   = {Higher order fluctuation expansions for nonlinear stochastic heat equations in singular limits},
  author  = {Gess, Benjamin and Wu, Zhengyan and Zhang, Rangrang},
  journal = {Stochastic Processes and their Applications},
  year    = {2025},
  volume  = {193},
  pages   = {104847},
  doi     = {10.1016/j.spa.2025.104847},
  url     = {https://www.sciencedirect.com/science/article/pii/S0304414925002911}
}

@article{WWZ22,
  author  = {Likun Wang and Zhengyan Wu and Rangrang Zhang},
  title   = {Well-Posedness of Dean--Kawasaki Equation with Singular Interactions},
  journal = {SIAM Journal on Mathematical Analysis},
  volume  = {58},
  number  = {3},
  pages   = {2738--2785},
  year    = {2026},
  doi     = {10.1137/25M1743818},
  publisher = {Society for Industrial and Applied Mathematics (SIAM)}
}

@misc{WZ24,
      title={McKean--Vlasov {PDE} with Irregular Drift and Applications to Large Deviations for Conservative {SPDE}s},
      author={Zhengyan Wu and Rangrang Zhang},
      year={2024},
      eprint={2208.13142},
      archivePrefix={arXiv},
      primaryClass={math.PR}
}

@article {FMJ25,
    AUTHOR = {Fenna M{\"u}ller and Max von Renesse and Johannes Zimmer},
     TITLE = {Well-Posedness for Dean-Kawasaki Models of Vlasov-Fokker-Planck Type},
   JOURNAL = {Proc. R. Soc. A},
    VOLUME = {481},
      YEAR = {2025},
     PAGES = {20250089},
       URL = {https://doi.org/10.1098/rspa.2025.0089},
}

@misc{HWZ25,
      title={Kinetic Theory with Fluctuations: Strong Well-Posedness of the Vlasov--Fokker--Planck--Dean--Kawasaki System},
      author={Zimo Hao and Zhengyan Wu and Johannes Zimmer},
      year={2025},
      eprint={2511.10194},
      archivePrefix={arXiv},
      primaryClass={math.PR}
}

@article {Krylov,
    AUTHOR = {Krylov, N. V.},
     TITLE = {A relatively short proof of {I}t\^o's formula for {SPDE}s and
              its applications},
   JOURNAL = {Stoch. Partial Differ. Equ. Anal. Comput.},
  FJOURNAL = {Stochastic Partial Differential Equations. Analysis and
              Computations},
    VOLUME = {1},
      YEAR = {2013},
    NUMBER = {1},
     PAGES = {152--174},
      ISSN = {2194-0401,2194-041X},
   MRCLASS = {60H15 (35B50 35R60)},
  MRNUMBER = {3327504},
MRREVIEWER = {Guangying\ Lv},
       DOI = {10.1007/s40072-013-0003-5},
       URL = {https://doi.org/10.1007/s40072-013-0003-5},
}

@article {JKO98,
    AUTHOR = {Jordan, Richard and Kinderlehrer, David and Otto, Felix},
     TITLE = {The variational formulation of the {F}okker-{P}lanck equation},
   JOURNAL = {SIAM J. Math. Anal.},
  FJOURNAL = {SIAM Journal on Mathematical Analysis},
    VOLUME = {29},
      YEAR = {1998},
    NUMBER = {1},
     PAGES = {1--17},
      ISSN = {0036-1410},
   MRCLASS = {35Q99 (35A15 49J99 60J60 82C31)},
  MRNUMBER = {1617171},
MRREVIEWER = {Thierry Goudon},
       DOI = {10.1137/S0036141096303359},
       URL = {https://doi.org/10.1137/S0036141096303359},
}

@article{Bec08,
 author = {Beckner, William},
 title = {Pitt's inequality with sharp convolution estimates},
 fjournal = {Proceedings of the American Mathematical Society},
 journal = {Proc. Am. Math. Soc.},
 issn = {0002-9939},
 volume = {136},
 number = {5},
 pages = {1871--1885},
 year = {2008},
 
}

@article{Bec08b,
 author = {Beckner, William},
 title = {Weighted inequalities and {Stein}-{Weiss} potentials},
 fjournal = {Forum Mathematicum},
 journal = {Forum Math.},
 issn = {0933-7741},
 volume = {20},
 number = {4},
 pages = {587--606},
 year = {2008},
 
 
}

@article{Bec12,
 author = {Beckner, William},
 title = {Pitt's inequality and the fractional {Laplacian}: {Sharp} error estimates},
 fjournal = {Forum Mathematicum},
 journal = {Forum Math.},
 issn = {0933-7741},
 volume = {24},
 number = {1},
 pages = {177--209},
 year = {2012},

}

@book{FL07,
 author = {Fonseca, Irene and Leoni, Giovanni},
 title = {Modern methods in the calculus of variations. $L^p$ spaces},
 fseries = {Springer Monographs in Mathematics},
 series = {Springer Monogr. Math.},
 issn = {1439-7382},
 isbn = {978-0-387-35784-3},
 year = {2007},
 publisher = {New York, NY: Springer},
 language = {English},
 zbMATH = {5114899},
 Zbl = {1153.49001}
}

@article{Fournier2009,
  author  = {Fournier, Nicolas},
  title   = {Particle approximation of some {Landau} equations},
  journal = {Kinetic and Related Models},
  volume  = {2},
  number  = {3},
  pages   = {451--464},
  year    = {2009},
  doi     = {10.3934/krm.2009.2.451}
}

@article{FournierHauray2016,
  author  = {Fournier, Nicolas and Hauray, Maxime},
  title   = {Propagation of chaos for the {Landau} equation with moderately soft potentials},
  journal = {The Annals of Probability},
  volume  = {44},
  number  = {6},
  pages   = {3581--3660},
  year    = {2016},
  doi     = {10.1214/15-AOP1056}
}

@article{FournierGuillin2017,
  author  = {Fournier, Nicolas and Guillin, Arnaud},
  title   = {From a {Kac-like} particle system to the {Landau} equation for hard potentials and {Maxwell} molecules},
  journal = {Annales scientifiques de l'\'{E}cole Normale Sup\'{e}rieure},
  series  = {4},
  volume  = {50},
  number  = {1},
  pages   = {157--199},
  year    = {2017},
  doi     = {10.24033/asens.2318}
}

@article{MischlerMouhot2013,
  author  = {Mischler, St\'{e}phane and Mouhot, Cl\'{e}ment},
  title   = {{Kac}'s program in kinetic theory},
  journal = {Inventiones Mathematicae},
  volume  = {193},
  number  = {1},
  pages   = {1--147},
  year    = {2013},
  doi     = {10.1007/s00222-012-0422-3}
}

@incollection{Villani2002,
  author    = {Villani, C\'{e}dric},
  title     = {A review of mathematical topics in collisional kinetic theory},
  booktitle = {Handbook of Mathematical Fluid Dynamics},
  volume    = {1},
  pages     = {71--305},
  publisher = {North-Holland},
  address   = {Amsterdam},
  year      = {2002},
  doi       = {10.1016/S1874-5792(02)80004-0}
}

@article{Hairer2014,
  author  = {Hairer, Martin},
  title   = {A theory of regularity structures},
  journal = {Inventiones Mathematicae},
  volume  = {198},
  number  = {2},
  pages   = {269--504},
  year    = {2014},
  doi     = {10.1007/s00222-014-0505-4}
}

\end{document}